\pdfoutput=1
\documentclass[a4paper,12pt]{article} 
\usepackage[utf8]{inputenc}
\usepackage[T1]{fontenc}
\usepackage{amsmath,amssymb,amsthm}
\usepackage[]{authblk}
\usepackage{bbm}
\usepackage{mathrsfs}

\usepackage{verbatim}
\usepackage{mathtools}
\usepackage{mathabx}
\usepackage{bbm}
\usepackage{floatrow}
\usepackage{array}
\usepackage{multirow}
\usepackage{diagbox}
\usepackage{caption}
\usepackage{subcaption}
\usepackage[numeric,initials,nobysame,msc-links,abbrev]{amsrefs}
\renewcommand{\eprint}[1]{\href{https://arxiv.org/abs/#1}{arXiv:#1}}
\newcommand{\pageafter}[1]{#1~pp.}
\BibSpec{article}{%
+{} {\PrintAuthors} {author}
+{,} { \textit} {title}
+{.} { } {part}
+{:} { \textit} {subtitle}
+{,} { \PrintContributions} {contribution}
+{.} { \PrintPartials} {partial}
+{,} { } {journal}
+{} { \textbf} {volume}
+{} { \PrintDatePV} {date}
+{,} { \issuetext} {number}
+{,} { \pageafter} {pages}
+{,} { } {status}
+{,} { \PrintDOI} {doi}
+{,} { available at \eprint} {eprint}
+{} { \parenthesize} {language}
+{} { \PrintTranslation} {translation}
+{;} { \PrintReprint} {reprint}
+{.} { } {note}
+{.} {} {transition}
+{} {\SentenceSpace \PrintReviews} {review}
}
\BibSpec{collection.article}{%
+{} {\PrintAuthors} {author}
+{,} { \textit} {title}
+{.} { } {part}
+{:} { \textit} {subtitle}
+{,} { \PrintContributions} {contribution}
+{,} { \PrintConference} {conference}
+{} {\PrintBook} {book}
+{,} { } {booktitle}
+{,} { \PrintDateB} {date}
+{,} { \pageafter} {pages}
+{,} { } {status}
+{,} { \PrintDOI} {doi}
+{,} { available at \eprint} {eprint}
+{} { \parenthesize} {language}
+{} { \PrintTranslation} {translation}
+{;} { \PrintReprint} {reprint}
+{.} { } {note}
+{.} {} {transition}
+{} {\SentenceSpace \PrintReviews} {review}
}
\usepackage{enumitem}
\usepackage{pgf,tikz}
\usepackage[capitalise]{cleveref}

\allowdisplaybreaks

\usetikzlibrary{arrows}
\usetikzlibrary[patterns]
\usetikzlibrary{shapes.misc}
\usetikzlibrary{decorations.pathreplacing}

\setlist[itemize]{leftmargin=*}
\setlist[enumerate]{leftmargin=*,label=(\arabic*),ref=(\arabic*)}

\tikzset{cross/.style={cross out, draw=black, minimum size=2*(#1-\pgflinewidth), inner sep=0pt, outer sep=0pt},
cross/.default={1pt}}

\makeatletter
\pgfdeclarepatternformonly[\LineSpace,\tikz@pattern@color]{my north east lines}{\pgfqpoint{-1pt}{-1pt}}{\pgfqpoint{\LineSpace}{\LineSpace}}{\pgfqpoint{\LineSpace}{\LineSpace}}%
{
    \pgfsetcolor{\tikz@pattern@color}
    \pgfsetlinewidth{0.4pt}
    \pgfpathmoveto{\pgfqpoint{0pt}{0pt}}
    \pgfpathlineto{\pgfqpoint{\LineSpace + 0.1pt}{\LineSpace + 0.1pt}}
    \pgfusepath{stroke}
}
\makeatother

\makeatletter
\pgfdeclarepatternformonly[\LineSpace,\tikz@pattern@color]{my north west lines}{\pgfqpoint{-1pt}{-1pt}}{\pgfqpoint{\LineSpace}{\LineSpace}}{\pgfqpoint{\LineSpace}{\LineSpace}}%
{
    \pgfsetcolor{\tikz@pattern@color}
    \pgfsetlinewidth{0.4pt}
    \pgfpathmoveto{\pgfqpoint{0pt}{\LineSpace}}
    \pgfpathlineto{\pgfqpoint{\LineSpace + 0.1pt}{-0.1pt}}
    \pgfusepath{stroke}
}
\makeatother

\makeatletter
\pgfdeclarepatternformonly[\LineSpace,\tikz@pattern@color]{my horizontal lines}{\pgfqpoint{-1pt}{-1pt}}{\pgfqpoint{\LineSpace}{\LineSpace}}{\pgfqpoint{\LineSpace}{\LineSpace}}%
{
    \pgfsetcolor{\tikz@pattern@color}
    \pgfsetlinewidth{0.4pt}
    \pgfpathmoveto{\pgfqpoint{0pt}{0pt}}
    \pgfpathlineto{\pgfqpoint{\LineSpace + 0.1pt}{0pt}}
    \pgfusepath{stroke}
}
\makeatother

\makeatletter
\pgfdeclarepatternformonly[\LineSpace,\tikz@pattern@color]{my vertical lines}{\pgfqpoint{-1pt}{-1pt}}{\pgfqpoint{\LineSpace}{\LineSpace}}{\pgfqpoint{\LineSpace}{\LineSpace}}%
{
    \pgfsetcolor{\tikz@pattern@color}
    \pgfsetlinewidth{0.4pt}
    \pgfpathmoveto{\pgfqpoint{0pt}{0pt}}
    \pgfpathlineto{\pgfqpoint{0pt}{\LineSpace + 0.1pt}}
    \pgfusepath{stroke}
}
\makeatother

\newdimen\LineSpace
\tikzset{
    line space/.code={\LineSpace=#1},
    line space=5pt
}

\definecolor{ffqqqq}{rgb}{1,0,0}
\definecolor{qqffqq}{rgb}{0,1,0}
\definecolor{ffffff}{rgb}{1,1,1}

\newtheorem{mainthm}{Theorem}

\newtheorem{thm}{Theorem}
\crefname{thm}{Theorem}{Theorems}
\newtheorem{cor}[thm]{Corollary}
\crefname{cor}{Corollary}{Corollaries}
\newtheorem{lem}[thm]{Lemma}
\crefname{lem}{Lemma}{Lemmas}
\newtheorem{prop}[thm]{Proposition}
\crefname{prop}{Proposition}{Propositions}

\crefname{conj}{Conjecture}{Conjectures}

\crefname{ques}{Question}{Questions}
\theoremstyle{definition}
\newtheorem{defn}[thm]{Definition}
\crefname{defn}{Definition}{Definitions}
\newtheorem{rem}[thm]{Remark}
\crefname{rem}{Remark}{Remarks}

\crefname{ex}{Example}{Examples}
\newtheorem{obs}[thm]{Observation}
\crefname{obs}{Observation}{Observations}

\crefname{claim}{Claim}{Claims}

\crefname{ass}{Assumption}{Assumptions}

\numberwithin{thm}{section}
\newcommand{\cA}{\ensuremath{\mathcal A}}
\newcommand{\cB}{\ensuremath{\mathcal B}}
\newcommand{\cC}{\ensuremath{\mathcal C}}
\newcommand{\cD}{\ensuremath{\mathcal D}}
\newcommand{\cE}{\ensuremath{\mathcal E}}
\newcommand{\cF}{\ensuremath{\mathcal F}}
\newcommand{\cG}{\ensuremath{\mathcal G}}
\newcommand{\cH}{\ensuremath{\mathcal H}}
\newcommand{\cI}{\ensuremath{\mathcal I}}

\newcommand{\cL}{\ensuremath{\mathcal L}}

\newcommand{\cS}{\ensuremath{\mathcal S}}
\newcommand{\cT}{\ensuremath{\mathcal T}}
\newcommand{\cU}{\ensuremath{\mathcal U}}

\newcommand{\cW}{\ensuremath{\mathcal W}}

\newcommand{\bbE}{{\ensuremath{\mathbb E}} }

\newcommand{\bbH}{{\ensuremath{\mathbb H}} }

\newcommand{\bbN}{{\ensuremath{\mathbb N}} }

\newcommand{\bbP}{{\ensuremath{\mathbb P}} }

\newcommand{\bbR}{{\ensuremath{\mathbb R}} }

\newcommand{\bbT}{{\ensuremath{\mathbb T}} }

\newcommand{\bbZ}{{\ensuremath{\mathbb Z}} }

\renewcommand{\a}{{\ensuremath{\alpha}}}
\let\oldd\d
\renewcommand{\d}{{\ensuremath{\delta}}}
\newcommand{\D}{{\ensuremath{\Delta}}}
\newcommand{\e}{{\ensuremath{\varepsilon}}}

\newcommand{\g}{{\ensuremath{\gamma}}}

\newcommand{\h}{{\ensuremath{\eta}}}
\let\oldk\k
\renewcommand{\k}{{\ensuremath{\kappa}}}

\let\oldl\l
\renewcommand{\l}{{\ensuremath{\lambda}}}
\let\oldL\L
\renewcommand{\L}{{\ensuremath{\Lambda}}}
\newcommand{\m}{{\ensuremath{\mu}}}
\newcommand{\n}{{\ensuremath{\nu}}}
\let\oldo\o
\renewcommand{\o}{{\ensuremath{\omega}}}
\let\oldO\O
\renewcommand{\O}{{\ensuremath{\Omega}}}
\newcommand{\p}{{\ensuremath{\pi}}}
\let\oldr\r
\renewcommand{\r}{{\ensuremath{\rho}}}

\let\oldt\t
\renewcommand{\t}{{\ensuremath{\tau}}}
\let\oldu\u
\renewcommand{\u}{{\ensuremath{\upsilon}}}
\newcommand{\x}{{\ensuremath{\xi}}}

\newcommand{\<}{\langle}
\renewcommand{\>}{\rangle}

\newcommand{\var}{\operatorname{Var}}

\newcommand{\1}{{\ensuremath{\mathbbm{1}}} }

\newcommand{\Et}{\ensuremath{\bbE_{\m}[\t_0]}}
\newcommand{\SG}{\cS\cG}
\newcommand{\bSG}{\overline{\SG}}

\newcommand{\ST}{\cS\cT}
\newcommand{\bcT}{\overline{\ST}}
\newcommand{\ue}{\ensuremath{\underline e}}
\newcommand{\ur}{\ensuremath{\underline r}}

\newcommand{\rme}{\ensuremath{r^{\mathrm{mes}}}}
\newcommand{\uri}{\underline r^{\mathrm{int}}}
\newcommand{\urm}{\underline r^{\mathrm{mes}}}
\newcommand{\sm}{\ensuremath{s^{\mathrm{mes}}}}
\newcommand{\si}{\ensuremath{s^{\mathrm{int}}}}
\newcommand{\us}{\ensuremath{\underline s}}

\newcommand{\usm}{\underline s^{\mathrm{mes}}}
\newcommand{\uv}{\ensuremath{\underline v}}

\newcommand{\uone}{\ensuremath{\underline 1}}
\newcommand{\hS}{\ensuremath{\widehat\cS}}
\newcommand{\Hb}{\ensuremath{\overline\bbH}}
\newcommand{\rd}{\ensuremath{\r_{\mathrm{D}}}}
\newcommand{\bone}{\ensuremath{\mathbf{1}}}
\newcommand{\bzero}{\ensuremath{\mathbf{0}}}
\newcommand{\li}{\ell^{\mathrm{int}}}
\newcommand{\lm}{\ell^{\mathrm{mes}}}
\newcommand{\lmp}{\ell^{\mathrm{mes}+}}
\newcommand{\lmm}{\ell^{\mathrm{mes}-}}
\newcommand{\lgl}{\ell^{\mathrm{gl}}}
\newcommand{\Lm}{\L^{\mathrm{mes}}}
\newcommand{\Lmp}{\L^{\mathrm{mes}+}}
\newcommand{\Lmm}{\L^{\mathrm{mes}-}}
\newcommand{\Nm}{N^{\mathrm{mes}}}
\newcommand{\Nmp}{N^{\mathrm{mes}+}}
\newcommand{\Nmm}{N^{\mathrm{mes}-}}
\newcommand{\Ni}{N^{\mathrm{int}}}
\newcommand{\Nc}{N^{\mathrm{cr}}}

\renewcommand{\le}{\leqslant}
\renewcommand{\ge}{\geqslant}
\renewcommand{\to}{\rightarrow}

\begin{document}
\title{Refined universality for critical KCM:\\upper bounds}
\author[,1]{Ivailo Hartarsky\thanks{\textsf{ivailo.hartarsky@tuwien.ac.at}}}
\affil{Technische Universit\"at Wien, Institut für Stochastik und Wirtschaftsmathematik, Wiedner Hauptstra\ss e 8-10, A-1040, Vienna, Austria}
\date{\vspace{-0.25cm}\today}
\maketitle
\vspace{-0.75cm}
\begin{abstract}
We study a general class of interacting particle systems called kinetically constrained models (KCM) in two dimensions. They are tightly linked to the monotone cellular automata called bootstrap percolation. Among the three classes of such models \cite{Bollobas15}, the critical ones are the most studied.

Together with the companion paper by Mar\^ech\'e and the author \cite{Hartarsky22univlower}, our work determines the logarithm of the infection time up to a constant factor for all critical KCM. This was previously known only up to logarithmic corrections \cites{Martinelli19a,Hartarsky21a,Hartarsky20}. We establish that on this level of precision critical KCM have to be classified into seven categories. This refines the two classes present in bootstrap percolation \cite{Bollobas23} and the two in previous rougher results \cites{Martinelli19a,Hartarsky21a,Hartarsky20}. In the present work we establish the upper bounds for the novel five categories and thus complete the universality program for equilibrium critical KCM. Our main innovations are the identification of the dominant relaxation mechanisms and a more sophisticated and robust version of techniques recently developed for the study of the Fredrickson-Andersen 2-spin facilitated model \cite{Hartarsky23FA}.
\end{abstract}


\tableofcontents

\section{Introduction}
\label{sec:intro}
Kinetically constrained models (KCM) are interacting particle systems. They have challenging features including non-ergodicity, non-attractiveness, hard constraints, cooperative dynamics and dramatically diverging time scales. This prevents the use of conventional mathematical tools in the field. 

KCM originated in physics in the 1980s \cites{Fredrickson84,Fredrickson85} as toy models for the liquid-glass transition, which is still a hot and largely open topic for physicists \cite{Arceri21}. The idea behind them is that one can induce glassy behaviour without the intervention of static interactions, disordered or not, but rather with simple kinetic constraints. The latter translate the phenomenological observation that at high density particles in a super-cooled liquid become trapped by their neighbours and require a scarce bit of empty space in order to move at all. We direct the reader interested in the motivations of these models and their position in the landscape of glass transition theories to \cites{Arceri21,Garrahan11,Ritort03}.

Bootstrap percolation is the natural monotone deterministic counterpart of KCM (see \cites{Morris17} for an overview). Nevertheless, the two subjects arose for different reasons and remained fairly independent until the late 2000s. That is when the very first rigorous results for KCM came to be \cite{Cancrini08}, albeit much less satisfactory than their bootstrap percolation predecessors. The understanding of these two closely related fields did not truly unify until the recent series of works \cites{Martinelli19,Martinelli19a,Mareche20Duarte,Hartarsky20,Hartarsky22univlower,Hartarsky21a,Hartarsky23FA} elucidating the common points, as well as the serious additional difficulties in the non-monotone stochastic setting. It is the goal of this series that is accomplished by the present paper.

\subsection{Models}
\label{subsec:models}
Let us introduce the class of $\cU$-KCM introduced in \cite{Cancrini08}. In $d\ge 1$ dimensions an \emph{update family} is a nonempty finite collection of finite nonempty subsets of $\bbZ^d\setminus\{0\}$ called \emph{update rules}. The $\cU$-KCM is a continuous time Markov chain with state space $\O=\{0,1\}^{\bbZ^d}$. Given a configuration $\h\in\O$, we write $\h_x$ for the state of $x\in\bbZ^d$ in $\h$ and say that \emph{$x$ is infected} (in $\h$) if $\h_x=0$. We write $\h_A$ for the restriction of $\h$ to $A\subset\bbZ^d$ and $\bzero_A$ for the completely infected configuration with $A$ omitted when it is clear from the context. We say that \emph{the constraint at $x\in\bbZ^d$ is satisfied} if there exists an update rule $U\in\cU$ such that $x+U=\{x+y:y\in U\}$ is fully infected. We denote the corresponding indicator by
\begin{equation}
\label{eq:def:cx}
c_x(\h)=\1_{\exists U\in\cU,\h_{x+U}=\bzero}.
\end{equation}

The final parameter of the model is its \emph{equilibrium density of infections} $q\in[0,1]$. We denote by $\m$ the product measure such that $\m(\h_x=0)=q$ for all $x\in\bbZ^d$ and by $\var$ the corresponding variance. Given a finite set $A\subset\bbZ^d$ and real function $f:\O\to\bbR$, we write $\m_A(f)$ for the average $\m(f(\h)|\h_{\bbZ^d\setminus A})$ of $f$ over the variables in $A$. We write $\var_A(f)$ for the corresponding conditional variance, which is thus also a function from $\O_{\bbZ^d\setminus A}$ to $\bbR$, where $\O_B=\{0,1\}^B$ for $B\subset\bbZ^d$.

With this notation the $\cU$-KCM can be formally defined via its generator
\begin{equation}
\label{eq:def:generator}\cL(f)(\h)=\sum_{x\in\bbZ^d}c_x(\h)\cdot\left(\m_x(f)-f\right)(\h)\end{equation}
and its Dirichlet form reads
\[\cD(f)=\sum_{x\in\bbZ^d}\m\left(c_x\cdot\var_x(f)\right),\]
where $\m_x$ and $\var_x$ are shorthand for $\m_{\{x\}}$ and $\var_{\{x\}}$. Alternatively, the process can be defined via a graphical representation as follows (see \cite{Liggett05} for background). Each site $x\in\bbZ^d$ is endowed with a standard Poisson process called \emph{clock}. Whenever the clock at $x$ rings we assess whether its constraint is satisfied by the current configuration. If it is, we update $\h_x$ to an independent Bernoulli variable with parameter $1-q$ and call this a \emph{legal update}. If the constraint is not satisfied, the update is \emph{illegal}, so we discard it without modifying the configuration. It is then clear that $\m$ is a reversible measure for the process (there are others, e.g.\ the Dirac measure on the fully non-infected configuration $\bone$).

Our regime of interest is $q\to0$, corresponding to the low temperature limit relevant for glasses. A quantitative observable, measuring the speed of the dynamics, is the infection time of $0$
\[\t_0=\inf\left\{t\ge 0:\h_0(t)=0\right\},\]
where $(\h(t))_{t\ge 0}$ denotes the $\cU$-KCM process. More specifically, we are interested in its expectation for the stationary process $\Et$, namely the process with random initial condition distributed according to $\m$. This quantifies the equilibrium properties of the system and is closely related e.g.\ to the more analytic quantity called relaxation time (i.e.\ inverse of the spectral gap of the generator) that the reader may be familiar with.

$\cU$-bootstrap percolation is essentially the $q=1$ case of $\cU$-KCM started out of equilibrium, from a product measure with $q_0\to0$ density of infections. More conventionally, it is defined as a synchronous cellular automaton, which updates all sites of $\bbZ^d$ simultaneously at each discrete time step, by infecting sites whose constraint is satisfied and never removing infections. As the process is monotone, it may alternatively be viewed as a growing subset of the grid generated by its initial condition. We denote by $[A]_\cU$ the set of sites eventually infected by the $\cU$-bootstrap percolation process with initial condition $A\subset \bbZ^d$, that is, the sites which can become infected in the $\cU$-KCM in finite time starting from $\h(0)=(\1_{x\not\in A})_{x\in\bbZ^d}$. Strictly speaking, other than this notation, bootstrap percolation does not feature in our proofs, but its intuition and techniques are omnipresent. On the other hand, some of our intermediate results can translate directly to recover some bootstrap percolation results of \cites{Bollobas23,Bollobas15}. 

\subsection{Universality setting}
We direct the reader to the companion paper by Mar\^ech\'e and the author \cite{Hartarsky22univlower}, a monograph of Toninelli and the author \cite{HartarskybookKCM} and the author's PhD thesis \cite{Hartarsky22phd}*{Chap.~1}, for comprehensive background on the universality results for two-dimensional KCM and their history. Instead, we provide a minimalist presentation of the notions we need. The definitions in this section were progressively accumulated in \cites{Gravner99,Bollobas15,Bollobas23,Martinelli19a,Hartarsky21a,Hartarsky22univlower} and may differ in phrasing from the originals, but are usually equivalent thereto (see \cite{Hartarsky22univlower} for more details).

Henceforth, we restrict our attention to models in two dimensions. The Euclidean norm and scalar product are denoted by $\|\cdot\|$ and $\<\cdot,\cdot\>$, and distances are w.r.t.\ $\|\cdot\|$. Let $S^1=\{x\in\bbR^2:\|x\|=1\}$ be the unit circle consisting of \emph{directions}, which we occasionally identify with $\bbR/2\pi\bbZ$ in the standard way. We denote the open half plane with outer normal $u\in S^1$ and offset $l\in\bbR$ by 
\begin{equation}
\label{def:Hul}
\bbH_u(l)=\left\{x\in\bbR^2:\<x,u\>< l\right\}
\end{equation}
and omit $l$ if it is $0$. We further denote its closure by $\Hb_u(l)$, omitting zero offsets. We often refer to continuous sets such as $\bbH_u$, but whenever talking about infections or sites in them, we somewhat abusively identify them with their intersections with $\bbZ^2$ without further notice.

Fix an update family $\cU$. 
\begin{defn}[Stability]
\label{def:stable}
A direction $u\in S^1$ is \emph{unstable} if there exists $U\in\cU$ such that $U\subset\bbH_u$ and \emph{stable} otherwise.
\end{defn}
It is not hard to see that unstable directions form a finite union of finite open intervals in $S^1$ \cite{Bollobas15}*{Theorem 1.10}. We say that a stable direction is \emph{semi-isolated} (resp.\ \emph{isolated}) if it is the endpoint of a nontrivial (resp.\ trivial) interval of stable directions.

\begin{defn}[Criticality]
Let $\cC$ be the set of open semicircles of $S^1$. An update family is 
\begin{itemize}
\label{def:tripartition}
    \item \emph{supercritical} if there exists $C\in\cC$ such that all $u\in C$ are unstable;
    \item \emph{subcritical} if every semicircle contains infinitely many stable directions;
    \item \emph{critical} otherwise.
\end{itemize}
\end{defn}
The following notion measures ``how stable'' a stable direction is.
\begin{defn}[Difficulty]
\label{def:alpha}
For $u\in S^1$ the \emph{difficulty} $\a(u)$ of $u$ is 
\begin{itemize}
    \item $0$ if $u$ is unstable;
    \item $\infty$ if $u$ is stable, but not isolated;
    \item $\min\{n:\exists Z\subset\bbZ^2,|Z|=n,|[\bbH_u\cup Z]_\cU\setminus\bbH_u|=\infty\}$ otherwise.
\end{itemize}
The \emph{difficulty} of $\cU$ is
\[\a=\min_{C\in\cC}\max_{u\in C}\a(u).\]
We say that a direction $u\in S^1$ is \emph{hard} if $\a(u)>\a$.
\end{defn}
See \cref{fig:example} for an example update family with $\a=3$ and its difficulties. It can be shown that $\a(u)\in[1,\infty)$ for isolated stable directions \cite{Bollobas23}*{Lemma 2.8}. Consequently, a model is critical iff $0<\a<\infty$ and supercritical iff $\a=0$, so difficulty is tailored for critical models and refines \cref{def:tripartition}. Furthermore, for supercritical models the notions of stable and hard direction coincide. Finally, observe that the definition implies that for any critical or supercritical update family there exists an open semicircle with no hard direction.
\begin{defn}[Refined types]
A critical or supercritical update family is
\begin{itemize}
    \item \emph{rooted} if there exist two non-opposite hard directions;
    \item \emph{unrooted} if it is not rooted;
    \item \emph{unbalanced} if there exist two opposite hard directions;
    \item \emph{balanced} if it is not unbalanced, that is, there exists a \emph{closed} semicircle containing no hard direction.
\end{itemize}
We further partition balanced unrooted update families into
\begin{itemize}
    \item \emph{semi-directed} if there is exactly one hard direction;
    \item \emph{isotropic} if there are no hard directions.
\end{itemize}
\end{defn}
We further consider the distinction between models with finite and infinite number of stable directions. The latter being necessarily rooted, but possibly balanced or unbalanced, we end up with a partition of all (two-dimensional non-subcritical) families into the seven classes studied in detail below in the critical case. We invite the interested reader to consult \cite{Hartarsky22univlower}*{Fig. 1} for simple representatives of each class with rules contained in the the lattice axes and reaching distance at most 2 from the origin. Naturally, many more examples have been considered in the literature (also see \cref{fig:example}).

Let us remark that models in each class may have one axial symmetry, but non-subcritical models invariant under rotation by $\pi$ are necessarily either isotropic or unbalanced unrooted (thus with finite number of stable directions), while invariance by rotation by $\pi/2$ implies isotropy.

\subsection{Results}
Our result, summarised in \cref{tab:universality}, together with the companion paper by Mar\^ech\'e and the author \cite{Hartarsky22univlower}, is the following complete refined classification of two-dimensional critical KCM (for the classification of supercritical ones, which only features the rooted/unrooted distinction, see \cites{Martinelli19a,Mareche20Duarte,Mareche20combi}).
\begin{table}
    \centering
    \begin{tabular}{r|p{4.5cm}| p{2.25cm}|p{2.25cm}|}
         &\multicolumn{1}{c|}{\multirow{2}{*}{Infinite stable directions}}
         & \multicolumn{2}{c|}{Finite stable directions}\\\cline{3-4}
         & & \multicolumn{1}{c|}{Rooted} & \multicolumn{1}{c|}{Unrooted}\\
        \hline
       Unbalanced & \ref{log4} $2,4,0$ & \ref{log3} $1,3,0$ & \ref{log2} $1, 2,0$\\
        \hline
        Balanced & \ref{log0} $2,0,0$ & \ref{log1} $1,1,0$ & \diagbox[dir=NE,innerwidth=2.25cm,height=3\line]{\ref{loglog} $1,0,1$\\
        \footnotesize{S.-dir.}}{\footnotesize{Iso.}\\\ref{iso} $1,0,0$}\\\hline
    \end{tabular}
    \caption{Classification of critical $\cU$-KCM with difficulty $\alpha$. For each class $\Et=\exp\left(\Theta(1)\left(\frac{1}{q^\alpha}\right)^{\beta}\left(\log \frac{1}{q}\right)^\gamma \left(\log\log\frac{1}{q}\right)^\delta\right)$ as $q \to 0$. The label of the class and the exponents $\beta,\gamma,\delta$ are indicated in that order.}
    \label{tab:universality}
\end{table}

\begin{mainthm}[Universality classification of two-dimensional critical KCM]
\label{th:main}
Let $\cU$ be a two-dimensional critical update family with difficulty $\alpha$. We have the following exhaustive alternatives as $q\to0$ for the expected infection time of the origin under the stationary $\cU$-KCM.\footnote{\label{foot:asymptotic}We write $f=\Theta(g)$ if there exist $c,C>0$ such that $cg(q)<f(q)<Cg(q)$ for all $q$ small enough and use other standard asymptotic notation (see e.g.\ \cite{Hartarsky22univlower}*{Section 1.2}).} If $\cU$ is 
\begin{enumerate}[label=(\alph*), ref=(\alph*), series=th]
    \item\label{log4} unbalanced with infinite number of stable directions (so rooted), then
    \[\Et=\exp\left(\frac{\Theta\left(\left(\log (1/q)\right)^4\right)}{q^{2\alpha}}\right);\]
    \item\label{log0} balanced with infinite number of stable directions (so rooted), then
    \[\Et=\exp\left(\frac{\Theta(1)}{q^{2\alpha}}\right);\]
\item\label{log3} unbalanced rooted with finite number of stable directions, then 
\[\Et=\exp\left(\frac{\Theta\left(\left(\log (1/q)\right)^3\right)}{q^{\alpha}}\right);\]
\item\label{log2} unbalanced unrooted (so with finite number of stable directions), then
\[\Et=\exp\left(\frac{\Theta\left(\left(\log(1/q)\right)^2\right)}{q^{\alpha}}\right);\]
\item\label{log1} balanced rooted with finite number of stable directions, then\footnote{See \cref{rem:logloglog}.}
\[\Et=\exp\left(\frac{\Theta\left(\log (1/q)\right)}{q^{\alpha}}\right);\]
\item\label{loglog} semi-directed (so balanced unrooted with finite number of stable directions), then
\[\Et=\exp\left(\frac{\Theta\left(\log\log (1/q)\right)}{q^{\alpha}}\right);\]
\item\label{iso} isotropic (so balanced unrooted with finite number of stable directions), then
\[\Et=\exp\left(\frac{\Theta(1)}{q^{\alpha}}\right).\]
\end{enumerate}
\end{mainthm}
This theorem is the result of a tremendous amount of effort by a panel of authors. It would be utterly unfair to claim that it is due to the present paper and its companion \cite{Hartarsky22univlower} alone. Indeed, parts of the result (sharp upper or lower bounds for certain classes) were established by (subsets of) Mar\^ech\'e, Martinelli, Morris, Toninelli and the author \cites{Martinelli19,Martinelli19a,Hartarsky21a,Hartarsky20}. Moreover, particularly for the lower bounds, the classification of two-dimensional critical $\cU$-bootstrap percolation models by Bollob\'as, Duminil-Copin, Morris and Smith \cite{Bollobas23} (featuring only the balanced/unbalanced distinction) is heavily used, while upper bounds additionally use prerequisites from \cites{Hartarsky22CBSEP,Hartarsky23FA}. Thus, a fully self-contained proof of \cref{th:main} from common probabilistic background is currently contained only in all the above references combined and spans hundreds of pages. Our contribution is but the conclusive step.

More precisely, the lower bound for classes \ref{log2} and \ref{iso} was deduced from \cite{Bollobas23} in \cite{Martinelli19}; the lower bound for class \ref{log0} was established in \cite{Hartarsky20}, while the remaining four were proved in \cite{Hartarsky22univlower}. Turning to upper bounds, the one for class \ref{log4} was given in \cite{Martinelli19a} and the one for class \ref{log3} is due to \cite{Hartarsky21a}. The remaining five upper bounds are new and those are the subject of our work. The most novel and difficult ones concern classes \ref{log1} and \ref{loglog}, the latter remaining quite mysterious prior to our work. Indeed, \cite{Hartarsky21a}*{Conjecture 6.2} predicted the above result with the exception of this class, whose behaviour was unclear. We should note that this conjecture itself rectified previous ones from \cites{Martinelli19a,Morris17}, which were disproved by the unexpected result of \cite{Hartarsky21a}, and was new to physicists, as well as mathematicians.

\begin{rem}
It should be noted that universality results including \cref{th:main} apply to KCM more general than the ones defined in \cref{subsec:models}. Namely, we may replace $c_x$ in \cref{eq:def:generator} by a fixed linear combination of the constraints $c_x$ associated to any finite set of update families. For instance, we may update vertices at rate proportional to their number of infected neighbours. This and other models along these lines have been considered e.g.\ in \cites{Fredrickson84,Blanquicett21,AlvesNaN}. For such mixtures of families, the universality class is determined by the family obtained as their union. Indeed, upper bounds follow by direct comparison of the corresponding Dirichlet forms, while lower bounds (e.g.\ \cite{Hartarsky22univlower}) generally rely on deterministic bottlenecks, which remain valid.
\end{rem}

\begin{rem}
\label{rem:logloglog}
Let us note that for reasons of extremely technical nature, we do not provide a full proof of (the upper bound of) \cref{th:main}\ref{log1}. More precisely, we prove it as stated for models with rules contained in the axes of the lattice. We also prove a fully general upper bound of
\begin{equation}
\label{eq:rem:logloglog}
\exp\left(\frac{O(\log(1/q))\log\log\log(1/q)}{q^\a}\right).
\end{equation}
Furthermore, with very minor modifications (see \cref{rem:scales}), the error factor can be reduced from $\log\log\log$ to $\log_*$, where $\log_*$ denotes the number of iterations of the logarithm before the result becomes negative (the inverse of the tower function). Unfortunately, removing this minuscule error term requires further work, which we omit for the sake of concision. Instead, we provide a sketch of how to achieve this in \cref{app:logloglog}.
\end{rem}

\subsection{Organisation}
The paper is organised as follows. In \cref{sec:mechanisms} we begin by outlining all the relevant relaxation mechanisms used by critical KCM, providing detailed intuition for the proofs to come. This section is particularly intended for readers unfamiliar with the subject, as well as physicists, for whom it may be sufficiently convincing on its own. In \cref{sec:preliminaries} we gather various notation and simple preliminaries. 

In \cref{sec:extensions} we formally state the two fundamental techniques we use to move from one scale to the next, namely East-extensions and CBSEP-extensions, which import and generalise ideas of \cite{Hartarsky23FA}. They are used in various combinations throughout the rest of the paper. The proofs of the results about those extensions, including the microscopic dynamics treated by \cite{Hartarsky21b} are deferred to \cref{app:extensions}, since they are quite technical and do not require new ideas. The bounds arising from extensions feature certain conditional expectations. We provide technical tools for estimating them in \cref{subsec:conditional}. We leave the entirely new proofs of these general analogues of \cite{Hartarsky23FA}*{Appendix A} to \cref{app:proba}.

\Cref{sec:iso,sec:log2,sec:loglog,sec:log1,sec:log0} are the core of our work and use the extensions mentioned above to prove the upper bounds of \cref{th:main} for classes \ref{iso}, \ref{log2}, \ref{loglog}, \ref{log1}, \ref{log0} respectively. As we will discuss in further detail (see \cref{sec:mechanisms,tab:mechanisms}), some parts of the proofs are common to several of these classes, making the sections interdependent. Thus, they are intended for linear reading.

We conclude in \cref{app:logloglog} by explaining how to remove the corrective $\log\log\log(1/q)$ factor discussed in \cref{rem:logloglog} to recover the result of \cref{th:main}\ref{log1} as stated in full generality. Due to their technical nature, we delegate \cref{app:extensions,app:logloglog,app:proba} to the arXiv version of the present work.

Familiarity with the companion paper \cite{Hartarsky22univlower} or bootstrap percolation \cite{Bollobas23} is not needed. Inversely, familiarity with \cites{Hartarsky21a,Hartarsky23FA} is strongly recommended for going beyond \cref{sec:mechanisms} and achieving a complete view of the proof of the upper bounds of \cref{th:main}. Nevertheless, we systematically state the implications of intermediate results of those works for our setting in a self-contained fashion, without re-proving them.

\section{Mechanisms}
\label{sec:mechanisms}

In this section we attempt a heuristic explanation of \cref{th:main} from the viewpoint of mechanisms, which is mostly related to upper bound proofs. Yet, let us say a few words about the lower bounds. The proof of the lower bounds in the companion paper \cite{Hartarsky22univlower} has the advantage and disadvantage of being unified for all seven classes. This is undeniably practical and spotlights the fact that all scaling behaviours can be viewed through the lens of the same bottleneck (few energetically costly configurations through which the dynamics has to go to infect the origin) on a class-dependent length scale. However, the downside is that it provides little insight on the particularities of each class, which turn out to be quite significant. To prove upper bounds we need a clear vision of an efficient mechanism for infecting the origin in each class. Since we work with the stationary process, efficient means that it should avoid configurations which are too unlikely w.r.t.\ $\m$. However, while lower bounds only identify what cannot be avoided, they do not tell us how to avoid everything else, nor indeed how to reach the unavoidable bottleneck.

Instead of outlining the mechanism used by each class, we focus on techniques which are somewhat generic and then apply combinations thereof to each class. In figurative terms, we develop several computer hardware components (three processors, four RAMs, etc.), give a general scheme of how to compose a generic computer out of generic components and, finally, assemble seven concrete computer configurations, using the appropriate components for each, sometimes changing a single component from a machine to the other. Moreover, within each component type different instances are strictly comparable, so, at the assembly stage, we might simply choose the best possible component fitting with the requirements of model at hand. This enables us to highlight the robust tools developed and refined recently, which correspond to the components and how they are manufactured, as well as give a clean universal proof scheme into which they are plugged.

Our different components are called the \emph{microscopic, internal, mesoscopic and global dynamics} and correspond to progressively increasing length scales on which we are able to relax, given a suitable infection configuration. As the notion of ``suitable,'' which we call \emph{super good} (SG), depends on the class and lower scale mechanisms used, we mostly use it as a black box input extended progressively over scales in a recursive fashion.

\begin{table}
    \centering
\begin{subtable}{\textwidth}
\centering\scriptsize{ 
\begin{tabular}{c|c|c|c|c|c|c}
         \multicolumn{2}{c|}{\normalsize Global}
         & \multicolumn{2}{c|}{\normalsize Mesoscopic}
         &\multicolumn{3}{c}{\normalsize Internal}
         \\
        \hline
        \rule{0pt}{2ex}CBSEP&East&CBSEP&East, Stair&CBSEP&East&Unbal.\\\hline
        \rule{0pt}{3ex}$\rd^{-1+o(1)}$&$ \rd^{-O(\log(1/\rd))}$&$e^{q^{-o(1)}}$&$\rd^{-O(\log (1/q))}$&$e^{q^{-o(1)}}$&$\rd^{-O(\log\log (1/q))}$&$\rd^{-O(1)}$
    \end{tabular}
}
    \caption{\normalsize The relaxation time cost associated to each choice of dynamics mechanism on each scale in terms of the probability of a droplet $\rd$.\smallskip}
    \label{tab:costs}
\end{subtable}
\begin{subtable}{\textwidth}
    \centering
{
\begin{scriptsize}\begin{tabular}{c|c|c|c|c|c|c|c}
         &\normalsize\ref{log4}& \normalsize\ref{log0}&\normalsize\ref{log3}&\normalsize\ref{log2}&\normalsize\ref{log1}&\normalsize\ref{loglog}&\normalsize\ref{iso}\\
        \hline
        \normalsize \rule{0pt}{2ex}Global&East*&East*&CBSEP&CBSEP*&CBSEP&CBSEP&CBSEP*\\\hline
        \normalsize \rule{0pt}{2ex}Mesoscopic&Stair&East&East*&CBSEP&East*&CBSEP&CBSEP\\\hline
        \normalsize \rule{0pt}{2ex}Internal&---&East&Unbal.&Unbal.*&East&East*&CBSEP
    \end{tabular}
\end{scriptsize}
}
\caption{\normalsize The fastest mechanism available to each class of \cref{th:main} on each scale. The * indicates a leading contribution for the class (column).}
    \label{tab:mechanisms}
    \end{subtable}
\caption{\label{tab:summary}Summary of the mechanisms and their costs. The microscopic one common to all classes and with negligible cost is not shown (see \cref{subsec:micro}).}
\end{table}

In order to guide the reader through \cref{sec:mechanisms} and beyond, in \cref{tab:summary}, we summarise the optimal mechanisms for each universality class on each scale and its cost. While its full meaning will only become clear in \cref{subsec:assembly}, the reader may want to consult it regularly, as they progress through \cref{sec:mechanisms}.

The SG events concern certain convex polygonal geometric regions called \emph{droplets}. These events are designed so as to satisfy several conditions ensuring that the configuration of infections inside the droplet is sufficient to infect the entire droplet. The SG events defined by extensions from smaller to larger scales require the presence of a lower scale droplet inside the large one (see \cref{fig:extensions}) in addition to well-chosen more sparse infections called \emph{helping sets} in the remainder of the larger scale droplet (see \cref{fig:example}). Helping sets allow the smaller one to move inside the bigger one.

We say that a droplet \emph{relaxes} in a certain \emph{relaxation time} if the dynamics restricted to the SG event and to this region ``mixes'' in this much time. Formally, this translates to a constrained Poincar\'e inequality for the conditional measure, but this is unimportant for our discussion.

One should think of droplets as extremely unlikely objects, which are able to move within a slightly favourable environment. Indeed, at all stages of our treatment, we need to control the inverse probability of droplets being SG and their relaxation times, keeping them as small as feasible. Furthermore, due to their inductive definition, the favourable environment required for their movement should not be too costly. Indeed, that would result in the deterioration of the probability of larger scale droplets, as those incorporate the lower scale environment in their internal structure. Hence, we seek a balance between asking for many infections to make the movement efficient and asking for few in order to keep the probability of droplets high enough.

\subsection{Scales}
\subparagraph*{Microscopic dynamics} refers to modifying infections at the level of the lattice \emph{along the boundary of a droplet}, while respecting the KCM constraint.

\subparagraph*{Internal dynamics} refers to relaxation on scales from the lattice level to the \emph{internal scale} $\li= C^2\log (1/q)/q^\a$, where $C$ is a large constant depending on $\cU$. This is the most delicate and novel step. Up to $\li$ we account for the main contribution to the probability of droplets. That is, at all larger scales the probability of a droplet essentially saturates at a certain value $\rd$, because finding helping sets becomes likely. Thus, on smaller scales, it is important to only very occasionally ask for more than $\a$ infections to appear close to each other in order to get the right probability $\rd$. This means that up to the internal scale hard directions are practically impenetrable, since they require helping sets of more that $\a$ infections.

\subparagraph*{Mesoscopic dynamics} refers to relaxation on scales from $\li$ to the \emph{mesoscopic scale} $\lm= 1/q^C$. As our droplets grow to the mesoscopic scale and past it, it becomes possible to require larger helping sets, which we call \emph{$W$-helping sets}. These allow droplets to move also in hard directions of finite difficulty, while nonisolated stable directions are still blocking.

\subparagraph*{Global dynamics} refers to relaxation on scales from $\lm$ to infinity. The extension to infinity being fairly standard (and not hard), one should rather focus on scales up to the \emph{global scale} given by $\lgl=\exp(1/q^{3\a})$, which is notably much larger than all time scales we are aiming for, but otherwise rather arbitrary.

\bigskip
Roughly speaking, on each of the last three scales, one should decide how to move a droplet of the lower scale in a domain on the larger scale.

For simplicity, in the remainder of \cref{sec:mechanisms}, we assume that the only four relevant directions are the axis ones so that droplets have rectangular shape (see \cref{subsec:geometry}). We further assume that all directions in the left semicircle have difficulties at most $\a$, while the down direction is hard, unless there are no hard directions (isotropic class).

\subsection{Microscopic dynamics}
\label{subsec:micro}
The microscopic dynamics (see \cref{sec:micro}) is the only place where we actually deal with the KCM directly and is the same, regardless of the size of the droplet and the universality class. Roughly speaking, from the outside of the droplet, we may think of it as fully infected, since it is able to relax and, therefore, bring infections where they are needed. Thus, the outer boundary of the droplet behaves like a 1-dimensional KCM with update family reflecting that we view the droplet as infected. Hence, provided there are enough helping sets at the boundary to infect it, we can apply results on 1-dimensional KCM supplied for this purpose by the author \cite{Hartarsky21b}. 

This way we establish that one additional column can relax in time $\exp(O(\log(1/q))^2)$, similarly to the East model described in \cref{subsubsec:East:extension} below. Assuming we know how to relax on the droplet itself, this allows us to relax on a droplet with one column appended. However, applying this procedure recursively line by line is not efficient enough to be useful for extending droplets more significantly.

\subsection{One-directional extensions}
\label{subsec:extensions}
We next explain two fundamental techniques beyond the microscopic dynamics which we use to extend droplets on any scale in a single direction (see \cref{sec:extensions}).

As mentioned above, our droplets are polygonal regions with a SG event (presence of a suitable arrangement of infections in the droplet). An extension takes as input a droplet and produces another one. In terms of geometry, it contains the original one 
and is obtained by extending it, say, horizontally, either to the left or both left and right (see \cref{fig:extensions}). The extended droplet's SG event requires that the smaller one is SG and, additionally, certain helping sets appear in the remaining volume. The choice of where we position the smaller droplet (at the right end of the bigger one, or anywhere inside it) depends on the type of extension. The additional helping sets are required in such a way that, with their help, the smaller droplet can, in principle, completely infect the larger one and, therefore, make it relax (resample its configuration within its SG event).

Thus, an \emph{extension} is a procedure for iteratively defining SG events on larger and larger scales. For each of our two types of extensions we need to provide corresponding iterative bounds on the probability of the SG event and on the relaxation time of droplets on this event. The former is a matter of careful computation. For the latter task we intuitively use a large-scale version of an underlying one-dimensional spin model, which we describe first.

\subsubsection{CBSEP-extension}
\label{subsubsec:CBSEP:extension}
In the one-dimensional spin version of CBSEP \cites{Hartarsky23FA,Hartarsky22CBSEP} we work on $\{\uparrow,\downarrow\}^\bbZ$. At rate 1 we resample each pair of neighbouring spins, provided that at least one of them is $\uparrow$. Their state is resampled w.r.t.\ the reference product measure, which is reversible, conditioned to still have a $\uparrow$ in at least one of the two sites. In other words, $\uparrow$ can perform coalescence, branching and symmetric simple exclusion moves, hence the name. The relaxation time of this model on volume $V$ is roughly $\min(V,1/q)^{2}$ in one dimension and $\min(V,1/q)$ in two and more dimensions \cites{Hartarsky22CBSEP,Hartarsky23FA}, where $q$ is the equilibrium density of $\uparrow$, which we think of as being small.

For us $\uparrow$ represent SG droplets, which we would like to move within a larger volume. However, as we would like them to be able to move possibly by an amount smaller than the size of the droplet, we need to generalise the model a bit. We equip each site of a finite interval of $\bbZ$ with a state space corresponding to the state of a column of the height of our droplet of interest in the original lattice $\bbZ^2$. Then the event ``there is a SG droplet'' may occur on a group of $\ell$ consecutive sites (columns). The long range generalised CBSEP, which, abusing notation, we call CBSEP, is defined as follows. We fix some \emph{range} $R>\ell$ and resample at rate 1 each group of $R$ consecutive sites, if they contain a SG droplet. The resampling is performed conditionally on preserving the presence of a SG droplet in those $R$ sites. Thus, one move of this process essentially delocalises the droplet within the range.

It is important to note (and this was crucial in \cite{Hartarsky23FA}) that CBSEP does not \emph{have to} create an additional droplet in order to evolve. Since SG droplets are unlikely, it suffices to move an initially available SG droplet through our domain in order to relax. Since infection needs to be able to propagate both left and right from the SG droplets, we will define (see \cref{subsec:CBSEP:extension} and particularly \cref{def:extension:CBSEP,subfig:extension:CBSEP}) \emph{CBSEP-extension} by extending our domain horizontally and asking for the SG droplet anywhere inside with suitable ``rightwards-pointing'' helping sets on its right and ``leftwards-pointing'' on its left.

While we now know that droplets evolve according to CBSEP, it remains to see how one can reproduce one CBSEP move via the original dynamics. This is done inductively on $R$ by a bisection procedure, the trickiest part being the case $R=\ell+1$. We then dispose with a droplet plus one column---exactly the setting for microscopic dynamics. However, we not only want to resample the state of the additional column, but also allow the droplet to move by one lattice step. To achieve this, we have to look inside the structure of the SG droplet and require for its infections (which have no rigid structure and may therefore move around like the organelles of an amoeba) to be somewhat more on the side we want to move towards (see e.g.\ \cref{fig:iso} and also \cref{def:bSG:iso,def:bSG:loglog:0,def:bSG:loglog:1,def:bSG:log2}). Then, together with a suitable configuration on the additional column provided by the microscopic dynamics, we easily recover our SG event shifted by one step, since most of the structure was already provided by the version of the SG event ``contracted'' towards the new column. 

This amoeba-like motion (moving a droplet, by slightly rearranging its internal structure) leads to a very small relaxation time of the dynamics. Indeed, the time needed to move the droplet is the product of three contributions: the relaxation time of the 1-dimensional spin model; the relaxation time of the microscopic dynamics; the time needed to see a droplet contracting as explained above (see \cref{cor:CBSEP:reduction}). The first of these is a power of the volume (number of sites); the second is $\exp(O(\log(1/q)))^2)$; the third is also small, as we discuss in \cref{subsubsec:East:extension}. 

However, CBSEP-extensions can only be used for sufficiently symmetric update families. That is, the droplet needs to be able to move indifferently both left and right and its position should not be biased in one direction or the other. Specifically, if we are working on a scale that requires the use of helping sets of size $\alpha$, these have to exist both for the left and right directions, so the model needs to be unrooted (if instead we use larger helping sets, having a finite number of stable directions suffices). The reason is that otherwise the position of the SG droplet is biased in one direction instead of being approximately uniform. This would break the analogy with the original one-dimensional spin model, which is totally symmetric. When symmetry is not available, we recourse to the East-extension presented next, which may also be viewed as a totally asymmetric version of the CBSEP-extension.

\subsubsection{East-extension}
\label{subsubsec:East:extension}
The East model \cite{Jackle91} is the one-dimensional KCM with $\cU=\{\{1\}\}$. That is, we are only allowed to resample the left neighbour of an infection. An efficient recursive mechanism for its relaxation is the following \cite{Mauch99}. Assume we start with an infection at $0$. In order to bring an infection to $-2^n+1$, using at most $n$ infections at a time (excluding $0$), we first bring one to $-2^{n-1}+1$, using $n-1$ infections. We then place an infection at $-2^{n-1}$ and reverse the procedure to remove all infections except $0$ and $-2^{n-1}$. Finally, start over with $n-1$ infections, viewing $-2^{n-1}$ as the new origin, thus reaching $-2^{n}+1$. It is not hard to check that this is as far as one can get with $n$ infections \cite{Chung01}. Thus, a number of infections logarithmic in the desired distance is needed. This is to be contrasted with CBSEP, for which only one infection is ever needed, as it can be moved indefinitely by SEP moves. The relaxation time of East on a segment of length $L$ is $q^{-O(\log \min(L,1/q))}$ \cites{Aldous02,Cancrini08,Chleboun14}, where $q$ is the equilibrium density of infections. This corresponds to the cost of $n$ infections when $2^n\sim \min(L,1/q)$ is the typical distance to the nearest infection.

The long-range generalised version of the East model is defined similarly to that of CBSEP. The only difference is that now $R>\ell$ consecutive columns are resampled together if there is a SG droplet on their extreme right. It is clear that this does not allow \emph{moving} the droplet, but rather forces us to recreate a new droplet at a shifted position before we can progress. The associated \emph{East-extension} of a droplet corresponds to extending its geometry to the left (see \cref{subsec:East:extension} and particularly \cref{def:extension:East,subfig:extension:East}). The extended SG event requires that the original SG droplet is present in the rightmost position and ``leftwards-pointing'' helping sets are available in the rest of the extended droplet.

The generalised East process goes back to \cite{Martinelli19a}, while the long range version is implicitly used in \cite{Hartarsky21a}. However, both works used a brutal strategy consisting of creating the new droplet from scratch. Instead, in this work we have to be much more careful, particularly for semi-directed models. Indeed, take $\ell$ large and $R=\ell+5$. Then it is intuitively clear that the presence of the original rightmost droplet overlaps greatly with the occurrence of the shifted SG one we would like to craft. Hence, the idea is to take advantage of this and only pay the \emph{conditional} probability of the droplet we are creating, given the presence of the original one.

This is not as easy as it sounds for several reasons. Firstly, we should make the SG structure soft enough (in contrast with e.g.\ \cites{Martinelli19a,Hartarsky21a}) so that small shifts do not change it much. Secondly, we need to actually have a quantitative estimate of the conditional probability of a complicated multi-scale event, given its translated version, which necessarily does not quite respect the same multi-scale geometry. To make matters worse, we do not have at our disposal a very sharp estimate of the probability of SG events (unlike in \cite{Hartarsky23FA}), so directly computing the ratio of two rough estimates would yield a very poor bound on the conditional probability. In fact, this problem is also present when contracting a droplet in the CBSEP-extension---we need to evaluate the probability of a contracted version of the droplet, conditionally on the original droplet being present. 

We deal with these issues in \cref{subsec:conditional} (see also \cref{app:proba}). We establish that, as intuition may suggest, to create a droplet shifted by $R-\ell$, given the original one, we roughly only need to pay the probability of a droplet on scale $R-\ell$ rather than $\ell$, which provides a substantial gain. Hence, the time necessary for an East-extension of a droplet to relax is essentially the product of the inverse probabilities of a droplet on scales of the form $2^m$ up to the extension length (see \cref{cor:East:reduction}).

\subsection{Internal dynamics}
The internal dynamics (see \cref{subsec:log1:internal,subsec:loglog:internal,subsec:iso,subsubsec:unbal:internal}) is where most of our work goes. This is not surprising, as the probability of SG events saturates at its final value $\rd$ at the internal scale. The value of $\rd$ is given by $\exp(-O(1)/q^\a)$ for balanced models and $\exp(-O(\log(1/q))^2/q^\a)$ for unbalanced ones, as in bootstrap percolation \cite{Bollobas23}. However, relaxation times for some classes keep growing past the internal scale, so the internal dynamics does not necessarily give the final answer in \cref{th:main} (see \cref{tab:mechanisms}).

\subsubsection{Unbalanced internal dynamics}
\label{subsubsec:unbal:int:outline}
Let us begin with the simplest case of unbalanced models. If $\cU$ is unbalanced with infinite number of stable directions (class \ref{log4}), droplets in \cite{Martinelli19a} on the internal scale consist of several infected consecutive columns, so that no relaxation is needed (the SG event is a singleton). The columns have size $\li$, which justifies the value of $\rd=q^{-O(\li)}=\exp(-O(\log(1/q))^2/q^\a)$.

Assume $\cU$ is unbalanced with finite number of stable directions (classes \ref{log3} and \ref{log2}, see \cref{subsubsec:unbal:internal}). Then droplets on the internal scale are fully infected square frames of thickness $O(1)$ and size $\li$. That is, the $\ell^\infty$ ball of radius $\li$ minus the one of radius $\li-O(1)$ (see \cite{Hartarsky21a}*{Figs. 2-4} or \cref{fig:bSG:log2} for more general geometry). This frame is infected with probability $\rd=q^{-O(\li)}$. In order to relax inside the frame, one can divide its interior into groups of $O(1)$ consecutive columns (see \cite{Hartarsky21a}*{Fig. 8}). We can then view them as performing a CBSEP dynamics with $\uparrow$ corresponding to a fully infected group of columns. This is possible, because with the help of the frame each completely infected group is able to completely infect the neighbouring ones. Here we are using that there are finitely many stable directions to ensure both the left and right directions have finite difficulty, so finite-sized helping sets, as provided by the frame, are sufficient to propagate our group of columns. This was already done in \cites{Hartarsky21a} and the time necessary for this relaxation is easily seen to be $\rd^{-O(1)}$ (the cost for creating a group of infected columns)---see \cref{prop:unbal:internal}.

\subsubsection{CBSEP internal dynamics}
If $\cU$ is isotropic (class \ref{iso}, see \cref{subsec:iso}), up to the conditioning problems of \cref{subsec:conditional} described above, we need only minor adaptations of the strategy of \cite{Hartarsky23FA} for the paradigmatic isotropic model called FA-2f. Droplets on the internal scale have an internal structure as obtained by iterating \cref{subfig:iso:SG} (see also \cite{Hartarsky23FA}*{Fig. 2}). Our droplets are extended little by little alternating between the horizontal and vertical directions, so that their size is multiplied essentially by a constant at each extension. Thus, roughly $\log (1/q)$ extensions are required to reach $\li$. As isotropic models do not have any hard directions, we can move in all directions and thus the symmetry required for CBSEP-extensions is granted. Hence, this mechanism leads to a very fast relaxation of droplets in time $\exp(q^{-o(1)})$---see \cref{th:isotropic}.\footnote{Note that in \cite{Hartarsky23FA}*{Proposition 4.7} a much larger internal relaxation time of order $\exp(q^{-1/2+o(1)})$ was obtained, but the cost $\rd^{-1}$ of SG droplets was much smaller than the one in the present work, so our treatment here is by no means as sharp for FA-2f as \cite{Hartarsky23FA}.} 

\begin{rem}
\label{rem:symmetrisation}
Note that for CBSEP-extensions to be used, we need a very strong symmetry. Namely, leftwards and rightwards pointing helping sets should be the same up to rotation by $\pi$. Yet, for a general isotropic model we only know that there are no hard directions, so helping sets have the same size (equal to the difficulty $\alpha$ of the model), but not necessarily the same shape. We circumvent this issue by artificially symmetrising our droplets and events. Namely, whenever we require helping sets in one direction, we also require the helping sets for the opposite direction rotated by $\pi$ (see \cref{def:ahelping,def:traversability,def:extension:CBSEP}). Although these are totally useless for the dynamics, they are important to ensure that the positions of droplets are indeed uniform rather than suffering from a drift towards an ``easier'' non-hard direction (see \cref{lem:T:ratio}).
\end{rem}

\subsubsection{East internal dynamics}
\label{subsubsec:East:internal:overview}
The most challenging case is the balanced non-isotropic one (classes \ref{log0}, \ref{log1} and \ref{loglog}). It is treated in \cref{subsec:loglog:internal,subsec:log1:internal}, but for the purposes of the present section only \cref{subsec:loglog:internal} is relevant. This is because we assume that only the four axis directions are relevant and our droplets are rectangular. The treatment of the general case for balanced rooted families is left to \cref{subsec:log1:internal,app:logloglog} (recall \cref{rem:logloglog}). 

For the internal dynamics the downwards hard direction prevents us from using CBSEP-extensions. To be precise, for semi-directed models (class \ref{loglog}) it is possible to perform CBSEP-extensions horizontally (and not vertically), but the gain is insignificant, so we treat all balanced non-isotropic models identically up to the internal scale as follows.

We still extend droplets, starting from a microscopic one, by a constant factor alternating between the horizontal and vertical directions (see \cref{fig:East:internal:loglog}). However, in contrast with the isotropic case (see \cref{subfig:iso:SG}), extensions are done in an oriented fashion, so that the original microscopic droplet remains anchored at the corner of larger ones. Thus, we may apply East-extensions on each step and obtain that the cost is given by the product of conditional probabilities from \cref{subsubsec:East:extension} over all scales and shifts of the form $2^n$:
\begin{equation}
\label{eq:product:subtle}\prod_{n=1}^{\log_2(\li)}\prod_{m=0}^{n}a_m^{(n)},\end{equation}
where $a_m^{(n)}$ is the inverse of the conditional probability of a SG droplet of size $2^n$ being present at position $2^m$, given that a SG droplet of size $2^n$ is present at position 0. It is crucial that \cref{eq:product:subtle} is \emph{not} the straightforward bound $\prod_{n}(\rd^{(n)})^{-n}$, with $\rd^{(n)}$ the probability of a droplet of scale $n$, that one would get by direct analogy with the East model (recall from \cref{subsubsec:East:extension} that the relaxation time of East on a small volume $L$ is $q^{-O(\log L)}$), as that would completely devastate all our results. Indeed, as mentioned in \cref{subsubsec:East:extension}, the term $a_m^{(n)}$ in \cref{eq:product:subtle} is approximately equal to $(\rd^{(m)})^{-1}$, rather than $(\rd^{(n)})^{-1}$. This is perhaps one of the most important points to our treatment.

Hence, \cref{eq:product:subtle} transforms into \[\prod_{n=1}^{\log_2(\li)}\prod_{m=0}^{n}\left(\rd^{(m)}\right)^{-1}.\] In other words, a droplet of size $2^m$ needs to be paid for once per scale larger than $2^m$ (see \cref{eq:gLni:decomposition:loglog}). A careful computation shows that only droplets larger than $q^{-\a}$ provide the dominant contribution and those all have probability essentially $\rd^{(m)}=\rd=\exp(-O(1)/q^\a)$ (see \cref{eq:mu:SG:bound:loglog}). Thus, the total cost would be 
\begin{equation}
\label{eq:product:subtle2}\prod_{n=\log_2(1/q^\alpha)}^{\log_2(\li)}\prod_{m=\log_2(1/q^\alpha)}^{n}\rd^{-1}=\rd^{-O(\log\log(1/q))^2}=e^{O(\log\log(1/q))^2/q^\alpha},\end{equation}
since there are $O(\log\log(1/q))$ scales from $q^{-\a}$ to $\li$, as they increase exponentially. 

\Cref{eq:product:subtle2} is unfortunately a bit too rough for the semi-directed class, overshooting \cref{th:main}\ref{loglog}. However, the solution is simple. It suffices to introduce scales growing double-exponentially above $q^{-\a}$ instead of exponentially (see \cref{eq:def:ln:internal:East}), so that the product over scales $n$ in \cref{eq:product:subtle2} becomes dominated by its last term, corresponding to droplet size $\li$. This gives the optimal final cost 
\[\rd^{-\log_2(q^\alpha\li)}=\rd^{-O(\log\log(1/q))}=e^{O(\log\log(1/q))/q^\a}\]
(see \cref{th:internal:loglog}).

\subsection{Mesoscopic dynamics}
For the mesoscopic dynamics (see \cref{subsec:iso,subsec:log0:meso,subsec:log2:meso,subsec:loglog:meso}) we are given as input a SG event for droplets on scale $\li=C^2\log(1/q)/q^\alpha$ and a bound on their relaxation time and occurrence probability $\rd$. We seek to output the same on scale $\lm=q^{-C}$. Taking $C\gg W$, once our droplets have size $\lm$, we are able to find $W$-helping sets (sets of $W$ consecutive infections, where $W$ is large enough).
\subsubsection{CBSEP mesoscopic dynamics}
\label{subsubsec:CBSEP:meso}
If $\cU$ is unrooted (classes \ref{log2}, \ref{loglog} and \ref{iso}, see \cref{subsec:log2:meso,subsec:loglog:meso}), recall that the hard directions (if any) are vertical. Then we can perform a horizontal CBSEP-extension directly from $\li$ to $\lm$, since $\li=C^2\log(1/q)/q^\a$ makes it likely for helping sets (of size $\alpha$) to appear along all segments of length $\li$ until we reach scale $\lm= q^{-C}$. The resulting droplet is very wide, but short (see \cref{subfig:bSG:log2}). However, this is enough for us to be able to perform a vertical CBSEP-extension (see \cref{subfig:bSG:log2:2}), requiring $W$-helping sets, since they are now likely to be found. Again, CBSEP dynamics being very efficient, its cost is negligible. Note that, in order to perform the vertical extension, we are using that there are no nonisolated stable directions, so that $W$ is larger than the difficulty of the up and down directions, making $W$-helping sets sufficient to induce growth in those directions. Thus, morally, there are no hard directions beyond scale $\lm$ for unrooted models.

\subsubsection{East mesoscopic dynamics}
If $\cU$ is rooted (classes \ref{log4}-\ref{log3} and \ref{log1}, see \cref{subsec:log0:meso}), CBSEP-extensions are still inaccessible. We may instead East-extend horizontally from $\li$ to $\lm$ in a single step. If the model is balanced or has a finite number of stable directions (classes \ref{log0}, \ref{log3} and \ref{log1}), we may proceed similarly in the vertical direction, reaching a droplet of size $\lm$ in time $\rd^{-O(\log(1/q))}$ (here we use the basic bound $q^{-O(\log L)}$ for East dynamics recalled in \cref{subsubsec:East:extension}, which is fairly tight in this case, since droplets are small compared to the volume: $\log \lm\approx \log (\lm/\li)$). For the unbalanced case (class \ref{log3}) here we require $W$-helping sets along the long side of the droplet like in  \cref{subsubsec:CBSEP:meso}. Another way of viewing this is simply as extending the procedure used for the East internal dynamics all the way up to the mesoscopic scale $\lm$ (see \cref{subsec:log0:meso}). 

It should be noted that a version of this mechanism, which coincides with the above for models with rectangular droplets, but differs in general, was introduced in \cite{Hartarsky21a}. Though their \emph{snail mesoscopic dynamics} can be replaced by our East one, for the sake of concision in \cref{subsec:global:FA} we directly import the results of \cite{Hartarsky21a} based on the snail mechanism.

\subsubsection{Stair mesoscopic dynamics}
For unbalanced families with infinite number of stable directions (class \ref{log4}) the following \emph{stair mesoscopic dynamics} was introduced in \cite{Martinelli19a}. Recall from \cref{subsubsec:unbal:int:outline} that for unbalanced models the internal droplet is simply a fully infected frame or group of consecutive columns. While moving the droplet left via an East motion, we pick up $W$-helping sets above or below the droplet. These sets allow us to make all droplets to their left shifted up or down by one row. Hence, we manage to create a copy of the droplet far to its left but also slightly shifted up or down (see \cite{Martinelli19a}*{Fig. 6}). Repeating this (with many steps in our staircase) in a two-dimensional East-like motion, we can now relax on a mesoscopic droplet with horizontal dimension much larger than $\lm$ but still polynomial in $1/q$ and vertical dimension $\lm$ in time $\rd^{-O(\log(1/q))}$. Here, one should again intuitively imagine we are using the bound $q^{-O(\log L)}$ but this time for the relaxation time of the 2-dimensional East model.

\subsection{Global dynamics}
The global dynamics (see \cref{subsec:global:FA,subsec:global:CBSEP,subsec:log2:global,subsec:loglog:global,subsec:global:East}) receives as input a SG event for a droplet on scale $\lm$ with probability roughly $\rd$ and a bound on its relaxation time, as provided by the mesoscopic dynamics. Its goal is to move such a droplet efficiently to the origin from its typical initial position at distance roughly $\rd^{-1/2}$.

\subsubsection{CBSEP global dynamics}
\label{subsubsec:global:CBSEP}
If $\cU$ has a finite number of stable directions (classes \ref{log3}-\ref{iso}) the mesoscopic droplet can perform a CBSEP motion in a typical environment. Indeed, the droplet is large enough for CBSEP-extensions with $W$-helping sets to be possible in all directions. Therefore, the cost of this mechanism is given by the relaxation time of CBSEP on a box of size $\lgl=\exp(1/q^{3\alpha})$ with density of $\uparrow$ given by $\rd$. Performing this strategy carefully and using the 2-dimensional CBSEP, this yields a relaxation time $\min((\lgl)^2,1/\rd)=1/\rd$ (recall \cref{subsubsec:CBSEP:extension} and see \cref{subsec:global:CBSEP}).

\subsubsection{East global dynamics}
If $\cU$ has infinite number of stable directions (classes \ref{log4} and \ref{log0}), the strategy is identical to the CBSEP global dynamics, but employs an East dynamics. Now the cost becomes the relaxation time of an East model with density of infections $\rd$, which yields a relaxation time of $\rd^{-O(\log \min(\lgl,1/\rd))}=\rd^{-O(\log(1/\rd))}$ (recall \cref{subsubsec:East:extension} and see \cref{subsec:global:East}).

\subsection{Assembling the components}
\label{subsec:assembly}
To conclude, let us return to the summary provided in \cref{tab:summary}. In \cref{tab:costs} we collect the mechanisms for each scale and their cost to the relaxation time. The results are expressed in terms of the probability of a droplet $\rd$, which equals $\exp(-O(\log(1/q))^2/q^\a)$ for unbalanced models and $\exp(-O(1)/q^\a)$ for balanced ones. The final bound on $\Et$ for each class then corresponds to the product of the costs of the mechanism employed at each scale. To complement this, in \cref{tab:mechanisms} we indicate the fastest mechanism available for each class on each scale. We further indicate which one gives the dominant contribution to the final result appearing in \cref{th:main}, once the bill is footed. 

Finally, let us alert the reader that, for the sake of concision, the proof below does not systematically implement the optimal strategy for each class as indicated in \cref{tab:mechanisms} if that does not deteriorate the final result. Similarly, when that is unimportant, we may give weaker bounds than the ones in \cref{tab:costs}. In \cref{subsec:global:FA} we tacitly import a weaker precursor of the CBSEP global mechanism from \cite{Hartarsky21a} not listed above.

\section{Preliminaries}
\label{sec:preliminaries}
\subsection{Harris inequality}
\label{subsec:correlation}
Let us recall a well-known correlation inequality due to Harris \cite{Harris60}. It is used throughout and we state some particular formulations that are useful to us.

For \cref{subsec:correlation} we fix a finite $\L\subset \bbZ^2$. We say that an event $\cA\subset \O_\L$ is \emph{decreasing} if adding infections does not destroy its occurrence.
\begin{prop}[Harris inequality]
Let $\cA,\cB\subset\O_\L$ be decreasing. Then
\begin{equation}
\label{eq:Harris:1}
    \m(\cA\cap \cB)\ge \m(\cA)\m(\cB).
\end{equation}
\end{prop}
\begin{cor}
Let $\cA,\cB,\cC,\cD\subset\O_\L$ be nonempty and decreasing events such that $\cB$ and $\cD$ are independent, then
\begin{align}
\label{eq:Harris:2}\m(\cA|\cB\cap \cD)&{}\ge \m(\cA|\cB)\ge\m(\cA),\\
\label{eq:Harris:3}\m(\cA\cap \cC|\cB\cap \cD)&{}\ge \m(\cA|\cB)\m(\cC|\cD).
\end{align}
\end{cor}
\begin{proof}
The first inequality in \cref{eq:Harris:2} is \cref{eq:Harris:3} for $\cC=\Omega_\Lambda$, the second follows from \cref{eq:Harris:1} and $\m(\cA|\cB)=\m(\cA\cap\cB)/\m(\cB)$, while \cref{eq:Harris:3} is
\[\m(\cA\cap \cC|\cB\cap \cD)=\frac{\m(\cA\cap\cC\cap\cB\cap\cD)}{\m(\cB\cap\cD)}\ge\frac{\m(\cA\cap\cB)\m(\cC\cap\cD)}{\m(\cB)\m(\cD)}=\m(\cA|\cB)\m(\cC|\cD),\]
using \cref{eq:Harris:1} in the numerator and independence in the denominator.
\end{proof}
We collectively refer to \cref{eq:Harris:1,eq:Harris:2,eq:Harris:3} as \emph{Harris inequality}.

\subsection{Directions}
\label{subsec:directions}
Throughout this work we fix a critical update family $\cU$ with difficulty $\a$. We call a direction $u\in S^1$ \emph{rational} if $u\bbR\cap\bbZ^2\neq\{0\}$. It follows from \cref{def:stable} that isolated and semi-isolated stable directions are rational \cite{Bollobas15}*{Theorem 1.10}. Therefore, by \cref{def:alpha} there exists an open semicircle with rational midpoint $u_0$ such that all directions in the semicircle have difficulty at most $
\a$. We may assume without loss of generality that the direction $u_0+\pi/2$ is hard unless $\cU$ is isotropic. It is not difficult to show (see e.g.\ \cite{Bollobas15}*{Lemma 5.3}) that one can find a nonempty set $\cS'$ of rational directions such that:
\begin{itemize}
    \item all isolated and semi-isolated stable directions are in $\cS'$;
    \item $u_0\in\cS'$;
    \item for every two consecutive directions $u,v$ in $\cS'$ either there exists a rule $X\in\cU$ such that $X\subset\Hb_{u}\cap\Hb_v$ or all directions between $u$ and $v$ are stable.
\end{itemize}
We further consider the set $\hS=\cS'+\{0,\pi/2,\pi,3\pi/2\}$ obtained by making $\cS'$ invariant by rotation by $\pi/2$. It is not hard to verify that the three conditions above remain valid when we add directions, so they are still valid for $\hat \cS$ instead of $\cS'$. We refer to the elements of $\hS$ as \emph{quasi-stable directions} or simply \emph{directions}, as they are the only ones of interest to us. We label the elements of $\hS=(u_i)_{i\in[4k]}$ clockwise and consider their indices modulo $4k$ (we write $[n]$ for $\{0,\dots,n-1\}$), so that $u_{i+2k}=-u_{i}$ (the inverse being taken in $\bbR^2$ and not w.r.t.\ the angle) is perpendicular to $u_{i+k}$. In figures we take $\hS=\frac\pi4(\bbZ/8\bbZ)$ and $u_0=(-1,0)$. Further observe that if all $U\in\cU$ are contained in the axes of $\bbZ^2$, then we may set $\hS=\frac\pi2(\bbZ/4\bbZ)$.

For $i\in[4k]$ we introduce $\r_i=\min\{\r>0:\exists x\in\bbZ^2,\<x,u_i\>=\r\}$ and $\l_i=\min\{\l>0:\l u_i\in\bbZ^2\}$, which are both well-defined, as the directions are rational (in fact $\r_i\l_i=1$, but we use both notations for transparency).

\subsection{Droplets}
\label{subsec:geometry}
We next define the geometry of the droplets we use. Recall half planes from \cref{def:Hul}.
\begin{defn}[Droplet]
\label{def:droplet}
A \emph{droplet} is a nonempty closed convex polygon of the form
\[\L(\ur)=\bigcap_{i\in [4k]}\Hb_{u_i}(r_i)\]
for some \emph{radii} $\ur\in\bbR^{[4k]}$ (see the black regions in \cref{fig:extensions}). For a sequence of radii $\ur$ we define the \emph{side lengths} $\us=(s_i)_{i\in[4k]}$ with $s_i$ the length of the side of $\L(\ur)$ with outer normal $u_i$.

We say that a droplet is \emph{symmetric} if it is of the form $x+\L(\ur)$ with $2x\in\bbZ^2$ and $r_i=r_{i+2k}$ for all $i\in[2k]$. If this is the case, we call $x$ the \emph{center} of the droplet.
\end{defn}
Note that if all $U\in\cU$ are contained in the axes of $\bbZ^2$, then droplets are simply rectangles with sides parallel to the axes.

We write $(\ue_i)_{i\in[4k]}$ for the canonical basis of $\bbR^{[4k]}$ and we write $\uone=\sum_{i\in[4k]}\ue_i$, so that $\L(r\uone)$ is a polygon with inscribed circle of radius $r$ and sides perpendicular to $\hS$. It is often more convenient to parametrise dimensions of droplets differently. For $i\in [4k]$ we set
\begin{equation}
\label{eq:def:uvi}\uv_i=\sum_{j=i-k+1}^{i+k-1}\<u_i,u_j\>\ue_j.\end{equation}
This way $\L(\ur+\uv_i)$ is obtained from $\L(\ur)$ by extending the two sides parallel to $u_i$ by $1$ in direction $u_i$ and leaving all other side lengths unchanged (see \cref{subfig:extension:East}). Note that if $\L(\ur)$ is symmetric, then so is $\L(\ur+\l_i\uv_i)$ for $i\in[4k]$.

\begin{defn}[Tube]
\label{def:tube}
Given $i\in[4k]$, $\ur$ and $l>0$, we define the \emph{tube of length $l$, direction $i$ and radii $\ur$} (see the thickened regions in \cref{fig:extensions})
\[T(\ur,l,i)=\L(\ur+l\uv_i)\setminus\L(\ur).\]
\end{defn}
We often need to consider boundary conditions for our events on droplets and tubes. Given two disjoint finite regions $A,B\subset \bbZ^2$ and two configurations $\h\in\O_A$ and $\o\in\O_{B}$, we define $\h\cdot\o\in\O_{A\cup B}$ as
\begin{equation}
\label{eq:def:boundary:condition}(\h\cdot\o)_x=\begin{cases}
\h_x&x\in A,\\
\o_x&x\in B.
\end{cases}\end{equation}

\subsection{Scales}
\label{subsec:scales}
Throughout the work we consider the positive integer constants
\[1/\e\gg1/\d\gg C\gg W.\]
Each one is assumed to be large enough depending on $\cU$ and, therefore, $\hS$ and $\a$ (e.g.\ $W>\a$), and much larger than any explicit function of the next (e.g.\ $e^W<C$). These constants are not allowed to depend on $q$. Whenever asymptotic notation is used, its implicit constants are not allowed to depend on the above ones, but only on $\cU$. Also recall \cref{foot:asymptotic}.

The following are our main scales corresponding to the mesoscopic and internal dynamics:
\begin{align*}
    \lmp&{}=q^{-C}/\sqrt{\d},&
    \lm&{}=q^{-C},\\
    \lmm&{}=q^{-C}\cdot\sqrt{\d},&
    \li&{}=C^2\log(1/q)/q^\a.
\end{align*}

\subsection{Helping sets}
\label{subsec:helping:sets}
We next introduce various constant-sized sets of infections sufficient to induce growth. As the definitions are quite technical in general, in \cref{fig:example} we introduce a deliberately complicated example, on which to illustrate them.

\begin{figure}
		\centering
		\begin{subfigure}{0.51\textwidth}
		\centering
		\begin{subfigure}{0.48\textwidth}
		\centering
		\begin{tikzpicture}[line cap=round,line join=round,x=0.15\textwidth,y=0.15\textwidth]
				\clip (-6.3,-0.5) rectangle (0.3,4.3);
				\draw [color=gray, xstep=1,ystep=1] (-6,0) grid (0,4);
				\fill (0,1) circle (2pt);
				\fill (0,2) circle (2pt);
				\fill (0,3) circle (2pt);
				\fill (-2,0) circle (2pt);
				\fill (-4,0) circle (2pt);
				\fill (-6,0) circle (2pt);
				\draw (0,0) node[cross=3pt,rotate=0] {};
		\end{tikzpicture}
		\end{subfigure}
		\begin{subfigure}{0.48\textwidth}
		\centering
		\begin{tikzpicture}[line cap=round,line join=round,x=0.15\textwidth,y=0.15\textwidth]
				\clip (-0.3,-0.5) rectangle (4.3,4.3);
				\draw [color=gray, xstep=1,ystep=1] (0,0) grid (4,4);
				\fill (0,1) circle (2pt);
				\fill (0,2) circle (2pt);
				\fill (0,3) circle (2pt);
				\fill (0,4) circle (2pt);
				\fill (1,0) circle (2pt);
				\fill (2,0) circle (2pt);
				\fill (3,0) circle (2pt);
				\fill (4,0) circle (2pt);
				\draw (0,0) node[cross=3pt,rotate=0] {};
		\end{tikzpicture}
		\end{subfigure}
		\begin{subfigure}{0.48\textwidth}
		\centering
		\begin{tikzpicture}[line cap=round,line join=round,x=0.15\textwidth,y=0.15\textwidth]
				\clip (-6.3,-4.5) rectangle (0.3,0.3);
				\draw [color=gray, xstep=1,ystep=1] (-6,-4) grid (0,0);
				\fill (0,-1) circle (2pt);
				\fill (0,-2) circle (2pt);
				\fill (0,-3) circle (2pt);
				\fill (0,-4) circle (2pt);
				\fill (-1,0) circle (2pt);
				\fill (-2,0) circle (2pt);
				\fill (-3,0) circle (2pt);
				\draw (0,0) node[cross=3pt,rotate=0] {};
		\end{tikzpicture}
		\end{subfigure}
		\begin{subfigure}{0.48\textwidth}
		\centering
		\begin{tikzpicture}[line cap=round,line join=round,x=0.15\textwidth,y=0.15\textwidth]
				\clip (-0.3,-4.5) rectangle (4.3,0.3);
				\draw [color=gray, xstep=1,ystep=1] (0,-4) grid (4,0);
				\fill (1,0) circle (2pt);
				\fill (2,0) circle (2pt);
				\fill (3,0) circle (2pt);
				\fill (0,-1) circle (2pt);
				\fill (0,-2) circle (2pt);
				\fill (0,-3) circle (2pt);
				\fill (0,-4) circle (2pt);
				\draw (0,0) node[cross=3pt,rotate=0] {};
		\end{tikzpicture}
		\end{subfigure}
		\phantom{\begin{subfigure}{0.48\textwidth}
		\centering
		\begin{tikzpicture}[line cap=round,line join=round,x=0.15\textwidth,y=0.15\textwidth]
				\clip (-6.3,-1.3) rectangle (0.3,1.3);
				\draw [color=gray, xstep=1,ystep=1] (-6,-3) grid (0,0);
				\fill (0,-1) circle (2pt);
				\fill (-1,0) circle (2pt);
				\fill (-2,0) circle (2pt);
				\draw (0,0) node[cross=3pt,rotate=0] {};
		\end{tikzpicture}
		\end{subfigure}}
		\begin{subfigure}{0.48\textwidth}
		\centering
		\begin{tikzpicture}[line cap=round,line join=round,x=0.15\textwidth,y=0.15\textwidth]
				\clip (-0.3,-1.5) rectangle (4.3,1.3);
				\draw [color=gray, xstep=1,ystep=1] (0,-1) grid (4,1);
				\fill (2,0) circle (2pt);
				\fill (0,-1) circle (2pt);
				\fill (0,1) circle (2pt);
				\draw (0,0) node[cross=3pt,rotate=0] {};
		\end{tikzpicture}
		\end{subfigure}
		\caption{The five update rules $U\in\cU$ given as dots. The cross marks the origin.}
		\end{subfigure}
		\quad
		\begin{subfigure}{0.45\textwidth}
		\centering
		\begin{tikzpicture}[line cap=round,>=triangle 45,line join=round,x=0.41\textwidth,y=0.41\textwidth]
				\draw(0,0) circle (0.41\textwidth);
				\fill [color=ffqqqq] (0,1) circle (3pt) node[anchor=south,black] {$3$};
			    \fill [color=ffqqqq] (1,0) circle (3pt) node[anchor=west,black] {$3$};
				\fill [color=ffqqqq] (-1,0) circle (3pt) node[anchor=east,black] {$3$};
				\fill [color=ffqqqq] (0,-1) circle (3pt) node[anchor=north,black] {$3$};
				\draw[->] (0,0)-- (1,0) node [below,midway] {$u_2$};
				\draw[->] (0,0)-- (0,1) node [right,midway] {$u_1$};
				\draw[->] (0,0)-- (-1,0) node [above,midway] {$u_0$};
				\draw[->] (0,0)-- (0,-1) node [left,midway] {$u_3$};
	    \end{tikzpicture}
	    \caption{The four stable directions, which coincide with $\hS$, and their difficulties.}
		\end{subfigure}
		\\
		\begin{subfigure}{\textwidth}
		\centering
		\begin{subfigure}{0.32\textwidth}
		\centering
		\begin{tikzpicture}[line cap=round,>=triangle 45,line join=round,x=0.15\textwidth,y=0.15\textwidth]
				\clip (-1.3,-0.3) rectangle (1.3,4.3);
				\draw [color=gray, xstep=1,ystep=1] (-1.2,0.8) grid (-0.8,5);
				\draw[->] (-0.3,0)-- (-1.3,0) node[midway,above] {$u_0$};
				\fill (-1,1) circle (2pt);
				\fill (-1,3) circle (2pt);
				\fill (-1,4) circle (2pt);
				\fill[pattern=my north east lines] (-0.3,-5) -- (2,-0.3) -- (2,5) -- (-0.3,5) -- cycle;
		\end{tikzpicture}
		\end{subfigure}
		\begin{subfigure}{0.32\textwidth}
		\centering
		\begin{tikzpicture}[line cap=round,>=triangle 45,line join=round,x=0.15\textwidth,y=0.15\textwidth]
				\clip (-0.3,1.3) rectangle (3.3,-1.3);
				\draw [color=gray, xstep=1,ystep=1] (0.8,1.2) grid (5,0.8);
				\draw[->] (0,0.3)-- (0,1.3) node[midway,right] {$u_1$};
				\fill (1,1) circle (2pt);
				\fill (2,1) circle (2pt);
				\fill (3,1) circle (2pt);
				\fill[pattern=my north east lines] (-5,0.3) -- (-0.3,-2) -- (5,-2) -- (5,0.3) -- cycle;
		\end{tikzpicture}
		\end{subfigure}
		\begin{subfigure}{0.32\textwidth}
		\centering
		\begin{tikzpicture}[line cap=round,>=triangle 45,line join=round,x=0.15\textwidth,y=0.15\textwidth]
				\clip (3.3,0.3) rectangle (-1.3,-3.3);
				\draw [color=gray, xstep=1,ystep=1] (3.2,-0.8) grid (0.8,-5);
				\draw[->] (0.3,0)-- (1.3,0) node[midway,below] {$u_2$};
				\fill (1,-1) circle (2pt);
				\fill (3,-2) circle (2pt);
				\fill (1,-3) circle (2pt);
				\fill[pattern=my north east lines] (0.3,5) -- (-2,0.3) -- (-2,-5) -- (0.3,-5) -- cycle;
		\end{tikzpicture}
		\end{subfigure}\vspace{0.3cm}
		\begin{subfigure}{0.96\textwidth}
		\centering
		\begin{tikzpicture}[line cap=round,line join=round,>=triangle 45,x=0.05\textwidth,y=0.05\textwidth]
				\clip (0.3,-3.3) rectangle (-15.3,1.3);
				\draw [color=gray, xstep=1,ystep=1] (-0.8,-0.8) grid (-20,-2.2);
				\fill[fill=white,opacity=1] (-9.2,-5) -- (-10.8,-5) -- (-10.8,-0.5) -- (-9.2,-0.5) -- cycle;
				\draw[->] (-9,-2.5)-- (-7,-2.5) node[midway,below] {$x_3$};
				\draw[->] (0,-0.3)-- (0,-1.3) node[midway,left] {$u_3$};
				\fill (-5,-1) circle (2pt);
				\fill (-1,-1) circle (2pt);
				\fill (-3,-2) circle (2pt);
			    \draw[decorate,decoration={brace,amplitude=5pt}] (-0.8,-2.2) -- (-5.2,-2.2) node [midway,below,yshift=-5pt] {$Z_3+k_1x_3+\l_0u_0$};
				\fill (-11,-1) circle (2pt);
				\fill (-13,-2) circle (2pt);
				\fill (-15,-1) circle (2pt);
				 \draw[decorate,decoration={brace,amplitude=5pt}] (-10.8,-2.2) -- (-15.2,-2.2) node [midway,below,yshift=-5pt] {$Z_3$};
				\fill[pattern=my north east lines] (5,-0.3) -- (-20,-0.3) -- (-20,5) -- (5,5) -- cycle;
				\draw (-10,-1.5) node {$\dots$};
		\end{tikzpicture}
		\end{subfigure}
		\caption{Possible choice of $u_i$-helping sets. The hatched region represents $\bbH_{u_i}\cap\bbZ^2$.}
		\end{subfigure}\vspace{0.3cm}
  \caption{An intricate isotropic  example.}
    \label{fig:example}
\end{figure}
\subsubsection{Helping sets for a line}
\label{subsubsec:helping:sets:line}
Recall $(u_i)_{i\in[4k]}$ and $(\lambda_i)_{i\in[4k]}$ from \cref{subsec:directions} and that for $i\in[4k]$, the direction $u_{i+k}$ is obtained by rotating $u_i$ clockwise by $\pi/2$.
\begin{defn}[$W$-helping set in direction $u_i$]
Let $i\in[4k]$. A $W$-helping set in direction $u_i$ is any set of $W$ consecutive infected sites in $\Hb_{u_i}\setminus\bbH_{u_i}$, that is, a set of the form $x+[W]\l_{i+k}u_{i+k}$ for some $x\in\Hb_{u_i}\setminus\bbH_{u_i}$.
\end{defn}
The relevance of $W$-helping sets in direction $u_i$ is that, since $W$ is large enough, $[Z\cup\bbH_{u_i}]_\cU=\Hb_{u_i}$ for any direction $u_i$ such that $\alpha(u_i)<\infty$ and $Z$ a $W$-helping set in direction $u_i$ (see \cite{Bollobas15}*{Lemma 5.2}).

We next define some smaller sets which are sufficient to induce such growth but have the annoying feature that they are not necessarily contained in $\Hb_{u_i}$ and do not necessarily induce growth in a simple sequential way like $W$-helping sets in direction $u_i$. Let us note that except in \cref{sec:micro} the reader will not lose anything conceptual by thinking that the sets $Z_i$, $u_i$-helping sets and $\a$-helping sets in direction $u_i$ defined below are simply single infected sites in $\Hb_{u_i}\setminus\bbH_{u_i}$ and the period $Q$ is 1.

In words, the set $Z_i$ provided by the following lemma together with $\bbH_{u_i}$ can infect a semi-sublattice of the first line outside $\bbH_{u_i}$ and only a finite number of other sites.
\begin{lem}
\label{lem:helping:sets}
Let $i\in[4k]$ be such that $0<\a(u_i)\le \a$. Then there exists a set $Z_i\subset\bbZ^2\setminus\bbH_{u_i}$ and $x_i\in\bbZ^2\setminus\{0\}$ such that \begin{align*}
\<x_i,u_i\>&{}=0,&|Z_i|&{}=\a,&\left|\left[Z_i\cup\bbH_{u_i}\right]_\cU\setminus\Hb_{u_i}\right|&{}<\infty,&\left[Z_i\cup\bbH_{u_i}\right]_\cU&{}\supset x_i\bbN,\end{align*} where $\bbN=\{0,1,\dots\}$.
\end{lem}
\begin{proof}
\Cref{def:alpha} supplies a set $Z\subset \bbZ^2\setminus\bbH_{u_i}$ such that $\overline Z=[\bbH_{u_i}\cup Z]_\cU\setminus\bbH_{u_i}$ is infinite and $|Z|=\alpha(u_i)$. Among all possible such $Z$, choose $Z$ to minimise $l=\max\{\<z,u_i\>:z\in Z\}$. Yet, $u_i$ is stable, since $\alpha(u_i)\neq 0$ (recall \cref{def:alpha}). Therefore, 
$\overline Z\subset \Hb_{u_i}(l)\setminus\bbH_{u_i}$, because $Z\cup\bbH_{u_i}\subset \Hb_{u_i}(l)$ (recall \cref{def:stable} and observe that it implies that $[\Hb_{u_i}(l)]_\cU=\Hb_{u_i}(l)$). 

Then \cite{Bollobas23}*{Lemma 3.3} asserts that $\overline Z\cap\Hb_{u_i}$ is either finite or contains $x_i\bbN$ for some $x_i\in\Hb_{u_i}\setminus (\bbH_{u_i}\cup\{0\})$. Assume that $|\overline Z\setminus\Hb_{u_i}|<\infty$, so that $|\overline Z\cap\Hb_{u_i}|=\infty$, since $|\overline Z|=\infty$. Then we conclude by setting $Z_i$ equal to the union of $Z$ with $\alpha-\alpha(u_i)$ arbitrarily chosen elements of $\overline Z\setminus Z$, so that $\overline{Z_i}=\overline Z$.

Assume for a contradiction that, on the contrary, $|\overline Z\setminus\Hb_{u_i}|=\infty$. Set $Z'=(Z-\rho_iu_i)\setminus\bbH_u$ (i.e.\ shift $Z$ one line closer to $\bbH_{u_i}$) and observe that $\overline{Z'}\supset (\overline Z\setminus\Hb_{u_i}-\rho_iu_i)$ is still infinite. Therefore, by \cref{def:alpha} $\alpha(u_i)\le |Z'|\le |Z|=\alpha(u_i)$. This contradicts our choice of $Z$ minimising $l$.
\end{proof}
In the example of \cref{fig:example} the $u_3$ direction admits a set $Z_3$ of cardinality 3 such that $[Z_3\cup\bbH_{u_3}]_\cU$ only contains every second site of the line $\Hb_{u_i}\setminus\bbH_{u_i}$, while at least $4$ sites are needed to infect the entire line. Thus, in order to efficiently infect $\Hb_{u_3}\setminus\bbH_{u_3}$, assuming $\bbH_{u_3}$ is infected, we may use two translates of $Z_3$ with different parity. This technicality is reflected in the next definition.
\begin{defn}[$u_i$-helping set]
\label{def:uihelping}
For all $i\in[4k]$ such that $0<\a(u_i)\le \a$ fix a choice of $Z_i$ and $x_i$ as in \cref{lem:helping:sets} in such a way that the \emph{period}
\[Q=\frac{\|x_i\|}{\l_{i+k}}\]
is independent of $i$ and sufficiently large so that the diameter of $\{0\}\cup Z_i$ is much smaller than $Q$. A \emph{$u_i$-helping set} is a set of the form
\begin{equation}
\label{eq:helping:decomposition}
\bigcup_{j\in\left[Q\right]}\left(Z_i+j\l_{i+k}u_{i+k}+k_jx_i\right),\end{equation}
for some integers $k_j$. For $i\in[4k]$ with $\a(u_i)=0$, we define $u_i$-helping sets to be empty. For $i\in[4k]$ with $\alpha(u_i)>\alpha$ there are no $u_i$-helping sets.
\end{defn}
Note that by \cref{lem:helping:sets} a $u_i$-helping set $Z$ is sufficient to infect a half-line, but since that contains a $W$-helping set in direction $u_i$, we have $[Z\cup\bbH_{u_i}]_\cU\supset\Hb_{u_i}$. 

We further incorporate the artificial symmetrisation alluded to in \cref{rem:symmetrisation} in the next definition.
\begin{defn}[$\a$-helping set in direction $u_i$]
\label{def:ahelping}Let $i\in[4k]$.
\begin{itemize}
    \item If $\a(u_i)\le \a$ and $\a(u_{i+2k})\le\a$, then a \emph{$\a$-helping set in direction $u_i$} is a set of the form $H\cup H'$ with $H$ a $u_i$-helping set and $-H'=\{-h:h\in H'\}$ a $u_{i+2k}$-helping set.
    \item If $\a(u_i)\le \a$ and $\a(u_{i+2k})>\a$, then a \emph{$\a$-helping set in direction $u_i$} is a $u_i$-helping set.
    \item If $\a<\a(u_i)\le \infty$, there are no \emph{$\a$-helping sets in direction $u_i$}.
\end{itemize}
If $\a(u_i)<\infty$, any set which is either a $W$-helping set  in direction $u_i$ or a $\a$-helping set in direction $u_i$ is called \emph{helping set in direction $u_i$}. If $\a(u_i)=\infty$, there are no \emph{helping sets in direction $u_i$}.
\end{defn}
In the example of \cref{fig:example} $u_0$ and $u_2$ are both of difficulty $\a=3$, so $\a$-helping sets in direction $u_0$ correspond to $(z_1+\{(0,0),(2,0),(3,0)\})\cup(z_2+\{(0,0),(-2,1),(0,2)\})$ for some $(z_1,z_2)\in(\{0\}\times\bbZ)^2$. The set $z_2+\{(0,0),(-2,1),(0,2)\}$ is not a $u_0$-helping set, but we include it in $\a$-helping sets in direction $u_0$. We do so, in order for $\a$-helping sets in direction $u_0$ and $u_2$ to be symmetric. Namely, they satisfy that $Z$ is a $\alpha$-helping set in direction $u_0$ if and only if $-Z$ is a $\alpha$-helping set in direction $u_2$.

\subsubsection{Helping sets for a segment}
\label{subsubsec:segment:helping}
For this section we fix a direction $u_{i}\in\hS$ with $\a(u_i)<\infty$ and a discrete segment $S$ perpendicular to $u_i$ of the form 
\begin{equation}
\label{eq:S:form}
\left\{x\in\bbZ^2: \<x,u_i\>=0,\<x,u_{i+k}\>/\l_{i+k}\in[0,a]\right\}\end{equation}
for some integer $a\ge W$. The direction $u_i$ is kept implicit in the notation, so it may be useful to view $S$ as having an orientation.
\begin{defn}
\label{def:HW}
For $d\ge 0$, we denote by $\cH^W_d(S)$ the event that there is an infected $W$-helping set in direction $u_i$ in $S$ at distance at least $d$ from its endpoints:
\begin{multline*}
\cH^W_d(S)=\big\{\eta\in\Omega:\exists x\in\bbZ\cap [d/\lambda_{i+k},a-(W-1)-d/\lambda_{i+k}],\\
\eta_{(x+[W])\lambda_{i+k}u_{i+k}}=\bzero\big\}.
\end{multline*}
We write $\cH^W(S)=\cH^W_0(S)$.
\end{defn}

For helping sets the definition is more technical, since they are not included in $S$. We therefore require that they are close to $S$ and at some distance from its endpoints.
\begin{defn}
\label{def:HdS}
For $d\ge 0$, we denote by $\cH_{d}(S)\subset \Omega$ the event such that $\eta\in\cH_d(S)$ if there exists $Z$ a helping set in direction $u_i$ such that for all $z\in Z$, we have $\eta_z=0$, 
\begin{align}
\label{eq:helping:set:domain}
\<z,u_i\>&{}\in[0,Q],&\<z,u_{i+k}\>&{}\in\left[d,a\l_{i+k}-d\right].
\end{align}
Given a domain $\L\supset S$ and a boundary condition $\o\in\O_{\bbZ^2\setminus\L}$ we define $\cH^{\o}_d(S)=\{\h\in\O_{\L}:\o\cdot\h\in\cH_d(S)\}$. We write $\cH^\o(S)=\cH^\o_0(S)$ and $\cH(S)=\cH_0(S)$.
\end{defn}
Note that in view of \cref{def:ahelping}, if $\alpha(u_i)<\infty$, then $\cH^\o(S)\supset\cH^W(S)$ for any $\o$ with equality if $\alpha(u_i)>\alpha$. The next observation bounds the probability of the above events.
\begin{obs}[Helping set probability]
\label{obs:mu:helping}
For any $\Lambda\supset S$ and $\o\in\Omega_{\bbZ^2\setminus\Lambda}$ we have: if $\a(u_i)<\infty$, then
\[\m\left(\cH^\o(S)\right)\ge \m\left(\cH^W(S)\right)\ge 1-\left(1-q^W\right)^{\lfloor|S|/W\rfloor}\ge \max\left(q^W,1-e^{-q^{2W}|S|}\right);\]
if $\a(u_i)\le\a$, then
\[\m(\cH(S))\ge \left(1-(1-q^\a)^{\O(|S|)}\right)^{O(1)}\ge \left(1-e^{-q^\a|S|/O(1)}\right)^{O(1)}.\]
\end{obs}
\begin{proof}
Assume $\alpha(u_i)<\infty$. As already observed, by \cref{def:ahelping,def:HW,def:HdS}, $\cH^\o(S)\supset\cH^W(S)$, as $W$-helping sets in direction $u_i$ are helping sets in direction $u_i$. For the second inequality follows by dividing $S$ into disjoint groups of $W$ consecutive sites (each of which is a $W$-helping set in direction $u_i$). The final inequality follows since $|S|\ge W$ and $(1-q^W)^{1/W}\le e^{-q^W/(2W)}\le e^{-q^{2W}}$.

The case $\alpha(u_i)\le \alpha$ is treated similarly. Indeed, in order for $\cH(S)$ to occur, we need to find each of the $Q=O(1)$ pieces of a $u_i$-helping set in \cref{eq:helping:decomposition}, each of which has cardinality $\alpha$. We direct the reader to \cite{Bollobas23}*{Lemma 4.2} for more details.
\end{proof}

\subsection{Constrained Poincar\'e inequalities}
\label{subsec:poincare}
We next define the (constrained) Poincar\'e constants of various regions. For $\L\subset\bbZ^2$, $\eta,\omega\in\O$ (or possibly $\eta$ defined on a set including $\Lambda$ and $\omega$ on a set including $\bbZ^2\setminus\Lambda$) and $x\in\bbZ^2$, we denote by $c_x^{\L,\o}(\h)=c_x(\h_\L\cdot\o_{\bbZ^2\setminus\L})$ (recall \cref{eq:def:cx,eq:def:boundary:condition}) the constraint at $x$ in $\Lambda$ with boundary condition $\omega$. Given a finite $\L\subset\bbZ^2$ and a nonempty event $\SG^\bone(\L)\subset\O_\L$, let $\g(\L)$ be the smallest constant $\g\in[1,\infty]$ such that the inequality
\begin{equation}
\label{eq:def:gamma}\var_{\L}\left(f|\SG^\bone(\L)\right)\le \g\sum_{x\in\L}\m_\L\left(c_x^{\Lambda,\bone}\var_x(f)\right)
\end{equation}
holds for all $f:\O\to\bbR$. Here we recall from \cref{subsec:models} that $\m$ denotes both the product Bernoulli probability distribution with parameter $q$ and the expectation with respect to it. Moreover, for any function $\phi:\O\to\bbR$, $\m_\L(\phi)=\m(\phi(\eta)|\eta_{\bbZ^2\setminus\L})$ is the average on the configuration $\eta$ of law $\m$ in $\L$, conditionally on its state in $\bbZ^2\setminus \L$. Thus, $\m_\L(\phi)$ is a function on $\O_{\bbZ^2\setminus\L}$. Similarly, $\var_x(f)=\m(f^2(\eta)|\eta_{\bbZ^2\setminus \{x\}})-\m^2(f(\eta)|\eta_{\bbZ^2\setminus \{x\}})$ and \begin{multline*}\var_\L\left(f|\SG^\bone(\L)\right)=\m\left(\left.f^2(\eta)\right|\h_\L\in\SG^{\bone}(\Lambda),\eta_{\bbZ^2\setminus\L}\right)\\-\m^2\left(f(\eta)|\h_\L\in\SG^{\bone}(\Lambda),\eta_{\bbZ^2\setminus\L}\right).\end{multline*}

\begin{rem}
\label{rem:deconditioning}
It is important to note that in the r.h.s.\ of \cref{eq:def:gamma} we average w.r.t.\ $\m_\L$ and not $\m_\L(\cdot|\SG^\bone(\L))$ (the latter would correspond to the usual definition of Poincar\'e constant, from which we deviate). In this respect \cref{eq:def:gamma} follows \cite{Hartarsky21a}*{Eq.\ (12)} and differs from \cite{Hartarsky23FA}*{Eq.\ (4.5)}. Although this nuance is not important most of the time, this choice is crucial for the proof of \cref{th:internal:East} below.
\end{rem}

\subsection{Boundary conditions, translation invariance, monotonicity}
Let us make a few conventions in order to lighten notation throughout the paper. As we already witnessed in \cref{subsec:helping:sets}, it is often the case that much of the boundary condition is actually irrelevant for the occurrence of the event. For instance, in \cref{def:HdS}, $\cH^\o(S)$ only depends on the restriction of $\o$ to a finite-range neighbourhood of the segment $S$. Moreover, even the state in $\o$ of sites close to $S$, but in $\bbH_{u_i}$ is of no importance. Such occasions arise frequently, so, by abuse, we allow ourselves to specify a boundary condition on any region containing the sites whose state actually matters for the occurrence of the event.

We also need the following natural notion of translation invariance.
\begin{defn}[Translation invariance]
\label{def:translation}
Let $A\subset \bbR^2$. Consider a collection of events $\cE^\o(A+x)$ for $x\in\bbZ^2$ and $\o\in\O_{\bbZ^2\setminus(A+x)}$. We say that $\cE(A)$ is \emph{translation invariant}, if for all $\eta\in\O_A$, $\o\in\O_{\bbZ^2\setminus A}$ and $x\in\bbZ^2$ we have 
\[\eta\in\cE^\o(A)\Leftrightarrow \eta_{\cdot-x}\in\cE^{\o_{\cdot -x}}(A+x).\]
Similarly, we say that $\cE^\o(A)$ is \emph{translation invariant}, if the above holds for a fixed $\o\in\O_{\bbZ^2\setminus A}$.
\end{defn}
We extend the events $\cH_d(S)$, $\cH_d^\o(S)$, $\cH_d^W(S)$ from \cref{def:HW,def:HdS} in a translation invariant way. Similarly, $\cT$ and $\ST$ events for tubes defined in \cref{subsec:traversability} below and $\SG$ events for droplets defined throughout the paper are translation invariant. Therefore, we sometimes only define them for a fixed region, as we did in \cref{subsubsec:segment:helping}, but systematically extended them in a translation invariant way to all translates of this region.

We also use the occasion to point out that, just like the event $\cH_d^\o(S)$, all our $\cT$, $\ST$ and $\SG$ events are decreasing in both the configuration and the boundary condition, so that we are able to apply \cref{subsec:correlation} as needed.

\section{One-directional extensions}
\label{sec:extensions}
In this section we define our crucial one-directional CBSEP-extension and East-extension techniques (recall \cref{subsec:extensions}).

\subsection{Traversability}
\label{subsec:traversability}
We first need the following traversability $\cT$ and symmetric traversability $\ST$ events for tubes (recall \cref{def:tube}) requiring infected helping sets (recall \cref{subsubsec:segment:helping}) to appear for each of the segments composing the tube. The definition is illustrated in \cref{fig:extensions}. Recall the constant $C$ from \cref{subsec:scales}
\begin{defn}[Traversability]
\label{def:traversability}
Fix a tube $T=T(\ur,l,i)$. Assume that $i\in[4k]$ is such that $\a(u_j)<\infty$ for all $j\in(i-k,i+k)$. For $m\ge 0$ and $j\in(i-k,i+k)$ write $S_{j,m}=\bbZ^2\cap\L(\ur+m\uv_i+\r_j\ue_j)\setminus\L(\ur+m\uv_i)$. Note that $S_{j,m}$ is a discrete line segment perpendicular to $u_j$ of length $s_j-O(1)$ (recall from \cref{def:droplet} that $\us$ is the sequence of side lengths of $\Lambda(\ur)$). For $\o\in\O_{\bbZ^2\setminus \L(\ur+l\uv_i)}$ we denote by
\[\cT^\o_d(T)=\bigcap_{j,m:\varnothing\neq S_{j,m}\subset T}\cH^\o_{C^2+d}\left(S_{j,m}\right)\]
the event that $T$ is $(\o,d)$-\emph{traversable}. We set $\cT^\o(T)=\cT^\o_0(T)$.

If moreover $\alpha(u_i)<\infty$ for all $i\in[4k]$, that is, $\cU$ has a finite number of stable directions, we denote by 
\[\ST^\o_d(T)=\cT_d^\o(T)\cap\bigcap_{j:\a(u_j)\le \alpha<\a(u_{j+2k}))}\bigcap_{m:\varnothing\neq S_{j,m}\subset T}\cH^W_{C^2+d}\left(S_{j,m}\right)\]
the event that $T$ is $(\o,d)$-\emph{symmetrically traversable}.
\end{defn}
Thus, if all side lengths of $\L(\ur)$ are larger than $C^2+d$ by a large enough constant, the event $\cT^\o_d(T(\ur,s,i))$ decomposes each of the hatched parallelograms in \cref{subfig:extension:East} into line segments parallel to its side that is not parallel to $u_i$. A helping set is required for each of these segments in the direction perpendicular to them which has positive scalar product with $u_i$. The last boundedly many segments may also use the boundary condition $\o$, but it is irrelevant for the remaining ones, since it is far enough from them.

For symmetric traversability, we rather require $W$-helping sets for opposites of hard directions (recall from \cref{def:ahelping} that if the direction itself is hard, helping sets are simply $W$-helping sets). In particular, if none of the directions $u_j$ for $j\in[4k]\setminus\{i+k,i-k\}$ is hard (implying that $\cU$ is unrooted), we have $\ST^\o_d(T(\ur,l,i))=\cT^\o_d(T(\ur,l,i))$. The reason for the name ``symmetric traversability'' is that if $\cU$ has a finite number of stable directions and $\L(\ur)$ is a symmetric droplet (recall \cref{subsec:geometry}), then, for any $l>0$, $i\in[4k]$, $\o\in\O_{\bbZ^2\setminus T(\ur,l,i)}$ and $\h\in\O_{T(\ur,l,i)}$, we have
\begin{equation}
\label{eq:symmetry:tubes}
\eta\in \ST^\o_d(T(\ur,l,i))\Leftrightarrow \eta'\in \ST^{\o'}_d(T(\ur,l,i+2k)),\end{equation}
denoting by $\o'\in\O_{\bbZ^2\setminus T(\ur,l,i+2k)}$ the boundary condition obtained by rotating $\o$ by $\pi$ around the center of $\L(\ur)$ and similarly for $\eta'$. To see this, recall from \cref{subsubsec:segment:helping} that $\cH^\o(S)\supset \cH^W(S)$ with equality when $\alpha(u_i)>\alpha$ and note that the same symmetry as in \cref{eq:symmetry:tubes} holds at the level of the segment $S_{j,m}$ and its symmetric one, $S'_{j+2k,m}=\bbZ^2\cap\L(\ur+m\uv_{i+2k}+\r_{j+2k}\ue_{j+2k})\setminus\L(\ur+m\uv_{i+2k})$:
\begin{equation*}
\eta\in\begin{cases}\cH^\o_{C^2+d}(S_{j,m})&\alpha(u_{j+2k})\le\alpha\\
\cH^W_{C^2+d}(S_{j,m})&\alpha(u_{j+2k})>\alpha\end{cases}\Leftrightarrow\eta'\in\begin{cases}\cH^{\o'}_{C^2+d}(S'_{j+2k,m})& \alpha(u_j)\le\alpha\\\cH^W_{C^2+d}(S'_{j+2k,m})&\alpha(u_j)>\alpha,\end{cases}\end{equation*}
all four cases following directly from \cref{def:ahelping,def:HdS,def:HW}.

We next state a simple observation which is used frequently to modify boundary conditions as we like at little cost.
\begin{lem}[Changing boundary conditions]
\label{lem:traversability:boundary}
Let $\L(\ur)$ be a droplet, $l>0$ be a multiple of $\l_i$ and $i\in[4k]$. Assume that for any $j\in[4k]\setminus\{i-k,i+k\}$ the side length $s_j$ of $\L(\ur)$ satisfies $s_j\ge C^3$. Set $T=T(\ur,l,i)$. Then there exists a decreasing event $\cW(T)\subset\O_T$ such that $\m(\cW(T))\ge q^{O(W)}$ for any $\o\in\O_{\bbZ^2\setminus T}$ and $\eta\in \cW(T)$ we have
\[\h\in \cT^\o(T)\Leftrightarrow\h\in\cT^\bone(T).\]
Moreover, $\m(\cT^\o(T))=q^{-O(W)}\m(\cT^\bone(T))$ for all $\o\in\O_{\bbZ^2\setminus T}$. The same holds with $\ST$ instead of $\cT$.
\end{lem}
\begin{proof}
Recall the segments $S_{j,m}$ from \cref{def:traversability}. Let $\cW(T)$ be the intersection of $\cH^W_{C^2}(S_{j,m})$ for the largest sufficiently large but fixed number of values of $m$ for each $j\in(i-k,i+k)$, such that $\varnothing\neq S_{j,m}\subset T$. By \cref{obs:mu:helping} $\m(\cW(T))\ge q^{O(W)}$. Moreover, the boundary condition is irrelevant for the remaining segments, so $\cW(T)$ is indeed as desired. Finally, by \cref{eq:Harris:1} we have
\begin{align*}
\m\left(\cT^\bone(T)\right)&{}\le \m\left(\cT^\o(T)\right)\le \frac{\m(\cW(T)\cap\cT^\o(T))}{\m(\cW(T))}\\
&{}\le q^{-O(W)}\m\left(\cW(T)\cap\cT^\bone(T)\right)\le q^{-O(W)}\m\left(\cT^\bone(T)\right).\qedhere\end{align*}
\end{proof}

Another convenient property allowing us to decompose a long tube into smaller ones is the following.
\begin{lem}[Decomposing tubes]
\label{lem:tube:decomposition}
Let $T=T(\ur,l,i)$ be a tube, $\o\in\O_{\bbZ^2\setminus T}$ be a boundary condition and $s\in[0,l]$ be a multiple of $\l_i$. Set $T_1=T(\ur,s,i)$ and $T_2=su_i+T(\ur,l-s,i)$. Then
\[\h\in\cT^\o(T(\ur,l,i))\Leftrightarrow\left(\h_{T_2}\in\cT^\o(T_2)\text{ and }\h_{T_1}\in\cT^{\h_{T_2}\cdot\o}(T_1)\right)\]
and the same holds for $\ST$ instead of $\cT$.
\end{lem}
\begin{proof}This follows immediately from \cref{def:traversability}, since for each of the segments $S_{j,m}$ in \cref{def:traversability} either $S_{j,m}\subset T_1$ or $S_{j,m}\cap T_1=\varnothing$ and similarly for $T_2$ (see \cref{subfig:extension:East}).
\end{proof}
\begin{figure}
    \centering
\subcaptionbox{East-extension. The thickened tube is traversable ($\cT$).\label{subfig:extension:East}}{\centering
\begin{tikzpicture}[line cap=round,line join=round,>=triangle 45,x=1.08cm,y=1.08cm]
\fill[fill=black,fill opacity=1.0] (1,0.41) -- (0.41,1) -- (-0.41,1) -- (-1,0.41) -- (-1,-0.41) -- (-0.41,-1) -- (0.41,-1) -- (1,-0.41) -- cycle;
\fill[pattern=my north east lines] (-0.52,0.89) -- (-3.02,0.89) -- (-3.39,0.52) -- (-0.89,0.52) -- cycle;
\fill[pattern=my vertical lines] (-3.5,0.26) -- (-3.5,-0.26) -- (-1,-0.26) -- (-1,0.26) -- cycle;
\fill[pattern=my north west lines] (-3.39,-0.52) -- (-3.02,-0.89) -- (-0.52,-0.89) -- (-0.89,-0.52) -- cycle;
\draw (1,0.41)-- (0.41,1);
\draw (0.41,1)-- (-0.41,1);
\draw (-0.41,1)-- (-1,0.41);
\draw (-1,0.41)-- (-1,-0.41);
\draw (-1,-0.41)-- (-0.41,-1);
\draw (-0.41,-1)-- (0.41,-1);
\draw (0.41,-1)-- (1,-0.41);
\draw (1,-0.41)-- (1,0.41);
\draw [very thick] (-2.91,1)-- (-3.5,0.41);
\draw [very thick] (-3.5,0.41)-- (-3.5,-0.41);
\draw [very thick] (-3.5,-0.41)-- (-2.91,-1);
\draw [very thick] (-2.91,-1)-- (-0.41,-1);
\draw [very thick] (-0.41,-1)-- (-1,-0.41);
\draw [very thick] (-1,-0.41)-- (-1,0.41);
\draw [very thick] (-1,0.41)-- (-0.41,1);
\draw [very thick] (-0.41,1)-- (-2.91,1);
\draw (-0.52,0.89)-- (-3.02,0.89);
\draw (-3.02,0.89)-- (-3.39,0.52);
\draw (-3.39,0.52)-- (-0.89,0.52);
\draw (-0.89,0.52)-- (-0.52,0.89);
\draw (-3.5,0.26)-- (-3.5,-0.26);
\draw (-3.5,-0.26)-- (-1,-0.26);
\draw (-1,-0.26)-- (-1,0.26);
\draw (-1,0.26)-- (-3.5,0.26);
\draw (-3.39,-0.52)-- (-3.02,-0.89);
\draw (-3.02,-0.89)-- (-0.52,-0.89);
\draw (-0.52,-0.89)-- (-0.89,-0.52);
\draw (-0.89,-0.52)-- (-3.39,-0.52);
\begin{footnotesize}\draw (-3.5,-0.8) node {$\o$};
\draw[->] (-3.5,0)-- (-4,0);
\draw (-3.75,0.2) node {$u_i$};
\draw [decorate,decoration={brace,amplitude=5pt}] (-0.41,-1) -- (-2.91,-1) node [midway,yshift=-0.3cm] {$s$};
\draw (0,0)--(0,-1.2) node [below] {$\L(\ur)$};
\draw (-2,1)--(-2,1.2) node [above] {$T(\ur,s,i)$};
\phantom{
\draw[->] (1,0)-- (1.5,0);
\draw (1.4,0.2) node {$u_{i+2k}$};
\draw (-1,0)--(-1,-1.2) node [below] {$\L(\ur)+xu_i$};
\draw (-2.75,1)--(-2.75,1.2) node [above] {$T(\ur,s-x,i)+xu_i$};
\draw (0.25,1)--(0.25,1.2) node [above] {$T(\ur,x,i+2k)+xu_i$};}
\end{footnotesize}
\end{tikzpicture}}\quad
\subcaptionbox{CBSEP-extension. Thickened tubes are symmetrically traversable ($\ST$).\label{subfig:extension:CBSEP}}{\centering
\begin{tikzpicture}[line cap=round,line join=round,>=triangle 45,x=1.08cm,y=1.08cm]
\fill[fill=black,fill opacity=1.0] (1,0.41) -- (0.41,1) -- (-0.41,1) -- (-1,0.41) -- (-1,-0.41) -- (-0.41,-1) -- (0.41,-1) -- (1,-0.41) -- cycle;
\fill[pattern=my north east lines] (-0.52,0.89) -- (-2.02,0.89) -- (-2.39,0.52) -- (-0.89,0.52) -- cycle;
\fill[pattern=my vertical lines] (-2.5,0.26) -- (-2.5,-0.26) -- (-1,-0.26) -- (-1,0.26) -- cycle;
\fill[pattern=my north west lines] (-2.39,-0.52) -- (-2.02,-0.89) -- (-0.52,-0.89) -- (-0.89,-0.52) -- cycle;
\fill[pattern=my north west lines] (0.52,0.89) -- (1.52,0.89) -- (1.89,0.52) -- (0.89,0.52) -- cycle;
\fill[pattern=my vertical lines] (1,-0.26) -- (1,0.26) -- (2,0.26) -- (2,-0.26) -- cycle;
\fill[pattern=my north east lines] (1.89,-0.52) -- (1.52,-0.89) -- (0.52,-0.89) -- (0.89,-0.52) -- cycle;
\draw (1,0.41)-- (0.41,1);
\draw (0.41,1)-- (-0.41,1);
\draw (-0.41,1)-- (-1,0.41);
\draw (-1,0.41)-- (-1,-0.41);
\draw (-1,-0.41)-- (-0.41,-1);
\draw (-0.41,-1)-- (0.41,-1);
\draw (0.41,-1)-- (1,-0.41);
\draw (1,-0.41)-- (1,0.41);
\draw [very thick] (-1.91,1)-- (-2.5,0.41);
\draw [very thick] (-2.5,0.41)-- (-2.5,-0.41);
\draw [very thick] (-2.5,-0.41)-- (-1.91,-1);
\draw [very thick] (-1.91,-1)-- (-0.41,-1);
\draw [very thick] (-0.41,-1)-- (-1,-0.41);
\draw [very thick] (-1,-0.41)-- (-1,0.41);
\draw [very thick] (-1,0.41)-- (-0.41,1);
\draw [very thick] (-0.41,1)-- (-1.91,1);
\draw (-0.52,0.89)-- (-2.02,0.89);
\draw (-2.02,0.89)-- (-2.39,0.52);
\draw (-2.39,0.52)-- (-0.89,0.52);
\draw (-0.89,0.52)-- (-0.52,0.89);
\draw (-2.5,0.26)-- (-2.5,-0.26);
\draw (-2.5,-0.26)-- (-1,-0.26);
\draw (-1,-0.26)-- (-1,0.26);
\draw (-1,0.26)-- (-2.5,0.26);
\draw (-2.39,-0.52)-- (-2.02,-0.89);
\draw (-2.02,-0.89)-- (-0.52,-0.89);
\draw (-0.52,-0.89)-- (-0.89,-0.52);
\draw (-0.89,-0.52)-- (-2.39,-0.52);
\draw [very thick] (0.41,1)-- (1.41,1);
\draw [very thick] (1.41,1)-- (2,0.41);
\draw [very thick] (2,0.41)-- (2,-0.41);
\draw [very thick] (2,-0.41)-- (1.41,-1);
\draw [very thick] (1.41,-1)-- (0.41,-1);
\draw [very thick] (0.41,-1)-- (1,-0.41);
\draw [very thick] (1,-0.41)-- (1,0.41);
\draw [very thick] (1,0.41)-- (0.41,1);
\draw (0.52,0.89)-- (1.52,0.89);
\draw (1.52,0.89)-- (1.89,0.52);
\draw (1.89,0.52)-- (0.89,0.52);
\draw (0.89,0.52)-- (0.52,0.89);
\draw (1,-0.26)-- (1,0.26);
\draw (1,0.26)-- (2,0.26);
\draw (2,0.26)-- (2,-0.26);
\draw (2,-0.26)-- (1,-0.26);
\draw (1.89,-0.52)-- (1.52,-0.89);
\draw (1.52,-0.89)-- (0.52,-0.89);
\draw (0.52,-0.89)-- (0.89,-0.52);
\draw (0.89,-0.52)-- (1.89,-0.52);
\begin{footnotesize}
\draw (-2.5,-0.8) node {$\o$};
\draw (2,-0.8) node {$\o$};
\draw[->] (-2.5,0)-- (-3,0);
\draw (-2.75,0.2) node {$u_i$};
\draw[->] (2,0)-- (2.5,0);
\draw (2.4,0.2) node {$u_{i+2k}$};
\draw [decorate,decoration={brace,amplitude=5pt}] (-0.41,-1) -- (-1.91,-1) node [midway,yshift=-0.3cm] {$s-x$};
\draw [decorate,decoration={brace,amplitude=5pt}] (1.41,-1) -- (0.41,-1) node [midway,yshift=-0.3cm] {$x$};
\draw (0,0)--(0,-1.2) node [below] {$\L(\ur)+xu_i$};
\draw (-1.75,1)--(-1.75,1.2) node [above] {$T(\ur,s-x,i)+xu_i$};
\draw (1.25,1)--(1.25,1.2) node [above] {$T(\ur,x,i+2k)+xu_i$};
\end{footnotesize}
\end{tikzpicture}
}
\caption{One-directional extensions. The black droplet is SG. Helping sets appear on each line of the hatched parallelograms as indicated by the hatching direction. The white strips have width $\Theta(C^2)$.\label{fig:extensions}}
\end{figure}
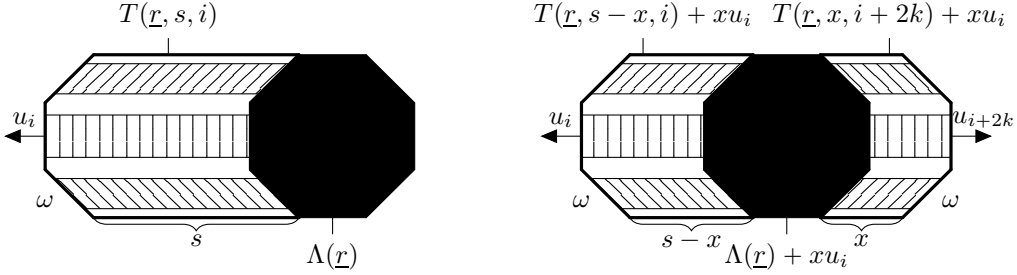

\subsection{East-extension}
\label{subsec:East:extension}
We start with the East-extension (see \cref{subfig:extension:East}), which is simpler to state. 
\begin{defn}[East-extension]
\label{def:extension:East}
Fix $i\in[4k]$, a droplet $\L(\ur)$, a multiple $l>0$ of $\l_i$ and an event $\SG^\bone(\L(\ur))\subset\O_{\L(\ur)}$. Assume that $\a(u_j)<\infty$ for all $j\in(i-k,i+k)$. We use the expression ``\emph{we East-extend $\L(\ur)$ by $l$ in direction $u_i$}'' to state that, for all $s\in(0,l]$ multiple of $\l_i$ and $\o\in\O_{\bbZ^2\setminus\L(\ur+s\uv_i)}$, we define the event $\SG^\o(\L(\ur+s\uv_i))\subset \O_{\L(\ur+s\uv_i)}$ to occur for $\h\in\O_{\L(\ur+s\uv_i)}$ if
\[\h_{\L(\ur)}\in\SG^\bone(\L(\ur))\quad\text{and}\quad
\h_{T(\ur,s,i)}\in\cT^{\o}(T(\ur,s,i)).
\]
\end{defn}
In other words, given the event $\SG^\bone$ for the droplet $\L(\ur)$, we define the event $\SG^\o$ (in particular for $\o=\bone$, but not only) for the larger droplet $\L(\ur+l\uv_i)=\L(\ur)\sqcup T(\ur,l,i)$. The event obtained on the larger droplet requires for the smaller one to be $\bone$-super good (SG) and for the remaining tube to be $\o$-traversable (recall \cref{def:traversability}). Note that these two events are independent. 
Further observe that if $\SG^\bone(\L(\ur))$ is translation invariant (recall \cref{def:translation}), then so is $\SG(\L(\ur+s\uv_i))$ for any $s\in(0,l]$ multiple of $\l_i$, defined by East-extending $\L(\ur)$ by $l$ in direction $u_i$. To get a grasp on \cref{def:extension:East}, let us note the following fact, even though it is not used directly in the proof of \cref{th:main}.
\begin{lem}[East-extension ergodicity]
\label{lem:closure}
Let $i\in[4k]$, $\L(\ur)$ be a droplet, $l$ be a multiple of $\l_i$ and $\SG^\bone(\L(\ur))\subset\O_{\L(\ur)}$ be an event. Assume that $\a(u_j)<\infty$ for all $j\in(i-k,i+k)$. Further assume that $\eta\in\SG^\bone(\L(\ur))$ implies that the $\cU$-KCM with initial condition $\h\cdot\bone_{\bbZ^2\setminus\L(\ur)}$ can entirely infect $\L(\ur)$. If we East-extend $\L(\ur)$ by $l$ in direction $u_i$, then for any $\o\in\O_{\bbZ^2\setminus\L(\ur+l\uv_i)}$ and $\eta\in\SG^\o(\L(\ur+l\uv_i))$ the $\cU$-KCM with initial condition $\o\cdot\eta$ can entirely infect $\L(\ur+l\uv_i)$.
\end{lem}
\begin{proof}
The proof is rather standard, so we only sketch the reasoning. Let $\eta\in\SG^\o(\L(\ur+l\uv_i))$. Since $\eta_{\L(\ur)}\in\SG^\bone(\L(\ur))$ by \cref{def:extension:East}, by hypothesis we can completely infect $\L(\ur)$, starting from $\o\cdot\h$. We next proceed by induction on $s\in[0,l]$ to show that we can infect $\L(\ur+s\uv_i)$. When a new site in $\bbZ^2$ is added to this set, as we increase $s$, we actually add to it an entire segment $S_{j,m}$ as in \cref{def:traversability} (at most one $m$ for each $j\in(i-k,i+k)$). Since $T(\ur,l,i)$ is $(\o,0)$-traversable, by \cref{def:HdS,def:traversability}, there is a helping set (in direction $u_j$) for this segment. As noted in \cref{subsubsec:helping:sets:line}, helping sets in direction $u_j$ together with the half-plane $\bbH_{u_j}$ infect the entire line $\Hb_{u_j}\setminus\bbH_{u_j}$ on the boundary of the half-plane. Since the helping set in our setting is only next to a finite fully infected droplet $\L(\ur+s\uv_i)$, infection spreads along its edge until it reaches a bounded distance from the corners (see \cite{Bollobas23}*{Lemma~3.4}). However, by our choice of $\hat\cS$ (recall \cref{subsec:directions}), for each $j\in[4k]$ there is a rule $X\in\cU$ such that $X\subset\Hb_{u_j}\cap\Hb_{u_{j+1}}$. Using this rule, we can infect even the remaining sites to fill up the corner between directions $u_j$ and $u_{j+1}$ of the droplet $\L(\ur+s'\uv_i)$ with $s'>s$ minimal such that $\L(\ur+s'\uv_i)\setminus\L(\ur+s\uv_i)\neq\varnothing$ (see \cite{Bollobas15}*{Lemma 5.5 and Fig. 6}).
\end{proof}

We next state a recursive bound on the Poincar\'e constant $\g$ from \cref{subsec:poincare} reflecting the recursive definition of SG events in an East-extension. In rough terms, it states that in order to relax on the larger volume, we need to be able to relax on the smaller one and additionally pay the cost of creating logarithmically many copies of it shifted by exponentially growing offsets, conditionally on the presence of the original droplet. We further need to account for the cost of microscopic dynamics (see the $e^{\log^2(1/q)}$ term below), but its contribution is unimportant. Recall $\lmp$ from \cref{subsec:scales}.
\begin{prop}[East-extension relaxation]
\label{cor:East:reduction}
Let $i\in[4k]$ be such that for all $j\in(i-k,i+k)$ we have $\a(u_j)<\infty$. Let $\L(\ur)$ be a droplet with $\ur=q^{-O(C)}$  and side lengths at least $C^3$. Let $l\in(0,\lmp]$ be a multiple of $\l_i$. Define $d_m=\l_i\lfloor(3/2)^{m}\rfloor$ for $m\in[1,M)$ and $M=\min\{m:\l_i(3/2)^m\ge l\}$. Let $d_M=l$, $\L^m=\L(\ur+d_m\uv_i)$ and $s_{m-1}=d_{m}-d_{m-1}$ for $m\in[2,M]$.

Let $\SG^\bone(\L(\ur))$ be a nonempty translation invariant decreasing event. Assume that we East-extend $\L(\ur)$ by $l$ in direction $u_i$. Then $\SG^\bone(\L(\ur+l\uv_i))$ is also nonempty, translation invariant, decreasing and satisfies
\[\g(\L(\ur+l\uv_i))\le \max\left(\g(\L(\ur)),\m^{-1}\left(\SG^\bone(\L(\ur))\right)\right)e^{O(C^2)\log^2(1/q)}\prod_{m=1}^{M-1}a_m,\]
with
\begin{equation}
\label{eq:def:am:East}
a_{m}=\m^{-1}\left(\left.\SG^\bone\left(\L^m+s_mu_i\right)\right|\SG^\bone(\L^m)\right).
\end{equation}
\end{prop}
The proof is left to \cref{app:subsec:extensions}.

\subsection{CBSEP-extension}
\label{subsec:CBSEP:extension}
We next turn our attention to CBSEP-extensions (see \cref{subfig:extension:CBSEP}). The definition differs from \cref{def:extension:East} (cf.\ \cref{subfig:extension:East}) in three ways. Firstly, we allow the smaller SG droplet to be anywhere inside the larger one (the exact position is specified by the offset below). Secondly, we ask for traversability on both sides of the smaller droplet in the direction away from it (so that infection can spread, starting from it), rather than just on one side. Thirdly, we require our tubes to be symmetrically traversable, instead of traversable. This makes the position of the small SG droplet roughly uniform.
\begin{defn}[CBSEP-extension]
\label{def:extension:CBSEP}
Assume that $\cU$ has a finite number of stable directions (equivalently, $\alpha(u_j)<\infty$ for all $j\in[4k]$). Fix $i\in[4k]$, a droplet $\L(\ur)$ and a multiple $l$ of $\l_i$. Let $\SG^\bone(\L(\ur))$ be a translation invariant event. We use the expression ``\emph{we CBSEP-extend $\L(\ur)$ by $l$ in direction $u_i$}'' to state that, for all $s\in(0,l]$ multiple of $\l_i$ and $\o\in\O_{\bbZ^2\setminus\L(\ur+s\uv_i)}$, we define the event $\SG^\o(\L(\ur+s\uv_i))\subset\O_{\L(\ur+s\uv_i)}$ as follows. 

For \emph{offsets} $x\in[0,s]$ divisible by $\l_i$ we define $\h\in\SG_x^\o(\L(\ur+s\uv_i))$ if the following all hold:
\begin{align*}
\eta_{T(\ur,s-x,i)+xu_i}&{}\in\ST^{\o}(T(\ur,s-x,i)+xu_i);\\
\h_{\L(\ur)+xu_i}&{}\in \SG^\bone(\L(\ur)+xu_i);\\
\eta_{T(\ur,x,i+2k)+xu_i}&{}\in\ST^{\o}(T(\ur,x,i+2k)+xu_i).
\end{align*}
We then set $\SG^\o(\L(\ur+s\uv_i))=\bigcup_x\SG^\o_x(\L(\ur+s\uv_i))$.
\end{defn}

Note that CBSEP-extending in direction $u_i$ gives the same result as CBSEP-extending in direction $u_{i+2k}$. We further reassure the reader that, in applications \cref{def:extension:East,def:extension:CBSEP}, are not used simultaneously for the same droplet $\L(\ur)$, so no ambiguity arises as to whether $\SG^\o(\L(\ur+l\uv_i))$ is obtained by CBSEP-extension or East-extension. However, as it is clear from \cref{tab:mechanisms}, it is sometimes necessary to CBSEP-extend a droplet itself obtained by East-extending an even smaller one. But for the time being, let us focus on a single CBSEP-extension.

The following analogue of \cref{lem:closure} holds for CBSEP-extension, which is also not used directly in the proof of \cref{th:main}. 
\begin{lem}[CBSEP-extension ergodicity]
\label{lem:closure:CBSEP}
Assume that $\cU$ has a finite number of stable directions. Let $i\in[4k]$, $\L(\ur)$ be a droplet and $l$ be a multiple of $\l_i$. Let $\SG^\bone(\L(\ur))\subset\O_{\L(\ur)}$ be translation invariant. Further assume that $\eta\in\SG^\bone(\L(\ur))$ implies that the $\cU$-KCM with initial condition $\h\cdot\bone_{\bbZ^2\setminus\L(\ur)}$ can entirely infect $\L(\ur)$. If we CBSEP-extend $\L(\ur)$ by $l$ in direction $u_i$, then for any $\o\in\O_{\bbZ^2\setminus\L(\ur+l\uv_i)}$ and $\eta\in\SG^\o(\L(\ur+l\uv_i))$ the $\cU$-KCM with initial condition $\o\cdot\eta$ can entirely infect $\L(\ur+l\uv_i)$.
\end{lem}
\begin{proof}
By \cref{def:extension:CBSEP}, it suffices to prove that for each offset $x\in[0,s]$ the conclusion holds for $\eta\in\SG^\o_x(\L(\ur+l\uv_i))$. By \cref{def:extension:CBSEP}, this implies that the events
$\SG^\bone(xu_i+\L(\ur))\cap\ST^\o(xu_i+T(\ur,s-x,i))$ and $\SG^\bone(xu_i+\L(\ur))\cap\ST^\o(xu_i+T(\ur,x,i+2k))$ hold. Moreover, by \cref{def:traversability}, $\ST^{\o'}(T)\subset \cT^{\o'}(T)$ for any tube $T$ and boundary condition $\o'$. Therefore, we may apply \cref{lem:closure} to each of the droplets $\L(\ur+x\uv_i)$ and $xu_i+\L(\ur+(s-x)\uv_i)$ (in directions $u_i$ and $u_{i+2k}$ respectively) to obtain the desired conclusion.
\end{proof}

We next state the CBSEP analogue of \cref{cor:East:reduction}, which is more involved, but also more efficient. Roughly speaking, we show that the time needed in order to relax on a CBSEP-extended droplet, is the product of four contributions: the Poincar\'e constant of the smaller droplet; the inverse probability of the symmetric traversability events in \cref{def:extension:CBSEP}; the cost of microscopic dynamics; the conditional probability of suitable contracted versions of the super good and symmetric traversability events, given the original ones (recall \cref{subsubsec:CBSEP:extension}). The last two contributions turn out to be negligible, but the last one requires some care and make the statement somewhat technical.

\begin{prop}[CBSEP-extension relaxation]
\label{cor:CBSEP:reduction}
Assume that $\cU$ has a finite number of stable directions. Let $i\in[4k]$. Let $\L(\ur)$ be a droplet with $\ur=q^{-O(C)}$ and side lengths at least $C^3$. Let  $l\in(0,\lmp]$ be a multiple of $\l_i$. Let $\SG^\bone(\L(\ur))$ be a nonempty translation invariant decreasing event.

Denote $\L_1=T(\ur,\l_i,i+2k)$, $\L_2=\L(\ur-\l_i\uv_i)$ and $\L_3=T(\ur-\l_i\uv_i,\l_i,i)$, so that $\L(\ur+\l_i\uv_i)-\l_iu_i=\L_1\sqcup\L_2\sqcup\L_3$ and $\L_2\cup\L_3=\L(\ur)=(\L_1\cup\L_2)+\l_iu_i$. Consider some nonempty decreasing events\footnote{We use a bar to denote ``contracted'' versions of events (recall \cref{subsubsec:CBSEP:extension}).} $\bSG(\L_2)\subset\O_{\L_2}$, $\bcT_{\h_2}(\L_1)\subset\O_{\L_1}$ and $\bcT_{\h_2}(\L_3)\subset\O_{\L_3}$ for all $\h_2\in\bSG(\L_2)$. Assume that
\begin{multline}
\label{eq:cor:CBSEP:reduction:condition}\left\{\h:\h_{\L_1}\in\bcT_{\h_{\L_2}}(\L_1),\h_{\L_2}\in \bSG(\L_2),\h_{\L_3}\in\bcT_{\h_{\L_2}}(\L_3)\right\}\\
\subset \SG^\bone(\L_1\cup\L_2)\cap\SG^\bone(\L_2\cup\L_3).
\end{multline}
Set $\bSG(\L_1\cup\L_2)=\{\h:\h_{\L_2}\in\bSG(\L_2),\h_{\L_1}\in\bcT_{\h_{\L_2}}(\L_1)\}$. 

If we CBSEP-extend $\L(\ur)$ by $l$ in direction $u_i$, then $\SG(\L(\ur+l\uv_i))$ is nonempty, translation invariant, decreasing and satisfies
\begin{multline*}\g(\L(\ur+l\uv_i))
\le \frac{\m(\SG^\bone(\L(\ur)))}{\m(\SG^\bone(\L(\ur+l\uv_i)))}\times\max\left(\m^{-1}\left(\SG^\bone(\L(\ur))\right),\g(\L(\ur))\right)
\\
\times\frac{e^{O(C^2)\log^2(1/q)}}{\m(\bSG(\L_1\cup\L_2)|\SG^\bone(\L_1\cup\L_2))\min_{\h_2\in\bSG(\L_2)}\m(\bcT_{\h_2}(\L_3)|\ST^{\bzero}(\L_3))}.\end{multline*}
\end{prop}
\Cref{cor:CBSEP:reduction} is proved in \cref{app:subsec:extensions} based on \cite{Hartarsky23FA}. We referring the reader to \cite{Hartarsky23FA}*{Section 
4.3} for the principles behind \cref{cor:CBSEP:reduction} in a less technical framework, but let us briefly discuss the contracted events. 

\Cref{eq:cor:CBSEP:reduction:condition} should be understood as follows. In the middle droplet $\L_2$, which has the shape of $\L(\ur)$, but contracted in direction $u_i$ by $O(1)$, we require an event $\bSG(\L_2)$. This event provides simultaneously as much of the structure of $\SG^\bone(\L_1\cup\L_2)$ and $\SG^\bone(\L_2\cup\L_3)$ (these regions both have the shape of $\L(\ur)$), as one can hope for, given that we are missing a tube of length $O(1)$ of these regions. Once such a favourable configuration $\h_2\in\bSG(\L_2)$ is fixed, the events $\bcT_{\h_{\L_2}}(\L_1)$ and $\bcT_{\h_{\L_2}}(\L_3)$ provide exactly the missing part of $\SG^\bone(\L_1\cup\L_2)$ and $\SG^\bone(\L_2\cup\L_3)$ respectively. In applications, these events necessarily need to be defined, taking into account the structure of $\SG^\bone(\L(\ur))$, on which we have made no assumptions at this point.

\subsection{Conditional probability tools}
\label{subsec:conditional}
In both \cref{cor:East:reduction,cor:CBSEP:reduction} our bounds feature certain conditional probabilities of SG events. We now provide two tools for bounding them.

The next result generalises \cite{Hartarsky23FA}*{Corollary A.3}, which relied on explicit computations unavailable in our setting. It shows that the offset of the core of a CBSEP-extended droplet (see \cref{subfig:extension:CBSEP} and recall the notation $\SG_x^\o$ from \cref{def:extension:CBSEP}) is roughly uniform. This result is the reason for the somewhat artificial \cref{def:ahelping} of helping sets and \cref{def:traversability} of $\ST$ (also see \cref{rem:symmetrisation}).
\begin{lem}[Uniform core position]
\label{lem:T:ratio}
Assume that $\cU$ has a finite number of stable directions. Fix $i\in[4k]$ and a symmetric droplet $\L=\L(\ur+l\uv_i)$ obtained by CBSEP-extension by $l$ in direction $u_i$. Assume that $l\le\lmp$ is divisible by $\l_i$ and that the side lengths of $\L(\ur)$ are at least $C^3$. Then for all $s\in[0,l]$ divisible by $\l_i$ and $\o,\o'\in\O_{\bbZ^2\setminus\L}$
\[\m\left(\left.\SG^\o_s(\L)\right|\SG^{\o'}(\L)\right)\ge q^{O(C)}.\]
\end{lem}
The proofs of \cref{lem:T:ratio,cor:perturbation} are left to \cref{app:proba}. The latter vastly generalises \cite{Hartarsky23FA}*{Lemma A.4} and is proved by different means. It is illustrated in \cref{fig:perturbation}. In words, \cref{cor:perturbation} states in a quantitative way that the conditional probability of a tube of ``critical'' size, $q^{-\a+o(1)}$, being traversable, given that a slightly perturbed version of it (shifted spatially, with different boundary condition, width of the white strips in \cref{subfig:extension:East}, radii and length) is traversable, is not very low. We note that sizes other than the critical one are not important, so cruder bounds suffice.
\begin{figure}
    \centering
    \begin{tikzpicture}[line cap=round,line join=round,>=triangle 45,x=3.0cm,y=3.0cm]
\fill[pattern=my north west lines] (-2.48,0.93) -- (-0.48,0.93) -- (-0.93,0.48) -- (-2.93,0.48) -- cycle;
\fill[pattern=my north west lines] (-3,0.31) -- (-1,0.31) -- (-1,-0.31) -- (-3,-0.31) -- cycle;
\fill[pattern=my north west lines] (-0.93,-0.48) -- (-0.48,-0.93) -- (-2.48,-0.93) -- (-2.93,-0.48) -- cycle;
\fill[pattern=my north east lines] (-2.63,0.85) -- (-0.73,0.85) -- (-0.89,0.69) -- (-2.79,0.69) -- cycle;
\fill[pattern=my north east lines] (-2.92,0.38) -- (-1.02,0.38) -- (-1.02,-0.16) -- (-2.92,-0.16) -- cycle;
\fill[pattern=my north east lines] (-2.79,-0.46) -- (-2.52,-0.73) -- (-0.62,-0.73) -- (-0.89,-0.46) -- cycle;
\fill[fill=black,fill opacity=0.5] (-2.92,0.38) -- (-2.92,0.31) -- (-1.02,0.31) -- (-1.02,0.38) -- cycle;
\fill[fill=black,fill opacity=0.5] (-2.79,-0.46) -- (-2.77,-0.48) -- (-0.93,-0.48) -- (-0.95,-0.46) -- cycle;
\draw [very thick] (-2.41,1)-- (-3,0.41);
\draw [very thick] (-3,0.41)-- (-3,-0.41);
\draw [very thick] (-3,-0.41)-- (-2.41,-1);
\draw [very thick] (-2.41,-1)-- (-0.41,-1);
\draw [very thick] (-0.41,-1)-- (-1,-0.41);
\draw [very thick] (-1,-0.41)-- (-1,0.41);
\draw [very thick] (-1,0.41)-- (-0.41,1);
\draw [very thick] (-0.41,1)-- (-2.41,1);
\draw [very thick] (-2.5,0.98)-- (-2.92,0.56);
\draw [very thick] (-2.92,0.56)-- (-2.92,-0.34);
\draw [very thick] (-2.92,-0.34)-- (-2.4,-0.86);
\draw [very thick] (-2.4,-0.86)-- (-0.5,-0.86);
\draw [very thick] (-0.5,-0.86)-- (-1.02,-0.34);
\draw [very thick] (-1.02,-0.34)-- (-1.02,0.56);
\draw [very thick] (-1.02,0.56)-- (-0.6,0.98);
\draw [very thick] (-0.6,0.98)-- (-2.5,0.98);
\draw (-2.48,0.93)-- (-0.48,0.93);
\draw (-0.48,0.93)-- (-0.93,0.48);
\draw (-0.93,0.48)-- (-2.93,0.48);
\draw (-2.93,0.48)-- (-2.48,0.93);
\draw (-3,0.31)-- (-1,0.31);
\draw (-1,0.31)-- (-1,-0.31);
\draw (-1,-0.31)-- (-3,-0.31);
\draw (-3,-0.31)-- (-3,0.31);
\draw (-0.93,-0.48)-- (-0.48,-0.93);
\draw (-0.48,-0.93)-- (-2.48,-0.93);
\draw (-2.48,-0.93)-- (-2.93,-0.48);
\draw (-2.93,-0.48)-- (-0.93,-0.48);
\draw (-2.63,0.85)-- (-0.73,0.85);
\draw (-0.73,0.85)-- (-0.89,0.69);
\draw (-0.89,0.69)-- (-2.79,0.69);
\draw (-2.79,0.69)-- (-2.63,0.85);
\draw (-2.92,0.38)-- (-1.02,0.38);
\draw (-1.02,0.38)-- (-1.02,-0.16);
\draw (-1.02,-0.16)-- (-2.92,-0.16);
\draw (-2.92,-0.16)-- (-2.92,0.38);
\draw (-2.79,-0.46)-- (-2.52,-0.73);
\draw (-2.52,-0.73)-- (-0.62,-0.73);
\draw (-0.62,-0.73)-- (-0.89,-0.46);
\draw (-0.89,-0.46)-- (-2.79,-0.46);

\draw [decorate,decoration={brace,amplitude=5pt}] (-0.41,-1) -- (-2.41,-1) node [midway,yshift= -0.4cm] {$l$};
\draw [decorate,decoration={brace,amplitude=5pt}] (-2.5,1) -- (-0.6,1) node [midway,yshift= 0.4cm] {$l'$};
\draw (-0.56,-0.8)--(-0.26,-0.5) node [above right] {$T'$};
\draw (-0.45,0.96)--(-0.26,0.77) node [below right] {$T$};
\draw[->] (-3.1,0)-- (-3.3,0);
\draw (-3.2,0) node [above] {$u_i$};
\draw[->] (-2.8,0.86)-- (-2.94,1);
\draw (-2.87,0.93) node [right] {$u_j$};
\end{tikzpicture}
    \caption{Illustration of the perturbation of \cref{cor:perturbation}. The two thickened tubes are $T$ and $T'$. The regions concerned by their traversability are hatched in different directions.}
    \label{fig:perturbation}
\end{figure}
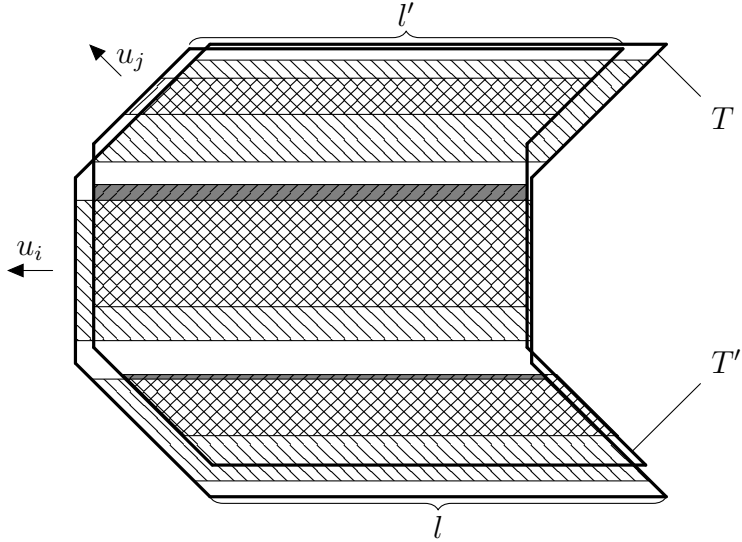
\begin{lem}[Perturbing a tube]
\label{cor:perturbation}
Let $i\in[4k]$ such that $\a(u_j)\le \a$ for all $j\in(i-k,i+k)$. Let $\L(\ur)$ be a droplet with side lengths $\us$ and let $T=T(\ur,l,i)$ be a tube. Assume that $l\in[\O(1),e^{q^{-o(1)}}]$, $s=\min_{i-k<j<i+k}s_j=q^{-\a+o(1)}$ and $\max_{i-k<j<i+k}s_j=q^{-\a+o(1)}$. For some $\D\in[C^2,s/W^2]$, let $\ur'$ and $l'$ be such that $0\le s_j-s_j'\le O(\D)$ for all $j\in(i-k,i+k)$ and $0\le l-l'\le O(\D)$, where $\us'$ are the side lengths of the droplet $\L(\ur')$. Further let $x\in\bbR^2$ be such that $\|x\|=O(\D)$ and $d,d'\in[0,O(\D)]$ with $d\le d'$. Denoting $T'=T(\ur',l',i)+x$, for any boundary conditions $\o\in\O_{\bbZ^2\setminus T}$ and $\o'\in\O_{\bbZ^2\setminus T'}$, we have
\begin{multline*}
\m\left(\left.\cT^{\o'}_{d'}(T')\right|\cT^\o_{d}(T)\right)\ge q^{O(W)}\left(1-(1-q^\a)^{\O(s)}\right)^{O(\D)}\\
\times\left(1-W\D/s-q^{1-o(1)}\right)^{O(l)}.
\end{multline*}
\end{lem}

\section{Isotropic models}
\label{sec:iso}
For this section we assume $\cU$ to be isotropic (class \ref{iso}). In this case the reasoning closely follows and generalises \cite{Hartarsky23FA}. We treat internal and mesoscopic dynamics simultaneously, since for this class there is no difference between the two.
\subsection{Isotropic internal and mesoscopic dynamics}
\label{subsec:iso}
We start by defining the geometry of our droplets and the corresponding length scales. They are all symmetric and every $2k$-th droplet is twice larger. Each such dilation is decomposed into $2k$ steps, so that their geometry fits the setting of our CBSEP-extensions from \cref{subsec:CBSEP:extension} (see \cref{subfig:iso:SG} and recall \cref{subfig:extension:CBSEP}). 

Recall \cref{subsec:geometry} and the constant $\e$ from \cref{subsec:scales}. Let $\ur^{(0)}$ be a sequence of radii with $r_i^{(0)}=r_{i+2k}^{(0)}$ for all $i\in[2k]$, such that for all $i\in[4k]$, $r_i^{(0)}=\Theta(1/\e)$ and the corresponding side length $s_i^{(0)}=\Theta(1/\e)$ is a multiple of $2\l_{i+k}$. For any integer $m\ge 0$, $i\in[2k]$ and $n=2km+r$ with $r\in[2k]$ we define
\begin{equation}
\label{eq:def:sin:iso}
s_{i}^{(n)}=s_{i+2k}^{(n)}=s^{(0)}_i2^{m}\times\begin{cases}2&k\le i< k+r\\
1&\text{otherwise}\end{cases}\end{equation}
and $\L^{(n)}=\L(\ur^{(n)})$ with $\ur^{(n)}$ the sequence of radii associated to $\us^{(n)}$ satisfying $r_{i}^{(n)}=r_{i+2k}^{(n)}$ for all $i\in[2k]$. Further set $\Nmp=2k\lceil \log (\e\lmp)/\log 2\rceil$ (recall $\lmp$ from \cref{subsec:scales}). 

Note that, as claimed, $\L^{(n)}$ are nested symmetric droplets extended in one direction at each step satisfying $\L^{(2km)}=2^m\L^{(0)}$. Moreover, they are nested so that we can define their SG events by extension (recall \cref{def:extension:CBSEP,subfig:extension:CBSEP} for CBSEP-extensions).
\begin{defn}[Isotropic SG]
\label{def:SG:iso}
Let $\cU$ be isotropic. We say that $\L^{(0)}$ is \emph{SG} ($\SG^\bone(\L^{(0)})$ occurs), if all sites in $\L^{(0)}$ are infected. We then recursively define $\SG^\bone(\L^{(n+1)})$ for $n\in[\Nmp]$ by CBSEP-extending $\L^{(n)}$ in direction $u_n$ by $l^{(n)}=s_{n+k}^{(n)}=\Theta(2^{n/2k}/\e)$ (recall from \cref{subsec:directions} that indices of directions and sequences are considered modulo $4k$ as needed and see \cref{subfig:iso:SG}).
\end{defn}
\begin{figure}
    \centering
\subcaptionbox{A generic realisation of $\SG^\bone(\L^{(n)})$ depicting the SG translates of $\L^{(n)},\dots,\L^{(n-2k)}$ involved in progressive shades of grey. Each extension is as in \cref{subfig:extension:CBSEP}.\label{subfig:iso:SG}}{\centering
    \begin{tikzpicture}[line cap=round,line join=round,>=triangle 45,x=1.5cm,y=1.5cm]
\fill[fill=black,fill opacity=1] (0.41,1) -- (-0.41,1) -- (-1,0.41) -- (-1,-0.41) -- (-0.41,-1) -- (0.41,-1) -- (1,-0.41) -- (1,0.41) -- cycle;
\fill[fill=black,fill opacity=0.2] (0.66,1) -- (1.25,0.41) -- (1.25,-0.41) -- (0.66,-1) -- (-0.99,-1) -- (-1.58,-0.41) -- (-1.58,0.41) -- (-0.99,1) -- cycle;
\fill[fill=black,fill opacity=0.2] (1.69,-0.03) -- (1.69,-0.85) -- (1.1,-1.44) -- (-0.55,-1.44) -- (-1.73,-0.27) -- (-1.73,0.56) -- (-1.14,1.15) -- (0.52,1.15) -- cycle;
\fill[fill=black,fill opacity=0.2] (0.52,1.81) -- (-1.14,1.81) -- (-1.73,1.22) -- (-1.73,-0.43) -- (-0.55,-1.61) -- (1.1,-1.61) -- (1.69,-1.02) -- (1.69,0.64) -- cycle;
\fill[fill=black,fill opacity=0.2] (1.69,0.64) -- (1.69,-1.02) -- (0.52,-2.19) -- (-1.14,-2.19) -- (-2.31,-1.02) -- (-2.31,0.64) -- (-1.14,1.81) -- (0.52,1.81) -- cycle;
\draw (0.41,1)-- (-0.41,1);
\draw (-0.41,1)-- (-1,0.41);
\draw (-1,0.41)-- (-1,-0.41);
\draw (-1,-0.41)-- (-0.41,-1);
\draw (-0.41,-1)-- (0.41,-1);
\draw (0.41,-1)-- (1,-0.41);
\draw (1,-0.41)-- (1,0.41);
\draw (1,0.41)-- (0.41,1);
\draw (0.66,1)-- (1.25,0.41);
\draw (1.25,0.41)-- (1.25,-0.41);
\draw (1.25,-0.41)-- (0.66,-1);
\draw (0.66,-1)-- (-0.99,-1);
\draw (-0.99,-1)-- (-1.58,-0.41);
\draw (-1.58,-0.41)-- (-1.58,0.41);
\draw (-1.58,0.41)-- (-0.99,1);
\draw (-0.99,1)-- (0.66,1);
\draw (1.69,-0.03)-- (1.69,-0.85);
\draw (1.69,-0.85)-- (1.1,-1.44);
\draw (1.1,-1.44)-- (-0.55,-1.44);
\draw (-0.55,-1.44)-- (-1.73,-0.27);
\draw (-1.73,-0.27)-- (-1.73,0.56);
\draw (-1.73,0.56)-- (-1.14,1.15);
\draw (-1.14,1.15)-- (0.52,1.15);
\draw (0.52,1.15)-- (1.69,-0.03);
\draw (0.52,1.81)-- (-1.14,1.81);
\draw (-1.14,1.81)-- (-1.73,1.22);
\draw (-1.73,1.22)-- (-1.73,-0.43);
\draw (-1.73,-0.43)-- (-0.55,-1.61);
\draw (-0.55,-1.61)-- (1.1,-1.61);
\draw (1.1,-1.61)-- (1.69,-1.02);
\draw (1.69,-1.02)-- (1.69,0.64);
\draw (1.69,0.64)-- (0.52,1.81);
\draw (1.69,0.64)-- (1.69,-1.02);
\draw (1.69,-1.02)-- (0.52,-2.19);
\draw (0.52,-2.19)-- (-1.14,-2.19);
\draw (-1.14,-2.19)-- (-2.31,-1.02);
\draw (-2.31,-1.02)-- (-2.31,0.64);
\draw (-2.31,0.64)-- (-1.14,1.81);
\draw (-1.14,1.81)-- (0.52,1.81);
\draw (0.52,1.81)-- (1.69,0.64);
\end{tikzpicture}
}\quad
\subcaptionbox{The setting of \cref{def:bSG:iso}. The tubes $\L_1^{(n)}$ and $\L_3^{(n)}$ of length $\l_r$ are hatched, $\L_2^{(n)}=\L^{(n)}\setminus\L_3^{(n)}$ is thickened, while the symmetrically traversable tubes are in progressive shades of grey.
\label{subfig:iso:bSG}}{\centering
\begin{tikzpicture}[line cap=round,line join=round,>=triangle 45,x=-1.5cm,y=-1.5cm]
\fill[fill=black,fill opacity=1.0] (-0.41,1) -- (0.41,1) -- (1,0.41) -- (1,-0.41) -- (0.41,-1) -- (-0.41,-1) -- (-1,-0.41) -- (-1,0.41) -- cycle;
\fill[fill=black,fill opacity=0.2] (0.66,1) -- (1.25,0.41) -- (1.25,-0.41) -- (0.66,-1) -- (-0.66,-1) -- (-1.25,-0.41) -- (-1.25,0.41) -- (-0.66,1) -- cycle;
\fill[fill=black,fill opacity=0.2] (1.54,0.12) -- (1.54,-0.71) -- (0.96,-1.29) -- (-0.37,-1.29) -- (-1.54,-0.12) -- (-1.54,0.71) -- (-0.96,1.29) -- (0.37,1.29) -- cycle;
\fill[fill=black,fill opacity=0.2] (0.37,1.71) -- (-0.96,1.71) -- (-1.54,1.12) -- (-1.54,-0.54) -- (-0.37,-1.71) -- (0.96,-1.71) -- (1.54,-1.12) -- (1.54,0.54) -- cycle;
\fill[fill=black,fill opacity=0.2] (1.83,0.83) -- (1.83,-0.83) -- (0.66,-2) -- (-0.66,-2) -- (-1.83,-0.83) -- (-1.83,0.83) -- (-0.66,2) -- (0.66,2) -- cycle;
\fill[pattern=my horizontal lines] (-0.99,2) -- (-0.83,2) -- (-2,0.83) -- (-2,-0.83) -- (-0.83,-2) -- (-0.99,-2) -- (-2.17,-0.83) -- (-2.17,0.83) -- cycle;
\fill[pattern=my horizontal lines] (0.66,2) -- (0.83,2) -- (2,0.83) -- (2,-0.83) -- (0.83,-2) -- (0.66,-2) -- (1.83,-0.83) -- (1.83,0.83) -- cycle;
\draw (-0.41,1)-- (0.41,1);
\draw (0.41,1)-- (1,0.41);
\draw (1,0.41)-- (1,-0.41);
\draw (1,-0.41)-- (0.41,-1);
\draw (0.41,-1)-- (-0.41,-1);
\draw (-0.41,-1)-- (-1,-0.41);
\draw (-1,-0.41)-- (-1,0.41);
\draw (-1,0.41)-- (-0.41,1);
\draw (0.66,1)-- (1.25,0.41);
\draw (1.25,0.41)-- (1.25,-0.41);
\draw (1.25,-0.41)-- (0.66,-1);
\draw (0.66,-1)-- (-0.66,-1);
\draw (-0.66,-1)-- (-1.25,-0.41);
\draw (-1.25,-0.41)-- (-1.25,0.41);
\draw (-1.25,0.41)-- (-0.66,1);
\draw (-0.66,1)-- (0.66,1);
\draw (1.54,0.12)-- (1.54,-0.71);
\draw (1.54,-0.71)-- (0.96,-1.29);
\draw (0.96,-1.29)-- (-0.37,-1.29);
\draw (-0.37,-1.29)-- (-1.54,-0.12);
\draw (-1.54,-0.12)-- (-1.54,0.71);
\draw (-1.54,0.71)-- (-0.96,1.29);
\draw (-0.96,1.29)-- (0.37,1.29);
\draw (0.37,1.29)-- (1.54,0.12);
\draw (0.37,1.71)-- (-0.96,1.71);
\draw (-0.96,1.71)-- (-1.54,1.12);
\draw (-1.54,1.12)-- (-1.54,-0.54);
\draw (-1.54,-0.54)-- (-0.37,-1.71);
\draw (-0.37,-1.71)-- (0.96,-1.71);
\draw (0.96,-1.71)-- (1.54,-1.12);
\draw (1.54,-1.12)-- (1.54,0.54);
\draw (1.54,0.54)-- (0.37,1.71);
\draw (1.83,0.83)-- (1.83,-0.83);
\draw (1.83,-0.83)-- (0.66,-2);
\draw (0.66,-2)-- (-0.66,-2);
\draw (-0.66,-2)-- (-1.83,-0.83);
\draw (-1.83,-0.83)-- (-1.83,0.83);
\draw (-1.83,0.83)-- (-0.66,2);
\draw (-0.66,2)-- (0.66,2);
\draw (0.66,2)-- (1.83,0.83);
\draw (-0.99,2)-- (-0.83,2);
\draw (-0.83,2)-- (-2,0.83);
\draw (-2,0.83)-- (-2,-0.83);
\draw (-2,-0.83)-- (-0.83,-2);
\draw (-0.83,-2)-- (-0.99,-2);
\draw (-0.99,-2)-- (-2.17,-0.83);
\draw (-2.17,-0.83)-- (-2.17,0.83);
\draw (-2.17,0.83)-- (-0.99,2);
\draw (0.66,2)-- (0.83,2);
\draw (0.83,2)-- (2,0.83);
\draw (2,0.83)-- (2,-0.83);
\draw (2,-0.83)-- (0.83,-2);
\draw (0.83,-2)-- (0.66,-2);
\draw (0.66,-2)-- (1.83,-0.83);
\draw (1.83,-0.83)-- (1.83,0.83);
\draw (1.83,0.83)-- (0.66,2);
\draw [very thick] (-0.83,2)-- (-2,0.83);
\draw [very thick] (-2,0.83)-- (-2,-0.83);
\draw [very thick] (-2,-0.83)-- (-0.83,-2);
\draw [very thick] (-0.83,-2)-- (0.66,-2);
\draw [very thick] (0.66,-2)-- (1.83,-0.83);
\draw [very thick] (1.83,-0.83)-- (1.83,0.83);
\draw [very thick] (1.83,0.83)-- (0.66,2);
\draw [very thick] (0.66,2)-- (-0.83,2);
\draw (-1.5,1.5)--(-1.67,1.5) node [right] {$\L_1^{(n)}$};
\draw [very thick] (-1.33,-1.5)--(-1.75,-1.5) node [right] {$\L_2^{(n)}$};
\draw (1.33,1.5)--(1.5,1.5) node [left] {$\L_3^{(n)}$};
\end{tikzpicture}
}
\caption{Geometry of isotropic $\SG$ and $\bSG$ events.\label{fig:iso}}
\end{figure}
Recall from \cref{subsec:poincare} that once $\SG^\bone(\L^{(n)})$ is defined, so is $\g(\L^{(n)})$. We next prove a bound on $\g(\L^{(n)})$.
\begin{thm}
\label{th:isotropic}
Let $\cU$ be isotropic (class \ref{iso}). Then for all $n\le\Nmp$
\begin{align*}
\g\left(\L^{(\Nmp)}\right)&{}\le \frac{\exp(1/(\log^{C/2}(1/q)q^\a))}{\m(\SG^\bone(\L^{(\Nmp)}))},&\m\left(\SG^\bone\left(\L^{(n)}\right)\right)&{}\ge \exp\left(\frac{-1}{q^\a\e^2}\right).
\end{align*}
\end{thm}
The rest of \cref{subsec:iso} is devoted to the proof of \cref{th:isotropic}. The bound on $\m(\SG^\bone(\L^{(n)}))$ is fairly standard in bootstrap percolation and could essentially be attributed to \cite{Bollobas23}, but we prove it in \cref{lem:mu:SG:bound:iso}, since we also need some better bounds on intermediate scales. Bounding $\g(\L^{(\Nmp)})$ is more demanding and is done by iteratively applying \cref{cor:CBSEP:reduction}, as suggested by \cref{def:SG:iso}.

Note that $\g(\L^{(0)})=1$, since \cref{eq:def:gamma} is trivial, because $\SG^\bone(\L^{(0)})$ is a singleton. We seek to apply \cref{cor:CBSEP:reduction}, in order to recursively upper bound $\g(\L^{(n)})$ for all $n\le \Nmp$. To that end, we need the following definition of contracted events. Since, in the language of \cref{cor:CBSEP:reduction}, the events $\bcT_{\h_2}$ we define do not depend on $\h_2$, we directly omit it from the notation.
\begin{defn}[Contracted isotropic events]
\label{def:bSG:iso}
For $n=2km+r\in[\Nmp+1]$ with $r\in[2k]$, as in \cref{cor:CBSEP:reduction} with $\ur=\ur^{(n)}$, $l=l^{(n)}$ and $i=r$, let 
\begin{align}
\nonumber\L^{(n)}_1&{}=T\left(\ur^{(n)},\l_r,n+2k\right)\\
\L^{(n)}_2&{}=\L\left(\ur^{(n)}-\l_r\uv_r\right)\label{def:L123}\\
\nonumber\L^{(n)}_3&{}=T\left(\ur^{(n)}-\l_r\uv_r,\l_r,r\right).
\end{align}

If $n<2k$, we define $\bcT(\L_1^{(n)})$, $\bSG(\L_2^{(n)})$ and $\bcT(\L_3^{(n)})$ to occur if $\L_1^{(n)}$, $\L_2^{(n)}$ and $\L_3^{(n)}$ is fully infected respectively.

For $n\ge2k$, we define $\bcT(\L_1^{(n)})\subset\O_{\L_1^{(n)}}$ (resp.\ $\bcT(\L_3^{(n)})\subset\O_{\L_3^{(n)}}$) to be the event that for every segment $S\subset\L_1^{(n)}$ (resp.\ $\L_3^{(n)}$) perpendicular to some $u_j$ with $j\neq r\pm k$ of length $2^m/(W\e)$ the event $\cH^W(S)$ occurs (recall \cref{def:HW}). Finally, for $n\ge2k$, we define $\bSG(\L_2^{(n)})$ as the intersection of the following events (see \cref{subfig:iso:bSG}):\footnote{\label{foot:STW}Recall from \cref{def:traversability} that $\ST_W$ refers to symmetric traversability with parallelograms in \cref{subfig:extension:East} shrunken by $W$, but not necessarily requiring $W$-helping sets. Further recall from \cref{subsec:traversability} that for isotropic models $\cT$ and $\ST$ events are the same.}
\begin{itemize}
    \item $\SG^\bone(\L^{(n-2k)})$;
    \item $\ST^\bone(T(\ur^{(n-2k)},l^{(n-2k)}/2-\l_r,r))\cap\ST^\bone(T(\ur^{(n-2k)},l^{(n-2k)}/2-\l_r,r+2k))$;
    \item for all $i\in(0,2k)$ \begin{multline*}\ST^\bone_W\left(T\left(\ur^{(n-2k+i)}-\l_r(\uv_{r}+\uv_{r+2k}),l^{(n-2k+i)}/2,r+i\right)\right)\\
    \cap\ST^\bone_W\left(T\left(\ur^{(n-2k+i)}-\l_r(\uv_{r}+\uv_{r+2k}),l^{(n-2k+i)}/2,r+i+2k\right)\right).
    \end{multline*}
    \item for every $i\in[2k]$, $j\in[4k]$ and segment $S\subset \L_2^{(n)}$, perpendicular to $u_j$ of length $2^m/(W\e)$ at distance at most $W$ from the $u_j$-side (parallel to $S$) of $\L^{(n-2k+i)}$, the event $\cH^W(S)$ holds.
\end{itemize}
\end{defn}
In words, $\bSG(\L_2^{(n)})$ is close to being the event that the central copy of $\L^{(n-2k)}$ in $\L_2^{(n)}$ is SG and several tubes are symmetrically traversable. Namely, for each $i\in[2k]$, the two tubes of equal length around $\L^{(n-2k+i)}$ corresponding to a CBSEP-extension by $l^{(n-2k+i)}$ in direction $u_{r}$, finally reaching $\L^{(n)}$ after $2k$ extensions. However, we have modified this event in the following ways. Firstly, the first extension is shortened by $2\l_r$, so that the final result after the $2k$ extensions fits inside $\L_2^{(n)}$ and actually only its $u_{r+k}$ and $u_{r-k}$-sides are shorter than those of $\L_2^{(n)}$ by $\l_r$ (see \cref{subfig:iso:bSG}). Secondly, the symmetric traversability events for tubes are required to occur with segments shortened by $W$ (recall \cref{def:traversability}) on each side. Finally, we roughly require $W$ helping sets for the last $O(W)$ lines of each tube, as well as the first $O(W)$ outside the tube (without going out of $\L^{(n)}_2$).

\begin{lem}[CBSEP-extension relaxation condition]
\label{lem:SG:condition:iso}
For all $n\in[\Nmp]$ we have $\bSG(\L_2^{(n)})\times\bcT(\L_3^{(n)})\subset\SG^\bone(\L_2^{(n)}\cup\L_3^{(n)})$ and similarly for $\L_1^{(n)}$ instead of $\L_3^{(n)}$.
\end{lem}
\begin{proof}
If $n<2k$, this follows directly from \cref{def:bSG:iso}, since $\bSG(\L_2^{(n)})\times\bcT(\L_3^{(n)})$ is only the fully infected configuration and similarly for $\L_1^{(n)}$. We therefore assume that $n\ge 2k$ and set $n=2km+r$ with $r\in[2k]$.

We start with the first claim. Note that $\L_2^{(n)}\cup\L_3^{(n)}=\L^{(n)}$. Let $\h\in\bSG(\L_2^{(n)})\times\bcT(\L_3^{(n)})$. We proceed by induction on $i$ to show that $\h_{\L^{(i)}}\in\SG^\bone(\L^{(i)})$ for $i\in[n-2k,n]$. 

The base is part of \cref{def:bSG:iso}. Assume $\h\in\SG^\bone(\L^{(i)})$ for some $i\in[n-2k,n)$. Then by \cref{def:extension:CBSEP}, it suffices to check that
\begin{equation}
\label{eq:iso:ST:intersection}\h\in\ST^\bone\left(T\left(\ur^{(i)},l^{(i)}/2,i\right)\right)\cap\ST^\bone\left(T\left(\ur^{(i)},l^{(i)}/2,i+2k\right)\right),
\end{equation}
since then $\h\in\SG^\bone_{l^{(i)}/2}(\L^{(i+1)})\subset\SG^\bone(\L^{(i+1)})$.

Let us first consider the case $i=n-2k$ and assume for concreteness that $m$ is even (so that $u_i=u_r$). Then 
\[\eta\in\bSG\left(\L_2^{(n)}\right)\subset\ST^\bone\left(T\left(\ur^{(i)},l^{(i)}/2-\l_r,r\right)\right),\]
so by \cref{lem:tube:decomposition} it suffices to check that $\h\in\ST^\bone(u_i(l^{(i)}/2-\l_r)+T(\ur^{(i)},\l_r,i))$, in order for the first symmetric traversability event in \cref{eq:iso:ST:intersection} to occur. We claim that this follows from $\h\in\bcT(\L_3^{(n)})$ and the fourth condition in \cref{def:bSG:iso}. To see this, notice that for each $j\in[4k]$ the $u_j$-side length of $\L(\ur^{(i)})$ satisfies $s_j^{(i)}=\Theta(s_j^{(0)}2^m)\gg 2^m/(W\e)$ by \cref{eq:def:sin:iso}. Further recall from \cref{subsec:helping:sets} that $\cH^W(S)\subset\cH^\o(S)$ for any segment $S$ of length at least $C$ and boundary condition $\o$. Thus, for each of the segments in \cref{def:traversability} for the tube $u_i(l^{(i)}/2-\l_r)+T(\ur^{(i)},\l_r,i)\subset\L^{(n)}$, we have supplied not only a helping set, but in fact several $W$-helping sets. For directions $u_j$  with $j\in(r-k,r+k)\setminus\{r\}$, they are in $\L^{(n)}_2$, while for $j=r$ they are found in $\L^{(n)}_3$, if $k=1$ and $m$ is even, and in $\L^{(n)}_2$ otherwise (see \cref{subfig:iso:bSG}). Hence, the claim is established. For the second event in \cref{eq:iso:ST:intersection} the reasoning is the same except that when $k>1$ or $m$ is even, the tube $T(\ur^{(i)},l^{(i)}/2,i+2k)$ is entirely contained in $\L_2^{(n)}$, so only $\bSG(\L_2^{(n)})$ is needed.

We next turn to the case $i\in(n-2k,n)$, which is treated similarly. Indeed,
\[\h\in\bSG\left(\L_2^{(n)}\right)\subset\ST_W^\bone\left(T\left(\ur^{(i)}-\l_r(\uv_r+\uv_{r+2k}),l^{(i)}/2,i\right)\right).\]
Comparing this tube to the desired one in \cref{eq:iso:ST:intersection}, $T(\ur^{(i)},l^{(i)}/2,i)$, we notice that the lengths and positions of their sides differ by $O(1)$ (see \cref{fig:perturbation}). However, recalling \cref{def:traversability,subfig:extension:East}, decreasing the width of each parallelogram there by $\O(W)\gg O(1)$ (using the event $\ST_W^\bone$ rather than $\ST^\bone$) is enough to compensate for this discrepancy (the shaded zones in \cref{fig:perturbation} are empty in this case). It remains to ensure that the first and last $O(1)$ segments in \cref{def:traversability} also have helping sets. But this is guaranteed by the fourth condition in \cref{def:bSG:iso} and (depending on the values of $k$, $i$ and $m$) $\bcT(\L_3^{(n)})$ exactly as in the case $i=n-2k$.

Finally, the statement for $\L_1^{(n)}$ is also proved analogously (with the offset for $i=n-2k$ modified by $\l_r$ in \cref{eq:iso:ST:intersection}), so the proof is complete.
\end{proof}

By \cref{lem:SG:condition:iso}, \cref{eq:cor:CBSEP:reduction:condition} holds, so we may apply \cref{cor:CBSEP:reduction}. This gives
\begin{align}
\nonumber\g\left(\L^{(n+1)}\right)\le{}& \max\left(\m^{-1}\left(\SG^\bone\left(\L^{(n)}\right)\right),\g\left(\L^{(n)}\right)\right)e^{O(C^2)\log^2(1/q)}\\
\label{eq:gamma:iso:step}&\times\frac{\m(\SG^\bone(\L^{(n)}))}{\m(\SG^\bone(\L^{(n+1)}))}\m^{-1}\left(\left.\bcT\left(\L_3^{(n)}\right)\right|\ST^\bzero\left(\L_3^{(n)}\right)\right)\\
\nonumber&\times\m^{-1}\left(\left.\bcT\left(\L_1^{(n)}\right)\cap\bSG\left(\L_2^{(n)}\right)\right|\SG^\bone\left(\L_1^{(n)}\cup\L_2^{(n)}\right)\right)
\end{align}
for $n\ge 2k$ and $\g(\L^{(n)})\le e^{O(C^2)\log^2(1/q)}$ for $n<2k$. We therefore assume that $n\ge 2k$. Recalling \cref{def:bSG:iso}, note that both $\bcT(\L_1^{(n)})$ and $\bcT(\L_3^{(n)})$ can be guaranteed by the presence of $O(W^2)$ well chosen infected $W$-helping sets, since only $O(W)$ disjoint segments of length $2^m/(W\e)$ perpendicular to $u_j$ for a given $j\in(r-k,r+k)$ can be fit in $\L_1^{(n)}$ or $\L^{(n)}_3$ (see \cref{subfig:iso:bSG}), so it suffices to have a $W$-helping set at each end of those. This and the Harris inequality, \cref{eq:Harris:2,eq:Harris:3}, give
\begin{gather}
\label{eq:iso:1}\m\left(\left.\bcT\left(\L_3^{(n)}\right)\right|\ST^\bzero\left(\L_3^{(n)}\right)\right)\ge \m\left(\bcT\left(\L_3^{(n)}\right)\right)\ge q^{W^{O(1)}},\\
\label{eq:iso:2}
\begin{multlined}\m\left(\left.\bcT\left(\L_1^{(n)}\right)\cap\bSG\left(\L_2^{(n)}\right)\right|\SG^\bone\left(\L_1^{(n)}\cup\L_2^{(n)}\right)\right)\\\ge q^{W^{O(1)}}\m\left(\left.\bSG\left(\L_2^{(n)}\right)\right|\SG^\bone\left(\L_1^{(n)}\cup\L_2^{(n)}\right)\right).
\end{multlined}
\end{gather}
To deal with the last term we prove the following.
\begin{lem}[Contraction rate]
\label{lem:iso:contraction:rate}
Setting $m=\lfloor n/(2k)\rfloor\ge 1$, we have
\begin{multline}
\label{eq:iso:3}\m\left(\left.\bSG\left(\L_2^{(n)}\right)\right|\SG^\bone\left(\L_1^{(n)}\cup\L_2^{(n)}\right)\right)\\\ge \begin{cases}
\m\left(\bSG\left(\L_2^{(n)}\right)\right)&2^m\le 1/\left(\log^C(1/q)q^\a\right),\\
q^{O(C)}\frac{\m(\bSG(\L_2^{(n)}))}{\m\left(\SG^\bone\left(\L^{(n-2k)}\right)\right)}&2^m\ge\log^C(1/q)/q^\a,\\
\exp\left(-2^m q^{1-o(1)}\right)&\text{otherwise}.
\end{cases}
\end{multline}
\end{lem}
\begin{proof}
The first case follows from the Harris inequality \cref{eq:Harris:2}. 

For the other two cases we start by noting that $\L^{(n)}_1\cup\L^{(n)}_2=\L^{(n)}-\l_ru_r$ may be viewed as a $2k$-fold CBSEP-extension of $\L^{(n-2k)}$. Recalling the offset in \cref{def:extension:CBSEP}, set 
\begin{align*}
\SG^\bullet_0&{}=\SG^\bone\left(\L^{(n)}-\l_ru_r\right),\\
\SG^\bullet_{i}&{}=\bigcap_{j=1}^{i}\SG^\bone_{l^{(n-j)}/2}\left(\L^{(n-j+1)}-\l_ru_r\right)\qquad i{}\in[1,2k-1],\\
\SG^\bullet_{2k}&{}=\SG^\bullet_{2k-1}\cap\SG^\bone_{l^{(n-2k)}/2-\l_r}\left(\L^{(n-2k+1)}-\l_ru_r\right),
\end{align*}
so that $\SG^\bullet_i$ corresponds to fixing the position of the core, which is a translate of $\L^{(n-i)}$, inside $\L^{(n)}-\l_ru_r$, but leaving its internal offsets unconstraint (see \cref{subfig:iso:bSG}). Thus, \cref{lem:T:ratio} applied $2k$ times gives
\[\m\left(\left.\SG^\bullet_{2k}\right|\SG^\bone\left(\L^{(n)}_1\cup\L^{(n)}_2\right)\right)=\prod_{i=1}^{2k}\m(\SG^\bullet_{i}|\SG^\bullet_{i-1})\ge q^{O(C)}.\]
Expanding the definition of $\SG^\bullet_{2k}$ via \cref{def:extension:CBSEP}, we see that this event is the intersection of $\SG^\bone(\L^{(n-2k)})$ with some increasing events (symmetrically traversable tubes) independent of the latter. Thus, the Harris inequality \cref{eq:Harris:2} gives
\begin{align}\m\left(\left.\bSG\left(\L_2^{(n)}\right)\right|\SG^\bone\left(\L_1^{(n)}\cup\L_2^{(n)}\right)\right)&{}\ge q^{O(C)}\m\left(\left.\bSG\left(\L_2^{(n)}\right)\right|\SG^\bullet_{2k}\right)\label{eq:mu:bSG:iso:bullet}\\
&{}\ge q^{O(C)}\m\left(\left.\bSG\left(\L_2^{(n)}\right)\right|\SG^\bone\left(\L^{(n-2k)}\right)\right).\nonumber\end{align}
Taking into account that $\bSG(\L^{(n)}_2)\subset\SG^\bone(\L^{(n-2k)})$ by \cref{def:bSG:iso}, this concludes the proof of the second case of \cref{eq:iso:3}.

For the third case, our starting point is again \cref{eq:mu:bSG:iso:bullet}. This time we observe that $\SG^\bullet_{2k}$ can be written as the intersection of $\SG^\bone(\L^{(n-2k)})$ with $4k$ symmetric traversability events, each of which is a perturbed version (in the sense of \cref{cor:perturbation,fig:perturbation}) of the ones appearing in \cref{def:bSG:iso} of $\bSG(\L_2^{(n)})$. Thus, the Harris inequality \cref{eq:Harris:3} allows us to lower bound $\m(\bSG(\L_2^{(n)})|\SG^\bullet_{2k})$ by 
\begin{align*}
\m(\cW)\times\m&\left(\ST^\bone\left(T\left(\ur^{(n-2k)},l^{(n-2k)}/2-\l_r,r\right)\right)\right.\\
&\,\left|\ST^\bone\left(T\left(\ur^{(n-2k)},l^{(n-2k)}/2+\l_r,r\right)-\l_ru_r\right)\right)
\\\times\m&\left(\ST^\bone\left(T\left(\ur^{(n-2k)},l^{(n-2k)}/2-\l_r,r+2k\right)\right)\right.\\
&\,\left|\ST^\bone\left(T\left(\ur^{(n-2k)},l^{(n-2k)}/2-\l_r,r+2k\right)-\l_ru_r\right)\right)
\\
\times\prod_{i=1}^{2k-1}\prod_{\xi=0}^1\m&\left(\ST_W^\bone\left(T\left(\ur^{(n-2k+i)}-\l_r\left(\uv_r+\uv_{r+2k}\right),l^{(n-2k+i)}/2,r+i+2k\xi\right)\right)\right.\\
&\,\left|\ST^\bone\left(T\left(\ur^{(n-2k+i)},l^{(n-2k+i)}/2,r+i+2k\xi\right)-\l_ru_r\right)\right),\end{align*}
where $\cW$ is the event appearing in the last item of \cref{def:bSG:iso}.

Firstly, each of the above conditional probabilities is bounded by 
\[q^{O(W)}\log^{-C^{O(1)}}(1/q)\left(1-q^{1-o(1)}\right)^{O(2^m/\e)}\ge\exp\left(-2^mq^{1-o(1)}\right),\]
using \cref{cor:perturbation} with $\D=C^2$ and recalling that $2^m=q^{-\a}\log^{O(C)}(1/q)$ and $\a\ge 1$. Secondly, $\m(\cW)\ge q^{W^{O(1)}}$ as in \cref{eq:iso:1}, concluding the proof of \cref{eq:iso:3}. We direct the reader to \cite{Hartarsky23FA}*{Appendix A} for the details of an analogous argument in a simpler setting.
\end{proof}

Iterating \cref{eq:gamma:iso:step} and plugging \cref{eq:iso:1,eq:iso:2} gives that $\g(\L^{(\Nmp)})$ is at most
\[\frac{e^{O(C^2)\Nmp\log^2(1/q)}q^{2\Nmp W^{O(1)}}}{\m(\SG^\bone(\L^{(\Nmp)}))}\prod_{n=2k}^{\Nmp-1}\m^{-1}\left(\bSG\left(\left.\L_2^{(n)}\right|\SG^\bone\left(\L^{(n)}_1\cup\L_2^{(n)}\right)\right)\right).\]
Further recalling that $\Nmp=O(\log(\lmp))=O(C\log(1/q))$ and inserting \cref{eq:iso:3}, we obtain
\begin{multline}\g\left(\L^{(\Nmp)}\right)\le \frac{e^{q^{-\a+1-o(1)}}}{\m(\SG^\bone(\L^{(\Nmp)}))}\prod_{n=2k}^{2^m\le1/(\log^C(1/q)q^\a) }\m^{-1}\left(\bSG\left(\L_2^{(n)}\right)\right)\\
\times\prod_{n:2^m\ge\log^C(1/q)/q^\a}^{\Nmp-1}\frac{\m(\SG^\bone(\L^{(n-2k)}))}{\m(\bSG(\L_2^{(n)}))}.
\label{eq:gLnmp:iso}
\end{multline}
with $m=\lfloor n/(2k)\rfloor$.
The final ingredient are the following probability bounds.

\begin{lem}[Probability of super good droplets]
\label{lem:mu:SG:bound:iso}
For $n\in[2k,\Nmp]$ and $m=\lfloor n/(2k)\rfloor$, the following bounds hold:
\begin{align}
\label{eq:mu:bSG:iso}\m\left(\bSG\left(\L_2^{(n)}\right)\right)&{}\ge \exp\left(\frac{-1}{\log^{C-3}(1/q)q^\a}\right)&\text{if }&2^m\le \frac{1}{\log^C(1/q)q^\a},\\
\label{eq:mu:bSG:iso:cond}\frac{\m(\bSG(\L_2^{(n)}))}{\m(\SG^\bone(\L^{(n-2k)}))}&{}\ge q^{W^{O(1)}}&\text{if }&2^m\ge\frac{\log^C(1/q)}{q^\a},\\
\m\left(\SG^\bone\left(\L^{(n)}\right)\right)&{}\ge \exp\left(\frac{-1}{q^\a\e^2}\right).\label{eq:mu:SG:LNm:iso}
\end{align}
\end{lem}
\begin{proof}
Let us first bound $\m(\SG^\bone(\L^{(n)}))$ for $n\le \Nmp$ by induction, starting with the trivial bound 
\begin{equation}
\label{eq:SG:bound:small:scales}\m\left(\SG^\bone\left(\L^{(2k)}\right)\right)\ge q^{|\L^{(2k)}|}\ge q^{O(1/\e)}.\end{equation}
From \cref{def:extension:CBSEP}, translation invariance and \cref{eq:symmetry:tubes}, for $n\in[2k,\Nmp-1]$ we have
\begin{align}
\nonumber\m\left(\SG^\bone\left(\L^{(n+1)}\right)\right)&{}\ge \m\left(\SG_0^\bone\left(\L^{(n+1)}\right)\right)\\
\label{eq:mu:ST:bound:iso}
&{}=\m\left(\SG^\bone\left(\L^{(n)}\right)\right)\m\left(\ST^\bone\left(T\left(\ur^{(n)},l^{(n)},n\right)\right)\right)\\
&{}\ge q^{O(1/\e)}\prod_{i=2k}^n\m\left(\ST^\bone\left(T\left(\ur^{(i)},l^{(i)},i\right)\right)\right),\nonumber
\end{align}
so we need to bound the last term. Applying \cref{def:traversability}, \cref{lem:traversability:boundary} and the Harris inequality \cref{eq:Harris:1} and then \cref{obs:mu:helping}, we get
\begin{align}
\nonumber\m\left(\ST^\bone\left(T\left(\ur^{(n)},l^{(n)},n\right)\right)\right)\ge{}&q^{O(W)}\prod_{j,m'}\cH_{C^2}\left(S_{j,m'}\right)\\
\label{eq:mu:ST:bound:iso2}\ge{}&q^{O(W)}\left(1-e^{-q^\a2^m/O(\e)}\right)^{O(2^m/\e)}\\
\nonumber\ge{}&q^{O(W)}\begin{cases}
\left(q^\a2^{m-1}\right)^{C2^m/\e}&2^m\le 1/q^{\a}\\
\exp\left(-2^m\exp\left(-q^\a2^m\right)\right)&2^m> 1/q^{\a},
\end{cases}
\end{align}
where the product runs over the segments $S_{j,m'}$ appearing in \cref{def:traversability} for the event $\ST^\bone(T(\ur^{(n)},l^{(n)},n))=\cT^\bone(T(\ur^{(n)},l^{(n)},n))$ (the last equality holds, since $\cU$ is isotropic). Plugging \cref{eq:mu:ST:bound:iso2} into \cref{eq:mu:ST:bound:iso} and iterating, we get
\begin{equation}
\label{eq:mu:SG:iso}
\m\left(\SG^\bone\left(\L^{(n)}\right)\right)\ge \begin{cases}
\exp\left(-1/\left(\log^{C-2}(1/q)q^\a\right)\right)&2^m\le 1/\left(\log^C(1/q)q^\a\right)\\
\exp\left(-1/\left(q^\a\e^2\right)\right)&2^m> 1/\left(\log^C(1/q)q^\a\right)
\end{cases}\end{equation}
since $\Nmp\le O(C)\log(1/q)$. This proves \cref{eq:mu:SG:LNm:iso}.

Recalling \cref{def:bSG:iso}, as in the proof of \cref{lem:iso:contraction:rate}, we have that for any $n\in[2k,\Nmp]$
\begin{multline*}\m\left(\bSG\left(\L_2^{(n)}\right)\right)=\m(\cW)\m\left(\SG^\bone\left(\L^{(n-2k)}\right)\right)\\
\begin{aligned}\times{}&\prod_{\xi=0}^1\m\left(\ST^\bone\left(T\left(\ur^{(n-2k)},l^{(n-2k)}/2-\l_r,r+2k\xi\right)\right)\right)\\
\times{}& \prod_{\xi=0}^1\prod_{i=1}^{2k-1}\m\left(\ST_W^\bone\left(T\left(\ur^{(n-2k+i)}-\l_r\left(\uv_r+\uv_{r+2k}\right),l^{(n-2k)}/2,r+2k\xi\right)\right)\right)
,\end{aligned}\end{multline*}
where $\cW$ is the event from the last item of \cref{def:bSG:iso} and $r=n-2km$. As in the proof of \cref{lem:iso:contraction:rate}, we have $\m(\cW)\ge q^{W^{O(1)}}$, while the factors in the products can be bounded exactly as in \cref{eq:mu:ST:bound:iso2}, entailing \cref{eq:mu:bSG:iso,eq:mu:bSG:iso:cond}, since we already have \cref{eq:mu:SG:iso}.
\end{proof}
\begin{proof}[Proof of \cref{th:isotropic}]
    The bound on $\m(\SG^\bone(\L^{(n)}))$ was proved in \cref{eq:mu:SG:LNm:iso,eq:SG:bound:small:scales}. The one on $\g(\L^{(\Nmp)})$ follows by plugging \cref{eq:mu:bSG:iso:cond,eq:mu:bSG:iso} into \cref{eq:gLnmp:iso}.
\end{proof}

\subsection{CBSEP global dynamics}
\label{subsec:global:CBSEP}
For the global dynamics we need to recall the global CBSEP mechanism introduced in \cite{Hartarsky23FA}. It is useful not only for class \ref{iso}, but also other unrooted models---classes \ref{log2} and \ref{loglog}.

Let $\Lmm$ and $\Lmp$ be droplets with side lengths $\Theta(\lmm)$ and $\Theta(\lmp)$ respectively (recall \cref{subsec:scales}). Consider a tiling of $\bbR^2$ with square boxes $Q_{i,j}=[0,\lm)\times[0,\lm)+\lm(i,j)$ for $(i,j)\in\bbZ^2$. 
\begin{defn}[Good and super good boxes]
\label{def:global:CBSEP}
We say that the box $Q_{i,j}$ is \emph{good} if for every segment $S\subset Q_{i,j}$, perpendicular to some $u\in\hS$ of length at least $\e\lmm$, $\cH^W(S)$ occurs (recall \cref{def:HW}). We denote the corresponding event by $\cG_{i,j}$. We further say that $\cG(\Lmp)$ occurs if for every segment $S\subset \Lmp$, perpendicular to some $u\in\hS$ of length at least $3\e\lmm$, the event $\cH^W(S)$ occurs.

Let $\SG^\bone(\Lmm)\subset\O_{\Lmp}$ be a nonempty translation invariant event. We say that $Q_{i,j}$ is \emph{super good} if it is good and $\SG^\bone(x+\Lmm)$ occurs for some $x\in\bbZ^2$ such that $x+\Lmm\subset Q_{i,j}$. We denote the corresponding event by $\SG_{i,j}$.
\end{defn}
In words, good boxes $Q_{i,j}$ and droplets $\Lmp$ contain $W$-helping sets in sufficient supply for a SG translate of $\Lmm$ to be able to move inside the box or droplet containing it. Our choice of $\lmm$ makes being good so likely that we are able to assume that all boxes and droplets are good at all times. Finally, a box is SG, if it also contains a SG translate of $\Lmm$ that we wish to move around. Thus, when looking at SG boxes, we essentially see a two-dimensional CBSEP dynamics, which leads to the following bound.

\begin{prop}[Global CBSEP relaxation]
\label{prop:global:CBSEP}
Let $\cU$ be unrooted (classes \ref{log2}, \ref{loglog} and \ref{iso}). Let $T=\exp(\log^4(1/q)/q^\a)$. Assume that $\SG^\bone(\Lmp)$ and $\SG^\bone(\Lmm)$ are nonempty translation invariant decreasing events such that the following conditions hold:
\begin{enumerate}
\item \label{condition:CBSEP:1} $(1-\m(\SG^\bone(\Lmm)))^TT^4=o(1)$;
\item \label{condition:CBSEP:2}for all $x\in\bbZ^2$ such that $x+\Lmm\subset \Lmp$ we have \[\SG^\bone(x+\Lmm)\cap \cG(\Lmp)\subset \SG^\bone(\Lmp).\]
\end{enumerate}Then
\[\Et\le \g\left(\Lmp\right)\frac{\log(1/\m(\SG^\bone(\Lmm)))}{q^{O(C)}}.\]
\end{prop}
We omit the proof, which is identical to \cite{Hartarsky23FA}*{Section 5}, given \cref{def:global:CBSEP},\footnote{Due to the difference between \cref{eq:def:gamma} and \cite{Hartarsky23FA}*{Eq.\ (4.5)}, the factor $\m_{\L_{i,j}}(\SG(\L_{i,j}))$ in \cite{Hartarsky23FA}*{last display of Section 5} cancels out with $\pi(\cS_1)^{-1}$ in \cite{Hartarsky23FA}*{Eq.\ (5.11)} up to a $q^{O(C)}$ factor, rather than compensating the conditioning in \cite{Hartarsky23FA}*{last display of Section 5}, which is absent in our setting.} and turn to the proof of \cref{th:main} for the isotropic class \ref{iso}.
\begin{proof}[Proof of \cref{th:main}\ref{iso}]
Let $\cU$ be isotropic. Recall the droplets $\L^{(n)}$ from \cref{subsec:iso}. Set $\Lmp=\L^{(\Nmp)}$, $\Nmm=2k\lceil\log(\e \lmm)/\log 2\rceil$ and $\Lmm=\L^{(\Nmm)}$. Thus, the side lengths of $\Lmm$ and $\Lmp$ are indeed $\Theta(\lmm)$ and $\Theta(\lmp)$ respectively by \cref{eq:def:sin:iso}. By \cref{th:isotropic}, condition \ref{condition:CBSEP:1} of \cref{prop:global:CBSEP} is satisfied:
\begin{align*}
(1-\m(\SG^\bone(\Lmm)))^TT^4&{}\le (1-e^{-1/(q^\a\e^2)})^TT^4\le T^4e^{-e^{\log^4(1/q)/q^\a-1/(q^\a\e^2)}}\\
&{}\le e^{4\log^4(1/q)/q^\a-e^{\log^4(1/q)/(2q^\a)}}=o(1).\end{align*}

We next seek to verify condition \ref{condition:CBSEP:2}. Proceeding by induction on $n\in[\Nmm,\Nmp]$, it suffices to show that for any $n\in[\Nmm,\Nmp)$ and $x,y\in\bbZ^2$ such that $x+\L^{(n)}\subset y+\L^{(n+1)}\subset \Lmp$, we have
\begin{equation}
\label{eq:cond:2:verification}\cG(\Lmp)\cap\SG^\bone(x+\L^{(n)})\subset \SG^\bone(y+\L^{(n+1)}).\end{equation}
Recalling \cref{def:SG:iso,def:extension:CBSEP}, we see that it suffices to show that for any tube $T$ of the form $z+T(\ur^{(n)},l,j)$ for some $l>0$, $j\in[4k]$ and $z\in\bbZ^2$ satisfying $T\subset y+\L^{(n+1)}$ also verifies $\cG(\Lmp)\subset\ST^\bone(T)$. Further recalling \cref{def:traversability}, we see that it suffices to show that on $\cG(\Lmp)$, each segment of length $\min_{j\in[4k]}s_j^{(n)}-C^2-O(1)$ perpendicular to $u_j$ for some $j\in[4k]$ contains an infected $W$-helping set (recall from \cref{subsubsec:segment:helping} that $\cH^W_d(S)\subset\cH^\o_d(S)$). Hence, \cref{eq:cond:2:verification} follows from \cref{def:global:CBSEP}, since 
\[\min_{j\in[4k]}s_j^{(n)}-C^2-O(1)=\Theta(\lmm)\ge 3\e\lmm.\]

Thus, we may apply \cref{prop:global:CBSEP}. Further plugging the bounds from \cref{th:isotropic}, we recover
\begin{align*}\Et&{}\le\frac{\exp(1/(\log^{C/2}(1/q)q^\a))}{\m(\SG^\bone(\L^{(\Nmp)}))}\frac{1}{q^\a\e^2q^{O(C)}}\\
&{}\le \frac{\exp(1/(\log^{C/3}(1/q)q^\a))}{\m(\SG^\bone(\L^{(\Nmp)}))}\le \exp\left(\frac{1+o(1)}{\e^2q^\a}\right),\end{align*}
concluding the proof.
\end{proof}

\section{Unbalanced unrooted models}
\label{sec:log2}
In this section we assume $\cU$ is unbalanced unrooted (class \ref{log2}). We deal with the internal, mesoscopic and global dynamics separately. The internal dynamics is very simple and already known since \cite{Hartarsky21a}. The mesoscopic and global ones are similar to the ones of \cref{sec:iso} with some adaptations needed for the mesoscopic one. 
\subsection{Unbalanced internal dynamics}
\label{subsubsec:unbal:internal}
For unbalanced unrooted $\cU$ (class \ref{log2}) the SG event on to the internal scale consists simply in having an infected ring of thickness $W$ (see \cref{fig:bSG:log2}). Recall $\li$ from \cref{subsec:scales}.
\begin{defn}[Unbalanced unrooted internal SG]
\label{def:SG:unbalanced:internal}
Assume $\cU$ is unbalanced unrooted. Let $\L^{(0)}=\L(\ur^{(0)})$ be a droplet with side lengths $s_j^{(0)}=2\l_j\lceil\li/(2\l_j)\rceil$ for $j\in[4k]$. We say that $\L^{(0)}$, is \emph{super good} ($\SG^\bone(\L^{(0)})$ occurs) if all sites in $\L^{(0)}\setminus\L(\ur^{(0)}-W\uone)$ are infected.
\end{defn}
The following result was proved in \cite{Hartarsky21a}*{Lemma 4.10} and provides the main contribution to the scaling for this class (see \cref{tab:mechanisms}).
\begin{prop}
\label{prop:unbal:internal}
For unbalanced unrooted $\cU$ (class \ref{log2}) we have
\[\max\left(\g\left(\L^{(0)}\right),\m^{-1}\left(\SG^\bone\left(\L^{(0)}\right)\right)\right)\le q^{-O(W\li)}\le\exp\left(C^3\log^2(1/q)/q^\a\right).\]
\end{prop}

\subsection{CBSEP mesoscopic dynamics}
\label{subsec:log2:meso}
Since $\cU$ is unbalanced unrooted, we may assume w.l.o.g.\ that $\a(u_j)\le\a$ for all $j\in[4k]\setminus\{k,-k\}$. We only use $4k$ scales for the mesoscopic dynamics. Recall \cref{subsec:geometry,subsec:scales}. For $i\in[0,2k]$ let $\L^{(i)}=\L(\ur^{(i)})$ be the symmetric droplet centered at $0$ with $\ur^{(i)}$ such that its associated side lengths are
\[s_j^{(i)}=s_{j+2k}^{(i)}=
\begin{cases}
2\l_j\lceil\li/(2\l_j)\rceil&i-k\le j< k\\
2\l_j\lceil\lmm/(2\l_j)\rceil&-k\le j<i-k.
\end{cases}
\]
For $i\in(2k,4k]$, we define $\L^{(i)}$ similarly by
\begin{equation}
\label{eq:def:si:2k:4k}
s_j^{(i)}=s_{j+2k}^{(i)}=
\begin{cases}
2\l_j\lceil\lmm/(2\l_j)\rceil&i-3k\le j< k\\
2\l_j\lceil\lmp/(2\l_j)\rceil&-k\le j<i-3k.
\end{cases}
\end{equation}
These droplets are exactly as in \cref{subfig:iso:SG}, except that the extensions are much longer. More precisely, we have $\L^{(i+1)}=\L(\ur^{(i)}+l^{(i)}(\uv_i+\uv_{i+2k})/2)$ with $l^{(i)}=s_{i+k}^{(i+1)}-s_{i+k}^{(i)}$, so that $l^{(i)}=(1-q^{C-\a+o(1)})\lmm$ if $i\in[2k]$ and $l^{(i)}=(1-O(\d))\lmp$ if $i\in[2k,4k)$. In particular, the droplets $\L^{(n)}$ for $n\in[4k+1]$ are nested in such a way that allows us to define their SG events by extension, as in \cref{def:SG:iso} (also recall \cref{def:SG:unbalanced:internal} for $\SG^\bone(\L^{(0)})$ and \cref{def:extension:CBSEP,subfig:extension:CBSEP} for CBSEP-extensions).
\begin{defn}[Unbalanced unrooted mesoscopic SG]
\label{def:SG:log2}
Let $\cU$ be unbalanced unrooted. For $n\in[4k]$ we define $\SG^\bone(\L^{(n+1)})$ by CBSEP-extending $\L^{(n)}$ by $l^{(n)}$ in direction $u_n$.
\end{defn}
With this definition we aim to prove the following (recall $\g(\L^{(4k)})$ from \cref{subsec:poincare}).
\begin{thm}
\label{th:log2}
Let $\cU$ be unbalanced unrooted (class \ref{log2}). Then
\[
\max\left(\g\left(\L^{(4k)}\right),\m^{-1}\left(\SG^\bone\left(\L^{(2k)}\right)\right)\right)\le\exp\left(\frac{\log^2(1/q)}{\d q^\a}\right).\]
\end{thm}
The remainder of \cref{subsec:log2:meso} is dedicated to the proof of \cref{th:log2}. Naturally, \cref{th:log2} results from $4k$ applications of \cref{cor:CBSEP:reduction} and using \cref{prop:unbal:internal} as initial input. The second step is somewhat special (see \cref{subfig:bSG:log2}), since there we need to take into account the exact structure of $\SG^\bone(\L^{(0)})$ from \cref{def:SG:unbalanced:internal} in the definition of the contracted events appearing in \cref{cor:CBSEP:reduction}. For the remaining steps the reasoning is identical to the proof of \cref{th:isotropic}, but computations are simpler, since there are only boundedly many scales. Following the proof of \cref{th:isotropic}, we start by defining our contracted events (cf.\ \cref{def:bSG:iso}).
\begin{figure}
    \centering
\begin{subfigure}{\textwidth}
\centering
\begin{tikzpicture}[line cap=round,line join=round,>=triangle 45,x=0.8cm,y=0.8cm]
\begin{scope}[shift={(0,0.05)},scale={0.95}]
\fill[fill=black,fill opacity=1.0] (-0.65,1) -- (0.65,1) -- (1.1,0.45) -- (1.1,-0.45) -- (0.45,-1.1) -- (-0.45,-1.1) -- (-1.1,-0.45) -- (-1.1,0.45) -- cycle;
\fill[color=black,fill=white,fill opacity=1.0] (-0.21,-0.5) -- (0.21,-0.5) -- (0.5,-0.21) -- (0.5,0.21) -- (0.21,0.5) -- (-0.21,0.5) -- (-0.5,0.21) -- (-0.5,-0.21) -- cycle;
\end{scope}
\fill[fill=black,fill opacity=0.5] (-0.34,0.93) -- (-3.04,0.93) -- (-3.63,0.34) -- (-3.63,-0.48) -- (-3.11,-1) -- (-0.41,-1) -- (-0.93,-0.48) -- (-0.93,0.34) -- cycle;
\fill[fill=black,fill opacity=0.5] (0.48,0.93) -- (1,0.41) -- (1,-0.41) -- (0.41,-1) -- (3.11,-1) -- (3.7,-0.41) -- (3.7,0.41) -- (3.18,0.93) -- cycle;
\draw (-0.34,0.93)-- (0.48,0.93);
\draw (0.48,0.93)-- (1,0.41);
\draw (1,0.41)-- (1,-0.41);
\draw (1,-0.41)-- (0.41,-1);
\draw (0.41,-1)-- (-0.41,-1);
\draw (-0.41,-1)-- (-0.93,-0.48);
\draw (-0.93,-0.48)-- (-0.93,0.34);
\draw (-0.93,0.34)-- (-0.34,0.93);
\draw (-3.7,-0.41)-- (-3.7,0.41);
\draw (-3.7,0.41)-- (-3.11,1);
\draw [thick] (-3.63,-0.48)-- (-3.7,-0.41);
\draw [thick] (-3.7,-0.41)-- (-3.7,0.41);
\draw [thick] (-3.7,0.41)-- (-3.11,1);
\draw [thick] (-3.11,1)-- (3.11,1);
\draw [thick] (3.11,1)-- (3.18,0.93);
\draw [thick] (3.18,0.93)-- (-3.04,0.93);
\draw [thick] (-3.04,0.93)-- (-3.63,0.34);
\draw [thick] (-3.63,0.34)-- (-3.63,-0.48);
\draw (-0.34,0.93)-- (-3.04,0.93);
\draw (-3.04,0.93)-- (-3.63,0.34);
\draw (-3.63,0.34)-- (-3.63,-0.48);
\draw (-3.63,-0.48)-- (-3.11,-1);
\draw (-3.11,-1)-- (-0.41,-1);
\draw (-0.41,-1)-- (-0.93,-0.48);
\draw (-0.93,-0.48)-- (-0.93,0.34);
\draw (-0.93,0.34)-- (-0.34,0.93);
\draw (0.48,0.93)-- (1,0.41);
\draw (1,0.41)-- (1,-0.41);
\draw (1,-0.41)-- (0.41,-1);
\draw (0.41,-1)-- (3.11,-1);
\draw (3.11,-1)-- (3.7,-0.41);
\draw (3.7,-0.41)-- (3.7,0.41);
\draw (3.7,0.41)-- (3.18,0.93);
\draw (3.18,0.93)-- (0.48,0.93);
\end{tikzpicture}
\caption{Case $n=1$. The tube $\L_3^{(n)}$ contains $W$-helping sets close to its boundaries except the one perpendicular to $u_k$.}
\label{subfig:bSG:log2}
\end{subfigure}
\begin{subfigure}{\textwidth}
\centering
    \begin{tikzpicture}[line cap=round,line join=round,>=triangle 45,x=0.8cm,y=0.8cm]
\fill[fill=black,fill opacity=1.0] (1,0.41) -- (0.41,1) -- (-0.41,1) -- (-1,0.41) -- (-1,-0.41) -- (-0.41,-1) -- (0.41,-1) -- (1,-0.41) -- cycle;
\fill[fill=black,fill opacity=0.3] (-2.91,1) -- (-3.5,0.41) -- (-3.5,-0.41) -- (-2.91,-1) -- (2.91,-1) -- (3.5,-0.41) -- (3.5,0.41) -- (2.91,1) -- cycle;
\fill[fill=black,fill opacity=0.3] (-4.68,2.77) -- (-5.27,2.18) -- (-5.27,1.35) -- (-1.15,-2.77) -- (4.68,-2.77) -- (5.27,-2.18) -- (5.27,-1.35) -- (1.15,2.77) -- cycle;
\fill[color=black,fill=white,fill opacity=1.0] (-0.25,-0.6) -- (0.25,-0.6) -- (0.6,-0.25) -- (0.6,0.25) -- (0.25,0.6) -- (-0.25,0.6) -- (-0.6,0.25) -- (-0.6,-0.25) -- cycle;
\draw (1,0.41)-- (0.41,1);
\draw (0.41,1)-- (-0.41,1);
\draw (-0.41,1)-- (-1,0.41);
\draw (-1,0.41)-- (-1,-0.41);
\draw (-1,-0.41)-- (-0.41,-1);
\draw (-0.41,-1)-- (0.41,-1);
\draw (0.41,-1)-- (1,-0.41);
\draw (1,-0.41)-- (1,0.41);
\draw (-2.91,1)-- (-3.5,0.41);
\draw (-3.5,0.41)-- (-3.5,-0.41);
\draw (-3.5,-0.41)-- (-2.91,-1);
\draw (-2.91,-1)-- (2.91,-1);
\draw (2.91,-1)-- (3.5,-0.41);
\draw (3.5,-0.41)-- (3.5,0.41);
\draw (3.5,0.41)-- (2.91,1);
\draw (2.91,1)-- (-2.91,1);
\draw (-4.68,2.77)-- (-5.27,2.18);
\draw (-5.27,2.18)-- (-5.27,1.35);
\draw (-5.27,1.35)-- (-1.15,-2.77);
\draw (-1.15,-2.77)-- (4.68,-2.77);
\draw (4.68,-2.77)-- (5.27,-2.18);
\draw (5.27,-2.18)-- (5.27,-1.35);
\draw (5.27,-1.35)-- (1.15,2.77);
\draw (1.15,2.77)-- (-4.68,2.77);
\draw [thick] (-5.02,2.91)-- (-5.61,2.32);
\draw [thick] (-5.61,2.32)-- (-5.61,2.22);
\draw [thick] (-5.61,2.22)-- (-5.02,2.81);
\draw [thick] (-5.02,2.81)-- (1.21,2.81);
\draw [thick] (1.21,2.81)-- (5.61,-1.59);
\draw [thick] (5.61,-1.59)-- (5.61,-1.49);
\draw [thick] (5.61,-1.49)-- (1.21,2.91);
\draw [thick] (1.21,2.91)-- (-5.02,2.91);
\draw (-5.61,1.49)-- (-5.61,2.32);
\draw (-5.61,2.32)-- (-5.02,2.91);
\draw (-5.02,2.91)-- (1.21,2.91);
\draw (1.21,2.91)-- (5.61,-1.49);
\draw (5.61,-1.49)-- (5.61,-2.32);
\draw (5.61,-2.32)-- (5.02,-2.91);
\draw (5.02,-2.91)-- (-1.21,-2.91);
\draw (-5.61,1.49)-- (-1.21,-2.91);

\end{tikzpicture}
\caption{Case $n=2$. Regions around all boundaries contain $W$-helping sets.}
\label{subfig:bSG:log2:2}
\end{subfigure}
\caption{The events $\bSG(\L^{(n)}_2)$ and $\bcT(\L_3^{(n)})$ of \cref{def:bSG:log2}. $\L_3^{(n)}$ is thickened. Black regions are entirely infected. Shaded tubes are $(\bone,W)$-symmetrically traversable.\label{fig:bSG:log2}}
\end{figure}

\begin{defn}[Contracted unbalanced unrooted events]
\label{def:bSG:log2}
For $n=2km+r\in[4k+1]$ and $r\in[2k]$, define $\L_1^{(n)},\L^{(n)}_2,\L_3^{(n)}$ by \cref{def:L123}.

Let $\bcT(\L^{(0)}_1)$ (resp.\ $\bcT(\L^{(0)}_3)$) be the events that $\L^{(0)}_1$ (resp.\ $\L^{(0)}_3$) is fully infected and $\bSG(\L^{(0)}_2)$ be the event that $\L^{(0)}_2\setminus\L(\ur^{(0)}-2W\uone)$ is fully infected.

Let $\bSG(\L^{(1)}_2)$ occur if the following all hold (see \cref{subfig:bSG:log2}):\textsuperscript{\ref{foot:STW}}
\begin{itemize}
    \item $\ST^\bone_W(T(\ur^{(0)}-\l_1\uv_1,l^{(0)}/2,0))$ occurs,
    \item $(\L(\ur^{(0)}+W\uone)\setminus\L(\ur^{(0)}-2W\uone))\cap\L_2^{(1)}$ is fully infected,
    \item $\ST^\bone_W(T(\ur^{(0)}-\l_1\uv_1,l^{(0)}/2,2k))$ occurs,
    \item for all $j\neq\pm k$ and segment $S\subset \L_2^{(1)}$, perpendicular to $u_j$ at distance at most $W$ from the $u_j$-side of $\L_2^{(1)}$ and of length $\li/W$, the event $\cH^W(S)$ occurs.
\end{itemize}
Further let $\bcT(\L_1^{(1)})$ occur if the following both hold (see \cref{subfig:bSG:log2}):
\begin{itemize}
\item $\L(\ur^{(0)}+W\uone)\cap\L_1^{(1)}$ is fully infected,
\item for all $j\neq\pm k$ and segment $S\subset \L_1^{(1)}$ perpendicular to $u_j$ of length $\li/W$ the event $\cH^W(S)$ occurs.
\end{itemize}
We define $\bcT(\L_3^{(1)})$ analogously.

Let $i\in[2,4k)$. We say that $\bcT(\L^{(i)}_1)$ occurs (see \cref{subfig:bSG:log2:2}) if for all $j\in[4k]$ and $m\in\{i-1,i\}$ every segment $S\subset\L^{(i)}_1$ perpendicular to $u_j$ of length $s_j^{(m)}/W$ at distance at most $W$ from the $u_j$-side (parallel to $S$) of $\L^{(m)}$, the event $\cH^W(S)$ occurs. We define $\bcT(\L^{(i)}_3)$ similarly. Let $\bSG(\L^{(i)}_2)$ occur if the following all hold (see \cref{subfig:bSG:log2:2}):
\begin{itemize}
    \item $\SG^\bone(\L^{(i-2)})$ occurs;
    \item for each $m\in\{0,2k\}$ the following occurs
    \begin{multline*}
    \ST^\bone_W\left(T\left(\ur^{(i-2)},l^{(i-2)}/2-\sqrt{W},i-2+m\right)\right)\\
    \cap\ST^\bone_W\left(T\left(\ur^{(i-1)}-\sqrt{W}\left(\uv_i+\uv_{i+2k}\right),l^{(i-1)}/2-\sqrt{W},i-1+m\right)\right);\end{multline*}
    \item for all $j\in[4k]$, $m\in\{i-2,i-1,i\}$ and segment $S\subset\L_2^{(i)}$, perpendicular to $u_j$ of length $s_j^{(m)}/W$ at distance at most $W$ from the $u_j$-side of $\L^{(m)}$, the event $\cH^W(S)$ holds.
\end{itemize}
\end{defn}
Before moving on, let us make a few comments on how \cref{def:bSG:log2} of $\bSG(\L^{(n)}_1)$ and $\bcT(\L^{(n)}_3)$ is devised. Recall that our goal is to satisfy \cref{eq:cor:CBSEP:reduction:condition}, that is, $\bSG(\L_2^{(n)})\times\bcT(\L_3^{(n)})\subset\SG^\bone(\L^{(n)})$, so as to apply \cref{cor:CBSEP:reduction}. For that reason, for the various values of $n$, we have required the (more than) parts of the event $\SG^\bone(\L^{(n)})$ which can be witnessed in each of $\L_2^{(n)}$ and $\L_3^{(n)}$. Since $\SG^\bone(\L^{(0)})$ corresponds to an infected ring of width roughly $W$ and radius being fully infected (see \cref{def:SG:unbalanced:internal}), we have required for $n\in\{0,1\}$
a ring of the same radius, but three times thicker to be infected. Similarly to \cref{def:bSG:iso}, we have slightly reduced the length of traversable tubes present in (recall \cref{def:SG:log2}), but thinned the corresponding parallelograms in \cref{subfig:extension:CBSEP}. We have further asked for $W$-helping sets around all boundaries so as to compensate for the shortening of the tubes. The construction takes advantage of the fact that for $n\ge 2$ the droplet $\L^{(n-2)}$ is far from the boundaries of $\L^{(n)}$ (see \cref{subfig:bSG:log2:2}), so the event $\SG^\bone(\L^{(n-2)})$ can be directly incorporated into $\bSG(\L^{(n)}_2)$, rather than being decomposed into one part in $\L_2^{(n)}$ and one in $\L_3^{(n)}$.

\begin{lem}[CBSEP-extension relaxation condition]
\label{lem:log2:condition}
For all $n\in[4k]$ we have $\bSG(\L_2^{(n)})\times\bcT(\L_3^{(n)})\subset\SG^\bone(\L_2^{(n)}\cup\L_3^{(n)})$ and similarly for $\L_1^{(n)}$ instead of $\L_3^{(n)}$.
\end{lem}
\begin{proof}
The proof for $n\ge 2$ is essentially identical to the one of \cref{lem:SG:condition:iso} and $n=0$ is immediate from \cref{def:bSG:log2,def:SG:unbalanced:internal}. We therefore focus on the case $n=1$ and on $\L_3^{(1)}$, since $\L_1^{(1)}$ is treated analogously. Assume $\bSG(\L_2^{(1)})$ and $\bcT(\L_3^{(1)})$ occur. Recalling \cref{def:extension:CBSEP}, it suffices to prove that $\SG^\bone_{l^{(0)}/2}(\L^{(1)})$ occurs.

Firstly, note that \begin{align*}
\ST^\bone\left(T\left(\ur^{(0)},l^{(0)}/2,2k\right)\right)\supset\ST^\bone_W\left(T\left(r^{(0)}-\l_1\uv_1,l^{(0)}/2,2k\right)\right),
\end{align*}
recalling from \cref{eq:def:uvi} that $\<\uv_1,\ue_j\>=0$ for all $j\in\{k+1,\dots,3k-1\}$ and $\<\uv_1,\ue_j\>\le O(1)\ll W$ for $j\in\{k,3k\}$. Similarly, for any $\eta\in\O_{\L^{(1)}}$ we have
\[\eta\in\ST^\bone_W\left(T\left(\ur^{(0)}-\l_1\uv_1,l^{(0)}/2,0\right)\right)\Rightarrow\eta\in \ST^{\eta_{\L^{(1)}\setminus T}\cdot\bone_{\bbZ^2\setminus\L^{(1)}}}(T),\]
where $T=(\ur^{(0)},l^{(0)}/2-\l_1/\<u_1,u_0\>,0)$. Furthermore, the fourth condition in the definition of $\bSG(\L_2^{(1)})$ and the second condition in the definition of $\bcT(\L_3^{(1)})$ (see \cref{def:bSG:log2}) imply the occurrence of $\ST^\bone(u_0(l^{(0)}/2-\l_1/\<u_1,u_0\>)+T(\ur^{(0)},\l_1/\<u_1,u_0\>,0))$. Using \cref{lem:tube:decomposition} to combine these two facts, we obtain that $\ST^\bone(\ur^{(0)},l^{(0)}/2,0)$ occurs.

Thus, it remains to show that $\SG^\bone(\L^{(0)})$ occurs. But, in view of \cref{def:SG:unbalanced:internal}, this is the case by the second condition in the definition of $\bSG(\L_2^{(1)})$ and the first condition of $\bcT(\L_3^{(1)})$ (see \cref{def:bSG:log2}).
\end{proof}

\begin{proof}[Proof of \cref{th:log2}]
By \cref{lem:log2:condition}, \cref{eq:cor:CBSEP:reduction:condition} holds, so we may apply \cref{cor:CBSEP:reduction}. Together with the Harris inequality \cref{eq:Harris:2}, this gives
\begin{equation}
\label{eq:log2:gamma}
\g\left(\L^{(4k)}\right)\le\frac{\g(\L^{(0)})\exp(O(C^{2})\log^2(1/q))}{\prod_{i\in[4k]}\m(\SG^\bone(\L^{(i+1)}))\m(\bcT(\L^{(i)}_1))\m(\bSG(\L_2^{(i)}))\m(\bcT(\L^{(i)}_3))}.
\end{equation}
In view of \cref{prop:unbal:internal}, in order to prove \cref{th:log2}, it suffices to prove that each of the terms in the denominator of \cref{eq:log2:gamma} is at least $\exp(-C^{O(1)}\log^2(1/q)/q^\a)$.

Inspecting \cref{def:bSG:log2,def:SG:log2}, we see that each $\SG$, $\bSG$ and $\bcT$ event in \cref{eq:log2:gamma} requires at most $C\li$ fixed infections, $W^{O(1)}$ $W$-helping sets and $O(1)$ $(\bone,W)$-symmetrically traversable tubes. We claim that the probability of each tube being $(\bone,W)$-symmetrically traversable is $q^{O(W)}$. Assuming this, the Harris inequality \cref{eq:Harris:1} and the above give that, for all $i\in[4k+1]$,
\[\m\left(\SG^\bone\left(\L^{(i)}\right)\right)\ge q^{C\li}q^{W^{O(1)}}q^{O(W)}=\exp\left(-C^{O(1)}\log^2(1/q)/q^\a\right)\]
and similarly for the other events.

To prove the claim, let us consider for concreteness and notational convenience the event 
\[\cE=\ST_W^\bone\left(T\left(\ur^{(1)},l^{(1)},1\right)\right),\]
all tubes being treated identically. As in \cref{eq:mu:ST:bound:iso2}, applying \cref{def:traversability,lem:traversability:boundary,obs:mu:helping}, we get 
\begin{equation}
\label{eq:mu:E:bound}\m(\cE)\ge q^{O(W)}\left(1-e^{-q^\a\li/O(1)}\right)^{O(l^{(1)})}\left(1-e^{-q^W\lmm/O(W)}\right)^{O(l^{(1)})}.\end{equation}
Here we noted that in directions $i\in(-k+2,k-1)$ symmetric traversability only requires helping sets (since the only hard directions are assumed to be $u_k$ and $u_{-k}$) and the corresponding side lengths of $\L^{(1)}$ are $\li+O(1)$, while for $i=k$ it requires $W$-helping sets, but the $u_k$-side of $\L^{(1)}$ has length $\lmm+O(1)$. Recalling \cref{subsec:scales} and the fact that $l^{(1)}=\Theta(\lmm)$, \cref{eq:mu:E:bound} becomes $\m(\cE)\ge q^{O(W)}$,
as claimed.
\end{proof}

\subsection{CBSEP global dynamics}
\label{subsec:log2:global}
With \cref{th:log2} established, we are ready to conclude the proof of \cref{th:main}\ref{log2} as in \cref{subsec:global:CBSEP}.
\begin{proof}[Proof of \cref{th:main}\ref{log2}]
Let $\cU$ be unbalanced unrooted. Recall the droplets $\L^{(n)}$ from \cref{subsec:log2:meso}. Set $\Lmp=\L^{(4k)}$ and $\Lmm=\L^{(2k)}$. Condition \ref{condition:CBSEP:1} of \cref{prop:global:CBSEP} is satisfied by \cref{th:log2}, while condition \ref{condition:CBSEP:2} is verified as in \cref{subsec:global:CBSEP}.

Thus, \cref{prop:global:CBSEP} applies and, together with \cref{th:log2}, it yields
\[\Et\le\exp\left(\frac{\log^2(1/q)}{\e q^\a}\right),\]
concluding the proof.
\end{proof}

\section{Semi-directed models}
\label{sec:loglog}
In this section we aim to treat semi-directed update families $\cU$ (class \ref{loglog}). The internal dynamics (\cref{subsec:loglog:internal}) based on East extensions is the most delicate. The mesoscopic and global dynamics (\cref{subsec:loglog:meso,subsec:loglog:global}) use the CBSEP mechanism along the same lines as in \cref{sec:iso,sec:log2}.

\subsection{East internal dynamics}
\label{subsec:loglog:internal}
In view of \cref{rem:logloglog}, in \cref{subsec:loglog:internal} we work not only with semi-directed models (class \ref{loglog}), but slightly more generally, in order to also treat balanced rooted models with finite number of stable directions (class \ref{log1}), whose update rules are contained in the axes of the lattice (in which case $k=1$---recall \cref{subsec:directions}). In either case we have that $\a(u_j)\le\a$ for all $j\in[4k]\setminus\{3k-1,3k\}$ and this is the only assumption on $\cU$ we use. 

Recalling \cref{subsec:scales}, set
\begin{align}
\Nc&{}=\min\{n:W^n\ge q^{-\a}\}=\lceil\a\log(1/q)/\log W\rceil,\nonumber\\
\Ni&{}=\min\left\{n:\left\lceil W^{\exp(n-\Nc)}/q^\a\right\rceil\ge\li\e\right\},\nonumber\\
&{}=\Nc+\log\log\log (1/q)+O(\log\log W),
\label{eq:def:ln:internal:East}\\
\ell^{(n)}&{}=\begin{cases}
W^n&n\le \Nc,\\
\left\lceil W^{\exp(n-\Nc)}/q^\a\right\rceil&\Nc<n\le \Ni.
\end{cases}\nonumber
\end{align}
\begin{rem}
\label{rem:scales}
Note that despite the extremely fast divergence of $\ell^{(n)}q^\a$, for $n\in(\Nc,\Ni]$ it holds that $W\le\ell^{(n+1)}/\ell^{(n)}<(\ell^{(n)}q^\a)^2<\log^4(1/q)$. The sharp divergence ensures that some error terms below sum to the largest one. This prevents additional factors of the order of $\Ni-\Nc$ in the final answer, particularly for the semi-directed class \ref{loglog} (recall \cref{subsubsec:East:internal:overview}). This technique was introduced in \cite{Hartarsky19}*{Eq.\ (16)}, while the geometrically increasing scale choice relevant for small $n$ originates from \cite{Gravner12}. It should be noted that this divergence can be further amplified up to a tower of exponentials of height linear in $n-\Nc$. In that case the $\log\log\log(1/q)$ error term in \cref{th:internal:East,eq:rem:logloglog} below becomes $\log_*(1/q)$, but is, alas, still divergent.
\end{rem}
Recall \cref{subsec:geometry}. Let $\ur^{(0)}=(r^{(0)}_j)_{j\in[4k]}$ be a symmetric sequence of radii such that $\ur=\Theta(1/\e)$, the vertices of $\L(\ur^{(0)})$ are in $2\bbZ^2$ and the corresponding side lengths $\us^{(0)}$ are also $\Theta(1/\e)$. For $n\in\bbN$ and $j\in[4k]$, we define $s_j^{(n)}=s_j^{(0)}\ell^{(n)}$. We denote $\L^{(n)}=\L(\ur^{(n)})$, where $\ur^{(n)}$ is the sequence of radii corresponding to $\us^{(n)}$ such that $r^{(n)}_{3k}=r_{3k}^{(0)}$ and $r^{(n)}_{3k-1}=r^{(0)}_{3k-1}$ (see \cref{fig:East:internal:loglog}).

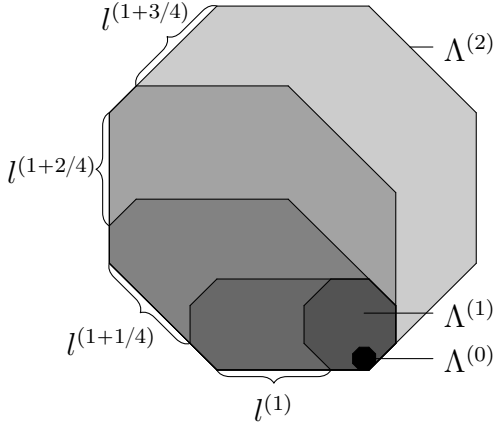
\begin{figure}
\floatbox[{\capbeside\thisfloatsetup{capbesideposition={right,top}}}]{figure}[\FBwidth]
{\centering
\begin{tikzpicture}[line cap=round,line join=round,>=triangle 45,x=0.022\textwidth,y=0.022\textwidth]
\fill[fill=black,fill opacity=1.0] (-0.21,0.5) -- (0.21,0.5) -- (0.5,0.21) -- (0.5,-0.21) -- (0.21,-0.5) -- (-0.21,-0.5) -- (-0.5,-0.21) -- (-0.5,0.21) -- cycle;
\fill[fill=black,fill opacity=0.2] (-1.45,3.5) -- (0.21,3.5) -- (1.38,2.33) -- (1.38,0.67) -- (0.21,-0.5) -- (-1.45,-0.5) -- (-2.62,0.67) -- (-2.62,2.33) -- cycle;
\fill[fill=black,fill opacity=0.2] (-6.42,15.5) -- (0.21,15.5) -- (4.89,10.81) -- (4.89,4.19) -- (0.21,-0.5) -- (-6.42,-0.5) -- (-11.11,4.19) -- (-11.11,10.81) -- cycle;
\fill[fill=black,fill opacity=0.2] (-6.42,-0.5) -- (-7.59,0.67) -- (-7.59,2.33) -- (-6.42,3.5) -- (0.21,3.5) -- (1.38,2.33) -- (1.38,0.67) -- (0.21,-0.5) -- cycle;
\fill[fill=black,fill opacity=0.2] (-11.11,4.19) -- (-11.11,5.84) -- (-9.94,7.01) -- (-3.31,7.01) -- (1.38,2.33) -- (1.38,0.67) -- (0.21,-0.5) -- (-6.42,-0.5) -- cycle;
\fill[fill=black,fill opacity=0.2] (-9.94,11.99) -- (-3.31,11.99) -- (1.38,7.3) -- (1.38,0.67) -- (0.21,-0.5) -- (-6.42,-0.5) -- (-11.11,4.19) -- (-11.11,10.81) -- cycle;
\draw (-0.21,0.5)-- (0.21,0.5);
\draw (0.21,0.5)-- (0.5,0.21);
\draw (0.5,0.21)-- (0.5,-0.21);
\draw (0.5,-0.21)-- (0.21,-0.5);
\draw (0.21,-0.5)-- (-0.21,-0.5);
\draw (-0.21,-0.5)-- (-0.5,-0.21);
\draw (-0.5,-0.21)-- (-0.5,0.21);
\draw (-0.5,0.21)-- (-0.21,0.5);
\draw (-1.45,3.5)-- (0.21,3.5);
\draw (0.21,3.5)-- (1.38,2.33);
\draw (1.38,2.33)-- (1.38,0.67);
\draw (1.38,0.67)-- (0.21,-0.5);
\draw (0.21,-0.5)-- (-1.45,-0.5);
\draw (-1.45,-0.5)-- (-2.62,0.67);
\draw (-2.62,0.67)-- (-2.62,2.33);
\draw (-2.62,2.33)-- (-1.45,3.5);
\draw (-6.42,15.5)-- (0.21,15.5);
\draw (0.21,15.5)-- (4.89,10.81);
\draw (4.89,10.81)-- (4.89,4.19);
\draw (4.89,4.19)-- (0.21,-0.5);
\draw (0.21,-0.5)-- (-6.42,-0.5);
\draw (-6.42,-0.5)-- (-11.11,4.19);
\draw (-11.11,4.19)-- (-11.11,10.81);
\draw (-11.11,10.81)-- (-6.42,15.5);
\draw (-6.42,-0.5)-- (-7.59,0.67);
\draw (-7.59,0.67)-- (-7.59,2.33);
\draw (-7.59,2.33)-- (-6.42,3.5);
\draw (-6.42,3.5)-- (0.21,3.5);
\draw (0.21,3.5)-- (1.38,2.33);
\draw (1.38,2.33)-- (1.38,0.67);
\draw (1.38,0.67)-- (0.21,-0.5);
\draw (0.21,-0.5)-- (-6.42,-0.5);
\draw (-11.11,4.19)-- (-11.11,5.84);
\draw (-11.11,5.84)-- (-9.94,7.01);
\draw (-9.94,7.01)-- (-3.31,7.01);
\draw (-3.31,7.01)-- (1.38,2.33);
\draw (1.38,2.33)-- (1.38,0.67);
\draw (1.38,0.67)-- (0.21,-0.5);
\draw (0.21,-0.5)-- (-6.42,-0.5);
\draw (-6.42,-0.5)-- (-11.11,4.19);
\draw (-9.94,11.99)-- (-3.31,11.99);
\draw (-3.31,11.99)-- (1.38,7.3);
\draw (1.38,7.3)-- (1.38,0.67);
\draw (1.38,0.67)-- (0.21,-0.5);
\draw (0.21,-0.5)-- (-6.42,-0.5);
\draw (-6.42,-0.5)-- (-11.11,4.19);
\draw (-11.11,4.19)-- (-11.11,10.81);
\draw (-11.11,10.81)-- (-9.94,11.99);
\draw [decorate,decoration={brace,amplitude=5pt}] (-1.45,-0.5) -- (-6.42,-0.5) node [midway,yshift= - 0.5cm] {$l^{(1)}$};
\draw [decorate,decoration={brace,amplitude=5pt}] (-7.59,0.67) -- (-11.11,4.19) node [midway,yshift= - 0.5cm,xshift=-0.5cm] {$l^{(1+1/4)}$};
\draw [decorate,decoration={brace,amplitude=5pt}] (-11.11,5.84) -- (-11.11,10.81) node [midway,xshift=-0.75cm] {$l^{(1+2/4)}$};
\draw [decorate,decoration={brace,amplitude=5pt}] (-9.94,11.99) -- (-6.42,15.5) node [midway,xshift=-0.4cm,yshift=0.4cm] {$l^{(1+3/4)}$};
\draw (2,13.7)--(3,13.7) node [right] {$\L^{(2)}$};
\draw (0,2)--(3,2) node [right] {$\L^{(1)}$};
\draw (0,0)--(3,0) node [right] {$\L^{(0)}$};
\end{tikzpicture}}{\caption{Geometry of the nested droplets $\L^{(n)}$ for $k=2$ in the setting of \cref{subsec:loglog:internal}. For $n\in\bbN$ droplets are symmetric and homothetic to the black $\L^{(0)}$. Intermediate ones $\L^{(1+1/4)}$, $\L^{(1+2/4)}$ and $\L^{(1+3/4)}$ obtained by East-extensions (see \cref{subfig:extension:East}) in directions $u_0$, $u_1$ and $u_2$ respectively are drawn in progressive shades of grey.\label{fig:East:internal:loglog}}}
\end{figure}

 For $j\in[2k]$, we write $l^{(n+j/(2k))}=s_{j+k}^{(n+1)}-s_{j+k}^{(n)}=\Theta(\ell^{(n+1)}/\e)$ and set $\ur^{(n+(j+1)/(2k))}=\ur^{(n+j/(2k))}+l^{(n+j/(2k))}\uv_j$, which is consistent with the definition of $\ur^{(n+1)}$ above. Thus, denoting $\L^{(n+j/(2k))}=\L(\ur^{(n+j/(2k))})$ for $n\in\bbN$ and $j\in(0,2k)$ (see \cref{fig:East:internal:loglog}), we may define SG events of these droplets by extension (recall \cref{def:extension:East,subfig:extension:East} for East-extensions).
\begin{defn}[Semi-directed internal SG]
\label{def:SG:internal:loglog}
Let $\cU$ be semi-directed or balanced rooted with finite number of stable directions and $k=1$. We say that $\L^{(0)}$ is SG ($\SG^\bone(\L^{(0)})$ occurs), if all sites in $\L^{(0)}$ are infected. We then recursively define $\SG^\bone(\L^{(n+(j+1)/(2k))})$, for $n\in[\Ni]$ and $j\in[2k]$, by East-extending $\L^{(n+j/(2k))}$ in direction $u_j$ by $l^{(n+j/(2k))}$ (see \cref{fig:East:internal:loglog}).
\end{defn}

As usual, we seek to bound the probability of $\SG^\bone(\L^{(\Ni)})$ and associated $\g(\L^{(\Ni)})$ (recall \cref{subsec:poincare}).
\begin{thm}
\label{th:internal:loglog}
Let $\cU$ be semi-directed (class \ref{loglog}) or balanced rooted with finite number of stable directions  (class \ref{log1}) and $k=1$. Then
\begin{align*}\g\left(\L^{(\Ni)}\right)&{}\le \exp\left(\frac{\log\log(1/q)}{\e^6q^\a}\right),&
\m\left(\SG^\bone\left(\L^{(\Ni)}\right)\right)&{}\ge \exp\left(\frac{-1}{\e^2q^\a}\right).
\end{align*}
\end{thm}
The rest of \cref{subsec:loglog:internal} is dedicated to the proof of \cref{th:internal:loglog}. The probability bound is fairly easy, as in \cref{eq:mu:SG:LNm:iso}, while the relaxation time is bounded by iteratively using \cref{cor:East:reduction} and then carefully estimating the product appearing there with the help of \cref{cor:perturbation}.

Note that $\g(\L^{(0)})=1$, since \cref{eq:def:gamma} is trivial, as $\SG^\bone(\L^{(0)})$ is a singleton. For $n\in 1/(2k)\bbN$, $j\in[2k]$ and $m\ge 1$, such that $n<\Ni$ and $n-j/(2k)\in\bbN$ set
\begin{equation}
\label{eq:def:amn:loglog}a^{(n)}_{m}=\m^{-1}\left(\left.\SG^\bone\left(\L^{(n)}+\left(\left\lfloor (3/2)^{m+1}\right\rfloor-\left\lfloor(3/2)^{m}\right\rfloor\right)\l_j u_j\right)\right|\SG^\bone\left(\L^{(n)}\right)\right).
\end{equation}
We further let
\begin{equation}
\label{eq:def:Mn}M^{(n)}=\min\left\{m:\l_j(3/2)^{m+1}\ge l^{(n)}\right\}=\log l^{(n)}/\log(3/2)+O(1).
\end{equation}
For the sake of simplifying expressions we abusively assume that $l^{(n)}=\l_j\lfloor(3/2)^{M^{(n)}+1}\rfloor$. Without this assumption, one would need to treat the term corresponding to $m=M^{(n)}$ below separately, but identically. 

We next seek to apply \cref{cor:East:reduction} with $\ur=\ur^{(n)}$ and $l=l^{(n)}$. Let us first analyse the term $a_m$ in \cref{eq:def:am:East}. By \cref{def:extension:East} and the Harris inequality \cref{eq:Harris:3}, we have
\begin{equation}\label{eq:loglog:internal:ambound}a_m\le \frac{a_m^{(n)}}{\m(\cT^\bone(T+(\lfloor (3/2)^{m+1}\rfloor-\lfloor(3/2)^{m}\rfloor)\l_ju_j)|\cT^\bone(T))}=\frac{a_m^{(n)}}{b_m^{(n)}},\end{equation}
using \cref{lem:tube:decomposition} in the equality and setting
\begin{align*}
T&{}=T\left(\ur^{(n)},\l_j\lfloor(3/2)^m\rfloor,j\right)\\
b_m^{(n)}&{}=\m\left(\cT^\bone\left(T\left(\ur^{(n)},\left(\left\lfloor (3/2)^{m+1}\right\rfloor-\left\lfloor(3/2)^{m}\right\rfloor\right)\l_j,j\right)\right)\right)\end{align*}
Moreover, by \cref{lem:tube:decomposition,lem:traversability:boundary} we have 
\begin{align}
\nonumber
\prod_{m=1}^{M^{(n)}}b_m^{(n)}&{}=q^{-O(WM^{(n)})}\m\left(\cT^\bone\left(T\left(\ur^{(n)},l^{(n)},j\right)\right)\right)\\
&{}=q^{-O(WM^{(n)})}\frac{\m(\SG^\bone(\L^{(n+1/(2k))}))}{\m(\SG^\bone(\L^{(n)}))},
\label{eq:loglog:internal:prod:tubes}
\end{align}
where the second equality uses \cref{def:SG:internal:loglog,def:extension:East}. 

Applying \cref{cor:East:reduction} successively and using \cref{eq:loglog:internal:ambound,eq:def:Mn}, we get
\begin{align}
\nonumber\g\left(\L^{(\Ni)}\right)&{}\le\max_{n\le\Ni}\m^{-1}\left(\SG^\bone\left(\L^{(n)}\right)\right) \prod_{n=0}^{\Ni-1/(2k)}e^{O(C^2)\log^{2}(1/q)}\prod_{m=1}^{M^{(n)}}\frac{a_m^{(n)}}{b_m^{(n)}}\\
\nonumber&{}\le \frac{\m(\SG^\bone(\L^{(0)}))e^{O(C^2)\Ni\log^2(1/q)}}{\m^2(\SG^\bone(\L^{(\Ni)}))}\prod_{n=0}^{\Ni-1/(2k)}q^{-O(WM^{(n)})}\prod_{m=1}^{M^{(n)}}a_m^{(n)}\\
&{}\le \frac{\exp(\log^{O(1)}(1/q))}{\m^2(\SG^\bone(\L^{(\Ni)}))}\prod_{n=0}^{\Ni-1/(2k)}\prod_{m=1}^{M^{(n)}}a_m^{(n)},
    \label{eq:gLni:decomposition:loglog}
    \end{align}
where in the second inequality we used \cref{eq:loglog:internal:prod:tubes} and the fact that $\m(\SG^\bone(\L^{(n)}))$ is non-increasing in $n$ (recall \cref{def:SG:internal:loglog,def:extension:East}); in the third inequality we used $\Ni\le \log(1/q)$ by \cref{eq:def:ln:internal:East} and $M^{(n)}\le O(\log(1/q))$ by \cref{eq:def:Mn,eq:def:ln:internal:East}. Note that in \cref{eq:gLni:decomposition:loglog} and below products on $n$ run over $1/(2k)\bbN$.

To evaluate the r.h.s.\ of \cref{eq:gLni:decomposition:loglog} we need the following lemma.
\begin{lem}
\label{lem:internal:loglog:probability}
Let $n\in 1/(2k)\bbN$ be such that $n\le\Ni$ and $m\ge 1$. Then
\begin{equation}
\label{eq:mu:SG:bound:loglog}
a_m^{(n)}\le\m^{-1}\left(\SG^\bone\left(\L^{(n)}\right)\right)\le\min\left((\d q^\a W^n)^{-W^n/\e^2},e^{1/(\e^2q^\a)}\right).
\end{equation}
Moreover, if
\begin{align}
\label{eq:conditions:internal:loglog:probability}
\ell^{(\lfloor n\rfloor)}&{}\ge 1/\left(q^{\a}\log^W(1/q)\right),&
M^{(n)}{}&\ge m+W,&
(3/2)^m&{}\le 1/q^\a,
\end{align}
setting
\begin{equation}
\label{eq:def:nm}
n_m=\min\left\{n'\in\bbN:\ell^{(n')}\ge1/\left(q^\a\log^W(1/q)\right),M^{(n')}\ge m+W\right\}\le n,
\end{equation}
the following improvements hold
\begin{align}
\label{eq:anm:bound:loglog}
a_m^{(n)}\le{}& \exp\left(\frac{(3/2)^m}{\e^4}\left((\Nc-n_m)^2+\1_{n\ge \Nc}\log^{2/3}\log(1/q)\right)\right)\\
&\times\begin{cases}
\exp\left(1/\left(q^\a\log^{W-O(1)}(1/q)\right)\right)&m\le \frac{\log(1/(q^\a\log^W(1/q)))}{\log(3/2)}\\
\exp\left(1/\left(q^\a\log^{W-O(1)}\log(1/q)\right)\right)&m> \frac{\log(1/(q^\a\log^W(1/q)))}{\log(3/2)}.
\end{cases}\nonumber
\end{align}
\end{lem}
Let us finish the proof of \cref{th:internal:loglog} before proving \cref{lem:internal:loglog:probability}. 
\begin{proof}[Proof of \cref{th:internal:loglog}]
The second inequality in \cref{th:internal:loglog} is contained in \cref{eq:mu:SG:bound:loglog}, so we focus on $\g(\L^{(\Ni)})$ based on \cref{eq:gLni:decomposition:loglog}. Set \begin{equation}
\label{eq:def:Ma}M_\alpha=\left\lceil \log(1/q^\a)/\log(3/2)\right\rceil.\end{equation}
Using the trivial bound $a_m^{(n)}\le\exp(1/(\e^2q^\a))$ from \cref{eq:mu:SG:bound:loglog} and then \cref{eq:def:ln:internal:East,eq:def:Mn}, we get
\begin{align}
\nonumber\prod_{n=\Nc-\lceil 1/\e\rceil}^{\Ni-1/(2k)}\prod_{m=M_\alpha}^{M^{(n)}}a_m^{(n)}&{}\le \exp\left(\frac{2}{\e^3q^\a}\sum_{n=\Nc}^{\Ni-1/(2k)}\left(M^{(n)}-M_\alpha+1\right)\right)\\
\nonumber&{}\le\exp\left(\frac{\sum_{n=\Nc}^{\Ni-1/(2k)} O(1+\log (\ell^{(\lceil n+1/(2k)\rceil)}q^\a/\e))}{\e^3q^\a}\right)\\
\nonumber&{}\le\exp\left(\frac{\sum_{n=\Nc}^{\Ni}e^{n+1-\Nc}}{\e^4q^\a}\right)\\
&{}\le\exp\left(\frac{\log\log(1/q)}{\e^5q^\a}\right),
\label{eq:gLni:part2:loglog}
\end{align}
which is the main contribution. Note that by \cref{eq:def:Mn,eq:def:ln:internal:East}, $n<\Nc-1/\e$ implies $M^{(n)}<M_\alpha$, so \cref{eq:gLni:part2:loglog} exhausts the terms in \cref{eq:gLni:decomposition:loglog} with $m\ge M_\a$.

Next set 
\begin{equation}
    \label{eq:def:NW}
    N_W=\left\lceil-\log\left(q^\a\log^W(1/q)\right)/\log W\right\rceil.
\end{equation}
Using the first bound on $a_m^{(n)}$ from \cref{eq:mu:SG:bound:loglog} and \cref{eq:def:Mn}, we obtain
\begin{align}
\label{eq:gLni:part2.5:loglog}
\nonumber\prod_{n=0}^{N_W}\prod_{m=1}^{M^{(n)}}a_m^{(n)}&{}\le\prod_{n=0}^{N_W}(\d q^\a W^n)^{-O(\log(1/q)W^n/\e^2)}\\
\nonumber&{}\le \exp\left(-\log^{O(1)}(1/q)\sum_{n=0}^{N_W}W^n\right)\\&{}\le\exp\left(1/\left(q^\a\log^{W-O(1)}(1/q)\right)\right).\end{align}

We next turn to the range $N_W\le n<n_m$ with $m<M_\alpha$. Recalling \cref{eq:def:nm,eq:def:Mn,eq:def:ln:internal:East}, we get that $N_W\le n<n_m$ implies $M^{(n)}<m+W$ and therefore $l^{(n)}\le O((3/2)^{m+W})$, so $W^n\le (3/2)^m$. Plugging this into the first bound on $a_m^{(n)}$ from \cref{eq:mu:SG:bound:loglog}, we get
\begin{align}
\prod^{M_\a-1}_{m=1}\prod_{n=N_W}^{n_m-1/(2k)}a_m^{(n)}&{}\le \exp\left(-\sum^{M_\a}_{m=1}\frac{(3/2)^m\log(\d q^\a (3/2)^m)}{\e^3}\right)\le e^{1/(q^\a\e^4)}.\label{eq:gLni:part3:loglog}
\end{align}

It remains to treat the range $n_m\le n<\Ni$ with $m< M_\a$. Note that by \cref{eq:def:NW,eq:def:nm,eq:def:ln:internal:East} $N_W\le n_m$ for any $m$ and set 
\begin{equation}
\label{eq:def:MW}
M_W=\left\lfloor \log\left(1/\left(q^\a\log^W(1/q)\right)\right)/\log(3/2)\right\rfloor.
\end{equation}
Then \cref{eq:anm:bound:loglog} gives 
\begin{align}
\nonumber\sum^{M_\a-1}_{m=1}\sum_{n=n_m}^{\Ni-1/(2k)}\log a_m^{(n)}\le{}& \frac{2k}{\e^4}\sum_{m=1}^{M_\a-1}(3/2)^m(\Nc-n_m)^2(\Ni-\Nc+\Nc-n_m)\\
\nonumber&+\frac{2k}{\e^4}(\Ni-\Nc)\log^{2/3}\log(1/q)\sum_{m=1}^{M_\a-1}(3/2)^m\\
\nonumber&+\frac{2kM_\a\Ni}{q^\a\log^{W-O(1)}(1/q)}+\frac{2k(M_\a-M_W)(\Ni-N_W)}{q^\a\log^{W-O(1)}\log (1/q)}\\
\nonumber\le{}& \frac{8k}{\e^4}\log\log\log(1/q)\sum_{m=1}^{M_\a-1}(3/2)^m(\Nc-n_m)^3\\
\nonumber&+\frac{\log^{2/3}\log(1/q)\log\log\log(1/q)}{\e^5q^\a}\\
&+\frac{1}{q^\a\log^{W-O(1)}(1/q)}+\frac{1}{q^\a\log^{W-O(1)}\log(1/q)},
\label{eq:loglog:huge}
\end{align}
where we used that $\Ni-\Nc\le 2\log\log\log(1/q)$ by \cref{eq:def:ln:internal:East}, $M_\a\le \log^{O(1)}(1/q)$ by \cref{eq:def:Ma}, $\Ni\le \log^{O(1)}(1/q)$ by \cref{eq:def:ln:internal:East}, $M_\a-M_W\le \log^{O(1)}\log(1/q)$ by \cref{eq:def:Ma,eq:def:MW} and $\Ni-N_W\le \log^{O(1)}\log(1/q)$ by \cref{eq:def:ln:internal:East,eq:def:NW}. In order to bound the last sum in \cref{eq:loglog:huge}, we note that by \cref{eq:def:MW,eq:def:nm,eq:def:ln:internal:East,eq:def:Ma}, for any $m\in[M_W,M_\a)$ we have $\Nc-n_m\le (M_\a-m)/\e$. Plugging this back into \cref{eq:loglog:huge}, we get
\begin{align*}
\sum^{M_\a-1}_{m=1}\sum_{n=n_m}^{\Ni-1/(2k)}\log a_m^{(n)}\le{}&\frac{\log\log\log(1/q)}{\e^{O(1)}}\left(M_\a^4(3/2)^{M_W}+(3/2)^{M_\a}\right)\\
&+\frac{\log^{3/4}\log(1/q)}{2q^\a}\\
\le{}&\frac{\log^{3/4}\log(1/q)}{q^\a}.\end{align*}
Plugging the last result and \cref{eq:gLni:part2:loglog,eq:gLni:part2.5:loglog,eq:gLni:part3:loglog} in \cref{eq:gLni:decomposition:loglog}, we conclude the proof of \cref{th:internal:loglog}, since $\m(\SG^\bone(\L^{(\Ni)}))\ge e^{-1/(\e^2q^\a)}$ by \cref{eq:mu:SG:bound:loglog}.
\end{proof}

\begin{proof}[Proof of \cref{lem:internal:loglog:probability}]
Let us fix $m$ and $n$ as in the statement for \cref{eq:mu:SG:bound:loglog}. The bound $a_m^{(n)}\le\m^{-1}(\SG^\bone(\L^{(n)}))$ follows from the Harris inequality \cref{eq:Harris:2}. To upper bound the latter term we note that by \cref{def:SG:internal:loglog,def:extension:East},
\begin{equation}
\label{eq:mu:bound:rec:loglog}
\m\left(\SG^\bone\left(\L^{(n)}\right)\right)=\m\left(\SG^\bone\left(\L^{(0)}\right)\right)\prod_{p=0}^{n-1/(2k)}\m\left(\cT^\bone\left(T\left(\ur^{(p)},l^{(p)},j(p)\right)\right)\right),
\end{equation}
setting $j(p)\in[2k]$ such that $p-j(p)/(2k)\in\bbN$ and letting products on $p$ run over $1/(2k)\bbN$. Clearly,
\begin{equation}
\label{eq:mu:bound:base:loglog}
\m\left(\SG^\bone\left(\L^{(0)}\right)\right)=q^{|\L^{(0)}|}=q^{\Theta(1/\e^2)}.
\end{equation}

Let us fix $p\in1/(2k)\bbN$, $p<\Ni$. Then, 
using \cref{lem:traversability:boundary}, \cref{def:traversability}, \cref{obs:mu:helping},  and the Harris inequality \cref{eq:Harris:1}, we get
\begin{multline}
\label{eq:mu:T:bound:loglog}
\m\left(\cT^\bone\left(T\left(\ur^{(p)},l^{(p)},j(p)\right)\right)\right)\\
\begin{aligned}\ge{}& q^{O(W)}\left(1-e^{-q^\a\ell^{(\lfloor p\rfloor)}/O(\e)}\right)^{O(l^{(p)})}\\
\ge{}&q^{O(W)}
\begin{cases}
(\d q^\a W^p)^{W^p/(\d\e)}&p\le \Nc,\\
\exp\left(-1/\left(q^{\a}\exp\left(W^{\exp(\lfloor p\rfloor-\Nc)}/\d\right)\right)\right)&p> \Nc.\end{cases}
\end{aligned}\end{multline}
In the last inequality we took into account $1/\e\gg1/\d\gg W\gg 1$, $\ell^{(\Nc)}=W^{O(1)}/q^\a$ and the explicit expressions \cref{eq:def:ln:internal:East}. From \cref{eq:mu:bound:base:loglog,eq:mu:bound:rec:loglog,eq:mu:T:bound:loglog} it is not hard to check \cref{eq:mu:SG:bound:loglog} (recalling \cref{subsec:scales}).

We next turn to proving \cref{eq:anm:bound:loglog}, so we fix $n\le\Ni$ and $m\ge 1$ satisfying \cref{eq:conditions:internal:loglog:probability}. Denote $s_m=(\lfloor(3/2)^{m+1}\rfloor-\lfloor(3/2)^{m}\rfloor)\l_j u_j$ for $j=j(n)$, so that \cref{eq:def:amn:loglog} spells
\[a_m^{(n)}=\m^{-1}\left(\left.\SG^\bone\left(\L^{(n)}+s_m\right)\right|\SG^\bone\left(\L^{(n)}\right)\right).\]
By the Harris inequality, \cref{eq:Harris:2,eq:Harris:3}, \cref{def:extension:East,def:SG:internal:loglog} we have
\begin{align}
\label{eq:anm:decomposition:loglog}
    a_m^{(n)}\le{}& \m^{-1}\left(\SG^\bone\left(\L^{(n_m)}\right)\right)\\
    &\times\prod_{p=n_m}^{n-1/(2k)}\m^{-1}\left(\left.\cT^\bone\left(T\left(\ur^{(p)},l^{(p)},j(p)\right)+s_m\right)\right|\cT^\bone\left(T\left(\ur^{(p)},l^{(p)},j(p)\right)\right)\right).\nonumber
\end{align}
Our goal is then to bound the last factor, using \cref{cor:perturbation}, which quantifies the fact that ``small perturbations $s_m$ do not modify traversability much.''

Let us fix $p\in[n_m,n)\cap(1/2k)\bbN$ and denote \begin{align*}
\cT&{}=\cT^\bone\left(T\left(\ur^{(p)},l^{(p)},j(p)\right)\right)&\cT'&{}=\cT^\bone\left(T\left(\ur^{(p)},l^{(p)},j(p)\right)+s_m\right).\end{align*}
In order to apply \cref{cor:perturbation} with $\D=\max(C^2,\|s_m\|)$, we check that $W^3 (3/2)^m\le \ell^{(\lfloor p\rfloor)}/\e$ (so that the sides of $\L^{(p)}$ are large enough). If $\ell^{(\lfloor p\rfloor)}\ge 1/q^\a$, this follows from the assumption $(3/2)^m\le 1/q^\a$ of \cref{lem:internal:loglog:probability}. Otherwise, by \cref{eq:def:nm,eq:def:Mn}
\begin{align*}
W^3(3/2)^m&{}\le (3/2)^{M^{(n_m)}-W/2}\le l^{(n_m)}/e^{\O(W)}=\Theta\left(\ell^{(n_m+1)}\right)/\left(\e e^{\O(W)}\right)\\
&{}\le \ell^{(n_m)}/\e\le \ell^{(\lfloor p\rfloor)}/\e,
\end{align*}
where in the last but one inequality we used that $\ell^{(n_m+1)}\le W^{O(1)}\ell^{(n_m)}$, since $n_m\le p$ and $\ell^{(p)}\le 1/q^\a$ (recall \cref{eq:def:ln:internal:East}). The remaining hypotheses of \cref{cor:perturbation} are immediate to verify. 

For $\|s_m\|=\Theta((3/2)^m)\le C^2$, \cref{cor:perturbation} gives
\[\m\left(\left.\cT'\right|\cT\right)\ge q^{O(C^2)}\left(1-q^{1-o(1)}\right)^{O(l^{(p)})}\ge \exp\left(-q^{-\a+1/2}\right),\]
as $l^{(p)}\le\ell^{(\Ni)}/\e\le q^{-\a-o(1)}$. If, on the contrary, $\|s_m\|\ge C^2$, \cref{cor:perturbation} gives
\begin{align}
\label{eq:conditional:T}\m(\cT'|\cT)\ge{}& q^{O(W)}\times\left(1-(1-q^\a)^{\ell^{(\lfloor p\rfloor)}/O(\e)}\right)^{O((3/2)^m)}\\
\nonumber&\times\left(1-O(W\e)(3/2)^m/\ell^{(\lfloor p\rfloor)}-q^{1-o(1)}\right)^{O(\ell^{(\lfloor p\rfloor+1)}/\e)}
\\
\nonumber\ge{}&q^{O(W)}\times\begin{cases}(\d q^\a W^p)^{O((3/2)^m)}&p\le \Nc\\
\exp\left(-(3/2)^m\exp\left(-W^{\exp(\lfloor p\rfloor-\Nc)}/\d\right)\right)&p>\Nc
\end{cases}\\
&\times\begin{cases}
\exp\left(-q^{-\a+1/2-o(1)}\right)&(3/2)^m\le q^{-\a+1/2-o(1)}\\
\exp\left(-W^2(3/2)^m\frac{\ell^{(\lfloor p\rfloor+1)}}{\ell^{(\lfloor p\rfloor)}}\right)&(3/2)^m> q^{-\a+1/2-o(1)},
\end{cases}\nonumber
\end{align}
in view of \cref{eq:def:ln:internal:East}. Further notice that if $(3/2)^m\le q^{-\a+1/2-o(1)}$ or $p>\Nc$, the third term dominates, while otherwise the second one does. Moreover, if $p\ge\Nc+\lceil\Psi\rceil$ with 
\begin{equation}
\label{eq:def:Delta}
\Psi=\log\frac{\log\log\log(1/q)}{3\log W},\end{equation}
then the Harris inequality \cref{eq:Harris:2}, translation invariance and \cref{eq:mu:T:bound:loglog} directly give the bound
\begin{equation}
\label{eq:conditional:T:trivial}
\m(\cT'|\cT)\ge\m(\cT')=\m(\cT)\ge  \exp\left(-1/\left(q^\a \log^W\log(1/q)\right)\right).
\end{equation}

Finally, we can plug \cref{eq:mu:SG:bound:loglog,eq:conditional:T,eq:conditional:T:trivial} in \cref{eq:anm:decomposition:loglog} to obtain the following bounds.
If $(3/2)^m\le q^{-\a+1/2-o(1)}$, then
\[a_m^{(n)}\le \exp\left(1/\left(q^\a\log^{W}(1/q)\right)\right),\]
because the contribution from \cref{eq:conditional:T} is negligible, since $n\le \Ni\le \log(1/q)$, while by \cref{eq:def:nm,eq:def:ln:internal:East}, $W^{n_m}=\ell^{(n_m)}\le W/(q^\a\log^W(1/q))$. 
If, on the contrary, $(3/2)^m>q^{-\a+1/2-o(1)}$, then
\begin{align*}
a_m^{(n)}\le{}&\begin{cases}
\exp\left(1/\left(q^\a\log^{W-O(1)}(1/q)\right)\right)&(3/2)^m\le 1/\left(q^\a\log^W(1/q)\right)\\
\left(\d q^\a W^{n_m}\right)^{-(3/2)^m/\e^3}&(3/2)^m> 1/\left(q^\a\log^W(1/q)\right)
\end{cases}
\\&\times\prod_{p=n_m}^{\min(n,\Nc)}(\d q^\a W^p)^{-O((3/2)^m)}\\
&{}\times\begin{cases}
1&n\le \Nc\\
\exp\left((3/2)^mW^{2\exp(\Psi)}/\d\right)&n>\Nc
\end{cases}\\
&\times\begin{cases}
\exp\left(1/\left(q^\a\log^{W-O(1)}(1/q)\right)\right)&(3/2)^m\le 1/\left(q^\a\log^W(1/q)\right)\\
\exp\left(1/\left(q^\a\log^{W-O(1)}\log(1/q)\right)\right)&(3/2)^m> 1/\left(q^\a\log^W(1/q)\right),
\end{cases}
\end{align*}
the terms corresponding to $\m^{-1}(\SG^\bone(\L^{(n_m)}))$ and to values of $p$ in the intervals $[n_m,\Nc]$, $(\Nc,\Nc+\lceil\Psi\rceil)$ and $[\Nc+\lceil\Psi\rceil,\Ni)$ respectively. Indeed, in the last term for small $m$ we used \cref{eq:conditional:T}, while for large $m$, we directly applied \cref{eq:conditional:T:trivial}. Observing that the product of the second case for the first term, the second term and the third term can be bounded by
\[\exp\left(\frac{(3/2)^m}{\e^4}\left((\Nc-n_m)^2 +\1_{n\ge \Nc}\log^{2/3}\log(1/q)\right)\right),\]
we obtain the desired \cref{eq:anm:bound:loglog}.
\end{proof}

\subsection{CBSEP mesoscopic dynamics}
\label{subsec:loglog:meso}
In this section we assume that $\cU$ is semi-directed (class \ref{loglog}) and w.l.o.g.\ $\a(u_i)\le \a$ for all $i\in[4k]\setminus\{3k\}$. The approach to the mesoscopic dynamics is very similar to the one of \cref{subsec:log2:meso}, employing a bounded number of CBSEP-extensions to go from the internal to the mesoscopic scale. Once again, the geometry of our droplets is as in \cref{subfig:iso:SG}, but extensions are much longer so that we go from scale $\li$ to $\lmm$ in $2k$ extensions and then to $\lmp$ in another $2k$ extensions.

Recall from \cref{subsec:loglog:internal} that we defined $\L^{(\Ni)}$, a symmetric droplet with side lengths $\us^{(\Ni)}$ equal to $\Theta(\ell^{(\Ni)}/\e)$, as well as $\SG^\bone(\L^{(\Ni)})$ in \cref{def:SG:internal:loglog}. Further recall \cref{subsec:scales}. Following \cref{subsec:log2:meso}, for $i\in[1,2k]$ we define
\[s^{(i+\Ni)}_j=s^{(i+\Ni)}_{j+2k}=\begin{cases}
s_j^{(\Ni)}&i-k\le j<k,\\
2\l_j\lceil\lmm/(2\l_j)\rceil&-k\le j<i-k,
\end{cases}\]
while for $i\in(2k,4k]$, we set
\begin{equation}
s_j^{(i+\Ni)}=s_{j+2k}^{(i+\Ni)}=
\begin{cases}
2\l_j\lceil\lmm/(2\l_j)\rceil&i-3k\le j< k\\
2\l_j\lceil\lmp/(2\l_j)\rceil&-k\le j<i-3k.
\end{cases}
\end{equation}
We then define $\L^{(\Ni+i)}=\L(\ur^{(\Ni+i)})$ with $\ur^{(\Ni+i)}$ the sequence of radii associated to $\us^{(\Ni+i)}$ satisfying 
\begin{align*}
\L\left(\ur^{(N_i+i)}\right)&{}=\L\left(\ur^{(\Ni+i-1)}+l^{(\Ni+i-1)}\left(\uv_{i-1}+\uv_{i+2k-1}\right)/2\right),\\
l^{(\Ni+i-1)}&{}=s^{(\Ni+i)}_{i+k-1}-s^{(\Ni+i-1)}_{i+k-1}=\begin{cases}(1-o(1))\lmm&i\in[1,2k],\\
(1-O(\d))\lmp&i\in(2k,4k].\end{cases}\end{align*}
We then define the corresponding SG events by CBSEP-extension as in \cref{def:SG:log2}.

\begin{defn}[Semi-directed mesoscopic SG] Let $\cU$ be semi-directed. For $i\in[4k]$ we define $\SG^\bone(\L^{(\Ni+i+1)})$ by CBSEP-extending $\L^{(\Ni+i)}$ by $l^{(\Ni+i)}$ in direction $u_i$.
\end{defn}
We then turn to bounding $\g(\L^{(\Ni+4k)})$ (recall \cref{subsec:poincare}).
\begin{thm}
\label{th:loglog:meso}
Let $\cU$ be semi-directed (class \ref{loglog}). Then 
\begin{align*}\g\left(\L^{(\Ni+4k)}\right)&{}\le \exp\left(\frac{\log\log(1/q)}{\e^{O(1)}q^\a}\right),\\
\m\left(\SG^\bone\left(\L^{(\Ni+2k)}\right)\right)&{}\ge \exp\left(\frac{-1}{\e^{O(1)}q^\a}\right).
\end{align*}
\end{thm}
The rest of \cref{subsec:loglog:meso} is dedicated to the proof of \cref{th:loglog:meso}. The proof proceeds exactly like \cref{th:log2}, except that the first two steps are much more delicate. Namely, they require taking into account the internal structure of $\SG^\bone(\L^{(\Ni)})$ on all scales down to $0$. This structure is, alas, rather complex (recall \cref{fig:East:internal:loglog}) and also not symmetric w.r.t.\ the reflection interchanging $u_0$ and $u_{2k}$. This is not unexpected and is, to some extent, the crux of semi-directed models.

As before, we define $\L_1^{(i)},\L_2^{(i)},\L_3^{(i)}$ by \cref{def:L123} for $i\in[\Ni,\Ni+4k)$. The next definitions are illustrated in \cref{fig:bSG:loglog} and are the analogue of \cref{def:bSG:log2}, but taking into account \cref{def:SG:internal:loglog}. Correspondingly, the intuition behind them is the same, the only difference being that we need to modify traversability events at all scales, because $\L^{(i)}$ touches the boundary of $\L^{(\Ni)}$ for all $i\le\Ni$ (compare \cref{fig:East:internal:loglog,subfig:iso:SG}).

\begin{figure}
    \centering
\begin{subfigure}{\textwidth}
\centering
\begin{tikzpicture}[line cap=round,line join=round,>=triangle 45,x=0.55cm,y=0.55cm]
\fill[fill=black,fill opacity=0.1] (-4.56,11.5) -- (0.21,11.5) -- (3.87,7.84) -- (3.87,2.66) -- (0.21,-1) -- (-4.56,-1) -- (-8.22,2.66) -- (-8.22,7.84) -- cycle;
\fill[fill=black,fill opacity=0.1] (-1.45,1) -- (-2.04,0.41) -- (-2.04,-0.41) -- (-1.45,-1) -- (-0.21,-1) -- (-0.79,-0.41) -- (-0.79,0.41) -- (-0.21,1) -- cycle;
\fill[fill=black,fill opacity=0.1] (-1.45,-1) -- (0.21,-1) -- (0.79,-0.41) -- (0.79,0.41) -- (0.21,1) -- (-1.45,1) -- (-2.04,0.41) -- (-2.04,-0.41) -- cycle;
\fill[fill=black,fill opacity=0.1] (-1.45,-1) -- (-2.91,0.46) -- (-2.91,1.29) -- (-2.33,1.88) -- (-0.67,1.88) -- (0.79,0.41) -- (0.79,-0.41) -- (0.21,-1) -- cycle;
\fill[fill=black,fill opacity=0.1] (-2.91,2.54) -- (-2.91,0.46) -- (-1.45,-1) -- (0.21,-1) -- (0.79,-0.41) -- (0.79,1.66) -- (-0.67,3.12) -- (-2.33,3.12) -- cycle;
\fill[fill=black,fill opacity=0.1] (-1.45,4) -- (-2.91,2.54) -- (-2.91,0.46) -- (-1.45,-1) -- (0.21,-1) -- (0.79,-0.41) -- (0.79,-0.41) -- (1.67,0.46) -- (1.67,2.54) -- (0.21,4) -- cycle;
\fill[fill=black,fill opacity=0.1] (-4.56,4) -- (-6.02,2.54) -- (-6.02,0.46) -- (-4.56,-1) -- (0.21,-1) -- (1.67,0.46) -- (1.67,2.54) -- (0.21,4) -- cycle;
\fill[fill=black,fill opacity=0.1] (-8.22,2.66) -- (-8.22,4.73) -- (-6.75,6.2) -- (-1.99,6.2) -- (1.67,2.54) -- (1.67,0.46) -- (0.21,-1) -- (-4.56,-1) -- cycle;
\fill[fill=black,fill opacity=0.1] (-8.22,7.84) -- (-6.75,9.3) -- (-1.99,9.3) -- (1.67,5.64) -- (1.67,0.46) -- (0.21,-1) -- (-4.56,-1) -- (-8.22,2.66) -- cycle;
\draw (-4.56,11.5)-- (0.21,11.5);
\draw (0.21,11.5)-- (3.87,7.84);
\draw (3.87,7.84)-- (3.87,2.66);
\draw (3.87,2.66)-- (0.21,-1);
\draw (0.21,-1)-- (-4.56,-1);
\draw (-4.56,-1)-- (-8.22,2.66);
\draw (-8.22,2.66)-- (-8.22,7.84);
\draw (-8.22,7.84)-- (-4.56,11.5);
\draw [very thick] (0.41,11.5)-- (0.62,11.5);
\draw [very thick] (0.62,11.5)-- (4.28,7.84);
\draw [very thick] (4.28,7.84)-- (4.28,2.66);
\draw [very thick] (4.28,2.66)-- (0.62,-1);
\draw [very thick] (0.62,-1)-- (0.41,-1);
\draw [very thick] (0.41,-1)-- (4.08,2.66);
\draw [very thick] (4.08,2.66)-- (4.08,7.84);
\draw [very thick] (4.08,7.84)-- (0.41,11.5);
\draw (-0.21,1)-- (-0.79,0.41);
\draw (-0.79,0.41)-- (-0.79,-0.41);
\draw (-0.79,-0.41)-- (-0.21,-1);
\draw (-0.21,-1)-- (0.21,-1);
\draw (0.21,-1)-- (0.79,-0.41);
\draw (0.79,-0.41)-- (0.79,0.41);
\draw (0.79,0.41)-- (0.21,1);
\draw (0.21,1)-- (-0.21,1);
\draw (-1.45,1)-- (-2.04,0.41);
\draw (-2.04,0.41)-- (-2.04,-0.41);
\draw (-2.04,-0.41)-- (-1.45,-1);
\draw (-1.45,-1)-- (-0.21,-1);
\draw (-0.21,-1)-- (-0.79,-0.41);
\draw (-0.79,-0.41)-- (-0.79,0.41);
\draw (-0.79,0.41)-- (-0.21,1);
\draw (-0.21,1)-- (-1.45,1);
\draw (-1.45,-1)-- (0.21,-1);
\draw (0.21,-1)-- (0.79,-0.41);
\draw (0.79,-0.41)-- (0.79,0.41);
\draw (0.79,0.41)-- (0.21,1);
\draw (0.21,1)-- (-1.45,1);
\draw (-1.45,1)-- (-2.04,0.41);
\draw (-2.04,0.41)-- (-2.04,-0.41);
\draw (-2.04,-0.41)-- (-1.45,-1);
\draw (-1.45,-1)-- (-2.91,0.46);
\draw (-2.91,0.46)-- (-2.91,1.29);
\draw (-2.91,1.29)-- (-2.33,1.88);
\draw (-2.33,1.88)-- (-0.67,1.88);
\draw (-0.67,1.88)-- (0.79,0.41);
\draw (0.79,0.41)-- (0.79,-0.41);
\draw (0.79,-0.41)-- (0.21,-1);
\draw (0.21,-1)-- (-1.45,-1);
\draw (-2.91,2.54)-- (-2.91,0.46);
\draw (-2.91,0.46)-- (-1.45,-1);
\draw (-1.45,-1)-- (0.21,-1);
\draw (0.21,-1)-- (0.79,-0.41);
\draw (0.79,-0.41)-- (0.79,1.66);
\draw (0.79,1.66)-- (-0.67,3.12);
\draw (-0.67,3.12)-- (-2.33,3.12);
\draw (-2.33,3.12)-- (-2.91,2.54);
\draw (-1.45,4)-- (-2.91,2.54);
\draw (-2.91,2.54)-- (-2.91,0.46);
\draw (-2.91,0.46)-- (-1.45,-1);
\draw (-1.45,-1)-- (0.21,-1);
\draw (0.21,-1)-- (0.79,-0.41);
\draw (0.79,-0.41)-- (0.79,-0.41);
\draw (0.79,-0.41)-- (1.67,0.46);
\draw (1.67,0.46)-- (1.67,2.54);
\draw (1.67,2.54)-- (0.21,4);
\draw (0.21,4)-- (-1.45,4);
\draw (-4.56,4)-- (-6.02,2.54);
\draw (-6.02,2.54)-- (-6.02,0.46);
\draw (-6.02,0.46)-- (-4.56,-1);
\draw (-4.56,-1)-- (0.21,-1);
\draw (0.21,-1)-- (1.67,0.46);
\draw (1.67,0.46)-- (1.67,2.54);
\draw (1.67,2.54)-- (0.21,4);
\draw (0.21,4)-- (-4.56,4);
\draw (-8.22,2.66)-- (-8.22,4.73);
\draw (-8.22,4.73)-- (-6.75,6.2);
\draw (-6.75,6.2)-- (-1.99,6.2);
\draw (-1.99,6.2)-- (1.67,2.54);
\draw (1.67,2.54)-- (1.67,0.46);
\draw (1.67,0.46)-- (0.21,-1);
\draw (0.21,-1)-- (-4.56,-1);
\draw (-4.56,-1)-- (-8.22,2.66);
\draw (-8.22,7.84)-- (-6.75,9.3);
\draw (-6.75,9.3)-- (-1.99,9.3);
\draw (-1.99,9.3)-- (1.67,5.64);
\draw (1.67,5.64)-- (1.67,0.46);
\draw (1.67,0.46)-- (0.21,-1);
\draw (0.21,-1)-- (-4.56,-1);
\draw (-4.56,-1)-- (-8.22,2.66);
\draw (-8.22,2.66)-- (-8.22,7.84);
\draw (0.21,11.5)-- (0.41,11.5);
\draw (0.41,-1)-- (0.21,-1);
\draw [very thick] (-4.56,11.5)-- (-4.76,11.5);
\draw [very thick] (-4.76,11.5)-- (-8.42,7.84);
\draw [very thick] (-8.42,7.84)-- (-8.42,2.66);
\draw [very thick] (-8.42,2.66)-- (-4.76,-1);
\draw [very thick] (-4.76,-1)-- (-4.56,-1);
\draw [very thick] (-4.56,-1)-- (-8.22,2.66);
\draw [very thick] (-8.22,2.66)-- (-8.22,7.84);
\draw [very thick] (-8.22,7.84)-- (-4.56,11.5);
\draw [shift={(0,0)},fill=black,fill opacity=1.0]  (0,0) --  plot[domain=-0.16:3.93,variable=\t]({1*1.41*cos(\t r)+0*1.41*sin(\t r)},{0*1.41*cos(\t r)+1*1.41*sin(\t r)}) -- cycle ;
\fill[fill=black,fill opacity=1.0] (0,0) -- (-1,-1) -- (0.62,-1) -- (1.39,-0.23) -- cycle;
\draw[very thick](-6.59,9.67)--(-6.92,10) node [above left] {$\L_3^{(i)}$};
\draw[very thick] (2.45,0.83)--(2.78,0.5) node [below right] {$\L_1^{(i)}$};
\end{tikzpicture}
\caption{Case $i=\Ni$  of \cref{def:bSG:loglog:0}.}
\label{subfig:bSG:loglog:0}
\end{subfigure}
\begin{subfigure}{\textwidth}
\centering
\begin{tikzpicture}[line cap=round,line join=round,>=triangle 45,x=0.55cm,y=0.55cm]
\fill[fill=black,fill opacity=0.1] (-0.85,0.26) -- (-0.27,0.85) -- (0.56,0.85) -- (0.85,0.56) -- (0.85,-0.27) -- (0.27,-0.85) -- (-0.56,-0.85) -- (-0.85,-0.56) -- cycle;
\fill[fill=black,fill opacity=0.1] (-1.51,0.85) -- (-2.1,0.26) -- (-2.1,-0.56) -- (-1.8,-0.85) -- (0.27,-0.85) -- (0.85,-0.27) -- (0.85,0.56) -- (0.56,0.85) -- cycle;
\fill[fill=black,fill opacity=0.1] (-2.97,0.31) -- (-1.8,-0.85) -- (0.27,-0.85) -- (0.85,-0.27) -- (0.85,0.56) -- (-0.17,1.59) -- (-0.32,1.73) -- (-2.39,1.73) -- (-2.97,1.14) -- cycle;
\fill[fill=black,fill opacity=0.1] (-2.97,2.39) -- (-2.39,2.97) -- (-0.32,2.97) -- (0.85,1.8) -- (0.85,-0.27) -- (0.27,-0.85) -- (-1.8,-0.85) -- (-2.97,0.31) -- cycle;
\fill[fill=black,fill opacity=0.1] (-2.97,2.39) -- (-1.51,3.85) -- (0.56,3.85) -- (1.73,2.68) -- (1.73,0.61) -- (0.27,-0.85) -- (-1.8,-0.85) -- (-2.97,0.31) -- cycle;
\fill[fill=black,fill opacity=0.1] (-4.62,3.85) -- (-6.08,2.39) -- (-6.08,0.31) -- (-4.91,-0.85) -- (0.27,-0.85) -- (1.73,0.61) -- (1.73,2.68) -- (0.56,3.85) -- cycle;
\fill[fill=black,fill opacity=0.1] (-8.28,2.51) -- (-8.28,4.58) -- (-6.81,6.05) -- (-1.64,6.05) -- (1.73,2.68) -- (1.73,0.61) -- (0.27,-0.85) -- (-4.91,-0.85) -- cycle;
\fill[fill=black,fill opacity=0.1] (-8.28,7.69) -- (-6.81,9.15) -- (-1.64,9.15) -- (1.73,5.79) -- (1.73,0.61) -- (0.27,-0.85) -- (-4.91,-0.85) -- (-8.28,2.51) -- cycle;
\fill[fill=black,fill opacity=0.1] (-8.28,7.69) -- (-4.62,11.35) -- (0.56,11.35) -- (3.93,7.99) -- (3.93,2.81) -- (0.27,-0.85) -- (-4.91,-0.85) -- (-8.28,2.51) -- cycle;
\fill[fill=black,fill opacity=0.1] (-9.79,11.35) -- (-13.46,7.69) -- (-13.46,2.51) -- (-10.09,-0.85) -- (5.45,-0.85) -- (9.11,2.81) -- (9.11,7.99) -- (5.74,11.35) -- cycle;
\fill[fill=black,fill opacity=1.0] (0,0) -- (-0.83,-1.15) -- (0.83,-1.15) -- cycle;
\draw (-0.85,0.26)-- (-0.27,0.85);
\draw (-0.27,0.85)-- (0.56,0.85);
\draw (0.56,0.85)-- (0.85,0.56);
\draw (0.85,0.56)-- (0.85,-0.27);
\draw (0.85,-0.27)-- (0.27,-0.85);
\draw (0.27,-0.85)-- (-0.56,-0.85);
\draw (-0.56,-0.85)-- (-0.85,-0.56);
\draw (-0.85,-0.56)-- (-0.85,0.26);
\draw (-1.51,0.85)-- (-2.1,0.26);
\draw (-2.1,0.26)-- (-2.1,-0.56);
\draw (-2.1,-0.56)-- (-1.8,-0.85);
\draw (-1.8,-0.85)-- (0.27,-0.85);
\draw (0.27,-0.85)-- (0.85,-0.27);
\draw (0.85,-0.27)-- (0.85,0.56);
\draw (0.85,0.56)-- (0.56,0.85);
\draw (0.56,0.85)-- (-1.51,0.85);
\draw (-2.97,0.31)-- (-1.8,-0.85);
\draw (-1.8,-0.85)-- (0.27,-0.85);
\draw (0.27,-0.85)-- (0.85,-0.27);
\draw (0.85,-0.27)-- (0.85,0.56);
\draw (0.85,0.56)-- (-0.17,1.59);
\draw (-0.17,1.59)-- (-0.32,1.73);
\draw (-0.32,1.73)-- (-2.39,1.73);
\draw (-2.39,1.73)-- (-2.97,1.14);
\draw (-2.97,1.14)-- (-2.97,0.31);
\draw (-2.97,2.39)-- (-2.39,2.97);
\draw (-2.39,2.97)-- (-0.32,2.97);
\draw (-0.32,2.97)-- (0.85,1.8);
\draw (0.85,1.8)-- (0.85,-0.27);
\draw (0.85,-0.27)-- (0.27,-0.85);
\draw (0.27,-0.85)-- (-1.8,-0.85);
\draw (-1.8,-0.85)-- (-2.97,0.31);
\draw (-2.97,0.31)-- (-2.97,2.39);
\draw (-2.97,2.39)-- (-1.51,3.85);
\draw (-1.51,3.85)-- (0.56,3.85);
\draw (0.56,3.85)-- (1.73,2.68);
\draw (1.73,2.68)-- (1.73,0.61);
\draw (1.73,0.61)-- (0.27,-0.85);
\draw (0.27,-0.85)-- (-1.8,-0.85);
\draw (-1.8,-0.85)-- (-2.97,0.31);
\draw (-2.97,0.31)-- (-2.97,2.39);
\draw (-4.62,3.85)-- (-6.08,2.39);
\draw (-6.08,2.39)-- (-6.08,0.31);
\draw (-6.08,0.31)-- (-4.91,-0.85);
\draw (-4.91,-0.85)-- (0.27,-0.85);
\draw (0.27,-0.85)-- (1.73,0.61);
\draw (1.73,0.61)-- (1.73,2.68);
\draw (1.73,2.68)-- (0.56,3.85);
\draw (0.56,3.85)-- (-4.62,3.85);
\draw (-8.28,2.51)-- (-8.28,4.58);
\draw (-8.28,4.58)-- (-6.81,6.05);
\draw (-6.81,6.05)-- (-1.64,6.05);
\draw (-1.64,6.05)-- (1.73,2.68);
\draw (1.73,2.68)-- (1.73,0.61);
\draw (1.73,0.61)-- (0.27,-0.85);
\draw (0.27,-0.85)-- (-4.91,-0.85);
\draw (-4.91,-0.85)-- (-8.28,2.51);
\draw (-8.28,7.69)-- (-6.81,9.15);
\draw (-6.81,9.15)-- (-1.64,9.15);
\draw (-1.64,9.15)-- (1.73,5.79);
\draw (1.73,5.79)-- (1.73,0.61);
\draw (1.73,0.61)-- (0.27,-0.85);
\draw (0.27,-0.85)-- (-4.91,-0.85);
\draw (-4.91,-0.85)-- (-8.28,2.51);
\draw (-8.28,2.51)-- (-8.28,7.69);
\draw (-8.28,7.69)-- (-4.62,11.35);
\draw (-4.62,11.35)-- (0.56,11.35);
\draw (0.56,11.35)-- (3.93,7.99);
\draw (3.93,7.99)-- (3.93,2.81);
\draw (3.93,2.81)-- (0.27,-0.85);
\draw (0.27,-0.85)-- (-4.91,-0.85);
\draw (-4.91,-0.85)-- (-8.28,2.51);
\draw (-8.28,2.51)-- (-8.28,7.69);
\draw (-9.79,11.35)-- (-13.46,7.69);
\draw (-13.46,7.69)-- (-13.46,2.51);
\draw (-13.46,2.51)-- (-10.09,-0.85);
\draw (-10.09,-0.85)-- (5.45,-0.85);
\draw (5.45,-0.85)-- (9.11,2.81);
\draw (9.11,2.81)-- (9.11,7.99);
\draw (9.11,7.99)-- (5.74,11.35);
\draw (5.74,11.35)-- (-9.79,11.35);
\draw [very thick] (-13.46,2.51)-- (-13.6,2.66);
\draw [very thick] (-13.6,2.66)-- (-13.6,7.84);
\draw [very thick] (-13.6,7.84)-- (-9.94,11.5);
\draw [very thick] (-9.94,11.5)-- (5.59,11.5);
\draw [very thick] (5.59,11.5)-- (5.74,11.35);
\draw [very thick] (5.74,11.35)-- (-9.79,11.35);
\draw [very thick] (-9.79,11.35)-- (-13.46,7.69);
\draw [very thick] (-13.46,7.69)-- (-13.46,2.51);
\draw (-10.09,-0.85)-- (-9.94,-1);
\draw (-9.94,-1)-- (5.59,-1);
\draw (5.59,-1)-- (9.25,2.66);
\draw (9.25,2.66)-- (9.25,7.84);
\draw (9.25,7.84)-- (9.11,7.99);
\draw (9.11,7.99)-- (9.11,2.81);
\draw (9.11,2.81)-- (5.45,-0.85);
\draw (5.45,-0.85)-- (-10.09,-0.85);
\draw [very thick] (-9.94,-1)-- (-9.79,-1.15);
\draw [very thick] (-9.79,-1.15)-- (5.74,-1.15);
\draw [very thick] (5.74,-1.15)-- (9.4,2.51);
\draw [very thick] (9.4,2.51)-- (9.4,7.69);
\draw [very thick] (9.4,7.69)-- (9.25,7.84);
\draw [very thick] (9.25,7.84)-- (9.25,2.66);
\draw [very thick] (9.25,2.66)-- (5.59,-1);
\draw [very thick] (5.59,-1)-- (-9.94,-1);\draw [shift={(0,0)},fill=black,fill opacity=1.0]  (0,0) --  plot[domain=-0.94:4.09,variable=\t]({1*1.41*cos(\t r)+0*1.41*sin(\t r)},{0*1.41*cos(\t r)+1*1.41*sin(\t r)}) -- cycle ;
\draw[very thick] (-11.77,9.67)--(-12.1,10) node [above left] {$\L_3^{(i)}$};
\draw[very thick] (7.57,0.68)--(7.9,0.35) node [below right] {$\L_1^{(i)}$};
\end{tikzpicture}
\caption{Case $i=\Ni+1$ of \cref{def:bSG:loglog:1}.}
\label{subfig:bSG:loglog:1}
\end{subfigure}
\caption{The events $\bcT(\L_1^{(i)})$, $\bSG(\L^{(i)}_2)$ and $\bcT(\L_3^{(i)})$. The microscopic black regions are entirely infected. Shaded tubes are $(\bone,W)$-traversable. $W$-helping sets are required close to all boundaries.\label{fig:bSG:loglog}}
\end{figure}
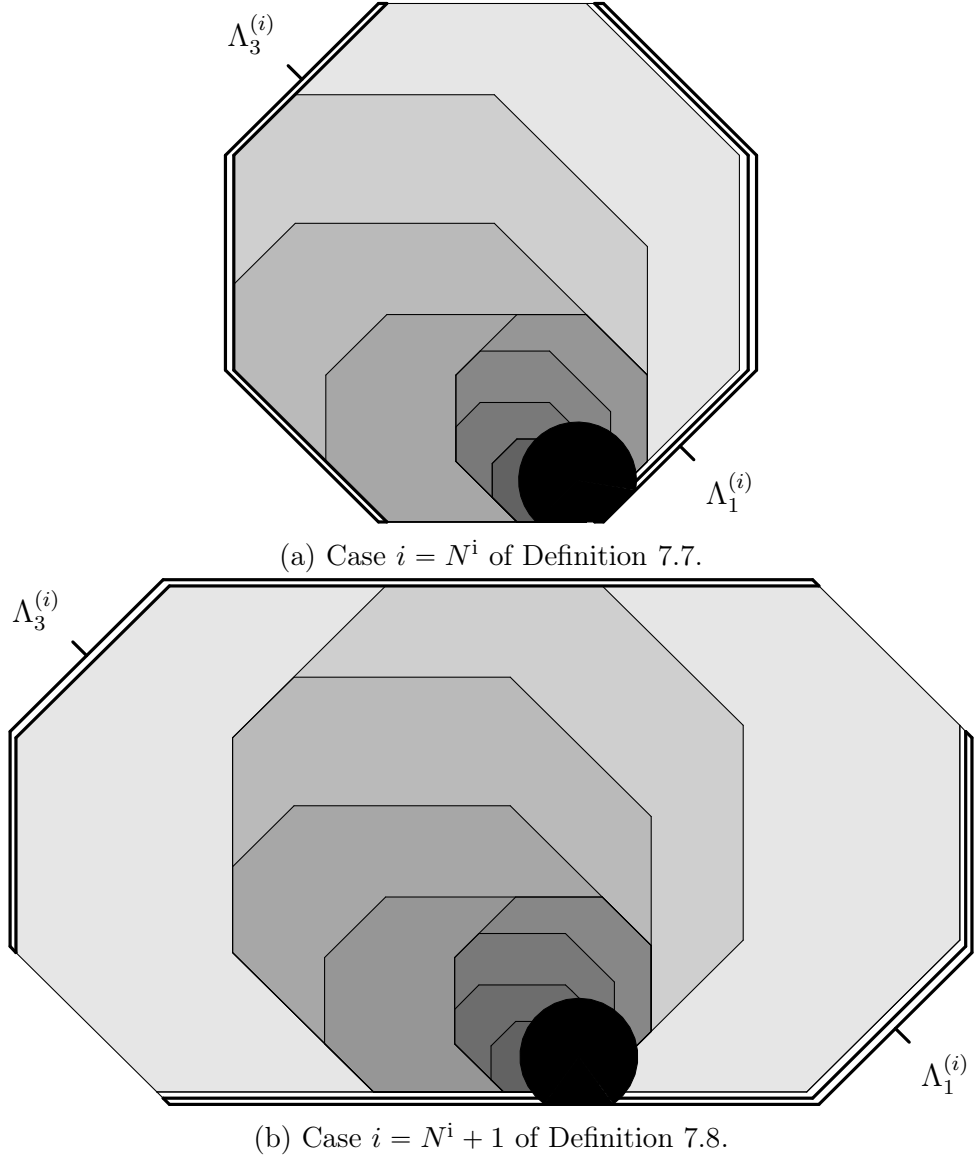

\begin{defn}[Contracted semi-directed events on scale $\Ni$]
\label{def:bSG:loglog:0}
Let us define $\bcT(\L_3^{(\Ni)})$ to be the event that for all $j\in[-k+1,k-1]$ and, for every segment $S\subset \L_3^{(\Ni)}$, perpendicular to $u_j$ of length $s^{(\Ni)}_{j}/W$, the event $\cH^W(S)$ occurs.

Let $\bcT(\L_1^{(\Ni)})$ be the event that for all $j\in[k+1,3k-2]$ every segment $S\subset \L_1^{(\Ni)}$, perpendicular to $u_j$ of length $s^{(\Ni)}_j/W$, the event $\cH^W(S)$ occurs and all sites in $\L_1^{(\Ni)}$ at distance at most  $\sqrt{W}/\e$ from the origin are infected.

For $n\in[0,\Ni]$ such that $2kn\in\bbN$ let $\L'^{(n)}=\L(\ur^{(n)}-\l_0(\uv_0+\uv_{2k}))$. Define $\SG'(\L'^{(n)})$ recursively exactly like $\SG^\bone(\L^{(n)})$ in \cref{def:SG:internal:loglog} with all droplets replaced by their contracted versions $\L'$ and all traversability events required in East-extensions (see \cref{def:extension:East}) replaced by the corresponding $(\bone,W)$-traversability events\textsuperscript{\ref{foot:STW}} ($\cT^\bone_W$, see \cref{def:traversability}). Let $\cW'$ be the event that for every $n\in[0,\Ni]$, $j\in[4k]$ and segment $S\subset \L^{(\Ni)}_2$, perpendicular to $u_j$ of length $s_j^{(n)}/W$ at distance at most $W$ from the $u_j$-side of $\L^{(n)}$, the event $\cH^W(S)$ holds. Let $\cI'$ be the event that all sites in $\L^{(\Ni)}_2$ at distance at most $\sqrt{W}/\e$ from the origin are infected. Finally, set 
\[\bSG\left(\L_2^{(\Ni)}\right) =\SG'\left(\L'^{(\Ni)}\right)\cap\cW'\cap\cI'.\]
\end{defn}

\begin{defn}[Contracted semi-directed events on scale $\Ni+1$]
\label{def:bSG:loglog:1}
We define $\bcT(\L_1^{(\Ni+1)})$ to be the event that for all $j\in[k+2,3k-1]$ and every segment $S\subset\L_1^{(\Ni+1)}$, perpendicular to $u_j$ of length $s_j^{(\Ni)}/W$, the event $\cH^W(S)$ occurs and all sites in $\L_1^{(\Ni+1)}$ at distance at most $\sqrt{W}/\e$ from the origin are infected. 

Let $\bcT(\L_3^{(\Ni+1)})$ be the event that for all $j\in[4k]$, $m\in\{\Ni,\Ni+1\}$ and every segment $S\subset \L_3^{(\Ni+1)}$, perpendicular to $u_j$ of length $s_j^{(m)}/W$ at distance at most $W$ from the $u_j$-side of $\L^{(m)}$, the event $\cH^W(S)$ occurs.

For $n\in[0,\Ni]$ such that $2k n\in\bbN$ let
\[\L''^{(n)}=\L\left(\ur''^{(n)}\right)=\L\left(\ur^{(n)}-\l_1\left(\uv_1+\uv_{2k+1}\right)\right)\]
and define $\SG''(\L''^{(n)})$ like $\SG'(\L'^{(n)})$ in \cref{def:bSG:loglog:0}. Further let
\[\SG''\left(\L''^{(\Ni+1)}\right)=\SG''\left(\L''^{(\Ni)}\right)\cap\bigcap_{j\in\{0,2k\}}\ST^\bone_W\left(T\left(\ur''^{(\Ni)},l^{(\Ni)}/2,j\right)\right).\]
Let $\cW''$ (resp.\ $\cI''$) be defined like $\cW'$ (resp.\ $\cI'$) in \cref{def:bSG:loglog:0} with $\L'$ replaced by $\L''$ and $\Ni$ replaced by $\Ni+1$. Finally, we set
\[\bSG\left(\L_2^{(\Ni+1)}\right)=\SG''\left(\L''^{(\Ni+1)}\right)\cap \cW''\cap\cI''.\]
\end{defn}

Notice that \cref{def:bSG:log2} for $i\in[2,4k)$ does not inspect the internal structure of $\SG^\bone(\L^{(0)})$ (see \cref{subfig:bSG:log2:2}). Thus, we may use the exact same definition for $\bcT(\L_1^{(\Ni+i)})$, $\bSG(\L_2^{(\Ni+i)})$ and $\bcT(\L_3^{(\Ni+i)})$ with $i\in[2,4k)$.

We may now turn to the analogue of \cref{lem:log2:condition}.
\begin{lem}
\label{lem:loglog:condition}
For all $n\in[\Ni,\Ni+4k)$ we have $\bSG(\L_2^{(n)})\times\bcT(\L_3^{(n)})\subset \SG^\bone(\L_2^{(n)}\cup\L_3^{(n)})$ and similarly for $\L_1^{(n)}$ instead of $\L_3^{(n)}$.
\end{lem}
\begin{proof}
For $n\ge \Ni+2$ the proof is the same as in \cref{lem:log2:condition,lem:SG:condition:iso}.

Assume that $\bSG(\L_2^{(\Ni)})$ and $\bcT(\L_3^{(\Ni)})$ occur. We seek to prove by induction that for all $n\le \Ni$ the event $\SG^\bone(\L^{(n)})$ occurs. For $n=0$ this is true, since $\cI'$ and the corresponding part of $\bcT(\L_3^{(\Ni)})$ in \cref{def:bSG:loglog:0} give that $\L^{(0)}$ is fully infected. By \cref{def:SG:internal:loglog,def:extension:East}, it remains to show that for all $n<\Ni$ the event $\cT=\cT^\bone(T(\ur^{(n)},l^{(n)},j))$ occurs, where $j\in[4k]$ is such that $n-j/(2k)\in\bbN$. But by \cref{def:bSG:loglog:0} the corresponding event $\cT'=\cT^\bone_W(T(\ur'^{(n)},l^{(n)},j)$ occurs, where $\L'^{(n)}=\L(\ur'^{(n)})$. It therefore remains to observe that $\cW'$, the $W$-helping sets in the definition of $\bcT(\L_3^{(\Ni)})$ and $\cT'$ imply $\cT$. Indeed, $W$-helping sets ensure the occurrence of $\cH^\bone_{C^2}(S)$ for the first and last $\Theta(W)$ segments $S$ in \cref{def:traversability} for $\cT$, while the remaining ones are provided by $\cT'$, since $\ur'^{(n)}$ and $\ur^{(n)}$ only differ by $O(1)\ll W$. We omit the details, which are very similar to those in the proof of \cref{lem:log2:condition} (see \cref{subfig:bSG:loglog:0}).

The remaining three cases ($\L_1^{(\Ni)}$ instead of $\L_3^{(\Ni)}$ and/or $\Ni+1$ instead of $\Ni$) are treated analogously (see \cref{fig:bSG:loglog}).
\end{proof}

\begin{proof}[Proof of \cref{th:loglog:meso}]
By \cref{lem:loglog:condition}, \cref{eq:cor:CBSEP:reduction:condition} holds, so we may apply \cref{cor:CBSEP:reduction}. Together with the Harris inequality, \cref{eq:Harris:1,eq:Harris:2}, this gives
\begin{multline}
\label{eq:loglog:gamma}
\g\left(\L^{(\Ni+4k)}\right)\\\le\frac{\g(\L^{(\Ni)})\exp(O(C^2)\log^2(1/q))}{\displaystyle\prod_{i=\Ni}^{\Ni+4k-1}\m(\SG^\bone(\L^{(i+1)}))\m(\bcT(\L^{(i)}_1))\m(\bSG(\L_2^{(i)}))\m(\bcT(\L^{(i)}_3))}.
\end{multline}
In view of \cref{th:internal:loglog}, it remains to bound each of the terms in the denominator by $\exp(-1/(\e^{O(1)}q^\a))$ in order to conclude the proof of \cref{th:loglog:meso}.

Notice that a total of $\e^{-O(1)}$ fixed infections and $W^{O(1)}\Ni=q^{o(1)}$ $W$-helping sets are required in all the events in \cref{eq:loglog:gamma}. This amounts to a negligible factor. The probability of $\SG'(\L'^{(\Ni)})$ and $\SG''(\L''^{(\Ni)})$ can be bounded exactly like $\SG^\bone(\L^{(\Ni)})$ in \cref{lem:internal:loglog:probability}. This yields a contribution of $\exp(1/(\e^{O(1)}q^\a))$. Finally, the remaining bounded number of $\ST^\bone_W$ events are treated as in \cref{th:log2} to give a negligible $q^{-O(W)}$ factor. Hence, the proof of \cref{th:loglog:meso} is complete.
\end{proof}

\subsection{Global CBSEP dynamics}
\label{subsec:loglog:global}
The global dynamics is also based on the CBSEP mechanism and proceeds as in \cref{subsec:global:CBSEP,subsec:log2:global}
\begin{proof}[Proof of \cref{th:main}\ref{loglog}]
Let $\cU$ be semi-directed. Recall the droplets $\L^{(\Ni+i)}$ for $i\in[4k+1]$ from \cref{subsec:loglog:meso}. Set $\Lmp=\L^{(\Ni+4k)}$ and $\Lmm=\L^{(\Ni+2k)}$. Condition \ref{condition:CBSEP:1} of \cref{prop:global:CBSEP} is satisfied by \cref{th:loglog:meso}, while condition \ref{condition:CBSEP:2} is verified as in \cref{subsec:global:CBSEP}.
 
Thus, \cref{prop:global:CBSEP} applies and, together with \cref{th:loglog:meso} it yields
\[\Et\le\exp\left(\frac{\log\log(1/q)}{\e^{O(1)} q^\a}\right),\]
concluding the proof.
\end{proof}

\section{Balanced rooted models with finite number of stable directions}
\label{sec:log1}
In this section we deal with balanced rooted models with finite number of stable directions (class \ref{log1}). The internal dynamics (\cref{subsec:log1:internal}) uses a two-dimensional version of East-extensions. As usual, it requires the most work, but applies directly also to balanced models with infinite number of stable directions (class \ref{log0}). The mesoscopic and global dynamics are imported from \cite{Hartarsky21a} in \cref{subsec:global:FA}.

\subsection{East internal dynamics}
\label{subsec:log1:internal}
In this section we simultaneously treat balanced rooted models (classes \ref{log0} and \ref{log1}). We may therefore assume that $\a(u_j)\le \a$ for all $j\in[-k+1,k]$ and this is the only assumption on $\cU$ we use.

Let us start by motivating the coming two-dimensional East-extension we need. By the above assumption on the difficulties, we are allowed to use East-extensions in directions $u_0$ and $u_1$. Indeed, recalling \cref{def:extension:East}, we see that for these directions the traversability events (recall \cref{def:traversability}) only require helping sets and not $W$-helping sets. In principle, one could alternate East-extensions in these two directions similarly to what we did e.g.\ in \cref{subsec:loglog:internal} for directions $u_0,\dots,u_{2k-1}$. However, this would not work, because extensions in directions $u_0$ and $u_1$ only increase the length of the sides parallel to $u_0$ and $u_1$, while all others remain unchanged (see \cref{subfig:extension:East}). Thus, the traversability events would be too unlikely, since they would require helping sets also for the other sides, e.g.\ the one with outer normal $u_{2-k}$, which are too small. This would make the probability of the SG event too large. Notice that this issue does not arise when $k=1$, as we saw in \cref{subsec:loglog:internal}.

For $k>1$, however, we therefore need to make the $u_j$-sides of our successive droplets grow for all $j\in[-k+1,k]$. A natural way to achieve this is as depicted in \cref{fig:East:internal}. The drawback is that we can no longer achieve this directly with one-directional East-extensions as in \cref{def:extension:East,subfig:extension:East}, so we need some more definitions. However, morally, one such two-dimensional extension can be achieved by two East-extensions in the sense that, East-extending in direction $u_0$ and then $u_1$ yields a droplet which contains the desired droplet as in \cref{fig:East:internal}. Unfortunately, our approach heavily relies on not looking at the configuration outside the droplet itself. For that reason we instead need to find for each point in the droplet appropriate lengths of the East-extensions in directions $u_0$ and $u_1$, so as to cover the point without going outside the target droplet (see \cref{fig:East}).

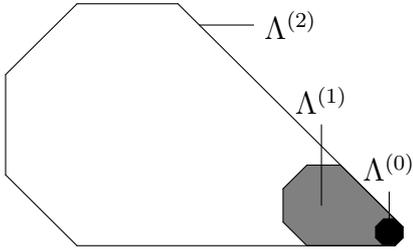
\begin{figure}
\floatbox[{\capbeside\thisfloatsetup{capbesideposition={right,top}}}]{figure}[\FBwidth]
{\centering
\begin{tikzpicture}[line cap=round,line join=round,>=triangle 45,x=0.013\textwidth,y=0.013\textwidth]
\fill[fill=black,fill opacity=1.0] (-1,0.41) -- (-1,-0.41) -- (-0.41,-1) -- (0.41,-1) -- (1,-0.41) -- (1,0.41) -- (0.41,1) -- (-0.41,1) -- cycle;
\fill[fill=black,fill opacity=0.5] (-3.59,5) -- (-6.07,5) -- (-7.83,3.24) -- (-7.83,0.76) -- (-6.07,-1) -- (0.41,-1) -- (1,-0.41) -- (1,0.41) -- cycle;
\draw (-1,0.41)-- (-1,-0.41);
\draw (-1,-0.41)-- (-0.41,-1);
\draw (-0.41,-1)-- (0.41,-1);
\draw (0.41,-1)-- (1,-0.41);
\draw (1,-0.41)-- (1,0.41);
\draw (1,0.41)-- (0.41,1);
\draw (0.41,1)-- (-0.41,1);
\draw (-0.41,1)-- (-1,0.41);
\draw (-15.59,17)-- (-23.04,17);
\draw (-23.04,17)-- (-28.31,11.73);
\draw (-28.31,11.73)-- (-28.31,4.27);
\draw (-28.31,4.27)-- (-23.04,-1);
\draw (-23.04,-1)-- (0.41,-1);
\draw (0.41,-1)-- (1,-0.41);
\draw (1,-0.41)-- (1,0.41);
\draw (1,0.41)-- (-15.59,17);
\draw (-3.59,5)-- (-6.07,5);
\draw (-6.07,5)-- (-7.83,3.24);
\draw (-7.83,3.24)-- (-7.83,0.76);
\draw (-7.83,0.76)-- (-6.07,-1);
\draw (-6.07,-1)-- (0.41,-1);
\draw (0.41,-1)-- (1,-0.41);
\draw (1,-0.41)-- (1,0.41);
\draw (1,0.41)-- (-3.59,5);
\draw (-14,15.42)--(-10,15.42) node [right] {$\L^{(2)}$};
\draw (-5,2)--(-5,8) node [above] {$\L^{(1)}$};
\draw (0,0)--(0,3) node [above] {$\L^{(0)}$};
\end{tikzpicture}}{\caption{Geometry of the droplets used for balanced rooted models in \cref{subsec:log1:internal} in the case $k=2$. The nested black, grey and white polygons are the droplets $\L^{(0)}$, $\L^{(1)}$ and $\L^{(2)}$ respectively.\label{fig:East:internal}}}
\end{figure}

Following \cref{subsec:loglog:internal} we define $\Nc,\Ni,\ell^{(n)}$ by \cref{eq:def:ln:internal:East}. In this section there are no fractional scales, so $n$ is an integer. Further let $\L^{(0)}$ be as in \cref{subsec:loglog:internal} with radii $\ur^{(0)}$ and side lengths $\us^{(0)}$. For $n\in[\Ni]$ set \[s_j^{(n)}=\begin{cases}s_j^{(0)}\ell^{(n)}&-k<j\le k\\
s_j^{(0)}&k+1<j<3k\end{cases}\]
and $s_{-k}^{(n)}$ and $s_{k+1}^{(n)}$ as required for $\us^{(n)}$ to be the side lengths of a droplet. Let $\ur^{(n)}$ be the corresponding radii such that $r_{-k}^{(n)}=r_{-k}^{(0)}$ and $r_{k+1}^{(n)}=r_{k+1}^{(0)}$. Finally, set $\L^{(n)}=\L(\ur^{(n)})$ as usual (see \cref{fig:East:internal}).

Fix $n\in[\Ni]$. Observe that we can cover $\L^{(n+1)}$ with droplets $(D_{\k})_{\k\in[K]}$ so that the following conditions all hold (see \cref{fig:Dk}).
\begin{itemize}
    \item For all $\k\in[K]$, $D_{\k}\subset \L^{(n+1)}$;
    \item $\bigcup_{\k=2}^{K-1}D_\k=\L^{(n+1)}$;
    \item $K=O(\ell^{(n+1)}/\ell^{(n)})$;
    \item any segment of length $\ell^{(n)}/(C\e)$ perpendicular to $u_j$ for some $j\in[4k]$ intersects at most $O(1)$ of the $D_{\k}$;
    \item droplets are assigned a \emph{generation} $g\in\{0,1,2\}$, so that only $D_0=\L^{(n)}$ is of generation $g=0$, only $D_1=\L(\ur^{(n)}+l_1\uv_1)$ is of generation $g=1$, where
    \[l_1=\frac{r^{(n+1)}_k-r^{(n)}_k}{\<u_1,u_k\>},\]
    so that $D_1$ spans the $u_{k+1}$-side of $\L^{(n+1)}$;
    \item if $\k\ge 2$, then $D_{\k}$ is of generation $g=2$, and is of the form \[D_\k=y_{\k}u_1+\L\left(\ur^{(n)}+l_{\k}\uv_0\right)\]
    for certain $l_{\k}\ge0$ and $y_{\k}\in[0,l_1]$ multiple of $\l_1$. 
\end{itemize}
To construct the $D_\k$ of generation $2$, it essentially suffices to increment $y_\k$ by $\Theta(\ell^{(n)}/\e)$ and define $l_\k$ to be the largest possible, so that $D_\k\subset \L^{(n+1)}$. Finally, we add to our collection of droplets the ones with $y_\k$ corresponding to a corner of $\L^{(n+1)}$ and again take $l_\k$ maximal (see \cref{fig:Dk}). Note that one is able to get $K=O(\ell^{(n+1)}/\ell^{(n)})$ thanks to the fact that $s^{(n)}_{-k}$ and $s_{k+1}^{(n)}$ are $\Theta(\ell^{(n)}/\e)$. We direct the interested reader to \cite{BalisterNaNb}*{Appendix E} for the explicit details of a similar construction in arbitrary dimension.

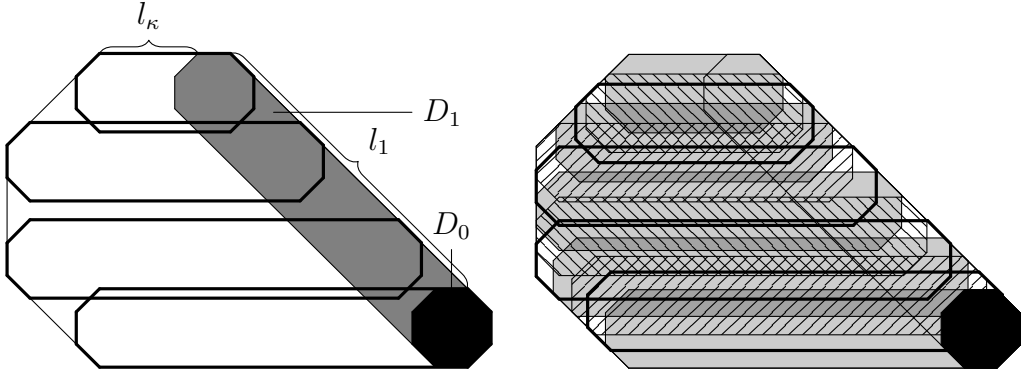
\begin{figure}
    \centering
\subcaptionbox{The droplets $D_\k$ corresponding to corners of $\L^{(n+1)}$. The generation $0$ droplet $D_0$ is given in black, while $D_1$ of generation $1$ is shaded.
\label{subfig:East:internal:corners}}{\centering
\begin{tikzpicture}[line cap=round,line join=round,>=triangle 45,x=0.52cm,y=0.52cm]
\fill[fill=black,fill opacity=1.0] (-0.41,1) -- (0.41,1) -- (1,0.41) -- (1,-0.41) -- (0.41,-1) -- (-0.41,-1) -- (-1,-0.41) -- (-1,0.41) -- cycle;
\fill[fill=black,fill opacity=0.5] (-7,6.41) -- (-6.41,7) -- (-5.59,7) -- (1,0.41) -- (1,-0.41) -- (0.41,-1) -- (-0.41,-1) -- (-7,5.59) -- cycle;
\draw (-0.41,1)-- (0.41,1);
\draw (0.41,1)-- (1,0.41);
\draw (1,0.41)-- (1,-0.41);
\draw (1,-0.41)-- (0.41,-1);
\draw (0.41,-1)-- (-0.41,-1);
\draw (-0.41,-1)-- (-1,-0.41);
\draw (-1,-0.41)-- (-1,0.41);
\draw (-1,0.41)-- (-0.41,1);
\draw (-5.59,7)-- (-8.9,7);
\draw (-8.9,7)-- (-11.24,4.66);
\draw (-11.24,4.66)-- (-11.24,1.34);
\draw (-11.24,1.34)-- (-8.9,-1);
\draw (-8.9,-1)-- (0.41,-1);
\draw (0.41,-1)-- (1,-0.41);
\draw (1,-0.41)-- (1,0.41);
\draw (1,0.41)-- (-5.59,7);
\draw [very thick] (-8.9,-1)-- (-9.49,-0.41);
\draw [very thick] (-9.49,-0.41)-- (-9.49,0.41);
\draw [very thick] (-9.49,0.41)-- (-8.9,1);
\draw [very thick] (-8.9,1)-- (0.41,1);
\draw [very thick] (0.41,1)-- (1,0.41);
\draw [very thick] (1,0.41)-- (1,-0.41);
\draw [very thick] (1,-0.41)-- (0.41,-1);
\draw [very thick] (0.41,-1)-- (-8.9,-1);
\draw [very thick] (-10.66,0.76)-- (-11.24,1.34);
\draw [very thick] (-11.24,1.34)-- (-11.24,2.17);
\draw [very thick] (-11.24,2.17)-- (-10.66,2.76);
\draw [very thick] (-10.66,2.76)-- (-1.34,2.76);
\draw [very thick] (-1.34,2.76)-- (-0.76,2.17);
\draw [very thick] (-0.76,2.17)-- (-0.76,1.34);
\draw [very thick] (-0.76,1.34)-- (-1.34,0.76);
\draw [very thick] (-1.34,0.76)-- (-10.66,0.76);
\draw [very thick] (-11.24,4.66)-- (-11.24,3.83);
\draw [very thick] (-11.24,3.83)-- (-10.66,3.24);
\draw [very thick] (-10.66,3.24)-- (-3.83,3.24);
\draw [very thick] (-3.83,3.24)-- (-3.24,3.83);
\draw [very thick] (-3.24,3.83)-- (-3.24,4.66);
\draw [very thick] (-3.24,4.66)-- (-3.83,5.24);
\draw [very thick] (-3.83,5.24)-- (-10.66,5.24);
\draw [very thick] (-10.66,5.24)-- (-11.24,4.66);
\draw [very thick] (-8.9,5)-- (-9.49,5.59);
\draw [very thick] (-9.49,5.59)-- (-9.49,6.41);
\draw [very thick] (-9.49,6.41)-- (-8.9,7);
\draw [very thick] (-8.9,7)-- (-5.59,7);
\draw [very thick] (-5.59,7)-- (-5,6.41);
\draw [very thick] (-5,6.41)-- (-5,5.59);
\draw [very thick] (-5,5.59)-- (-5.59,5);
\draw [very thick] (-5.59,5)-- (-8.9,5);
\draw (-7,6.41)-- (-6.41,7);
\draw (-6.41,7)-- (-5.59,7);
\draw (-5.59,7)-- (1,0.41);
\draw (1,0.41)-- (1,-0.41);
\draw (1,-0.41)-- (0.41,-1);
\draw (0.41,-1)-- (-0.41,-1);
\draw (-0.41,-1)-- (-7,5.59);
\draw (-7,5.59)-- (-7,6.41);
\draw [decorate,decoration={brace,amplitude=5pt}] (-8.9,7) -- (-6.41,7) node [midway,yshift=  0.5cm] {$l_\k$};
\draw [decorate,decoration={brace,amplitude=5pt}] (-5.59,7) -- (0.41,1) node [midway,yshift= 0.4cm,xshift=0.4cm] {$l_1$};
\draw (-4.5,5.5)--(-1,5.5) node [right] {$D_1$};
\draw (0,0)--(0,2) node [above] {$D_0$};
\end{tikzpicture}
}\quad
\subcaptionbox{All droplets $D_\k$. In the second generation, for visibility, droplets alternate between shaded, thickened and hatched.\label{subfig:East:internal:all}}{\centering
\begin{tikzpicture}[line cap=round,line join=round,>=triangle 45,x=0.52cm,y=0.52cm]
\fill[fill=black,fill opacity=1.0] (-0.41,1) -- (0.41,1) -- (1,0.41) -- (1,-0.41) -- (0.41,-1) -- (-0.41,-1) -- (-1,-0.41) -- (-1,0.41) -- cycle;
\fill[fill=black,fill opacity=0.2] (-8.9,-1) -- (-9.49,-0.41) -- (-9.49,0.41) -- (-8.9,1) -- (0.41,1) -- (1,0.41) -- (1,-0.41) -- (0.41,-1) -- cycle;
\fill[pattern=my north east lines] (-11.24,4.66) -- (-11.24,3.83) -- (-10.66,3.24) -- (-3.83,3.24) -- (-3.24,3.83) -- (-3.24,4.66) -- (-3.83,5.24) -- (-10.66,5.24) -- cycle;
\fill[fill=black,fill opacity=0.2] (-8.9,5) -- (-9.49,5.59) -- (-9.49,6.41) -- (-8.9,7) -- (-5.59,7) -- (-5,6.41) -- (-5,5.59) -- (-5.59,5) -- cycle;
\fill[fill=black,fill opacity=0.2] (-10.22,0.32) -- (-10.81,0.91) -- (-10.81,1.74) -- (-10.22,2.32) -- (-0.91,2.32) -- (-0.32,1.74) -- (-0.32,0.91) -- (-0.91,0.32) -- cycle;
\fill[pattern=my north east lines] (-9.75,-0.15) -- (-10.34,0.44) -- (-10.34,1.27) -- (-9.75,1.85) -- (-0.44,1.85) -- (0.15,1.27) -- (0.15,0.44) -- (-0.44,-0.15) -- cycle;
\fill[pattern=my north west lines] (-11.24,2.77) -- (-10.66,3.36) -- (-1.95,3.36) -- (-1.36,2.77) -- (-1.36,1.95) -- (-1.95,1.36) -- (-10.66,1.36) -- (-11.24,1.95) -- cycle;
\fill[fill=black,fill opacity=0.2] (-11.24,3.42) -- (-10.66,4) -- (-2.59,4) -- (-2,3.42) -- (-2,2.59) -- (-2.59,2) -- (-10.66,2) -- (-11.24,2.59) -- cycle;
\fill[fill=black,fill opacity=0.2] (-10.74,5.16) -- (-10.74,4.33) -- (-10.15,3.75) -- (-4.33,3.75) -- (-3.75,4.33) -- (-3.75,5.16) -- (-4.33,5.75) -- (-10.15,5.75) -- cycle;
\fill[pattern=my north west lines] (-9.4,6.5) -- (-9.98,5.92) -- (-9.98,5.09) -- (-9.4,4.5) -- (-5.09,4.5) -- (-4.5,5.09) -- (-4.5,5.92) -- (-5.09,6.5) -- cycle;
\draw (-0.41,1)-- (0.41,1);
\draw (0.41,1)-- (1,0.41);
\draw (1,0.41)-- (1,-0.41);
\draw (1,-0.41)-- (0.41,-1);
\draw (0.41,-1)-- (-0.41,-1);
\draw (-0.41,-1)-- (-1,-0.41);
\draw (-1,-0.41)-- (-1,0.41);
\draw (-1,0.41)-- (-0.41,1);
\draw (-5.59,7)-- (-8.9,7);
\draw (-8.9,7)-- (-11.24,4.66);
\draw (-11.24,4.66)-- (-11.24,1.34);
\draw (-11.24,1.34)-- (-8.9,-1);
\draw (-8.9,-1)-- (0.41,-1);
\draw (0.41,-1)-- (1,-0.41);
\draw (1,-0.41)-- (1,0.41);
\draw (1,0.41)-- (-5.59,7);
\draw (-8.9,-1)-- (-9.49,-0.41);
\draw (-9.49,-0.41)-- (-9.49,0.41);
\draw (-9.49,0.41)-- (-8.9,1);
\draw (-8.9,1)-- (0.41,1);
\draw (0.41,1)-- (1,0.41);
\draw (1,0.41)-- (1,-0.41);
\draw (1,-0.41)-- (0.41,-1);
\draw (0.41,-1)-- (-8.9,-1);
\draw [very thick] (-10.66,0.76)-- (-11.24,1.34);
\draw [very thick] (-11.24,1.34)-- (-11.24,2.17);
\draw [very thick] (-11.24,2.17)-- (-10.66,2.76);
\draw [very thick] (-10.66,2.76)-- (-1.34,2.76);
\draw [very thick] (-1.34,2.76)-- (-0.76,2.17);
\draw [very thick] (-0.76,2.17)-- (-0.76,1.34);
\draw [very thick] (-0.76,1.34)-- (-1.34,0.76);
\draw [very thick] (-1.34,0.76)-- (-10.66,0.76);
\draw (-11.24,4.66)-- (-11.24,3.83);
\draw (-11.24,3.83)-- (-10.66,3.24);
\draw (-10.66,3.24)-- (-3.83,3.24);
\draw (-3.83,3.24)-- (-3.24,3.83);
\draw (-3.24,3.83)-- (-3.24,4.66);
\draw (-3.24,4.66)-- (-3.83,5.24);
\draw (-3.83,5.24)-- (-10.66,5.24);
\draw (-10.66,5.24)-- (-11.24,4.66);
\draw (-8.9,5)-- (-9.49,5.59);
\draw (-9.49,5.59)-- (-9.49,6.41);
\draw (-9.49,6.41)-- (-8.9,7);
\draw (-8.9,7)-- (-5.59,7);
\draw (-5.59,7)-- (-5,6.41);
\draw (-5,6.41)-- (-5,5.59);
\draw (-5,5.59)-- (-5.59,5);
\draw (-5.59,5)-- (-8.9,5);
\draw [very thick] (-9.35,-0.55)-- (-9.94,0.04);
\draw [very thick] (-9.94,0.04)-- (-9.94,0.87);
\draw [very thick] (-9.94,0.87)-- (-9.35,1.45);
\draw [very thick] (-9.35,1.45)-- (-0.04,1.45);
\draw [very thick] (-0.04,1.45)-- (0.55,0.87);
\draw [very thick] (0.55,0.87)-- (0.55,0.04);
\draw [very thick] (0.55,0.04)-- (-0.04,-0.55);
\draw [very thick] (-0.04,-0.55)-- (-9.35,-0.55);
\draw (-10.22,0.32)-- (-10.81,0.91);
\draw (-10.81,0.91)-- (-10.81,1.74);
\draw (-10.81,1.74)-- (-10.22,2.32);
\draw (-10.22,2.32)-- (-0.91,2.32);
\draw (-0.91,2.32)-- (-0.32,1.74);
\draw (-0.32,1.74)-- (-0.32,0.91);
\draw (-0.32,0.91)-- (-0.91,0.32);
\draw (-0.91,0.32)-- (-10.22,0.32);
\draw (-9.75,-0.15)-- (-10.34,0.44);
\draw (-10.34,0.44)-- (-10.34,1.27);
\draw (-10.34,1.27)-- (-9.75,1.85);
\draw (-9.75,1.85)-- (-0.44,1.85);
\draw (-0.44,1.85)-- (0.15,1.27);
\draw (0.15,1.27)-- (0.15,0.44);
\draw (0.15,0.44)-- (-0.44,-0.15);
\draw (-0.44,-0.15)-- (-9.75,-0.15);
\draw [very thick] (-11.24,4.05)-- (-10.66,4.64);
\draw [very thick] (-10.66,4.64)-- (-3.22,4.64);
\draw [very thick] (-3.22,4.64)-- (-2.64,4.05);
\draw [very thick] (-2.64,4.05)-- (-2.64,3.22);
\draw [very thick] (-2.64,3.22)-- (-3.22,2.64);
\draw [very thick] (-3.22,2.64)-- (-10.66,2.64);
\draw [very thick] (-10.66,2.64)-- (-11.24,3.22);
\draw [very thick] (-11.24,3.22)-- (-11.24,4.05);
\draw (-11.24,2.77)-- (-10.66,3.36);
\draw (-10.66,3.36)-- (-1.95,3.36);
\draw (-1.95,3.36)-- (-1.36,2.77);
\draw (-1.36,2.77)-- (-1.36,1.95);
\draw (-1.36,1.95)-- (-1.95,1.36);
\draw (-1.95,1.36)-- (-10.66,1.36);
\draw (-10.66,1.36)-- (-11.24,1.95);
\draw (-11.24,1.95)-- (-11.24,2.77);
\draw (-11.24,3.42)-- (-10.66,4);
\draw (-10.66,4)-- (-2.59,4);
\draw (-2.59,4)-- (-2,3.42);
\draw (-2,3.42)-- (-2,2.59);
\draw (-2,2.59)-- (-2.59,2);
\draw (-2.59,2)-- (-10.66,2);
\draw (-10.66,2)-- (-11.24,2.59);
\draw (-11.24,2.59)-- (-11.24,3.42);
\draw (-10.74,5.16)-- (-10.74,4.33);
\draw (-10.74,4.33)-- (-10.15,3.75);
\draw (-10.15,3.75)-- (-4.33,3.75);
\draw (-4.33,3.75)-- (-3.75,4.33);
\draw (-3.75,4.33)-- (-3.75,5.16);
\draw (-3.75,5.16)-- (-4.33,5.75);
\draw (-4.33,5.75)-- (-10.15,5.75);
\draw (-10.15,5.75)-- (-10.74,5.16);
\draw (-9.4,6.5)-- (-9.98,5.92);
\draw (-9.98,5.92)-- (-9.98,5.09);
\draw (-9.98,5.09)-- (-9.4,4.5);
\draw (-9.4,4.5)-- (-5.09,4.5);
\draw (-5.09,4.5)-- (-4.5,5.09);
\draw (-4.5,5.09)-- (-4.5,5.92);
\draw (-4.5,5.92)-- (-5.09,6.5);
\draw (-5.09,6.5)-- (-9.4,6.5);
\draw [very thick] (-9.66,6.24)-- (-4.83,6.24);
\draw [very thick] (-4.83,6.24)-- (-4.24,5.66);
\draw [very thick] (-4.24,5.66)-- (-4.24,4.83);
\draw [very thick] (-4.24,4.83)-- (-4.83,4.24);
\draw [very thick] (-4.83,4.24)-- (-9.66,4.24);
\draw [very thick] (-9.66,4.24)-- (-10.24,4.83);
\draw [very thick] (-10.24,4.83)-- (-10.24,5.66);
\draw [very thick] (-10.24,5.66)-- (-9.66,6.24);
\draw (-6.41,7)-- (-5.59,7);
\draw (-5.59,7)-- (1,0.41);
\draw (1,0.41)-- (1,-0.41);
\draw (1,-0.41)-- (0.41,-1);
\draw (0.41,-1)-- (-0.41,-1);
\draw (-0.41,-1)-- (-7,5.59);
\draw (-7,5.59)-- (-7,6.41);
\draw (-7,6.41)-- (-6.41,7);
\end{tikzpicture}
}
\caption{\label{fig:Dk}Geometry of the droplets $(D_\k)_{\k\in[K]}$ used in the two-dimensional East-extension in \cref{def:extension:East:multid}. Also recall \cref{fig:East:internal}.}
    \label{fig:East}
\end{figure}

\begin{defn}[$n$-traversability]
\label{def:n-traversable}
Fix $n\in[\Ni]$ and let $R\subset\L^{(n+1)}$ be a region of the form
\begin{equation}
\label{eq:R:form}
\bigcup_{I\in\cI}\left(\bigcap_{\k\in I} D_\k\setminus \bigcup_{\k\in[K]\setminus I} D_\k\right)
\end{equation}
for some family $\cI$ of subsets of $[K]$. We say that $R$ is $n$-\emph{traversable} ($\cT_n(R)$ occurs\footnote{The $n$-traversability $\cT_n$ should not be confused with $(\o,d)$-traversability $\cT_d^\o$ from \cref{def:traversability}, which only features with $d=0$ and $\o=\bone$ in the present section.}) if for all $j\in(-k,k)$ and every segment $S\subset R$ perpendicular to $u_j$ of length at least $\d \ell^{(n)}/\e$ the following two conditions hold.
\begin{itemize}
\item If $S$ is at distance at least $W$ from the boundary of all $D_{\k}$, then the event $\cH(S)$ occurs. 
\item If $S$ is at distance at most $W$ from a side of a $D_{\k}$ parallel to $S$ for some $\k\in[K]$, but $S$ does not intersect any non-parallel side of any $D_{\k'}$, then the event $\cH^W(S)$ occurs.
\end{itemize}
\end{defn}
Roughly speaking, $R$ must be one of the polygonal pieces into which the boundaries of all $D_\k$ cut $\L^{(n+1)}$. It is $n$-traversable, if segments of the size slightly smaller than $\L^{(n)}$ contain helping sets for the directions in $(-k,k)$. However, we only require this slightly away from the boundaries of $D_\k$ and instead add $W$-helping sets close to boundaries, so that we can still cross them but keep the following independence.
\begin{rem}
\label{rem:product}
Note that $n$-traversability events are product over the disjoint regions into which all the boundaries of $(D_\k)_{\k\in[K]}$ partition $\L^{(n+1)}$.
\end{rem}
\begin{defn}[Two-dimensional East-extension]
\label{def:extension:East:multid}
For $n\in[\Ni]$ we say that we \emph{East-extend} $\L^{(n)}$ to $\L^{(n+1)}$ if $\SG^\bone(D_1)$ is defined by East-extending $\L^{(n)}$ by $l_1$ in direction $u_1$ and $\SG^\bone(\L^{(n+1)})=\SG^\bone(D_1)\cap\cT_n(\L^{(n+1)}\setminus D_1)$.
\end{defn} 
Indeed, \cref{def:n-traversable} gives $\cT_n(\L^{(n+1)}\setminus D_1)$, since \cref{eq:R:form} is satisfied:
\[\L^{(n+1)}\setminus D_1=\bigcup_{\k\in[K]}D_\k\setminus D_1=\bigcup_{I\subset[K]\setminus\{0,1\}}\left(\bigcap_{\k\in I}D_\k\setminus\bigcup_{\k\not\in I}D_\k\right).\]
Armed with this notion, we are ready to define our SG events up to the internal scale for our models of interest.
\begin{defn}[Balanced rooted internal SG]
\label{def:SG:balanced:internal}
Let $\cU$ be balanced rooted. We say that $\L^{(0)}$ is SG ($\SG^\bone(\L^{(0)}$ occurs), if all sites in $\L^{(0)}$ are infected. We then recursively define $\SG^\bone(\L^{(n+1)})$ for $n\in[\Ni]$ by East-extending $\L^{(n)}$ to $\L^{(n+1)}$ (see \cref{def:extension:East:multid}).
\end{defn}

We are now ready to state our bound on the probability of $\SG^\bone(\L^{(\Ni)})$ and $\g(\L^{(\Ni)})$ (recall \cref{subsec:poincare}).
\begin{thm}
\label{th:internal:East}
Let $\cU$ be balanced rooted (classes \ref{log0} and \ref{log1}). Then
\begin{align*}
\g\left(\L^{(\Ni)}\right)&{}\le \exp\left(\frac{\log(1/q)\log\log\log(1/q)}{\e^3q^\a}\right),\\\m\left(\SG^\bone\left(\L^{(\Ni)}\right)\right)&{}\ge\exp\left(\frac{-1}{\e^2q^\a}\right).
\end{align*}
\end{thm}
The rest of \cref{subsec:log1:internal} is dedicated to the proof of \cref{th:internal:East}. As usual, the probability bound is not hard (see \cref{lem:internal:East:probability} below), while the relaxation time is bounded recursively. However, we need to obtain such a recursive relation, using \cref{cor:East:reduction} twice (see \cref{lem:internal:East:step} below). Yet, thanks to the additional $\log(1/q)$ factor as compared to \cref{th:internal:loglog} (and the $\log\log\log(1/q)$ one, see \cref{rem:logloglog}), the computations need not be as precise and, in particular, do not rely on \cref{cor:perturbation}.

Note that $\g(\L^{(0)})=1$, since \cref{eq:def:gamma} is trivial, as $\SG^\bone(\L^{(0)})$ is a singleton. For $m\ge 1$ and $n\in [\Ni]$ denote
\begin{equation}
\label{eq:def:amn}a^{(n)}_{m}=\max_{j\in\{0,1\}}\m^{-1}\left(\left.\SG^\bone\left(\L^{(n)}+\left(\lfloor (3/2)^{m+1}\rfloor-\lfloor(3/2)^{m}\rfloor\right)\l_j u_j\right)\right|\SG^\bone\left(\L^{(n)}\right)\right).
\end{equation}
For the sake of simplifying expressions we abusively assume that for all $\k\in[K]$ the length $l_\k$ is of the form $\l_0\lfloor(3/2)^m\rfloor$ with integer $m$. Without this assumption, one would need to treat the term corresponding to $m=M-1$ in \cref{cor:East:reduction} separately, but identically. We next deduce \cref{th:internal:East} from the following two lemmas.
\begin{lem}
\label{lem:internal:East:step}
For $n<\Ni$ we have
\[\g\left(\L^{(n+1)}\right)\le \frac{\g(\L^{(n)})e^{O(C^2)\log^2(1/q)}}{(\m(\SG^\bone(\L^{(n+1)}))\m(\cT_n(\L^{(n+1)})))^{O(1)}}\prod_{m=1}^{M^{(n)}} a_m^{(n)}
,\]
where $M^{(n)}=\lceil1/\e\rceil+\lceil \log \ell^{(n+1)}/\log(3/2)\rceil$.
\end{lem}

\begin{lem}
\label{lem:internal:East:probability}
For any $n\le\Ni$ and $m\ge1$ we have
\begin{align}
\nonumber
a_m^{(n)}\le{}&\m^{-1}\left(\SG^\bone\left(\L^{(n)}\right)\right)\le \m^{-1}\left(\SG^\bone\left(\L^{(n)}\right)\right)\m^{-1}\left(\cT_{n-1}\left(\L^{(n)}\right)\right)\\
\label{eq:mu:SG:bound}\le{}&\min\left(\left(\d q^\a W^n\right)^{-W^n/\e^2},e^{1/(\e^2q^\a)}\right).
\end{align}
\end{lem}
\begin{proof}[Proof of \cref{th:internal:East}]From \cref{lem:internal:East:step,lem:internal:East:probability} and the explicit expressions \cref{eq:def:ln:internal:East}, we get
\begin{align*}
\label{eq:gLni:decomposition}
\g\left(\L^{(\Ni)}\right)&{}\le e^{\log^{O(1)}(1/q)}\prod_{n=0}^{\Ni-1}\left(\m\left(\SG^\bone\left(\L^{(n+1)}\right)\right)\m\left(\cT_n\left(\L^{(n+1)}\right)\right)\right)^{-O(1)}\prod_{m=1}^{M^{(n)}}a_m^{(n)}\\
&{}\le e^{\log^{O(1)}(1/q)}\prod_{n=0}^{\Ni-1}\left(\m\left(\SG^\bone\left(\L^{(n+1)}\right)\right)\m\left(\cT_n\left(\L^{(n+1)}\right)\right)\right)^{-O(\log(1/q))}\\
&{}\le\exp\left(\frac{\log(1/q)\log\log\log(1/q)}{\e^3q^\a}\right).
\end{align*}
Since the second inequality in \cref{th:internal:East} is contained in \cref{lem:internal:East:probability}, this concludes the proof of the theorem modulo \cref{lem:internal:East:probability,lem:internal:East:step}.    
\end{proof}

\begin{proof}[Proof of \cref{lem:internal:East:step}]
Let us start by recalling a general fact about product measures. Consider two disjoint regions $A,B\subset\bbZ^2$ and a product measure $\n$ on $\O_{A}\times\O_B$. The law of total variance and convexity give
\begin{equation}
\label{eq:var:decomposition:AB}
\var_{\n_{A\cup B}}(f)=\n_{B}\left(\var_{\n_A}(f)\right)+\var_{\n_{B}}\left(\n_{A}(f)\right)\le \n(\var_{\n_A}(f)+\var_{\n_B}(f)).
\end{equation}

Fix $n\in[\Ni]$. Applying \cref{eq:var:decomposition:AB} several times (in view of \cref{rem:product,def:extension:East:multid}), we obtain
\begin{align}
\label{eq:cover}&\var_{\L^{(n+1)}}\left(f|\SG^\bone\left(\L^{(n+1)}\right)\right)\\
\nonumber&{}\le \m_{\L^{(n+1)}}\left(\left.\var_{D_1}\left(f|\SG^\bone(D_1)\right)+\sum_{\k=2}^{K-1}\var_{R_\k}\left(f|\cT_n\left(R_\k\right)\right)\right|\SG^\bone\left(\L^{(n+1)}\right)\right)\\
\nonumber&{}\le \sum_{\k=1}^{K-1}\m_{\L^{(n+1)}}\left(\left.\var_{D_\k\cup D_1}\left(f|\SG^\bone(D_1),\cT_n(D_\k\setminus D_1)\right)\right|\SG^\bone\left(\L^{(n+1)}\right)\right),
\end{align}
where $R_\k=D_\k\setminus\bigcup_{\k'=1}^{\k-1}D_{\k'}$. Since the terms above are treated identically (except $\k=1$, which is actually simpler), without loss of generality we focus on $\k=2$.

Recall from \cref{def:extension:East:multid} that $\SG^\bone(D_1)$ was defined by East-extending $D_0$ in direction $u_1$. Further East-extend $D_0$ by $l_2$ (recall that $D_2=y_2u_1+\L(\ur^{(n)}+l_2\uv_0)$) in direction $u_0$, so that $\SG^\bone(D_2)$ is also defined. Let $V=D_1\cup D_2$ (that is a $\dashv$ shaped region in \cref{fig:Dk}) and
\begin{equation}
\label{eq:def:SG:V}\SG^\bone(V)=\SG^\bone(D_1)\cap \cT_n(D_2\setminus D_1).
\end{equation}

Using a two-block dynamics (see e.g.\ \cref{lem:aux:east}), we have
\begin{multline}
\label{eq:varvg}
\var_{V}(f|\SG^\bone(V))\\\le\frac{\m_{V}(\var_{D_1}(f|\SG^\bone(D_1))+\1_{\cE}\var_{V\setminus D_1}(f|\cT_n(V\setminus D_1))|\SG^\bone(V))}{\O(\m(\cE|\SG^\bone(V)))},
\end{multline}
where
\begin{equation}
\label{eq:def:cE}\cE=\SG^\bone\left(\L^{(n)}+y_2u_1\right)\cap\cT_n\left(D_1\cap D_2\right)\subset\O_{D_1}.\end{equation}
Recalling \cref{def:traversability,def:n-traversable,def:extension:East:multid}, \cref{eq:def:cE} and the fact that each segment of length $\ell^{(n)}/(\e C)\gg \d \ell^{(n)}/\e$ intersects at most $O(1)$ droplets, we see that
\begin{align}
\nonumber\cE\cap\cT_n(V\setminus D_1)&{}\subset\SG^\bone\left(\L^{(n)}+y_2u_1\right)\cap\cT^\bone\left(D_2\setminus \left(\L^{(n)}+y_2u_1\right)\right)\\
\label{eq:ETn:SG}&{}= \SG^\bone(D_2).
\end{align}
By \cref{eq:ETn:SG} and convexity of the variance, we obtain
\begin{multline}
\label{eq:mVgE}
\m_{V}\left(\left.\1_{\cE}\var_{V\setminus D_1}(f|\cT_n(V\setminus D_1))\right|\SG^\bone(V)\right)\\
\begin{aligned}
&{}\le\frac{\m(\cE)}{\m(\SG^\bone(V))}\m_{V}\left(\var_{D_2}\left(f|\cE\cap\cT_n\left(V\setminus D_1\right)\right)\right)\\
&{}\le\frac{\m(\cE)\m(\SG^\bone(D_2))\m_{V}(\var_{D_2}(f|\SG^\bone(D_2)))}{\m(\SG^\bone(V))\m(\cE\cap\cT_n(V\setminus D_1))}\\
&{}\le \frac{\m_{V}(\var_{D_2}(f|\SG^\bone(D_2)))}{\m^2(\cT_n(\L^{(n+1)}))}.
\end{aligned}
\end{multline}
Indeed, in the last line we recalled the definitions of $\SG^\bone(D_2)$, $\SG^\bone(V)$ and $\cE$ (see \cref{def:extension:East,eq:def:cE,eq:def:SG:V}), while in the second one we took into account that for any events $\cA\subset\cB$ with $\m(\cA)>0$ it holds that 
\begin{equation}
\label{eq:variance:conditioning:enlarge}
\var(f|\cA)=\min_{c\in\bbR}\m\left(\left.(f-c)^2\right|\cA\right)\le \frac{\m((f-\m(f|\cB))^2\1_{\cA})}{\m(\cA)}\le\frac{\m(\cB)}{\m(\cA)}\var(f|\cB)
\end{equation}
and \cref{eq:ETn:SG}.

We plug \cref{eq:mVgE} in \cref{eq:varvg} and note that by the Harris inequality, \cref{eq:Harris:1,eq:Harris:2}, $\m(\cE|\SG^\bone(V))\ge \m(\cE)\ge \m(\SG^\bone(\L^{(n)}))\m(\cT_n(\L^{(n+1)}))$. This yields
\begin{align}
\nonumber\var_{V}\left(f|\SG^\bone(V)\right)&{}\le\frac{O(1)\m_{V}(\var_{D_1}(f|\SG^\bone(D_1))+\var_{D_2}(f|\SG^\bone(D_2)))}{\m(\SG^\bone(\L^{(n)}))\m(\SG^\bone(V))\m^3(\cT_n(\L^{(n+1)}))}\\
&{}\le\frac{O(1)\m_{V}(\var_{D_1}(f|\SG^\bone(D_1))+\var_{D_2}(f|\SG^\bone(D_2)))}{\m^2(\SG^\bone(\L^{(n+1)}))\m^3(\cT_n(\L^{(n+1)}))}
\label{eq:varvg2}
\end{align}
where the second inequality uses \cref{eq:def:SG:V,def:extension:East:multid}.

As in \cref{eq:loglog:internal:ambound,eq:loglog:internal:prod:tubes}, \cref{cor:East:reduction} gives
\begin{multline}
\label{eq:gDk}\g(D_2)\le \max\left(\g\left(\L^{(n)}\right),\m^{-1}\left(\SG^\bone\left(\L^{(n)}\right)\right)\right)e^{O(C^2)\log^2(1/q)}q^{-O(WM)}\\
\times\frac{\m(\SG^\bone(\L^{(n)}))}{\m(\SG^\bone(D_2))}\prod_{m=1}^{M}a_m^{(n)}
\end{multline}
with $M=\min\{m:\l_0(3/2)^{m+1}\ge l_2\}\le M^{(n)}$. Plugging \cref{eq:gDk,eq:def:gamma} (and their analogues for $D_1$) into \cref{eq:varvg2}, we obtain
\begin{align*}\g(V)&{}\le\frac{\g(\L^{(n)})e^{O(C^2)\log^2(1/q)}\prod_{m=1}^{M^{(n)}}a_m^{(n)}}{\m^3(\SG^\bone(\L^{(n+1)}))\m^3(\cT_n(\L^{(n+1)}))\min_\k\m(\SG^\bone(D_\k))}\\
&{}\le \frac{\g\left(\L^{(n)}\right)e^{O(C^2)\log^2(1/q)}\prod_{m=1}^{M^{(n)}}a_m^{(n)}}{\m^4\left(\SG^\bone\left(\L^{(n+1)}\right)\right)\m^4(\cT_n(\L^{(n+1)}))},\end{align*}
where the last inequality uses \cref{eq:ETn:SG} and that $\SG^\bone(D_1)\supset\SG^\bone(\L^{(n+1)})$ by \cref{def:extension:East:multid}. Plugging this into \cref{eq:cover}, concludes the proof of \cref{lem:internal:East:step}, since $K=O(\ell^{(n+1)}/\ell^{(n)})\le O(\log^4(1/q))$, as noted in \cref{rem:scales}.
\end{proof}

\begin{proof}[Proof of \cref{lem:internal:East:probability}]
The first inequality in \cref{eq:mu:SG:bound} follows from the Harris inequality \cref{eq:Harris:2}, while the second one is trivial. Therefore, we turn to the last one and fix $n\in[\Ni]$. Note that by \cref{def:extension:East,def:n-traversable,def:extension:East:multid}
\begin{equation}
\label{eq:mu:bound:rec}
\m\left(\SG^\bone\left(\L^{(n+1)}\right)\right)\ge\m\left(\SG^\bone\left(\L^{(n)}\right)\right)\m\left(\cT_n\left(\L^{(n+1)}\right)\right)\m\left(\cT^\bone\left(D_1\setminus D_0\right)\right).
\end{equation}
We therefore proceed by induction starting with 
\begin{equation}
\label{eq:mu:bound:base}
\m\left(\SG^\bone\left(\L^{(0)}\right)\right)=q^{|\L^{(0)}|}=q^{\Theta(1/\e^2)}.\end{equation}

We observe that from \cref{def:n-traversable}, in order to ensure the occurrence of $\cT_n(\L^{(n+1)})$, it suffices to have $O(WK\ell^{(n+1)})/(\ell^{(n)}\d)$ well-placed $W$-helping sets and $O((\ell^{(n+1)})^2)/(\ell^{(n)}\d\e)$ helping sets for segments of length $\d\ell^{(n)}/(3\e)$. Indeed, we may split lines perpendicular to each $u_j$ for $j\in(-k,k)$ into successive disjoint segments of length $\d\ell^{(n)}/(3\e)$ with a possible smaller leftover. It is then sufficient to place $W$-helping sets or helping sets depending on whether the segment under consideration is close to a parallel boundary of one of the $D_\k$ or not. Note that here we crucially use the assumption that each segment of length $\ell^{(n)}/(C\e)\gg \d\ell^{(n)}/\e$ intersects only $O(1)$ droplets.

Recall that $1/\e\gg1/\d\gg W\gg 1$, $\ell^{(\Nc)}=W^{O(1)}q^\a$, $K=O(\ell^{(n+1)}/\ell^{(n)})\le\log^{O(1)}(1/q)$, the explicit expressions \cref{eq:def:ln:internal:East} and \cref{obs:mu:helping}. Then the Harris inequality \cref{eq:Harris:1}, yields
\begin{multline}\m\left(\cT_n\left(\L^{(n+1)}\right)\right)\\
\begin{aligned}\ge{}& 
q^{O(W^2K\ell^{(n+1)})/(\ell^{(n)}\d)}\left(1-e^{-q^\a\d\ell^{(n)}/O(\e)}\right)^{O((\ell^{(n+1)})^2/(\ell^{(n)}\d\e))}\\
\ge{}&e^{-\log^{O(1)}(1/q)}\times
\begin{cases}
\left(\d q^\a W^n\right)^{W^n/(\d^2\e)}&n\le \Nc\\
\exp\left(-1/\left(q^{\a}\exp\left(W^{\exp(n-\Nc)}\right)\right)\right)&n> \Nc.
\end{cases}
\end{aligned}\label{eq:mu:T:bound}\end{multline}
Essentially the same computation leads to the same bound for $\m(\cT^\bone(D_1\setminus D_0))$ (see \cref{eq:mu:T:bound:loglog}). The only difference is that only $O(1)$ $W$-helping sets and $O(\ell^{(n+1)}/\e)$ helping sets are needed. Further recalling \cref{eq:mu:bound:base,eq:mu:bound:rec}, it is not hard to check \cref{eq:mu:SG:bound}.
\end{proof}

\subsection{FA-1f global dynamics}
\label{subsec:global:FA}
We next import the global FA-1f dynamics together with much of the mesoscopic multi-directional East one simultaneously from \cite{Hartarsky21a}.

\begin{prop}
\label{prop:global:FA}
Let $\cU$ have a finite number of stable directions, $T=\exp(\log^4(1/q)/q^\a)$ and $\uri$ be such that the associated side lengths satisfy $C\le \si_j\le O(\li)$ for all $j\in[4k]$. Assume that for all $l\in[0,\lm]$ multiple of $\l_0$ the event $\SG^\bone(\L(\uri+l\uv_0))$ is nonempty, decreasing, translation invariant and satisfies
\[\left(1-\m\left(\SG^\bone\left(\L\left(\uri+l\uv_0\right)\right)\right)\right)^T T^W=o(1).\] Then,
\[\Et\le \frac{\max_{l\in[0,\lm]}\g(\L(\uri+l\uv_0))}{(q^{1/\d}\min_{l\in[0,\lm]}\m(\SG^\bone(\L(\uri+l\uv_0))))^{\log(1/q)/\d}}.\]
\end{prop}
The proof is as in \cite{Hartarsky21a}, up to the following minor modifications. Firstly, one needs to replace the base of the snail by $\Lm=\L(\uri+\l_0\lceil\lm/\l_0\rceil \uv_0)$, which has a similar shape by hypothesis. Secondly, the event that the base is super good on \cite{Hartarsky21a} should be replaced by $\SG^\bone(\Lm)$. Finally, \cite{Hartarsky21a}*{Proposition 4.9} is substituted by the definition \cref{eq:def:gamma} of $\g(\Lm)$. As \cref{prop:global:FA} is essentially the entire content of \cite{Hartarsky21a} (see particularly Proposition 4.12 and Remark 4.8 there), we refer the reader to that work for the details.

\begin{proof}[Proof of \cref{th:main}\ref{log1}]
Let $\cU$ be balanced rooted with finite number of stable directions. Recall $\L^{(\Ni)}=\L(\ur^{(\Ni)})$ with $\ur^{(\Ni)}=:\uri$ from \cref{subsec:loglog:internal} if $k=1$ and from \cref{subsec:log1:internal} if $k\ge 2$. Fix $l\in[0,\lm]$ multiple of $\l_0$ and East-extend $\L^{(\Ni)}$ by $l$ in direction $u_0$. It is not hard to check from \cref{def:extension:East,obs:mu:helping} that \[\frac{\m(\SG^\bone(\L(\uri+l\uv_0)))}{\m(\SG^\bone(\L(\uri)))}=\m\left(\cT^\bone\left(T\left(\uri,l,0\right)\right)\right)=q^{O(W)}\]
(see \cref{eq:mu:E:bound}).
Then, by \cref{cor:East:reduction}, \cref{th:internal:East,th:internal:loglog} and the Harris inequality \cref{eq:Harris:1}, we obtain
\begin{align*}
\m\left(\SG^\bone\left(\L\left(\uri+l\uv_0\right)\right)\right)&{}\ge\exp\left(\frac{-2}{\e^2q^\a}\right)\\
\g\left(\L\left(\uri+l\uv_0\right)\right)&{}\le \begin{cases}
\exp\left(\frac{\log(1/q)}{\e^3q^\a}\right)&k=1,\\
\exp\left(\frac{2\log(1/q)\log\log\log(1/q)}{\e^3q^\a}\right)&k\ge 2.\end{cases}
\end{align*}
Plugging this in \cref{prop:global:FA}, we obtain
\begin{equation}
\label{eq:poluted}\Et\le\begin{cases}
\exp\left(\frac{2\log(1/q)}{\e^3q^\a}\right)&k=1,\\
\exp\left(\frac{3\log(1/q)\log\log\log(1/q)}{\e^3q^\a}\right)&k\ge 2,\end{cases}\end{equation}
which concludes the proof of \cref{th:main}\ref{log1} in the case $k=1$ and of \cref{eq:rem:logloglog} for $k\ge 2$. The full result of \cref{th:main}\ref{log1} for $k\ge 2$ is proved identically, replacing \cref{th:internal:East} by the stronger \cref{th:internal:East:improved}.
\end{proof}

\section{Balanced models with infinite number of stable directions}
\label{sec:log0}
We finally turn to balanced models with infinite number of stable directions (class \ref{log0}). The internal dynamics was already handled in \cref{subsec:log1:internal}. The mesoscopic one (\cref{subsec:log0:meso}) is essentially the same as the the internal one, using two-dimensional East-extensions. The global dynamics (\cref{subsec:global:East}) also uses an East mechanism analogous to the FA-1f one from \cite{Hartarsky21a} used in \cref{subsec:global:FA}.

\subsection{East mesoscopic dynamics}
\label{subsec:log0:meso}
Given that the bound we are aiming for in \cref{th:main}\ref{log0} is much larger than those in previous sections, there is a lot of margin and our reasoning is far from tight for the sake of simplicity.

Recall $\Ni$ and $\ell^{(n)}$ for $n\le\Ni$ from \cref{eq:def:ln:internal:East}, the droplets $\L^{(n)}$ from \cref{subsec:log1:internal}, their SG events from \cref{def:SG:balanced:internal}. For $n>\Ni$, we set $\ell^{(n)}=W^{n-\Ni}\ell^{(\Ni)}$ and define $\us^{(n)},\ur^{(n)},\L^{(n)}$ as in \cref{subsec:log1:internal}. Recall \cref{subsec:scales}. Further let $\Nm=\inf\{n:\ell^{(n)}/\e\ge\lm=q^{-C}\}=\Theta(C\log(1/q)/\log W)$ and assume for simplicity that $\ell^{(\Nm)}=q^{-C}\e$. We are only be interested in $n\le \Nm$ and extend \cref{def:n-traversable,def:SG:balanced:internal,def:extension:East:multid} to such $n$ without change. With these conventions, our goal is the following.
\begin{thm}
\label{prop:meso:east}
Let $\cU$ be a balanced model with infinitely many stable directions (class \ref{log0}). Then
\begin{align*}
\g\left(\L^{(\Nm)}\right)&{}\le\exp\left(\frac{\log^{2}(1/q)}{\e^3q^\a}\right),&\m\left(\SG^\bone\left(\L^{(\Nm)}\right)\right)&{}\ge\exp\left(\frac{-2}{\e^2q^\a}\right).
\end{align*}
\end{thm}
\begin{proof}[Proof of \cref{prop:meso:east}]
The proof is essentially identical to the one of \cref{th:internal:East}, so we only indicate the necessary changes. To start with, \cref{lem:internal:East:step} applies without change for $n\in[\Ni,\Nm)$. Also, the Harris inequality \cref{eq:Harris:2} still implies that $a_m^{(n)}\le \m^{-1}(\SG^\bone(\L^{(n)}))\le \m^{-1}(\SG^\bone(\L^{(\Nm)}))$. Therefore, 
\[\g\left(\L^{(N^m)}\right)\le\frac{\g(\L^{(\Ni)})e^{\log^{O(1)}(1/q)}}{(\m(\SG^\bone(\L^{(\Nm)}))\min_{n\in[\Nm]}\m(\cT_n(\L^{n+1})))^{O(\Nm M^{(\Nm-1)})}}.\]
Recalling the bound on $\g(\L^{(\Ni)})$ established in \cref{th:internal:East}, together with the fact that $\Nm\le C\log(1/q)$ and $M^{(\Nm-1)}\le O(C\log(1/q))$, it suffices to prove that
\begin{equation}
\label{eq:SG:bound:meso:East}
\m\left(\SG^\bone\left(\L^{(\Nm)}\right)\right)\min_{n\in[\Nm]}\m\left(\cT_n\left(\L^{n+1}\right)\right)\ge \exp\left(-2/\left(\e^2q^\a\right)\right),
\end{equation}
in order to conclude the proof of \cref{prop:meso:east}.

Once again, the proof of \cref{eq:SG:bound:meso:East} proceeds similarly to the one of \cref{eq:mu:SG:bound} in \cref{lem:internal:East:probability}. Indeed, the same computation as \cref{eq:mu:T:bound} in the present setting gives that for $n\in[\Ni, \Nm)$ we have
\begin{equation}
\label{eq:mu:T:bound:meso:East}\m\left(\cT_n\left(\L^{(n+1)}\right)\right)\ge q^{O(W^3/\d)}\exp\left(-e^{-q^\a\d\ell^{(n)}/O(\e)}O\left(W^2\ell^{(n)}/(\d\e)\right)\right)
\end{equation}
and similarly for $\m(\cT^\bone(D_1\setminus D_0))$ (as in the proof of \cref{lem:internal:East:probability}, also see \cref{eq:mu:T:bound:loglog}). From \cref{eq:mu:bound:rec} it follows that
\begin{multline*}\m\left(\SG^\bone\left(\L^{(\Nm)}\right)\right)\ge\m\left(\SG^\bone\left(\L^{(\Ni)}\right)\right)\\
\times\prod_{n=\Ni}^{\Nm-1}\m\left(\cT^\bone(D_1\setminus D_0)\right)\m\left(\cT_n\left(\L^{(n+1)}\setminus\L^{(n)}\right)\right).\end{multline*}
Plugging \cref{eq:mu:SG:bound,eq:mu:T:bound:meso:East} in the r.h.s., this yields \cref{eq:SG:bound:meso:East} as desired.
\end{proof}

\subsection{East global dynamics}
\label{subsec:global:East}
For the global dynamics we use a simpler version of the procedure of \cite{Hartarsky21a}*{Section 5} with East dynamics instead of FA-1f.

\begin{figure}
    \centering\begin{tikzpicture}[line cap=round,line join=round,>=triangle 45,x=0.20cm,y=0.20cm]
\fill[fill opacity=0.5] (-7.59,9) -- (-11.73,9) -- (-14.66,6.07) -- (-14.66,1.93) -- (-11.73,-1) -- (0.41,-1) -- (1,-0.41) -- (1,0.41) -- cycle;
\draw (1,-2) -- (1,21);
\draw [domain=-48:3] plot(\x,{(--1-0*\x)/-1});
\draw (-7.59,9)-- (-11.73,9);
\draw (-11.73,9)-- (-14.66,6.07);
\draw (-14.66,6.07)-- (-14.66,1.93);
\draw (-14.66,1.93)-- (-11.73,-1);
\draw (-11.73,-1)-- (0.41,-1);
\draw (0.41,-1)-- (1,-0.41);
\draw (1,-0.41)-- (1,0.41);
\draw (1,0.41)-- (-7.59,9);
\draw [domain=-48:3] plot(\x,{(--37.28-0*\x)/4.14});
\draw (-14.66,-2) -- (-14.66,21);
\draw (-30.31,-2) -- (-30.31,21);
\draw (-45.97,-2) -- (-45.97,21);
\draw [domain=-48:3] plot(\x,{(--78.7-0*\x)/4.14});
\draw [very thick] (-43.04,9)-- (-45.97,6.07);
\draw [very thick] (-45.97,6.07)-- (-45.97,1.93);
\draw [very thick] (-45.97,1.93)-- (-43.04,-1);
\draw [very thick] (-43.04,-1)-- (0.41,-1);
\draw [very thick] (0.41,-1)-- (1,-0.41);
\draw [very thick] (1,-0.41)-- (1,0.41);
\draw [very thick] (1,0.41)-- (-7.59,9);
\draw [very thick] (-7.59,9)-- (-43.04,9);
\draw (-38.14,4) node {$Q_{i+1}$};
\draw (-22.49,4) node {$Q_{i}$};
\draw (-2.13,7) node {$Q_{i-1}$};
\draw (-9.66,4) node {$\Lm$};
\draw [decorate,decoration={brace,amplitude=5pt}] (-30.31,9) -- (-14.66,9) node [midway,yshift= 0.4cm] {$l=\Theta(\lm)$};
\draw (-44.5,7.5)--(-44.5,5.5) node [below] {$\L$};
\end{tikzpicture}
    \caption{Illustration of the East global dynamics (\cref{subsec:global:East}). The shaded droplet $\Lm$ inscribed in the box $Q$ is extended by $2l$ to the thickened one $\L$.}
    \label{fig:global:east}
\end{figure}
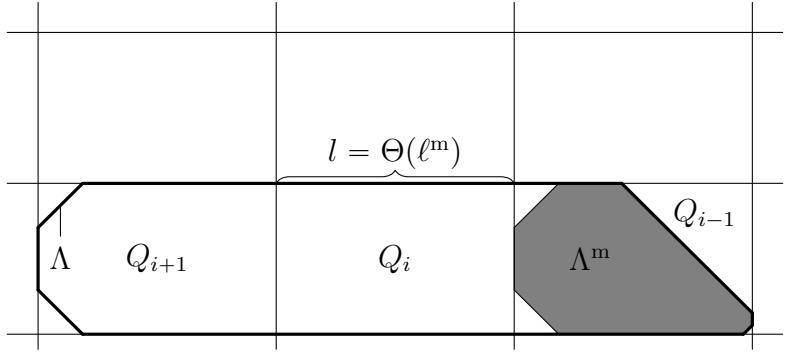

\begin{proof}[Proof of \cref{th:main}\ref{log0}]
Let $\cU$ be balanced with infinite number of stable directions and recall \cref{subsec:log0:meso}. Set $T=\exp(1/q^{3\a})$, $\usm=\us^{(\Nm)}$, $\urm=\ur^{(\Nm)}$ and $\Lm=\L^{(\Nm)}$. In particular, $\sm_j=\Theta(\lm)$ for $j\in[-k, k+1]$ and $\sm_j=\Theta(1/\e)$ for $j\in[k+2,3k-1]$. We East-extend $\Lm$ by $2l=2(\l_0+\rme_0+\rme_{2k})$ in direction $u_0$ to obtain $\L=\L(\urm+2l\uv_0)$. \Cref{cor:East:reduction,prop:meso:east,def:extension:East}, the Harris inequality \cref{eq:Harris:2} and the simple fact that $\m(\cT^\bone(T(\urm,2l,0)))=q^{O(W)}$ (by \cref{obs:mu:helping,lem:traversability:boundary} as usual) give
\begin{align}
\label{eq:log0:global}
\g(\L)\le{}& \exp\left(\frac{\log^2(1/q)}{\e^{O(1)}q^\a}\right),&\m\left(\SG^\bone(\L)\right)\ge{}&\exp\left(\frac{-3}{\e^2q^\a}\right).
\end{align}

A similar argument to the rest of the proof was already discussed thoroughly in \cite{Hartarsky21a}*{Section 5} and then in \cite{Hartarsky23FA}*{Section 5}, so we only provide a sketch. The adapted approach of \cite{Hartarsky21a}*{Section 5} proceeds as follows.
\begin{enumerate}
    \item Denoting $t_*=\exp(-1/(\e^W q^{2\a}))$, by the main result of \cite{Martinelli19a} it suffices to show that $T\bbP_{\m}(\t_0>t_*)=o(1)$, in order to deduce $\Et\le t_*+o(1)$.
    \item By finite speed of propagation we may work with the $\cU$-KCM on a large discrete torus of size $T\gg t_*$.
    \item\label{item:events} We partition the torus into strips and the strips into translates of the box $Q=\bbH_{u_0}(\l_0+\rme_0)\cap\bbH_{u_k}(\r_k+\rme_k)\cap\Hb_{u_{-k}}(\rme_{-k})\cap \Hb_{u_{2k}}(\rme_{2k})$ as shown in \cref{fig:global:east}. We say $Q$ is \emph{good} ($\cG(Q)$ occurs) if for each segment $S\subset Q$ perpendicular to some $u\in\hS$ of length $\e\lm$ the event $\cH^W(S)$ occurs. Further define $\SG(Q)$ to occur if the only (integer) translate of $\Lm$ contained in $Q$ is SG. We say that \emph{the environment is good} ($\cE$ occurs) if all boxes are good and in each strip at least one box is super good. The sizes are chosen so that it is sufficiently likely for this event to always occur up to time $t_*$. Indeed, we have $(1-\m(\SG^\bone(\Lm)))^TT^W=o(1)$ by \cref{prop:meso:east} and $(1-\m_Q(\cG))T^W=o(1)$ by \cref{obs:mu:helping}.
    \item By a standard variational technique it then suffices to prove a Poincar\'e inequality, bounding the variance of a function conditionally on $\cE$ by the Dirichlet form on the torus. Moreover, since $\m$ and $\cE$ are product w.r.t.\ the partition of \cref{fig:global:east}, it suffices to prove this inequality on a single strip.
    \item\label{item:conclusion} Finally, we prove such a bound, using an auxiliary East dynamics for the boxes and the definition of $\g$ to reproduce the resampling of the state of a box by moves of the original $\cU$-KCM.
\end{enumerate}
Let us explain the last step above in more detail, as it is the only one that genuinely differs from \cite{Hartarsky21a}.

Let $Q_i=Q+ilu_0$ and $\bbT=\bigcup_{i\in[T]}Q_i$ be our strip of interest (indices are considered modulo $T$, since the strip is on the torus). As explained above, our goal is to prove that for all $f:\O_\bbT\to\bbR$ it holds that
\begin{equation}
\label{eq:log0:goal}
\var_{\bbT}(f|\cE)\le \exp\left(1/\left(\e^{O(1)}q^{2\a}\right)\right)\sum_{x\in\bbT}\m_{\bbT}\left(c_x^{\bbT,\bone}\var_x(f)\right),
\end{equation}
where $c_x^{\bbT,\bone}$ takes into account the periodic geometry of $\bbT$. 

By \cite{Martinelli19a}*{Proposition 3.4} on the generalised East chain we have
\begin{equation}
\label{eq:log0:final}
\var_\bbT(f|\cE)\le \exp\left(1/\left(\e^5q^{2\a}\right)\right)\sum_{i\in[T]}\m_{\bbT}\left(\left.\1_{\SG(Q_{i-1})}\var_{Q_i}\left(f|\cG\left(Q_i\right)\right)\right|\cE\right),
\end{equation}
since \cref{prop:meso:east} and the Harris inequality \cref{eq:Harris:2} give $\m(\SG(Q)|\cG(Q))\ge \exp(-2/(\e^2q^\a))$.\footnote{Strictly speaking \cite{Martinelli19a} does not deal with the torus conditioned on having an infection, but this issue is easily dealt with by the method of \cite{Blondel13}.}

Next observe that $\L_i\supset Q_i$, where $\L_i=\L+(i-1)lu_0$ (see \cref{fig:global:east}). Let $\cG(\L_i\setminus Q_i)\subset\cG(Q_{i+1})\cap\cG(Q_{i-1})$ be the event that $\cH^W(S)$ holds for all segments $S\subset \L_i\setminus Q_i$ of length $2\e\lm$ perpendicular to some $u\in\hS$. Hence, by convexity of the variance and the fact that $\m(\cE)=1-o(1)$ we have 
\begin{multline*}
\m_\bbT\left(\left.\1_{\SG(Q_{i-1})}\var_{Q_i}(f|\cG(Q_i))\right|\cE\right)\\
\begin{aligned}
\le{}&(1+o(1))\m_\bbT\left(\var_{\L_{i}}(f|\SG(Q_{i-1})\cap\cG(Q_i)\cap\cG(\L_i\setminus Q_i))\right),\\
\le{}&(1+o(1))\m_\bbT\left(\var_{\L_i}\left(f|\SG^\bone(\L_i)\right)\right).
\end{aligned}
\end{multline*}
Here we used \cref{eq:variance:conditioning:enlarge} and $\SG(Q_{i-1})\cap\cG(Q_i)\cap\cG(\L_i\setminus Q_i)\subset\SG^\bone(\L_i)$ (recall \cref{def:extension:East}) for the second inequality. Finally, recalling \cref{eq:def:gamma,eq:log0:global,eq:log0:final}, we obtain \cref{eq:log0:goal} as desired.
\end{proof}

As already noted, all lower bounds in \cref{th:main} are known from \cite{Hartarsky22univlower} and the upper ones for classes \ref{log4} and \ref{log3} were proved in \cite{Martinelli19a} and \cite{Hartarsky21a} respectively. Thus, the proof of \cref{th:main} is complete.

\section*{Acknowledgements}
This work is supported by ERC Starting Grant 680275 ``MALIG'' and Austrian Science Fund (FWF) P35428-N. We thank Fabio Martinelli and Cristina Toninelli for their involvement in early stages of this project (among other things), Laure Mar\^ech\'e for interesting discussions about the universality of critical KCM and anonymous referees for careful reading and helpful comments on the presentation.

\appendix

\section{Extensions}
\label{app:extensions}
This appendix aims to prove our main building blocks---\cref{cor:East:reduction,cor:CBSEP:reduction} for the East- and CBSEP-extensions.
\subsection{Auxiliary two-block chain}
\label{app:three:block}
We begin with a non-product variant of the standard two-block technique for the purposes of the proof of the East-extension \cref{cor:East:reduction}. Let $(\O_1,\pi_1)$ and $(\O_2,\pi_2)$ be finite positive probability spaces, $(\O,\pi)$ denote the associated product space and $\n=\p(\cdot|\cH)$ for some event $\cH\subset\O$. For $\omega\in\O$ we write $\o_i\in \O_i$ for its $i$\textsuperscript{th} coordinate. Consider an event $\cF\subset\O_1$ and set
\[\cD(f) =\n\left(\var_\n(f|
   \o_2) + \1_{\cF}\var_\n(f|\o_1)\right)\]
for any $f:\cH\to \bbR$. Observe that $\cD$ is the Dirichlet form of the continuous time Markov chain on $\cH$ in which $\o_1$ is resampled at rate one from
$\n(\cdot | \o_2)$ and, if $\o_1\in \cF$, then $\o_2$ is resampled with rate one from
$\n(\cdot |\o_1)$. This chain is
reversible w.r.t.\ $\n$.

\begin{lem}
\label{lem:aux:east}
Assume that $\cF\times\O_2\subset\cH$. Then, for all $f:\cH\to \bbR$ we have
\[\var_\n(f)\le O(1)\max_{\o_2\in \O_2}\n^{-1}(\cF|\o_2)\cD(f).\]
\end{lem}
\begin{proof}
We follow \cite{Hartarsky23FA}*{Proposition 3.5}. Consider the Markov chain $(\o(t))_{t\ge 0}$ described above. Given two arbitrary initial conditions $\o(0)$ an $\o'(0)$ we construct a coupling of two of such chains with these initial conditions such that with probability $\O(1)$ we have $\o(t)=\o'(t)$ for $t>T=\max_{\o_2\in \O_2}\n^{-1}(\cF|\o_2)$. Standard arguments \cite{Levin09} then prove that the mixing time of the chain is $O(T)$ and the lemma follows.

To construct our coupling, we use the following representation of the Markov chain. We are given two independent Poisson clocks with rate one and the chain transitions occur only at the clock rings. When the first clock rings, a Bernoulli variable
$\xi$ with probability of success $\n(\cF|\o_2)$
is sampled. If
$\xi=1$, then $\o_1$ is resampled w.r.t.\ the measure $\pi(\cdot|
\cF)=\n(\cdot|\cF,\o_2)$, while if $\xi=0$, then $\o_1$ is
resampled w.r.t.\ the measure $\n(\cdot|
\cF^c,\o_2)$. Clearly, in doing so 
$\o_1$ is resampled w.r.t.\ $\n(\cdot|
\o_2)$. If the second clock rings,
we resample $\o_2$ from $\p_2$ if $\o_1\in\cF$ and ignore the ring otherwise.

Both chains use the same clocks. When the first clock rings and the current couple of configurations is
$(\o,\o')$, we first maximally couple the two Bernoulli variables $\xi,\xi'$ corresponding to $\o,\o'$ respectively. Then:
\begin{itemize}
\item if $\xi=\xi'=1$, we update both $\o_1$ and $\o'_1$ to the \emph{same} $\eta_1\in \cF$ with probability $\pi(\eta_1| \cF)$;
\item otherwise, we resample $\o_1$ and $\o'_1$ independently from their respective laws, given $\xi,\xi'$.
\end{itemize}
When the second clock rings, the two chains attempt to update to two maximally coupled configurations with the corresponding distributions. 
 
Suppose now that two consecutive rings occur at times $t_1<t_2$ at the first and second clocks respectively and the Bernoulli variables at time $t_1$ are both $1$. Then the two configurations are clearly identical at $t_2$. To conclude the proof, observe that for any time interval $\D$ of length one the probability that  there exist $t_1<t_2$ in $\D$ as above is at least $1/(4T)$.
\end{proof}

\subsection{Microscopic dynamics}
\label{sec:micro}
We next turn to the microscopic dynamics (recall \cref{subsec:micro}).

Recall \cref{def:uihelping}. Let $\L=\L(\ur)$ be a droplet with side lengths at least $C^3$. Given $\o\in\O_{\bbZ^2\setminus\L}$ and $i\in[4k]$, we define $\L\subset \L_i^\o\subset\L(\ur+O(1)\uv_i)$ by $\L_i^\o=\L$ if $\a(u_i)=0$ or $\a(u_i)>\a$. If $\a(u_i)\in(0,\a]$, we rather set
\[\L_i^\o=\L\cup\bigcup_{x}\left(\left(\left[Z_i\cup\bbH_{u_i}\right]_\cU\setminus\Hb_{u_i}\right)+x\right)\setminus \left\{y\in\bbZ^2\setminus\L:\o_y=0\right\},\]
the union being over $x\in\L$ such that $\o_{(x+Z_i)\setminus\L}=\bzero$ and $x$ is at distance at least $C$ from all sides of $\L$ except the $u_i$-side. In words, we essentially look at pieces of $u_i$-helping sets for the last few lines of the droplet sticking out of it and add to $\L$ the sites which each piece can infect. The reason for introducing this is that helping sets may need to infect a few sites outside $\L$ before creating their periodic infections on the corresponding line and it is those sites that we wish to include in $\L_i^\o$. We set $\L_I^\o=\bigcup_{i\in I}\L_i^\o$ for $I\subset[4k]$.
Let $i\in[4k]$ be such that $\a(u_j)<\infty$ for all $j\in I=\{i-k+1,\dots,i+k-1\}$. Fix $\L=\L(\ur)$ with side lengths at least $C^3$ and at most $q^{-O(C)}$. Let $l\in[0,O(1)]$ be a multiple of $\l_i$, $\o\in\O_{\bbZ^2\setminus\L(\ur+l\uv_i)}$, $\L^+=(\L(\ur+l\uv_i))_I^\o$ and $T=T(\ur,l,i)$. Our goal is to provide a relaxation mechanism for an East-extension of bounded length.
\begin{lem}
\label{lem:paths}
In the above setting we have
\begin{equation}
\label{eq:paths}\m_{\L^+\setminus\L}\left(\var_{T}\left(f|\cT^\o(T)\right)\right)\le e^{O(\log^2(1/q))}\sum_{x\in \L^+\setminus\L}\m_{\L^+\setminus\L}\left(c_x^{\L^+\setminus\L,\bzero_\L\cdot\o_{\bbZ^2\setminus\L^+}}\var_x(f)\right)\end{equation}
and the same holds for $\ST$ instead of $\cT$.
\end{lem}
Though it is possible to prove this directly via canonical paths, we rather deduce it from the main result of \cite{Hartarsky21b} proved much more elegantly. That work was developed for the purpose of its present application.
\begin{proof}
We only treat $\cT$, the proof for $\ST$ being identical. Let us denote by $\cE^\o$ the event that the $\cU$-KCM restricted to $\L^+\setminus\L$ with boundary condition $\bzero_\L\cdot\o_{\bbZ^2\setminus\L^+}$ is able to fully infect $\L^+\setminus\L$. As in the proof of \cref{lem:closure}, we see that $\cT^\o(T)\subset\cE^\o$. Moreover, recalling \cref{lem:traversability:boundary}, we have $\m(\cT^\o(T))\ge \m(\cW(T))\ge q^{O(W)}$, since $T$ has bounded length. Hence, by \cref{eq:var:decomposition:AB,eq:variance:conditioning:enlarge},
\[\m_{\L^+\setminus\L}\left(\var_T\left(f|\cT^\o(T)\right)\right)\le \var_{\L^+\setminus\L}\left(f|\cT^\o(T)\right)\le q^{-O(W)}\var_{\L^+\setminus\L}\left(f|\cE^\o\right).\]

We next observe that the process defining $\cE^\o$ is in fact a one-dimensional inhomogeneous KCM of the type considered in \cite{Hartarsky21b} and called general KCM there (enumerate the sites of $\L^+\setminus\L$ so that neighbouring sites remain at bounded distance, e.g.\ in lexicographical order for $(\<\cdot,u_{i+k}\>,\<\cdot,u_{i}\>)$). Therefore, \cite{Hartarsky21b}*{Theorem 1} yields \cref{eq:paths} as desired, taking into account that $\m(\cE^\o)\ge \m(\cT^\o(T))\ge q^{O(W)}$.
\end{proof}

\begin{cor}
\label{cor:single:line}
In the same setting as above, we have
\begin{multline}
\m_{\L^+}\left(\left.\var_{T}(f|\cT^\o(T))\right|\SG^\bone(\L)\right)\\\le e^{O(\log^2(1/q))}\max\left(\g(\L),\m^{-1}\left(\SG^\bone(\L)\right)\right)\sum_{x\in\L^+}\m_{\L^+}\left(c_x^{\L^+,\o}\var_x(f)\right)\label{eq:cor:single:line}
\end{multline}
and the same holds with $\ST$ instead of $\cT$.
\end{cor}
\begin{proof}
By \cref{lem:paths}, it suffices to bound
\[\m_{\L^+}\left(\left.c_x^{\L^+\setminus\L,\bzero_\L\cdot\o_{\bbZ^2\setminus\L^+}}\var_x(f)\right|\SG^\bone(\L)\right)\]
from above by the r.h.s.\ of \cref{eq:cor:single:line} for any $x\in\L^+\setminus\L$. By \cref{eq:var:decomposition:AB} this is at most
\[\m_{\L^+}\left(c_x^{\L^+\setminus\L,\bzero_\L\cdot\o_{\bbZ^2\setminus\L^+}}\var_{\L\cup\{x\}}\left(f|\SG^\bone(\L)\right)\right).\]

By the two-block \cref{lem:aux:east} we have
\begin{multline*}\var_{\L\cup\{x\}}\left(f|\SG^\bone(\L)\right)\\\le q^{-O(1)}\m_{\L\cup\{x\}}\left(\left.\var_\L\left(f|\SG^\bone(\L)\right)+\1_{\cI}\var_x(f)\right|\SG^\bone(\L)\right),\end{multline*}
where $\cI$ is the event that all sites in $\L$ at distance at most some large constant from $x$ are infected. Putting this together and observing that
$\1_\cI \cdot c_x^{\L^+\setminus\L,\bzero_\L\cdot\o_{\bbZ^2\setminus\L^+}}\le c_x^{\L^+,\o}$, we get
\begin{multline*}\m_{\L^+}(\var_T(f|\cT^\o(T))|\SG^\bone(\L))\le e^{O(\log^2(1/q))}\times\\
\left(|\L^+\setminus\L|\m_{\L^+}\left(\var_\L\left(f|\SG^\bone(\L)\right)\right)+\sum_{x\in\L^+\setminus\L}\m_{\L^+}\left(\left.c_x^{\L^+,\o}\var_x(f)\right|\SG^\bone(\L)\right)\right).\end{multline*}
Finally, recalling \cref{eq:def:gamma} and $|\L^+|\le q^{-O(C)}$, we recover \cref{eq:cor:single:line}.
\end{proof}

\subsection{Proofs of the one-directional extensions}
\label{app:subsec:extensions}
We require a more technical version of \cref{eq:def:gamma} accounting for a boundary condition. For a droplet $\L=\L(\ur)$, boundary condition $\o\in\O_{\bbZ^2\setminus\L}$, nonempty event $\SG^\o(\L)\subset\O_\L$ and set of directions $I\subset [4k]$, let $\g_I^\o(\L)$ be the smallest constant $\g\in[1,\infty]$ such that
\begin{equation}
\label{eq:def:gamma:boundary}\m_{\L^\o_I}\left(\var_{\L}\left(f|\SG^\o\left(\L\right)\right)\right)\le \g\sum_{x\in\L_I^\o}\m_{\L^\o_I}\left(c_x^{\L^\o_I,\o}\var_x(f)\right).
\end{equation}
holds for all $f:\O\to \bbR$. 

For the rest of the section we recall the following notation from \cref{cor:East:reduction}. Let $i\in[4k]$ be such that $\a(u_j)<\infty$ for all $j\in(i-k,i+k)$. Let $\L=\L(\ur)$ be a droplet with $\ur=q^{-O(C)}$ and side lengths at least $C^3$. Let $l\in(0,\lmp]$ be a multiple of $\l_i$. Let $d_m=\l_i\lfloor(3/2)^{m}\rfloor$ for $m\in[1,M)$ and $M=\min\{m:\l_i(3/2)^m\ge l\}$. Let $d_M=l$, $\L^m=\L(\ur+d_m\uv_i)$ for $m\in[1,M]$ and $s_{m-1}=d_{m}-d_{m-1}$ for $m\in[2,M]$.

\begin{lem}
\label{prop:east:reduction}
Set $I=\{i-k+1,\dots,i+k-1\}$. Let $\SG^\bone(\L(\ur))$ be a nonempty translation invariant decreasing event. Assume that we East-extend $\L(\ur)$ by $l$ in direction $u_i$. Then
\[\g\left(\L^M\right)\le \max_{\o}\g_I^{\o}\left(\L^1\right)\prod_{m=1}^{M-1} \frac{a_m}{q^{O(W)}}\]
where $a_m$ is defined in \cref{eq:def:am:East}.
\end{lem}
\begin{proof}
We loosely follow \cite{Hartarsky23FA}*{Eq.\ (4.10)}. Note that by \cref{eq:def:gamma,eq:def:gamma:boundary} $\g_I^\bone(\L^M)=\g(\L^M)$. Proceeding by induction it then suffices to prove that for any $m\in[1,M)$ and $\o\in\O_{\bbZ^2\setminus\L^{m+1}}$
\begin{equation}
\label{eq:East:extension:step}
\g_I^{\o}\left(\L^{m+1}\right)\le \max_{\o'\in\O_{\bbZ^2\setminus\L^m}}\g_I^{\o'}\left(\L^m\right)\frac{a_{m}}{q^{O(W)}}.
\end{equation}

Fix such $m$ and $\o$ and partition $\L^{m+1}=V_1\sqcup V_2\sqcup V_3$ so that
\begin{align*}
V_1\cup V_2&{}=\L^m,&V_2\cup V_3&{}=\L^m+s_{m}u_i.
\end{align*}
That is, set $V_1=s_mu_i+T(\ur,s_m,i+2k)$, $V_2=s_m+\L(\ur+(d_m-s_m)\uv_i)$ and $V_3=d_m u_i+T(\ur,s_m,i)$.

In order to apply \cref{lem:aux:east}, we define $\O_1=\O_{\L^m}$, $\O_2=\cT^\o(V_3)$ and equip them with $\p_1=\m_{\L^m}$ and $\p_2=\m_{V_3}(\cdot|\cT^\o(V_3))$ respectively. We set $\cH=\SG^\o(\L^{m+1})$ and $\cF=\SG^\bone(\L^m)\cap \SG^\bone(V_2)$. Note that these these $\SG$ events were defined when East-extending $\L(\ur)$ by $l$ in direction $u_i$, since $0\le d_m-s_m\le d_m\le d_{m+1}\le d_M=l$ (for $V_2$ we also use translation invariance). Notice that $\cF\times\O_2\subset\cH$, since, by \cref{def:extension:East}, $\SG^\bone(\L^m)=\SG^\bone(\L(\ur))\cap\cT^\bone(T(\ur,d_m,i))$ and 
\[\cT^\bone(T(\ur,d_m,i))\cap\cT^\o(V_3)\subset \cT^\o(T(\ur,d_{m+1},i))\]
by \cref{lem:tube:decomposition}. We may therefore apply \cref{lem:aux:east} to get
\begin{multline}
\label{eq:East:ext:proof:1}\var_{\L^{m+1}}\left(f|\SG^\o\left(\L^{m+1}\right)\right)\le O(1)\max_{\eta_{V_3}\in\cT^\o(V_3)}\m^{-1}\left(\cF|\SG^\o(\L^{m+1}),\eta_{V_3}\right)\\
\times\m_{\L^{m+1}}\left(\left.\var_{\L^m}(f|\cH,\eta_{V_3})+\1_{\cF}\var_{V_3}\left(f|\cH,\eta_{\L^m}\right)\right|\SG^\o\left(\L^{m+1}\right)\right).
\end{multline}

Note that by \cref{def:extension:East} for any $\eta_{V_3}\in\cT^\o(V_3)$ we have 
\[\eta_{\L^m}\cdot\eta_{V_3}\in \SG^\o\left(\L^{m+1}\right)\Leftrightarrow \eta_{\L^m}\in\SG^{\eta_{V_3}\cdot\o}(\L^m),\]
which implies that
\begin{align*}\var_{\L^m}\left(f|\cH,\eta_{V_3}\right)&{}=\var_{\L^m}\left(f|\SG^{\h_{V_3}\cdot\o}(\L^m)\right)\\
\m\left(\cF|\SG^\o\left(\L^{m+1}\right),\h_{V_3}\right)&{}=\m\left(\cF|\SG^{\h_{V_3}\cdot\o}(\L^m)\right).\end{align*}
Further note that by \cref{def:extension:East,def:traversability},
\begin{multline*}
\cF=\SG^\bone\left(s_mu_i+\L(\ur)\right)\cap \cT^\bone\left(s_mu_i+T\left(\ur,d_m-s_m,i\right)\right)\\\cap\SG^{\bone}(\L(\ur))\cap\cT^{\h_{s_mu_i+T(\ur,d_m-s_m,i)}\cdot\bone}\left(T\left(\ur,s_m,i\right)\right),\end{multline*}
the second SG event being implied by $\SG^{\eta_{V_3}\cdot\o}(\L^m)$ again by \cref{def:extension:East}. Applying \cref{lem:traversability:boundary,eq:Harris:3}, we get that for any $\o'\in\O_{\bbZ^2\setminus \L^m}$ 
\begin{align*}&{}\m(\cF|\SG^{\o'}(\L^m))\\
&{}\ge \begin{multlined}[t]\m\Big(\SG^\bone(V_2)\cap\cT^{\h_{s_mu_i+T(\ur,d_m-s_m,i)}\cdot\bone}\left(T(\ur,s_m,i)\right)\\
\cap\cW\left(T\left(\ur,s_m,i\right)\right)\big|\SG^{\o'}(\L^m)\Big)\end{multlined}\\
&{}= \m\left(\left.\SG^\bone(V_2)\cap\cT^{\bzero}\left(T(\ur,s_m,i)\right)\cap\cW\left(T(\ur,s_m,i)\right)\right|\SG^{\o'}(\L^m)\right)\\
&{}= \m\left(\left.\SG^\bone(V_2)\cap\cW\left(T(\ur,s_m,i)\right)\right|\SG^{\o'}(\L^m)\right)\\
&{}\ge q^{O(W)}\m\left(\left.\SG^\bone(V_2)\right|\SG^{\o'}(\L^m)\right),\end{align*}
where in the second inequality we used that $\cT^\bzero(T(\ur,s_m,i))\supset\SG^{\o'}(\L^m)$, using \cref{def:extension:East} and $s_m\le d_m$. Moreover, since $\cF\times\O_2\subset\cH$ and $\cF\subset \SG^\bone(V_2)$, we have
\[\1_{\cF}\var_{V_3}\left(f|\cH,\eta_{\L^m}\right)\le \1_{\SG^\bone(V_2)}\var_{V_3}\left(f|\cT^\o(V_3)\right).\]

Plugging the above back into \cref{eq:East:ext:proof:1} yields
\begin{align}
\label{eq:East:ext:proof}
\var_{\L^{m+1}}\left(f|\SG^\o\left(\L^{m+1}\right)\right)\le{}& q^{-O(W)}\max_{\o'}\m^{-1}\left(\left.\SG^\bone(V_2)\right|\SG^{\o'}(\L^m)\right)\\
\nonumber&\times\m_{\L^{m+1}}\big(\var_{\L^m}\left(f|\SG^{\h_{V_3}\cdot\o}\left(\L^{m+1}\right)\right)\\&+\1_{\SG^\bone(V_2)}\var_{V_3}\left(f|\cT^\o(V_3)\right)\big|\SG^\o\left(\L^{m+1}\right)\big).\nonumber
\end{align}
From \cref{eq:def:gamma:boundary} we have
\begin{multline*}
\m_{(\L^{(m+1)})_I^\o}\left(\var_{\L^m}\left(f|\SG^{\h_{V_3}\cdot\o}\left(\L^{m+1}\right)\right)\right)\\
\le \max_{\o'} \g^{\o'}_I\left(\L^{m}\right)\sum_{x\in(\L^{m+1})^\o_I}\m_{(\L^{m+1})_I^\o}\left(c_x^{(\L^{m+1})_I^\o,\o}\var_x(f)\right).
\end{multline*}
On the other hand, recalling by \cref{def:extension:East} that $\SG^\o(\L^{m+1})\subset\cT^\o(V_3)$,
\begin{multline*}\m_{(\L^{(m+1)})_I^\o}\left(\left.\1_{\SG(V_2)}\var_{V_3}(f|\cT^\o(V_3))\right|\SG^\o\left(\L^{m+1}\right)\right)\\
\begin{aligned}
\le{}&\frac{\m_{(\L^{m+1})_I^\o}(\1_{\SG^\bone(V_2)}\1_{\cT^\o(V_3)}\var_{V_3}(f|\cT^\o(V_3)))}{\SG^\o(\L^{m+1})}\\
={}&\frac{\m(\SG^\bone(V_2)\cap\cT^\o(V_3))}{\SG^\o(\L^{m+1})}\m_{(\L^{m+1})_I^\o}\left(\left.\var_{V_3}(f|\cT^\o(V_3))\right|\SG^\bone(V_2)\cap\cT^\o(V_3)\right)\\
\le{}&\frac{\m(\SG^\o(s_mu_i+\L^m))\m_{(\L^{m+1})_I^\o}(\var_{\L^m+s_mu_i}(f|\SG^\bone(V_2)\cap\cT^\o(V_3)))}{\m(\SG^\bone(\L(\ur)))\m(\cT^\bone(T(\ur,s_m,i)))\m(\cT^\o(s_mu_i+T(\ur,d_m,i)))}\\
={}&\frac{\m_{(\L^{m+1})_I^\o}(\var_{\L^m+s_mu_i}(f|\SG^\bone(V_2)\cap\cT^\o(V_3)))}{\m(\cT^\bone(V_3))}\\
\le{}&\frac{\m(\cT^\o(s_mu_i+T(\ur,d_m,i)))\m_{(\L^{m+1})_I^\o}(\var_{\L^m+s_mu_i}(f|\SG^\o(s_mu_i+\L^m)))}{\m(\cT^\bone(s_mu_i+T(\ur,d_m-s_m,i)))\m(\cT^\o(d_mu_i+T(\ur,s_m,i)))\m(\cT^\bone(V_3))}\\
\le{}&\frac{\g^\o_I(\L^m)}{\m(\cT^\bone(V_3))q^{O(W)}}\sum_{x\in(\L^m+s_mu_i)_I^\o}\m_{(\L^{m+1})_I^\o}\left(c_x^{(\L^m+s_mu_i)^\o_I,\o}\var_x(f)\right),
\end{aligned}
\end{multline*}
where we used \cref{def:extension:East,lem:tube:decomposition,eq:var:decomposition:AB} in the second inequality; translation invariance and \cref{def:extension:East} in the second equality; \cref{eq:variance:conditioning:enlarge,def:extension:East,lem:tube:decomposition} in the third inequality; and \cref{lem:tube:decomposition,lem:traversability:boundary,eq:def:gamma:boundary} in the last one. Plugging these bounds into \cref{eq:East:ext:proof}, we obtain
\[\g^\o_I\left(\L^{m+1}\right)\le \frac{\max_{\o'}\g^{\o'}_I(\L^m)}{q^{O(W)}\m(\cT^\bone(V_3))\min_{\o'}\m(\SG^\bone(V_2)|\SG^{\o'}(\L^m))}.\]

It remains to transform the denominator in the last expression, fixing some $\o'$. Note that
\begin{multline*}\m\left(\cT^\bone(V_3)\cap\SG^\bone(V_2)\cap\SG^{\o'}\left(\L^m\right)\right)\\
\begin{aligned}\ge{}&\m\left(\cT^\bone(V_3)\cap\SG^\bone(V_2)\cap\cW\left(s_mu_i+T\left(\ur,d_m-s_m,i\right)\right)\cap\SG^{\bone}\left(\L^m\right)\right),\\
\ge{}& \m\left(\SG^\bone\left(s_mu_i+\L^m\right)\cap\cW\left(s_mu_i+T\left(\ur,d_m-s_m,i\right)\right)\cap\SG^{\bone}\left(\L^m\right)\right)\\
\ge{}& q^{O(W)}\m\left(\SG^\bone\left(s_mu_i+\L^m\right)\cap\SG^{\bone}\left(\L^m\right)\right)\end{aligned}\end{multline*}
using that $\SG$ is decreasing in the boundary condition, then \cref{lem:tube:decomposition,lem:traversability:boundary,def:extension:East} and finally \cref{lem:traversability:boundary} and the Harris inequality \cref{eq:Harris:1}. Moreover, by \cref{def:extension:East,lem:traversability:boundary},
\[\m\left(\SG^{\o'}(\L^m)\right)\le q^{-O(W)}\m\left(\SG^\bone(\L^m)\right),\]
so that we recover
\[\m\left(\cT^\bone(V_3)\right)\m\left(\left.\SG^\bone(V_2)\right|\SG^{\o'}\left(\L^m\right)\right)\ge q^{O(W)}/a_m\]
completing the proof of \cref{eq:East:extension:step} and \cref{prop:east:reduction}.
\end{proof}

\begin{proof}[Proof of \cref{cor:East:reduction}]
The fact that $\SG(\L(\ur+l\uv_i))$ is nonempty, translation invariant and decreasing follows directly from \cref{def:extension:East}. By \cref{prop:east:reduction} it suffices to relate $\max_{\o}\g^\o_I(\L^1)$ and $\g(\L(\ur))$, using \cref{cor:single:line}. Notice that by \cref{def:extension:East} we have \begin{equation}
\label{eq:SG:decomposition:East}\SG^\o\left(\L^1\right)=\SG^\bone(\L(\ur))\times\cT^{\o}(T(\ur,\l_i,i)).
\end{equation}
Therefore, (see e.g.\ \cite{Hartarsky21a}*{Lemma 3.9} or \cref{eq:var:decomposition:AB})
\begin{multline}
\var_{\L^1}\left(f|\SG^\o\left(\L^1\right)\right)\le \m_{\L(\ur)}\left(\left.\var_{T(\ur,\l_i,i)}\left(f|\cT^\o\left(T(\ur,\l_i,i)\right)\right)\right|\SG^\bone\left(\L\left(\ur\right)\right)\right)\\
+\m_{T(\ur,\l_i,i)}\left(\left.\var_{\L(\ur)}(f|\SG^\bone\left(\L(\ur)\right)\right|\cT^\o\left(T(\ur,\l_i,i)\right)\right).
\label{eq:varL1}
\end{multline}
The former term is treated by \cref{cor:single:line}, which gives 
\begin{multline*}
\m_{\L^+}\left(\left.\var_{T(\ur,\l_i,i)}\left(f|\cT^\o\left(T(\ur,\l_i,i)\right)\right)\right|\SG^\bone\left(\L\left(\ur\right)\right)\right)\\\le e^{O(\log^2(1/q))}\max\left(\g(\L(\ur)),\m^{-1}\left(\SG^\bone(\L(\ur))\right)\right)\sum_{x\in\L^+}\m_{\L^+}\left(c^{\L^+,\o}_x\var_x(f)\right),\end{multline*}
where $\L^+=(\L(\ur+\l_i\uv_i)_I^\o)$. For the second term in \cref{eq:varL1}, \cref{eq:def:gamma} and $\m_{T(\ur,\l_i,i)}(\cT^\o(T(\ur,\l_i,i)))\ge q^{O(W)}$ (see the proof of \cref{lem:traversability:boundary}) give
\begin{multline*}\m_{T(\ur,\l_i,i)}\left(\left.\var_{\L(\ur)}(f|\SG^\bone\left(\L(\ur)\right)\right|\cT^\o\left(T(\ur,\l_i,i)\right)\right)\\
\le q^{-O(W)}\g(\L(\ur))\sum_{x\in\L(\ur)}\m_{\L(\ur+\l_i\uv_i)}\left(c_x^{\L(\ur),\bone}\var_x(f)\right).\end{multline*}
Plugging these into \cref{eq:varL1} and recalling \cref{eq:def:gamma:boundary}, we get 
\[\g^\o\left(\L^1\right)\le e^{O(\log^2(1/q))}\max\left(\g(\L(\ur)),\m^{-1}\left(\SG^\bone(\L(\ur))\right)\right),\]
which concludes the proof of \cref{cor:East:reduction} together with \cref{prop:east:reduction}, since $M=O(\log(\lmp))\le O(C)\log(1/q)$.
\end{proof}
We next turn to CBSEP-extensions.
\begin{lem}
\label{prop:CBSEP:reduction}
Assume that $\cU$ has a finite number of stable directions. Set $J=[4k]\setminus\{i+k,i-k\}$. Let $\SG^\bone(\L(\ur))$ be a nonempty translation invariant decreasing event. Assume we CBSEP-extend $\L(\ur)$ by $l$ in direction $u_i$. Then
\[\g\left(\L(\ur+l\uv_i)\right)\le \max_{\o} \g_J^{\o}\left(\L(\ur+\l_i\uv_i)\right)\frac{\m(\SG^\bone(\L(\ur+\l_i\uv_i)))}{\m(\SG^\bone(\L(\ur+l\uv_i)))}e^{O(C^2)\log^2(1/q)}.\]
\end{lem}
\begin{proof}
As in \cite{Hartarsky23FA}*{Eq.\ (4.10)} (with minor amendments as in \cref{prop:east:reduction}), we have
\begin{equation}
\label{eq:CBSEP:reduction:deconditioning}
\g\left(\L^M\right)\le \max_{\o}\g_J^{\o}\left(\L^1\right)\frac{\m(\SG^\bone(\L^1))}{\m(\SG^\bone(\L^M))q^{O(MW)}}\prod_{m=1}^{M-1}b_m\end{equation}
with 
\[
b_m=\max_\o\m^{-2}_{\L^{m+1}}\left(\left.\SG^\bone_{s_m}\right|\SG^\o\right)\max_\o\m^{-1}_{\L^{m+1}}\left(\left.\SG_0^\o\right|\SG^\o\right).\]
Let us mention that the only difference of \cref{eq:CBSEP:reduction:deconditioning} w.r.t.\ \cite{Hartarsky23FA} is the fraction in the r.h.s. It comes from the absence of the conditioning in the r.h.s.\ of \cref{eq:def:gamma} as compared to \cite{Hartarsky23FA}*{Eq.\ (4.5)} pointed out in \cref{rem:deconditioning}. This leads to \cite{Hartarsky23FA}*{Eq.\ (4.16)} being slightly simpler in our setting. Namely, there one should use the fact that for any finite $A\subset B\subset \bbZ^2$, $\cA\subset\O_A$, $\cB\subset\O_B$ and $f:\O_B\to[0,\infty)$ we have
\begin{equation}
\label{eq:deconditioning}\m_B(\1_\cA\m_A(f)|\cB)=\frac{\m_B(\1_\cA\1_\cB\m_A(f))}{\m(\cB)}\le \frac{\m_B(\1_{\cA}\m_A(f))}{\m(\cB)}=\frac{\m(\cA)}{\m(\cB)}\m_B(f).\end{equation}
Using this yields 
\[\prod_{m=1}^{M-1}\frac{\SG^\bone(\L^m)}{\SG^\bzero(\L^{m+1})}\le \prod_{m=1}^{M-1}\frac{\m(\SG^\bone(\L^m))}{\m(\SG^\bone(\L^{m+1}))q^{O(W)}}=\frac{\m(\SG^\bone(\L^1))}{\m(\SG^\bone(\L^M))q^{O(MW)}},\]
using \cref{lem:traversability:boundary}. Up to this modification the proof is the same as in \cite{Hartarsky23FA}, so we do not repeat it.

Given \cref{eq:CBSEP:reduction:deconditioning}, we are left with proving $b_m\le q^{-O(C)}$ for all $m$. The last statement is simply \cref{lem:T:ratio}---the analogue of \cite{Hartarsky23FA}*{Corollary A.3}, so we are done.
\end{proof}

\begin{proof}[Proof of \cref{cor:CBSEP:reduction}]
The fact that $\SG(\L(\ur+l\uv_i))$ is nonempty, translation invariant and decreasing follows directly from \cref{def:extension:East}. By \cref{prop:CBSEP:reduction} it suffices to relate $\g_J^\o(\L^1)$ and $\g\left(\L(\ur)\right)$. This is done exactly as in \cite{Hartarsky23FA}*{Lemma\ 4.10} (see particularly Eqs.\ (4.20) and (4.22) there), replacing \cite{Hartarsky23FA}*{Eq.\ (4.30)} by \cref{cor:single:line} and using \cref{eq:deconditioning,lem:traversability:boundary} as in the proof of \cref{cor:East:reduction}.

\end{proof}

\section{Conditional probabilities}
\label{app:proba}

The goal of this appendix is to prove \cref{lem:T:ratio,cor:perturbation}. Recall \cref{subsec:conditional} and \cref{def:extension:CBSEP}. \begin{proof}[Proof of \cref{lem:T:ratio}]
We prove that for all $s,s'\in[0,l]$  divisible by $\l_i$ and $\o,\o'\in\O_{\bbZ^2\setminus\L}$ we have
\begin{equation}
\label{eq:T:ratio}
\frac{\m(\SG^\o_s(\L))}{\m(\SG^{\o'}_{s'}(\L))}=q^{O(W)}.
\end{equation}
Once this is established, we note that by \cref{def:extension:CBSEP},
\begin{align*}
\max_{s'} \m\left(\SG^{\o'}_{s'}(\L)\right)&{}\le \m\left(\SG^{\o'}(\L)\right)=\m\left(\bigcup_{s'}\SG^{\o'}_{s'}(\L)\right)\\
&{}\le O(l)\max_{s'}\m\left(\SG^{\o'}_{s'}(\L)\right).\end{align*}
Further recalling from \cref{subsec:scales}, that $l\le\lmp\le q^{-O(C)}$ and $W\ll C$, we get
\[\m\left(\left.\SG^\o_s(\L)\right|\SG^{\o'}(\L)\right)\ge\frac{\m(\SG^\bone_s(\L))}{\m(\SG^{\o'}(\L))}\ge q^{-O(C)},\]
since $\SG^\bone_s(\L)\subset\SG^\bone(\L)\subset \SG^{\o'}(\L)$. Thus, it remains to prove \cref{eq:T:ratio}. Moreover, it clearly suffices to establish \cref{eq:T:ratio} for $s'=0$ and $\o'=\bone$.

To prove \cref{eq:T:ratio} in that case, let us first observe that by translation invariance, \cref{def:extension:CBSEP,eq:symmetry:tubes},\begin{equation}
    \label{eq:T:ratio:1}
    \frac{\m(\SG^\o_s(\L))}{\m(\SG^{\bone}_{0}(\L))}=\frac{\m(\ST^{\o_s}(T_s))\m(\ST^{\o_{l-s}}(T_{l-s}))}{\m(\ST^{\bone}(T_l))},
\end{equation}
where for $x\in\{s,l-s,l\}$, $T_x=T(\ur,x,i)$ and the $\o_x$ is a boundary condition that can be expressed in terms of $\o$ and $x$. Applying \cref{lem:traversability:boundary,lem:tube:decomposition} to \cref{eq:T:ratio:1}, we obtain \cref{eq:T:ratio} as desired.
\end{proof}

Our next goal is to treat certain perturbations of traversability events. To do that we not only require the Harris inequality but also the van den Berg--Kesten \cite{BK85} one. We should note that the BK inequality is not natural to use for an upper bound in our setting and has not been employed to this purpose until now. Nevertheless, since we aim to bound conditional probabilities, it will prove useful.
\begin{defn}[Disjoint occurrence]
Given $\L\subset \bbZ^2$ and two decreasing events $\cA,\cB\subset\O_\L$, we say that $\cA$ and $\cB$ \emph{occur disjointly} in $\o\in\O_\L$ if there exist disjoint sets $X,Y\subset \L$, such that $\o_{X\cup Y}=\bzero$; $\o'_X=\bzero$ implies $\o'\in \cA$ for $\o'\in\O_\L$; and $\o'_Y=\bzero$ implies $\o'\in \cB$ for $\o'\in\O_\L$.
\end{defn}
\begin{prop}[BK inequality]
\label{prop:BK}
For any decreasing events $\cA,\cB\subset\O_\L$,
\[\m(\text{$\cA$ and $\cB$ occur disjointly})\le \m(\cA)\m(\cB).\]
\end{prop}

We may now start building conditional probability bounds up progressively for segments, parallelograms and, eventually, tubes. For segments, recall \cref{subsec:helping:sets}.
\begin{lem}[Perturbing a segment]
\label{lem:shrinking:ratio:line}
Fix $i\in[4k]$ such that $\a(u_i)\le\a$. Let $S$ be a discrete segment perpendicular to $u_i$ and $S',S''\subset S$ be discrete segments partitioning $S$. Assume that $|S|\ge \O(W)|S''|$ and $|S|=q^{-\a+o(1)}$. Then
\[\m\left(\cH(S')|\cH(S)\right)\ge 1-\frac{W^{1/3}|S''|}{|S|}-q^{1-o(1)}.\]
\end{lem}
\begin{proof}
Let us note that a stronger version of this result can be proved more easily by counting circular shifts of the configuration in a $O(1)$ neighbourhood of $S$ such that a given helping set remains at distance at least some constant from $S''$ (see the proof of \cite{Duminil-Copin23}*{Proposition 3.2(3)} for a subsequent implementation of this technique). We prefer to give the proof below as a preparation for \cref{lem:shrinking:ratio}.

By translation invariance, we may assume that $S$ is of the form in \cref{eq:S:form}. In view of \cref{def:ahelping}, we need to distinguish cases, depending on whether $\alpha(u_{i+2k})>\alpha$. We first assume that $\a(u_{i+2k})>\a$. Thus, helping sets are just $u_i$-helping sets or $W$-helping sets. By \cref{def:uihelping}, if $\alpha(u_i)=0$, there is nothing to prove, since $u_i$-helping sets are empty, so $\cH(S')$ always occurs. We therefore assume that $\alpha(u_i)>0$. We further assume $S''\neq\varnothing$, since otherwise the statement is trivial.

Recall from \cref{def:uihelping} that a $u_i$-helping set is composed of $Q$ translates of the set $Z_i$. For $r\in[Q]$ we denote by $\cH_{(r)}(S)$ the event that there is an infected translate of $Z_i$ by a vector of the form $(r+k_rQ)\l_{i+k}u_{i+k}$ with $k_r\in\bbZ$ satisfying \cref{eq:helping:set:domain} (for $d=0$). Similarly define $\cH_{(r)}(S')$. In words, we look for the part of the helping set with a specified reminder $r$ modulo $Q$. In particular, by \cref{def:HdS,def:HW}, we have
\begin{equation}
\label{eq:HS:decomposition}\cH(S)=\cH^W(S)\cup\bigcap_{r\in[Q]}\cH_{(r)}(S)\end{equation}
and similarly for $S'$.

Since $|S|=q^{-\a+o(1)}$, the probability that there are $\a+1$ infected sites at distance $O(1)$ from each other and from $S$ is $q^{1-o(1)}$. Furthermore, if this does not happen, but $\cH(S)$ occurs, then all $\cH_{(r)}(S)$ for $r\in[Q]$ occur disjointly. Therefore, by the BK inequality \cref{prop:BK},
\begin{equation}
\label{eq:muHS}
\m(\cH(S))\le q^{1-o(1)}+\prod_{r\in[Q]}\m\left(\cH_{(r)}(S)\right)\le \left(1+q^{1-o(1)}\right)\prod_{r\in[Q]}\m\left(\cH_{(r)}(S)\right),\end{equation}
since, as in \cref{obs:mu:helping}, we have
\begin{equation}
\label{eq:muHrS}
\m(\cH_{(r)}(S))\ge 1-\left(1-q^{\a}\right)^{\O(|S|)}\ge q^{o(1)}.
\end{equation}
Using \cref{eq:muHS,eq:HS:decomposition} and applying the Harris inequality \cref{eq:Harris:1}, we get
\[\frac{\m(\cH(S'))}{\m(\cH(S))}\ge \frac{\m(\bigcap_{r\in[Q]}\cH_{(r)}(S'))}{(1+q^{1-o(1)})\prod_{r\in[Q]}\m(\cH_{(r)}(S))}\ge\left(1-q^{1-o(1)}\right)\prod_{r\in[Q]}\frac{\m(\cH_{(r)}(S'))}{\m(\cH_{(r)}(S))}.\]

For $r\in[Q]$ and $j\in\bbZ$, let us denote by $I_{(r)}^{j}$ the indicator of the event that $Z_i+(r+jQ)\l_{i+k}u_{i+k}$ is fully infected and denote by $J_{(r)}(S)$ the set of values of $j$ such that this set satisfies \cref{eq:helping:set:domain}. Since $Z_i$ has diameter (much) smaller than $Q$, for all $r\in[Q]$, the random variables $I^j_{(r)}$ are i.i.d.\ for $j\in\bbZ$ (and therefore exchangeable). Further noting that $J_{(r)}(S)\supset J_{(r)}(S')$, and setting $\Sigma=\sum_{j\in J_{(r)}(S)}I^j_{(r)}$, we obtain
\begin{align*}
\frac{\m(\cH_{(r)}(S'))}{\m(\cH_{(r)}(S))}&{}=\m\left(\left.\sum_{j\in J_{(r)}(S')}I_{(r)}^j\ge1\right|\Sigma\ge 1\right)\\
&{}=\sum_{s=1}^{|J_{(r)}(S)|}\frac{\m(\Sigma=s)}{\m(\Sigma\ge 1)}\left(1-\prod_{l=0}^{s-1}\frac{|J_{(r)}(S)\setminus J_{(r)}(S')|-l}{|J_{(r)}(S)|-l}\right)\\
&{}\ge \frac{|J_{(r)}(S')|}{|J_{(r)}(S)|}\sum_{s=1}^{|J_{(r)}(S)|}\frac{\m(\Sigma=s)}{\m(\Sigma\ge 1)}=\frac{|J_{(r)}(S')|}{|J_{(r)}(S)|}\ge \frac{|S'|-O(1)}{|S|}.\end{align*}
Recalling that $|S|\ge \O(W)|S''|$ and $W\gg Q=O(1)$, this entails that
\begin{align*}
\m(\cH(S')|\cH(S))&{}=\frac{\cH(S')}{\m(\cH(S))}\ge \left(1-q^{1-o(1)}\right)\left(1-\frac{|S''|+O(1)}{|S|}\right)^Q
\\&{}\ge 1-q^{1-o(1)}-\frac{O(Q)|S''|}{|S|},\end{align*}
concluding the proof for the case $\alpha(u_{i+2k})>\alpha$.

Turning to the case, $\alpha(u_{i+2k})\le\alpha$, there is little to change. Firstly, if $\alpha(u_{i+2k})=0$, the proof above applies, since $\alpha$-helping sets in direction $u_i$ are the same (since $u_{i+2k}$-helping sets are empty). Moreover, if $-Z_{i+2k}=Z_i+x\l_{i+k}u_{i+k}$ for some $x\in\bbZ$, there is nothing more to prove either, since $\alpha$-helping sets and $u_i$-helping sets coincide again. We may therefore assume this is not the case. If $\alpha(u_i)=0$, then the proof proceeds as above, but with $Z_i$ replaced by $-Z_{i+2k}$. Finally, if $1\le\alpha(u_i),\alpha(u_{i+2k})\le \alpha$, then the proof proceeds as above, but one needs to consider not only $\cH_{(r)}(S)$, but also their analogues with $Z_i$ replaced by $-Z_{i+2k}$.
\end{proof}

In the next lemma, we next focus on a parallelogram, which plays the role of one of the hatched ones in \cref{fig:perturbation}. Informally, the statement is as follows. The $u_i$-side of the parallelogram is of critical size, so that each segment $S_{i,m}$, into which it is decomposed in \cref{def:traversability}, is also of critical size, allowing us to apply \cref{lem:shrinking:ratio:line} to it. The other dimension of the parallelogram is left unconstrained. The lemma provides a bound on the probability that a parallelogram of slightly smaller $u_i$-side is traversable (has helping sets for each segment $S_{i,m}$, given that the original one is.
\begin{lem}[Perturbing a parallelogram]
\label{lem:shrinking:ratio}
Let $i,j\in[4k]$ be such that $j\not\in\{i,i+2k\}$ and $\a(u_i)\le \a$. Consider the parallelogram 
\[R=R(l,h)=\Hb_{u_i}(l)\cap\Hb_{u_j}(h)\cap\Hb_{u_{j+2k}}(0)\cap\Hb_{u_{i+2k}}(0)\]
for $l\in[\r_i,e^{q^{-o(1)}}]$ and $h=q^{-\a+o(1)}$. We say that $R$ is \emph{traversable} in direction $u_i$ ($\cT(R)$ occurs), if for each nonempty segment of the form
\[S=\bbZ^2\cap R\cap \Hb_{u_i}(h')\setminus\bbH_{u_i}(h')\]
the event $\cH_{C^2}^{\bone_{\bbZ^2\setminus R(l+W,h)}}(S)$ occurs. Let $R'=R(l,h')$ with $1\ge h'/h\ge 1-1/W$. Then
\[\m\left(\cT(R')|\cT(R)\right)\ge \left(1-\sqrt W\left(1-\frac{h'}{h}\right)-q^{1-o(1)}\right)^{O(l)}.\]
\end{lem}
\begin{proof}
We start by noting that if $\alpha(u_i)=0$, there is nothing to prove, since $\cT(R')$ always occurs, so we assume $\alpha(u_i)>0$. Furthermore, we may assume that $h-h'>\O(1)$, since otherwise either $R\cap\bbZ^2=R'\cap \bbZ^2$ or $R'\cap\bbZ^2=R''\cap\bbZ^2$ for some $R''=R(l,h-\O(1))$. Let $M=1+\lfloor l/\r_i\rfloor$, so that $R$ consists of $M$ segments perpendicular to $u_i$. Let us emphasise that the boundary condition is irrelevant for $\cT(R)$, as it is imposed far from the boundary of the domain concerned. Therefore, this event may also depend on the configuration outside $R$. 

We partition $R$ into its first and second halves $R_1=R(\rho_i\lfloor l/(2\rho_i)\rfloor,h)$ and $R_2=R\setminus R_1$. Thus, $R_1$ and $R_2$ consist of $\lceil M/2\rceil$ and $\lfloor M/2\rfloor$ segments perpendicular to $u_i$ respectively. Recalling \cref{def:uihelping}, we see that if $\cT(R)$ occurs, then one of the following must occur.
\begin{itemize}
    \item The parallelograms $R_1$ and $R_2$ are disjointly traversable.
    \item There is a set of $\a+1$ infections at distance $O(1)$ from each other and from both $R_1$ and $R_2$. Furthermore, the parallelograms $R_1'$ and $R_2'$, formed by removing in each of $R_1$ and $R_2$ the $Q$ lines closest to their common boundary, are both traversable.
\end{itemize}
Using the BK inequality \cref{prop:BK}, this gives
\begin{align}\m(\cT(R))\le{}&\m\left(\cT(R_1))\m(\cT(R_2)\right) + q^{1-o(1)}\m\left(\cT(R_1')\right)\m\left(\cT(R_2')\right)\nonumber\\
={}&\m(\cT(R_1))\m(\cT(R_2))\left(1+q^{1-o(1)}\right).\label{eq:perturbed:parallelogram:iteration}
\end{align}
The last estimate follows as in \cref{eq:muHrS} from the fact that traversing the $O(1)$ lines at the boundary of $R_1$ and $R_2$ happens with probability $q^{o(1)}$ together with the Harris inequality \cref{eq:Harris:1}. 

Let us write simply $\cH_m$ for $\cH_{C^2}^{\bone_{\bbZ^2\setminus R(l+W,h)}}(R\cap\Hb(m\r_i)\setminus\bbH_{u_i}(m\r_i))$ and similarly define $\cH_m'$ for $R'$. Iterating \cref{eq:perturbed:parallelogram:iteration}, we obtain
\[\cT(R)\le \left(1+q^{1-o(1)}\right)\prod_{m\in[M]}\m(\cH_m),\]
since $l=e^{q^{-o(1)}}$. Hence, by the Harris inequality \cref{eq:Harris:1}
\[\frac{\m(\cT(R'))}{\m(\cT(R))}\ge \left(1-q^{1-o(1)}\right)\prod_{m\in[M]}\frac{\m(\cH_m')}{\m(\cH_m)}.\]
The last fraction can be bounded, using \cref{lem:shrinking:ratio:line}, to obtain
\[\m\left(\cT(R')|\cT(R)\right)=\frac{\m(\cT(R'))}{\m(\cT(R))}\ge \left(1-O\left(W^{1/3}\right)\left(1-\frac{h'}{h}\right)-q^{1-o(1)}\right)^{M}.\qedhere\]
\end{proof}

Turning to the proof of \cref{cor:perturbation}, recall \cref{fig:perturbation}. There the regions introduced in the proof below are depicted as follows. The parallelograms $R_j$ are North-West hatched, while $R'_j$ are North-East hatched. Thus, $R''_j$ are double hatched. The shaded parallelograms are $R_j^2$, while $R_j^1$ are the remainder of the area which is North-East but not double hatched.
\begin{proof}[Proof of \cref{cor:perturbation}] Recalling \cref{def:traversability}, it is clear that $\cT^{\o}_d(T)$ is the intersection of $2k-1$ independent traversability events for parallelograms of length $l$ in the sense of \cref{lem:shrinking:ratio}. Let us denote them by $(R_j)_{j=i-k+1}^{i+k-1}$ and, similarly, $(R_j')_{j=i-k+1}^{i+k-1}$ for $T'$ with $R_j$ and $R_j'$ having sides perpendicular to $u_j$ (see \cref{fig:perturbation}). Thus, fixing $j\in(i-k,i+k)$, the role of $i$ and $j$ in \cref{lem:shrinking:ratio} is played by $j$ and $i+k$ here. 

Further set $R''_j=R_j\cap R_j'=R(l-O(\D),s_j-O(\D+C^2))$.\footnote{This equality only holds up to translation, but for simplicity we leave out the translation vector. Note that, although we stated \cref{lem:shrinking:ratio} for parallelograms with an integer point at one of their corners, analogous bounds hold for non-integer translates thereof.} Notice that (see \cref{fig:perturbation}) $R'_j\setminus R_j''$ consists of two disjoint possibly empty parallelograms $R^1_j=R(O(\D),s_j-O(\D+C^2))$ and $R^2_j=R(l-O(\D),O(\D))$ with the notation of \cref{lem:shrinking:ratio} (up to translation). Note that here we used that $R'_j$ has smaller length and height than $R_j$, because $s_j\ge s_j'$, $l\ge l'$ and $d'\ge d$.

By \cref{lem:traversability:boundary,eq:Harris:1} we have
\begin{align*}\m\left(\left.\cT_{d'}^{\o'}(T')\right|\cT_d^\o(T)\right)&{}\ge \frac{\m(\cW(T)\cap\cT_d(T)\cap\cW(T')\cap\cT_{d'}(T'))}{q^{-O(W)}\m(\cT_d(T))}\\
&{}\ge q^{O(W)}\m\left(\left.\cT_{d'}(T')\right|\cT_{d}(T)\right),\end{align*}
where $\cT_d(T)$ denotes the event that $T$ is $(\cdot,d)$-traversable without boundary condition (also depending on the states of sites outside $T$) and similarly for $T'$. Moreover, 
\begin{align*}
\cT_{d'}(T')&{}=\bigcap_j \cT(R'_j)\supset \bigcap_j (\cT(R''_j)\cap\cT(R^1_j))&\cT_d(T)&{}=\bigcap_j\cT(R_j),
\end{align*}
so the Harris inequality \cref{eq:Harris:3} gives
\[\m\left(\left.\cT_{d'}(T')\right|\cT_{d}(T)\right)\ge \prod_j\m\left(\cT\left(R^1_j\right)\right)\m\left(\left.\cT\left(R''_j\right)\right|\cT(R_j)\right).\]
We may then conclude, using \cref{lem:shrinking:ratio} and that by \cref{obs:mu:helping}
\[\m\left(\cT\left(R_j^1\right)\right)\ge \left(1-(1-q^\a)^{\O(s_j)}\right)^{O(\D)}.\qedhere\]
\end{proof}

\section{The surplus factor for balanced rooted models with finite number of stable directions.}
\label{app:logloglog}
To conclude, let us briefly sketch how to remove the $\log\log\log(1/q)$ factor appearing in \cref{th:internal:East}, which would also propagate to pollute \cref{th:main}\ref{log1} (see \cref{eq:poluted}).
\begin{thm}
\label{th:internal:East:improved}
Let $\cU$ be balanced rooted (classes \ref{log0} and \ref{log1}). Let $\L^{(\Ni)}$ be as in \cref{subsec:log1:internal}. Instead of \cref{def:SG:balanced:internal}, one can define $\SG^\bone(\L^{(\Ni)})$ in such a way that
\begin{align*}
\g\left(\L^{(\Ni)}\right)\le{}&\exp\left(\frac{\log^{O(1)}\log(1/q)}{q^\a}\right),&
\m(\SG^\bone(\L^{(N^i)}))\ge{}& \exp\left(\frac{-1}{\e^2q^\a}\right).
\end{align*}
\end{thm}
\begin{proof}[Sketch proof of \cref{th:internal:East:improved}]
To prove this, one should combine the techniques of \cref{subsec:loglog:internal,subsec:log1:internal}. More precisely, a bound on $a_m^{(n)}$ less crude than \cref{eq:mu:SG:bound} should be established along the lines of \cref{eq:anm:bound:loglog}. As in \cref{eq:anm:decomposition:loglog}, we may further decompose $a_m^{(n)}$ into a product over scales $p\le n$. 

The relevant values of the parameters correspond to $m$ such that $(3/2)^m\le 1/(\log^{W}(1/q)q^\a)$, say, and $p\in[\Nc,n]$, as other cases can be dealt with using the crude bound \cref{eq:mu:SG:bound}. Further, as in \cref{eq:conditional:T:trivial}, we can also discard $p\ge \Nc+\Psi$. Hence, we need to focus for the remaining values of $m$ and $p$ on lower bounding
\begin{equation}
\label{eq:app:logloglog:cond:prob}\m\left(\left.\cT_{p}\left(\left(\L^{(p+1)}\setminus D_1\right) +s_m\right)\right|\cT_{p}\left(\L^{(p+1)}\setminus D_1\right)\right)\end{equation}
and $\m(\cT^\bone((D_1\setminus D_0)+s_m)|\cT^\bone(D_1\setminus D_0))$, the latter being treated exactly like $\m(\cT'|\cT)$ in \cref{eq:conditional:T}. \Cref{eq:app:logloglog:cond:prob} can be further decomposed as a product over elementary regions delimited by the boundaries of the $(D_\k)_{\k\in[K]}$ (recall \cref{fig:Dk,rem:product,eq:Harris:3}).

Unfortunately, for such (non-convex) polygonal regions $R$, bounding
\[\m\left(\left.\cT_p\left(R+s_m\right)\right|\cT_p(R)\right)\]
is no easy feat. Indeed, \cref{cor:perturbation} only treats tubes and, more importantly deals, with helping sets for one direction only in each part of the tube (recall \cref{subfig:extension:East}), while $\cT_p(R)$ requires helping sets in various directions, which are all dependent. To make matters worse, for certain families $\cU$ it may happen that a single set of $\a$ infections is simultaneously a helping set for different directions and this would create complex and heavy dependency among different directions, which could, \emph{a priori}, make boundary regions attract such sets.

To deal with this issue, one could further elaborate \cref{def:n-traversable}. Indeed, we may split $\L^{(p+1)}\setminus D_1$ into disjoint horizontal strips (recall \cref{subfig:East:internal:all}) of width $ \ell^{(p)}/(W\e)$. Each strip is assigned a direction $u_j$, $j\in(-k,k)$ and we only ask for helping sets for this direction to be present. These requirements are again cut at a small distance from the boundaries of all $D_\k$ into parallelograms like the ones treated in \cref{lem:shrinking:ratio}. We further demand $W$-helping sets on segments close to the boundaries of the various $D_\k$ as in \cref{def:n-traversable}. Naturally, some leftover regions remain without helping sets as in \cref{def:extension:East:multid}, but they are unimportant like in \cref{subsec:log1:internal}.

By doing this, we make the event $\cT_n(R)$ the intersection of traversability events of parallelograms in the sense of \cref{lem:shrinking:ratio}, so that its result can be applied as in the proof of \cref{cor:perturbation}, leading to a calculation similar to the one in \cref{th:internal:loglog}. The only significant change is that now there are $O(W\ell^{(p+1)}/\ell^{(p)})$ parallelograms instead of a constant number. This is not really a problem. However, if we wish to avoid careful computations, given that we are interested in the range $p\in(\Nc,\Nc+\Psi)$, we can brutally bound $W\ell^{(p+1)}/\ell^{(p)}$ by its maximum, which is $\log^{O(1)}\log(1/q)$ by the definition of $\Psi$, \cref{eq:def:Delta}.
\end{proof}

\let\d\oldd
\let\k\oldk
\let\l\oldl
\let\L\oldL
\let\o\oldo
\let\O\oldO
\let\r\oldr
\let\t\oldt
\let\u\oldu

\bibliographystyle{plain}
\bibliography{Bib}
\end{document}